\documentclass[3p]{elsarticle}
\journal{Computer Methods in Applied Mechanics and Engineering}

\usepackage{graphicx}
\usepackage{mathtools}
\usepackage{amssymb}
\usepackage{amsthm}
\usepackage{latexsym}
\usepackage{bm}
\usepackage{color}
\usepackage{stmaryrd}
\usepackage{mathrsfs}
\usepackage{array}
\usepackage{algorithm}
\usepackage{algpseudocode}
\usepackage{multirow}
\usepackage{url}
\usepackage{mdframed}

\makeatletter
\newcommand*{\rom}[1]{\expandafter\@slowromancap\romannumeral #1@}
\makeatother

\newenvironment{myenv}[1]
  {\mdfsetup{
    frametitle={\colorbox{white}{\space#1\space}},
    innertopmargin=10pt,
    frametitleaboveskip=-\ht\strutbox,
    frametitlealignment=\center
    }
  \begin{mdframed}
  }
  {\end{mdframed}}

\DeclareMathAlphabet\mathbfcal{OMS}{cmsy}{b}{n}

\SetSymbolFont{stmry}{bold}{U}{stmry}{m}{n}

\newtheorem{definition}{Definition}

\newtheorem{theorem}{Theorem}
\newtheorem{remark}{Remark}

\AfterEndEnvironment{proof}{\noindent\ignorespaces}

\numberwithin{equation}{section}

\begin{document}
\begin{frontmatter}
\title{A continuum and computational framework for viscoelastodynamics: finite deformation linear models}

\author[label1,label2,label3]{Ju Liu}
\ead{liuj36@sustech.edu.cn,liujuy@gmail.com}

\author[label4]{Marcos Latorre}
\ead{marcos.latorre@yale.edu}

\author[label3,label5]{Alison L. Marsden}
\ead{amarsden@stanford.edu}

\address[label1]{Department of Mechanics and Aerospace Engineering, Southern University of Science and Technology, Shenzhen, Guangdong 518055, P.R.China}

\address[label2]{Guangdong-Hong Kong-Macao Joint Laboratory for Data-Driven Fluid Mechanics and Engineering Applications, Southern University of Science and Technology, Shenzhen, Guangdong, 518055, P.R.China}

\address[label3]{Department of Pediatrics (Cardiology) and Institute for Computational and Mathematical Engineering, Stanford University, Clark Center E1.3, 318 Campus Drive, Stanford, CA 94305, USA}

\address[label4]{Department of Biomedical Engineering, Yale University, New Haven, CT 06520, USA}

\address[label5]{Department of Bioengineering, Stanford University, Clark Center E1.3, 318 Campus Drive, Stanford, CA 94305, USA}

\begin{abstract}
This work concerns the continuum basis and numerical formulation for deformable materials with viscous dissipative mechanisms. We derive a viscohyperelastic modeling framework based on fundamental thermomechanical principles. Since most large deformation problems exhibit isochoric properties, our modeling work is constructed based on the Gibbs free energy in order to develop a continuum theory using pressure-primitive variables, which is known to be well-behaved in the incompressible limit. A set of general evolution equations for the internal state variables is derived. With that, we focus on a family of free energies that leads to the so-called finite deformation linear model. Our derivation elucidates the origin of the evolution equations of that model, which was originally proposed heuristically and thus lacked formal compatibility with the underlying thermodynamics. In our derivation, the thermodynamic inconsistency is clarified and rectified. A classical model based on the identical polymer chain assumption is revisited and is found to have non-vanishing viscous stresses in the equilibrium limit, which is counter-intuitive in the physical sense. Because of that, we then discuss the relaxation property of the non-equilibrium stress in the thermodynamic equilibrium limit and its implication on the form of free energy. A modified version of the identical polymer chain model is then proposed, with a special case being the model proposed by G. Holzapfel and J. Simo.  Based on the consistent modeling framework, a provably energy stable numerical scheme is constructed for incompressible viscohyperelasticity using inf-sup stable elements. In particular, we adopt a suite of smooth generalization of the Taylor-Hood element based on Non-Uniform Rational B-Splines (NURBS) for spatial discretization. The temporal discretization is performed via the generalized-$\alpha$ scheme. We present a suite of numerical results to corroborate the proposed numerical properties, including the nonlinear stability, robustness under large deformation, and the stress accuracy resolved by the higher-order elements. Additionally, the pathological behavior of the original identical polymer chain model is numerically identified with an unbounded energy decaying. This again underlines the importance of demanding vanishing non-equilibrium stress in the equilibrium limit.
\end{abstract}

\end{frontmatter}

\section{Introduction}
Many polymeric and biological materials may undergo large deformations, with mechanical behavior characterized by hyperelasticity with intrinsic viscous dissipative mechanisms \cite{Ferry1980,Humphrey2013,Shaw2018,Benitez2017,Holzapfel2002}. To incorporate the viscous dissipative mechanism into the solid model, there have been several different modeling approaches developed to date. From a big picture perspective, prior modeling work can be categorized into at least two different groups: one based on the concept of internal state variables \cite{Coleman1967,Horstemeyer2010,Maugin2015} and one based on the hereditary integral \cite{Coleman1961,Gurtin1962,Christensen1980,Drapaca2007}. The former approach invokes a set of internal state variables, which is associated with the physical process occurring at the microscopic level and is manifested at the macroscopic scale. The evolution of the internal state variables is constrained by the second law of thermodynamics and generally results in dissipative behavior. This concept has been successfully applied to modeling inelasticity and phase transitions in general \cite{Horstemeyer2010,Maugin2015}. In the second approach, the model works directly on the stress-strain relationship. The viscous effect is modeled via convolution of a relaxation function with the strain history to describe the fading memory effect on the stress \cite{Coleman1964}, and the convolution is also known as the hereditary integral. This approach includes the notable quasi-linear viscoelasticity theory that has been implemented by the finite element method \cite{Puso1998} and applied to investigating soft tissues \cite[Chapter~7]{Fung2013}. It has also been generalized by utilizing fractional-order derivatives to capture the continuous relaxation spectrum, giving rise to a propitious alternative candidate for viscoelasticity modeling \cite{Yu2016,Perdikaris2014,Craiem2008}. The constrained mixture theory, which has been developed and applied for vascular growth and remodeling, can be categorized into this general modeling framework, in which the fading memory effect due to the continual turnover is fitted into the hereditary integral \cite{Humphrey2002,Valentin2013}. We also mention that some formulations based on internal state variables can be equivalently written into the hereditary integral formulation. The model discussed in this article, as will be shown, can be represented in terms of the hereditary integral with exponential relaxation kernels \cite{Park2001,Zeng2017}.

Focusing on the internal state variable approach, different underlying modeling assumptions have given rise to different models and computational procedures for viscoelasticity at finite strains. In recent works, it has been common to employ a \textit{multiplicative} decomposition of the deformation gradient to characterize elastic and viscous deformations \cite{Sidoroff1974}. With that, viscoelastic materials at finite strains may be established \cite{LeTallec1994,Reese1998,Reese1998b,Peric2003,Nedjar2007,Liu2019b,Hong2008}, and when combined with Ogden's hyperelastic function \cite{Ogden1972}, are able to describe responses that may adequately deviate from the thermodynamic equilibrium. Among those works, the model constructed by Reese and Govindjee \cite{Reese1998,Reese1998b} is rather representative. In their works, the non-equilibrium stresses are consistently derived from a Helmholtz free energy, that can be additively split into equilibrium and non-equilibrium parts \cite{Lubliner1985}. The derivation is similar to that of the finite strain elastoplasticity theory \cite{Simo1992b}, and material isotropy was assumed in their derivation. Recently, with the isotropy assumption released, the modeling approach of \cite{Reese1998} was further extended to account for material anisotropy \cite{Latorre2015,Latorre2016}.

Alternatively, a model proposed by Simo in \cite{Simo1987} gained popularity over the years and inspires numerous subsequent modeling works \cite{Holzapfel1996a,Holzapfel2001,Gasser2011,Govindjee1992,Gueltekin2016}. In his approach, the non-equilibrium stresses due to the viscous effect, which are also termed the ``over-stress", are governed by a set of linear evolution equations, with the elastic stress derived from strain energy. The evolution equations were proposed in a \textit{heuristic} manner as a straightforward generalization of the standard Zener solid model. Due to the linear nature, the modeling approach is often termed as \textit{finite deformation linear viscoelasticity}, or \textit{finite linear viscoelasticity} for short. The linear evolution equations also enable one to express the non-equilibrium stresses in terms of a simple convolution integral, which can then be computed via a one-step second-order accurate recurrence formula. The algorithm incrementally integrates the constitutive laws by only using the information from the past one step \cite{Simo1987,Simo2006,Taylor1970}. Thanks to this recurrence formula, the viscoelasticity model becomes amenable to finite element implementations \cite{Simo2006,Holzapfel2000,Kaliske1997}. Another appealing feature is that the strain energy function may conveniently account for material anisotropy \cite{Gasser2011,Holzapfel1996,Pena2008,Pena2011}, as the over-stresses are governed by a set of linear evolution equations.  On the other side, although large deformation is allowed for such models, the inherent linearity assumed for the evolution equation is regarded as a drawback of this approach, as it is regarded to be suited only for small deviations from thermodynamic equilibrium \cite[Section~6.10]{Holzapfel2000} (see also \cite[Chapter~10]{Haupt2013}). Yet, at least under physiological settings, we may reasonably expect the small deviation assumption to remain sound, and thus the model may still be well-suited to many biomechanical applications. A more critical issue is its lack of a rational thermodynamic foundation. Indeed, the linear evolution equations have not been explained by a rational thermomenchanical theory, although some related discussions were made in \cite{Holzapfel1996}. In recent work, a finite-time blow-up solution has been identified for this type of model \cite{Govindjee2014}, signifying its inconsistent nature. This alerting evidence hinders further adoption of this model for viscoelastic materials.


In this work, we consider a general continuum formulation for fully coupled thermomechanical models with viscous-type dissipation. The viscous deformation is characterized by a set of internal state variables. Unlike the prior approach \cite{Simo1987}, the non-equilibrium stresses are distinguished from the variables conjugating to the internal state variables. A set of fully nonlinear evolution equations for those conjugate variables are also obtained following the standard Coleman-Noll argument \cite{Coleman1963}. The origin and relations of the conjugate variables with the non-equilibrium stresses are elucidated through a careful derivation. In particular, it is shown that the non-equilibrium stresses and the conjugate variables can be identical only under a very special circumstance, an issue that has long been ignored in the literature \cite{Simo1987,Simo2006,Holzapfel2000}. Next, we consider a special form of the non-equilibrium part of the free energy, or the configurational free energy, which governs the viscous responses. The particular form of the energy can be viewed as a generalization of the form proposed in \cite{Holzapfel1996}. If one further demands the energy to be quadratic in terms of the internal state variables, a set of linear evolution equations for the conjugate variables is obtained, without invoking any linearization technique \cite{Reese1998}. To the best of our knowledge, this is the first time that the evolution equations' thermodynamic origin gets elucidated. We also mention that the right-hand side of the evolution equations is governed by the \textit{fictitious} second Piola-Kirchhoff stress, which is slightly different from the classical model \cite[Chapter~10]{Simo2006}. As the first instantiation of this framework, we considered the identical polymer chain model proposed in \cite{Holzapfel1996}. Interestingly, it can be easily seen that the non-equilibrium stresses of this model do not relax to zero in the thermodynamic equilibrium limit, which is physically counter-intuitive. It is indeed a requirement that the viscous stresses should vanish for static processes in general \cite[Chapter~10]{Haupt2013}. Without a prior multiplicative decomposition of the deformation, the non-equilibrium stresses are not guaranteed to fully relax in the limit. Therefore, to characterize the relaxation, we discuss the necessary and sufficient conditions from the perspective of the free energy design. Following that, we analyze first a simple model that has been considered in \cite[Section~4.2]{Holzapfel1996}. This model can be viewed as an extension of the St. Venant-Kirchhoff model to the viscous part and is shown to be the only model in which the non-equilibrium stresses equal the corresponding conjugate variables. Furthermore, the configurational free energy of this model suggests that there is an \textit{additive} split of the quadratic strain, making it parallel to the elastoplasticity theory developed by Green and Naghdi \cite{Green1965,Green1971}. The last example considered is a modified version of the identical polymer chain model, inspired from the aforesaid one \cite{Govindjee1992,Holzapfel1996} and is guaranteed to have the non-equilibrium stresses fully relaxed in the thermodynamic equilibrium limit.

\begin{figure}[!htb]
\begin{center}
	\begin{tabular}{c}
		\includegraphics[angle=0, trim=0 0 0 0, clip=true, scale = 0.32]{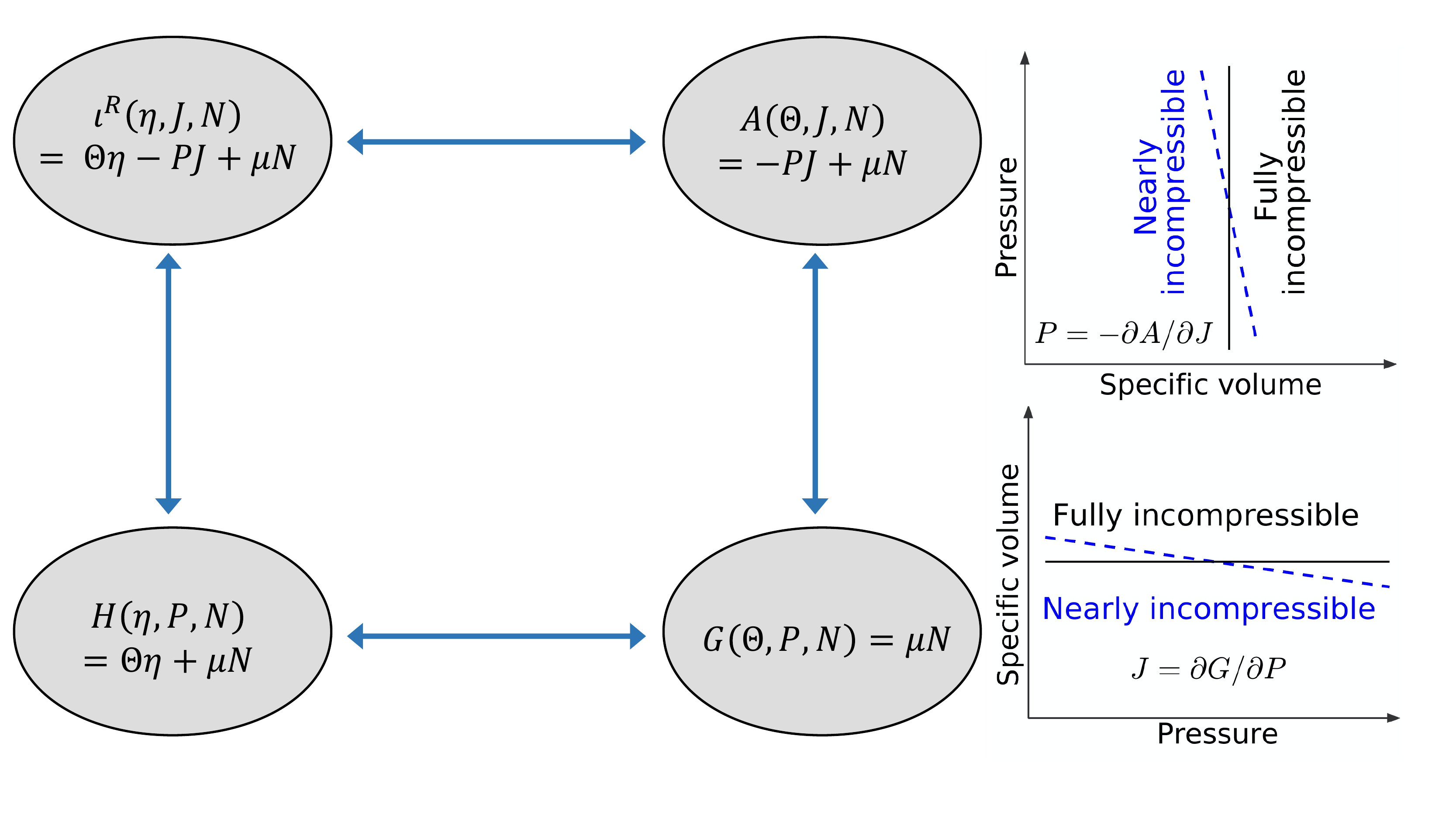} 
	\end{tabular}
	\caption{Illustration of the Legendre transformation of thermodynamic potentials: the internal energy $\iota^R$, Helmholtz free energy $A$, Gibbs free energy $G$, and enthalpy $H$. In particular, the Helmholtz free energy $A\left(\Theta, J, N\right)$ can be transformed to $G\left(\Theta, P, N\right)$ assuming convexity of $A$ with respect to $J$. The resulting constitutive laws based on $A$ and $G$ are given in the right two figures. In the theory based on the Helmholtz free energy, one has $P = - \partial A / \partial J$. In the incompressible limit, the pressure-volume (i.e., stress-strain) curve becomes a vertical line, and the corresponding components in the elasticity tensor blow up to infinity, which in turn causes numerical issues. This partly explains that, for most displacement-based formulations, the tangent matrix is ill-conditioned and demands a direct solver oftentimes. In the constitutive laws based on $G$, the volume is given by $J = \partial G / \partial P$, which becomes a horizontal line in the thermodynamic limit. The resulting continuum model constitutes a saddle-point problem and requires inf-sup stable or stabilized discretization techniques.} 
	\label{fig:thermo-commuting-diagram}
\end{center}
\end{figure}

For finite deformation problems, the deformation is typically highly isochoric. If one considers the Helmholtz free energy as the thermodynamic potential, the resulting system becomes singular in the incompressible limit (see Figure \ref{fig:thermo-commuting-diagram}). The pressure-volume curve, as part of the stress-strain curve, has a very large slope and becomes a vertical line in the incompressible limit. The value of the slope enters into the elasticity tensor and engenders an ill-conditioned stiffness matrix in the incompressible limit, a well-known issue in the displacement-based elasticity formulation \cite{Auricchio2013}. To handle this issue, one may consider performing a Legendre transformation for the free energy and switch the independent variable from the specific volume to its conjugate counterpart, the pressure \cite{Liu2018}. In doing so, the problem enjoys a saddle-point nature with the pressure acting as an independent variable. Correspondingly, in the incompressible limit, the volume-pressure curve becomes horizontal with a slope approaching zero. With the Legendre transformation performed, the resulting thermodynamic potential becomes the Gibbs free energy, and it is, therefore, a well-suited potential for thermomechanical analysis. This argument is supported by an analysis made in the realm of computational fluid dynamics (CFD). It has been shown that the pressure primitive variable is among the two sets of variables for interpolation if one wants to have a compressible CFD code that survives in the low Mach number limit \cite{Hauke1994,Hauke1998}. The aforementioned Gibbs free energy can be viewed as the thermodynamic explanation of the effectiveness of the pressure primitive variables. A theory for finite elasticity based on the Gibbs free energy was established \cite{Liu2018}; the saddle-point nature has also been exploited to design preconditioners \cite{Liu2019,Liu2020a}; it can also be shown to enjoy a provably nonlinearly stable semi-discrete formulation using inf-sup stable elements \cite{Liu2019a}; it can be conveniently utilized to construct a strongly-coupled fluid-structure interaction formulation \cite{Liu2018,Liu2020b}. In this work, we follow our prior approach and use the Gibbs free energy to construct a mixed formulation for viscoelastodynamics. The thermodynamic consistent nature of the continuum model naturally allows us to devise a discretization that inherits the energy stability property. The resulting mixed formulation necessitates inf-sup stable element pairs, and we adopt a suite of smooth generalization of the Taylor-Hood element by NURBS that has been numerically demonstrated to be stable \cite{Liu2019a,Buffa2011,Hosseini2015,Liu2021}. In particular, the same set of basis functions is utilized to represent the geometry and to construct the spaces for the displacement and velocity fields, which fits well into the paradigm of NURBS-based isogeometric analysis \cite{Hughes2005,Cottrell2009}. One appealing feature of this choice is that it achieves higher-order accuracy \cite{Evans2009} without sacrificing robustness \cite{Lipton2010}, in contrast to conventional $C^0$-continuous finite elements that often becomes ``fragile" when raising the polynomial order \cite{Elguedj2008}. Lastly, we mention other promising candidates for the spatial discretization, such as the smooth generalization of the Raviart-Thomas elements \cite{Buffa2011,Evans2013} as well as divergence-free discontinuous Galerkin methods \cite{Cockburn2004}.


The body of this work starts in Section \ref{sec:continuum} where we derive a general continuum theory for viscoelasticity with the Gibbs free energy as the thermodynamic potential. Following, we restrict the discussion to a special form of the free energy and provide a definition for the finite linear viscoelasticity. After revisiting the identical polymer chain model, the relaxation condition for the viscous stress is discussed. Two material models are then presented, which completes the continuum modeling section. In Section \ref{sec:numerical_formulation}, we consider the spatial and temporal discretization of the constructed continuum model. In particular, we show the energy stability of the proposed numerical formulation. In Section \ref{sec:numerical_results}, a suite of three numerical tests is performed as an examination of the continuum model and verification of the numerical scheme. We give concluding remarks in Section \ref{sec:conclusion}. In Appendices, the implementation details of the three material models are presented.

\section{Continuum Basis}
\label{sec:continuum}
\subsection{Kinematics}
Let $\Omega_{\bm X}$ and $\Omega_{\bm x}^t$ be bounded open sets in $\mathbb R^{3}$ with Lipschitz boundaries. The motion of the body is described by a family of smooth mappings parameterized by the time coordinate $t$,
\begin{align*}
\bm\varphi_t(\cdot) = \bm\varphi(\cdot, t) : \Omega_{\bm X} &\rightarrow \Omega_{\bm x}^t = \bm \varphi(\Omega_{\bm X}, t) = \bm \varphi_t(\Omega_{\bm X}), \quad \forall t \geq 0, \quad
\bm X \mapsto \bm x = \bm \varphi(\bm X, t) = \bm \varphi_t(\bm X), \quad \forall \bm X \in \Omega_{\bm X}.
\end{align*}
In the above, $\bm x$ represents the current position of a material particle originally located at $\bm X$, which implies $\bm \varphi(\bm X, 0) = \bm X$. The displacement and velocity of the material particle are defined as
\begin{align*}
\bm U := \bm \varphi(\bm X, t) - \bm \varphi(\bm X, 0) = \bm \varphi(\bm X, t) - \bm X, \qquad
\bm V := \left. \frac{\partial \bm \varphi}{\partial t}\right|_{\bm X}= \left. \frac{\partial \bm U}{\partial t}\right|_{\bm X} = \frac{d\bm U}{dt}.
\end{align*}
In this work, we use $d\left( \cdot \right)/dt$ to denote a total time derivative. The spatial velocity is defined as $\bm v := \bm V \circ \bm \varphi_t^{-1}$. Analogously, we define $\bm u := \bm U \circ \bm \varphi_t^{-1}$. The deformation gradient, the Jacobian determinant, and the right Cauchy-Green tensor are defined as
\begin{align*}
\bm F := \frac{\partial \bm \varphi}{\partial \bm X}, \qquad
J := \textup{det}\left(\bm F \right), \qquad \bm C := \bm F^T \bm F.
\end{align*}
Since most materials of interest behaves differently in bulk and shear under large strains, the deformation is multiplicatively decomposed into a volumetric part $J^{1/3}\bm I$ and an isochoric part $\tilde{\bm F} := J^{-1/3}\bm F$ \cite{Flory1961,Simo1985}. Clearly, by construction one has the multiplicative decomposition of the deformation gradient as
\begin{align}
\label{eq:Flory-decomposition}
\bm F = \left( J^{\frac13} \bm I \right) \tilde{\bm F}.
\end{align}
The corresponding modified right Cauchy-Green tensor $\tilde{\bm C}$ is defined as
\begin{align*}
\tilde{\bm C} := J^{-\frac23}\bm C.
\end{align*}
To facilitate the following discussion, we note the differentiation relation
\begin{align*}
\frac{\partial \tilde{\bm C}}{\partial \bm C} = J^{-\frac23} \mathbb P^T \quad \mbox{ with } \quad \mathbb P := \mathbb I - \frac{1}{3} \bm C^{-1} \otimes \bm C,
\end{align*}
wherein $\mathbb I$ is the fourth-order identity tensor. The projection tensor $\mathbb P$ furnishes deviatoric behaviors in the Lagrangian description \cite{Holzapfel2000}. Furthermore, in this work, the magnitude of a second-order tensor $\bm A$ is denoted as 
\begin{align}
\label{eq:def_mag_2nd_tensor}
\left\lvert \bm A \right\rvert^2 := \mathrm{tr}[\bm A \bm A^T ].
\end{align}

\subsection{Balance equations and constitutive relations}
The motion of the continuum body has to satisfy the local balance equations as well as the second law of thermodynamics \cite{Marsden1994,Scovazzi2007}. The advective form of the mass balance equation can be written as
\begin{align}
\label{eq:current-mass-eqn}
\frac{d\rho}{dt} + \rho \nabla_{\bm x} \cdot \bm v = 0,
\end{align}
wherein $\rho$ is the mass density in the current configuration. This mass balance equation is equivalent to the equation for the volumetric strain $J$,
\begin{align}
\label{eq:reference-mass-eqn}
\frac{dJ}{dt} = J \nabla_{\bm x} \cdot \bm v, \quad \mbox{ or equivalently, } \quad \frac{dJ}{dt} = J \nabla_{\bm X} \bm V : \bm F^{-T}.
\end{align}
The balance of linear momentum can be written as
\begin{align}
\label{eq:current-linear-momentum-eqn}
\rho \frac{d \bm{v} }{dt} = \nabla_{\bm x} \cdot \bm \sigma + \rho \bm b, \quad \mbox{ or equivalently, } \quad \rho_0 \frac{d \bm V}{dt} = \nabla_{\bm X} \cdot \bm P + \rho_0 \bm B,
\end{align}
wherein $\bm \sigma$ denote the Cauchy stress, $\bm b$ represents the body force per unit mass, $\bm P := J \bm \sigma \bm F^{-T}$ is the first Piola-Kirchhoff stress, and $\bm B := \bm b \circ \bm \varphi_t$. The balance of angular momentum is satisfied by imposing symmetry on the Cauchy stress, i.e., $\bm \sigma = \bm \sigma^T$. The balance of internal energy is stated as
\begin{align}
\label{eq:current-internal-energy-eqn}
\rho \frac{d\iota}{dt} = \bm \sigma : \nabla_{\bm x} \bm v - \nabla_{\bm x} \cdot \bm q + \rho r,
\end{align}
in which $\iota$ is the internal energy per unit mass, $\bm q$ denote the heat flux, and $r$ is the heat source per unit mass. Further, we introduce the following quantities: the internal energy $\iota^R$ defined with respect to the reference volume is related to $\iota$ by $\iota^R := \rho_0 \iota$; the heat flux defined per unit referential surface area is denoted as $\mathbfcal Q$, and it is related to $\bm q$ by the Piola transformation $\mathbfcal Q = J \bm F^{-1} \bm q$; the heat supply per unit referential volume $R := \rho_0 r$. With these definitions, the balance of the internal energy can also be expressed as
\begin{align}
\label{eq:reference-internal-energy-eqn}
\frac{d \iota^R}{dt} = \bm P : \frac{d\bm F }{dt}- \nabla_{\bm X} \cdot \mathbfcal Q + R.
\end{align}
Let us introduce $\mathfrak S$ as the entropy per unit volume in the referential configuration, $\Theta$ as the absolute temperature field in the reference configuration, and $\mathcal D$ as the dissipation. The second law of thermodynamics states the dissipation of the system is non-negative, i.e.,
\begin{align}
\label{eq:2nd_law_thermodynamics}
\mathcal D := \frac{d\mathfrak S}{dt} + \nabla_{\bm X} \cdot \left( \frac{\mathbfcal Q}{\Theta} \right) - \frac{R}{\Theta} \geq 0.
\end{align}
The Gibbs free energy \textit{per unit volume in the referential configuration} is defined as \cite[Chapter~5]{Schroeder1999}
\begin{align}
\label{eq:def_Gibbs_free_energy}
G := \iota^R - \Theta \mathfrak S + PJ,
\end{align}
wherein $P$ is the thermodynamic pressure defined on the referential configuration, which is the conjugate variable to $J$. The thermodynamic pressure defined on the current configuration is denoted by $p:= P \circ \bm \varphi_t^{-1}$. Taking material time derivatives at both sides of \eqref{eq:def_Gibbs_free_energy} results in
\begin{align*}
\frac{dG}{dt} + \mathfrak S \frac{d\Theta}{dt} - J \frac{dP}{dt} = \frac{d\iota^R}{dt} - \Theta \frac{d\mathfrak S}{dt} + P\frac{dJ}{dt}.
\end{align*}
Substituting the balance equation of the internal energy \eqref{eq:reference-internal-energy-eqn}, the second law of thermodynamics \eqref{eq:2nd_law_thermodynamics}, and the mass balance equation \eqref{eq:reference-mass-eqn} into the above equation, one readily obtains
\begin{align}
\label{eq:dissipation_form_1}
\Theta \mathcal D = J \bm \sigma : \nabla_{\bm x} \bm v - \frac{\mathbfcal Q \cdot \nabla_{\bm X} \Theta}{\Theta} + P J \nabla_{\bm x} \cdot \bm v  - \mathfrak S \frac{d\Theta}{dt} + J \frac{dP}{dt} - \frac{dG}{dt} .
\end{align}
We may additively split the Cauchy stress into deviatoric and hydrostatic parts,
\begin{align*}
\bm \sigma = \bm \sigma_{\mathrm{dev}} + \frac13 \left( \mathrm{tr}\left[ \bm \sigma \right]\right) \bm I,
\end{align*}
with which one may show that 
\begin{align*}
J \bm \sigma : \nabla_{\bm x} \bm v = \frac12 J \tilde{\bm F}^{-1} \bm \sigma_{\mathrm{dev}} \tilde{\bm F}^{-T} : \frac{d}{dt} \tilde{\bm C} + \frac13 J \mathrm{tr}\left[ \bm \sigma \right] \nabla_{\bm x} \cdot \bm v.
\end{align*}
Consequently, the relation \eqref{eq:dissipation_form_1} can be rewritten as
\begin{align}
\label{eq:gibbs_truesdell_hint}
\Theta \mathcal D = \frac12 J \tilde{\bm F}^{-1} \bm \sigma_{\mathrm{dev}} \tilde{\bm F}^{-T} : \frac{d}{dt} \tilde{\bm C} +  J \left(\frac{1}{3}\textup{tr}\left[ \bm \sigma \right] + p \right)\nabla_{\bm x}\cdot \bm v - \frac{\mathbfcal Q \cdot \nabla_{\bm X} \Theta}{\Theta} - \mathfrak S \frac{d\Theta}{dt} + J \frac{dP}{dt} - \frac{dG}{dt} .
\end{align}
From the above, we postulate that the Gibbs free energy $G$ is a function of $\tilde{\bm C}$, $P$, and $\Theta$ by invoking Truesdell's principle of equipresence \cite{Truesdell1965}. Additioanlly we assume that $G$ also depends on a set of strain-like internal state variables $\{\bm \Gamma^{\alpha}\}_{\alpha=1}^{m}$,
\begin{align}
\label{eq:Gibbs_free_energy_function}
G = G(\tilde{\bm C}, P, \Theta, \bm \Gamma^1, \cdots, \bm \Gamma^m).
\end{align}
Here $m$ is the number of relaxation processes characterizing the viscous property of the material. Oftentimes, the internal state variables $\bm \Gamma^{\alpha}$ is viewed as a strain tensor akin to the right Cauchy-Green strain tensor \cite{Holzapfel1996a,Holzapfel1996}. We demand that for a given homogeneous reference temperature $\Theta_0 > 0$,
\begin{align}
\label{eq:Gibbs_free_energy_normalization_condition}
G(\bm I, \textit{0}, \Theta_0, \bm I, \cdots, \bm I) = 0,
\end{align}
which is commonly known as the normalization condition \cite{Holzapfel2000}. With the above function form of the Gibbs free energy, material time derivative of the Gibbs free energy can be written explicitly as
\begin{align}
\label{eq:dG_dt_chain_rule}
\frac{dG}{dt} = \frac12 \tilde{\bm S} : \frac{d}{dt} \tilde{\bm C} + \frac{\partial G}{\partial P} \frac{dP}{dt} + \frac{\partial G}{\partial \Theta} \frac{d\Theta}{dt} + \sum_{\alpha =1}^{m} \frac{\partial G}{\partial \bm \Gamma^{\alpha}} : \frac{d}{dt}\bm \Gamma^{\alpha},
\end{align}
in which we introduced the fictitious second Piola-Kirchhoff stress
\begin{align*}
\tilde{\bm S} := 2\frac{\partial G(\tilde{\bm C}, P, \Theta, \bm \Gamma^1, \cdots, \bm \Gamma^m)}{\partial \tilde{\bm C}}.
\end{align*}
Substituting \eqref{eq:dG_dt_chain_rule} into \eqref{eq:dissipation_form_1} leads to 
\begin{align}
\label{eq:dissipation_form_2}
\Theta \mathcal D =& \left( \frac{1}{2}J \tilde{\bm F}^{-1} \bm \sigma_{\mathrm{dev}} \tilde{\bm F}^{-T} - \frac{1}{2} \tilde{\bm S} \right) : \frac{d}{dt}\tilde{\bm C} +  J \left(\frac{1}{3}\textup{tr}\left[ \bm \sigma \right] + p \right)\nabla_{\bm x}\cdot \bm v - \frac{\mathbfcal Q \cdot \nabla_{\bm X} \Theta}{\Theta} \nonumber \\
& - \left( \mathfrak S + \frac{\partial G}{\partial \Theta}\right) \frac{d\Theta}{dt} + \left( J - \frac{\partial G}{\partial P} \right)\frac{dP}{dt} - 2 \sum_{\alpha =1}^{m} \frac{\partial G}{\partial \bm \Gamma^{\alpha}} : \frac12 \frac{d}{dt}\bm \Gamma^{\alpha}.
\end{align}
We introduce $\bm Q^{\alpha}$ as the conjugate variables to the internal state variables $\bm \Gamma^{\alpha}$,
\begin{align}
\label{eq:definition_Q}
\bm Q^{\alpha} := - 2 \frac{\partial G(\tilde{\bm C}, P, \Theta, \bm \Gamma^1, \cdots, \bm \Gamma^m)}{\partial \bm \Gamma^{\alpha}}.
\end{align}
The relation \eqref{eq:dissipation_form_2} furnishes a natural means of choosing constitutive relations that ensure the satisfaction of the second law of thermodynamics. We may thereby make the following choices for the constitutive relations,
\begin{align}
\label{eq:constitutive_dev_sigma}
& \bm \sigma_{\mathrm{dev}} = J^{-1} \tilde{\bm F} \left( \mathbb P : \tilde{\bm S} \right) \tilde{\bm F}^T = J^{-1} \tilde{\bm F} \left( \mathbb P : \left( 2 \frac{\partial G}{\partial \tilde{\bm C}} \right) \right) \tilde{\bm F}^T, \displaybreak[2]\\
\label{eq:constitutive_tr_sigma}
& \frac{1}{3}\mathrm{tr}\left[\bm \sigma \right] = -p, \displaybreak[2] \\
\label{eq:constitutive_heat_flux}
& \mathbfcal Q = - \bar{\kappa} \nabla_{\bm X} \Theta, \displaybreak[2] \\
\label{eq:constitutive_s}
& \mathfrak S = -\frac{\partial G}{\partial \Theta}, \displaybreak[2] \\
\label{eq:constitutive_rho}
& \rho = \rho_0 \left( \frac{\partial G}{\partial P} \right)^{-1}, \displaybreak[2] \\
\label{eq:constitutive_Q}
& \bm Q^{\alpha} = \mathbb V^{\alpha} : \left( \frac12 \frac{d}{dt}\bm \Gamma^{\alpha} \right).
\end{align}
In the above, $\bar{\kappa}$ is the thermal conductivity and $\mathbb V^{\alpha}$ is a positive definite fourth-order viscosity tensor. Here, we also assume that there exists a fourth-order tensor $\left(\mathbb V^{\alpha}\right)^{-1}$ such that an inverse relation for \eqref{eq:constitutive_Q} holds,
\begin{align}
\label{eq:constitutive_Q_inverse}
\frac{1}{2} \frac{d}{dt} \bm \Gamma^{\alpha} = \left( \mathbb V^{\alpha} \right)^{-1} : \bm Q^{\alpha}.
\end{align}
Based on \eqref{eq:constitutive_dev_sigma} and \eqref{eq:constitutive_tr_sigma}, the Cauchy stress $\bm \sigma$ can be represented as
\begin{align*}
\bm \sigma := \bm \sigma_{\mathrm{dev}} + \frac{1}{3}\mathrm{tr}\left[\bm \sigma \right] \bm I = J^{-1} \tilde{\bm F} \left( \mathbb P : \tilde{\bm S} \right) \tilde{\bm F}^T - p \bm I = J^{-1} \tilde{\bm F} \left( \mathbb P : \left( 2 \frac{\partial G}{\partial \tilde{\bm C}} \right) \right) \tilde{\bm F}^T - p \bm I.
\end{align*}
Accordingly, the second Piola-Kirchhoff stress $\bm S$, as the pull-back operation performed on $J\bm \sigma$, can be written as
\begin{align*}
& \bm S := J \bm F^{-1} \bm \sigma \bm F^{-T} = J \bm F^{-1} \left( \bm \sigma_{\mathrm{dev}} + \frac{1}{3}\mathrm{tr}\left[\bm \sigma \right] \bm I \right) \bm F^{-T} = \bm S_{\mathrm{iso}} + \bm S_{\mathrm{vol}}, \\
& \bm S_{\mathrm{iso}} := J \bm F^{-1} \bm \sigma_{\mathrm{dev}} \bm F^{-T} = J^{-\frac23} \mathbb P : \tilde{\bm S} = J^{-\frac23} \mathbb P : 2 \frac{\partial G}{\partial \tilde{\bm C}}, \\
& \bm S_{\mathrm{vol}} := \frac13 \mathrm{tr}\left[ \bm \sigma \right] J \bm F^{-1} \bm F^{-T} = -J P \bm C^{-1}.
\end{align*}
With the above choices, it can be shown that the dissipation relation \eqref{eq:dissipation_form_2} reduces to
\begin{align}
\label{eq:dissipation-relation}
\Theta \mathcal D = \bar{\kappa} |\nabla_{\bm X} \Theta |^2 - \sum_{\alpha=1}^{m} \bm Q^{\alpha} : \frac12 \left(\frac{d}{dt}\bm \Gamma^{\alpha}\right) &= \bar{\kappa} |\nabla_{\bm X} \Theta |^2 + \sum_{\alpha=1}^{m} \frac14 \left(\frac{d}{dt}\bm \Gamma^{\alpha}\right) : \mathbb V^{\alpha} : \left(\frac{d}{dt}\bm \Gamma^{\alpha}\right) \nonumber \\
&= \bar{\kappa} |\nabla_{\bm X} \Theta |^2 + \sum_{\alpha=1}^{m} \bm Q^{\alpha} : \left( \mathbb V^{\alpha} \right)^{-1} : \bm Q^{\alpha},
\end{align}
which remains non-negative for arbitrary kinematic processes, thereby automatically satisfying the second law of thermodynamics. The choices of the constitutive relations \eqref{eq:constitutive_dev_sigma}-\eqref{eq:constitutive_Q} are certainly not a unique way for ensuring the second law of thermodynamics. Yet, as will be shown, these choices are general enough to describe the mechanical behavior of a variety of materials. The definition of $\bm Q^{\alpha}$ \eqref{eq:definition_Q} suggests that it is a function of $\bm \Gamma^{\alpha}$, $\tilde{\bm C}$, $P$, and $\Theta$. Thereby, taking material time derivatives of $\bm Q^{\alpha}$ yields
\begin{align*}
\frac{d}{dt} \bm Q^{\alpha} = 2 \frac{\partial \bm Q^{\alpha}}{\partial \bm \Gamma^{\alpha}} : \frac12 \frac{d}{dt} \bm \Gamma^{\alpha} + 2 \frac{\partial \bm Q^{\alpha}}{\partial \tilde{\bm C}} : \frac12 \frac{d}{dt} \tilde{\bm C} + \frac{\partial \bm Q^{\alpha}}{\partial \Theta} \frac{d}{dt} \Theta + \frac{\partial \bm Q^{\alpha}}{\partial P} \frac{d}{dt} P.
\end{align*}
Invoking the constitutive relation \eqref{eq:constitutive_Q_inverse} to replace the rate of $\bm \Gamma^{\alpha}$ in the above relation, we obtain a set of evolution equations for $\bm Q^{\alpha}$ after rearranging terms,
\begin{align}
\label{eq:evolution_eqn_Q}
\frac{d}{dt} \bm Q^{\alpha} - 2 \frac{\partial \bm Q^{\alpha}}{\partial \bm \Gamma^{\alpha}} : \left( \mathbb V^{\alpha} \right)^{-1} : \bm Q^{\alpha} = 2 \frac{\partial \bm Q^{\alpha}}{\partial \tilde{\bm C}} : \frac12 \frac{d}{dt} \tilde{\bm C} + \frac{\partial \bm Q^{\alpha}}{\partial \Theta} \frac{d}{dt}\Theta + \frac{\partial \bm Q^{\alpha}}{\partial P} \frac{d}{dt} P.
\end{align}
The above constitutes a set of general evolution equations for $\bm Q^{\alpha}$, and in the latter part of this work we will reveal how it may recover the familiar linear evolution equations \cite[Chapter~6]{Holzapfel2000}. In the last, we characterize the thermodynamic equilibrium state by the following two conditions, 
\begin{align}
\label{eq:characterization-thermodynamic-equilibrium}
\nabla_{\bm X}\Theta \Big|_{\mathrm{eq}} = \bm 0, \quad \mbox{ and } \quad \frac{d}{dt}\bm \Gamma^{\alpha} \Big|_{\mathrm{eq}} = \bm O, \quad \mbox{ for } \alpha = 1, \cdots, m.
\end{align}
In the above, the notation $\big|_{\mathrm{eq}}$ indicates that equality holds at the thermodynamic equilibrium limit. The conditions \eqref{eq:characterization-thermodynamic-equilibrium} suggest that, in the equilibrium limit, there is no more dissipation in the system. In particular, \eqref{eq:characterization-thermodynamic-equilibrium}$_2$ implies that there is no more change of the internal state variables in the thermodynamic equilibrium limit. According to the constitutive relation \eqref{eq:constitutive_Q}, \eqref{eq:characterization-thermodynamic-equilibrium}$_2$ is equivalent to
\begin{align}
\bm Q^{\alpha} \Big|_{\mathrm{eq}} = \bm O, \quad \mbox{ for } \alpha = 1, \cdots, m.
\end{align}

\subsection{An additive split of the free energy and refined constitutive relations}
Here we consider a more refined structure of the thermodynamic potential $G$. First, we note that the density $\rho$ has to be independent of $\tilde{\bm C}$, since the modified right Cauchy-Green tensor describes volume-preserving deformations. Therefore, one may reasonably demand that the density is a function of the pressure and temperature. Then the constitutive relation \eqref{eq:constitutive_rho} implies that
\begin{align}
\label{eq:partial-G-partial-P-rho}
\frac{\partial G(\tilde{\bm C}, P, \Theta, \bm \Gamma^1, \cdots, \bm \Gamma^m)}{\partial P} = \rho_0 \rho^{-1}(p , \theta, \bm \Gamma^1, \cdots, \bm \Gamma^m),
\end{align}
with the equality holds by recalling that $P(\bm X) = p(\bm \varphi_t(\bm X), t)$ and $\Theta(\bm X) = \theta(\bm \varphi_t(\bm X), t)$. The relation \eqref{eq:partial-G-partial-P-rho} suggests that the free energy can be additively split as
\begin{align}
\label{eq:G_first_additive_split}
G = G(\tilde{\bm C}, P, \Theta, \bm \Gamma^1, \cdots, \bm \Gamma^m) = G^{\infty}_{\textup{vol}}(P,\Theta, \bm \Gamma^1, \cdots, \bm \Gamma^m) + G_{\mathrm{iso}}(\tilde{\bm C}, \Theta, \bm \Gamma^1, \cdots, \bm \Gamma^m),
\end{align}
by performing a partial integration of \eqref{eq:partial-G-partial-P-rho} with respect to $P$ (see also \cite[p.~559]{Liu2018}). This split structure of the free energy has also been justified based on physical observations \cite{Criscione2000,Sansour2008}. It needs to be pointed out that the split is a logical consequence of the multiplicative decomposition \eqref{eq:Flory-decomposition}. Second, there is a widely-adopted assumption with experimental justifications that viscous effects only affect the isochoric motion, which suggests that the volumetric energy $G^{\infty}_{\mathrm{vol}}$ is independent of the internal state variables. Third, as a commonly adopted approach \cite{Holzapfel2000}, the isochoric part of the free energy $G_{\mathrm{iso}}$ can be further split into an equilibrium part as well as a dissipative, or non-equilibrium, part. The last two assumptions further lead to the following form of the Gibbs free energy,
\begin{align}
\label{eq:Gibbs-free-energy-structure}
G = G(\tilde{\bm C}, P, \Theta, \bm \Gamma^1, \cdots, \bm \Gamma^m) = G^{\infty}_{\textup{vol}}(P,\Theta) + G^{\infty}_{\textup{iso}}(\tilde{\bm C},\Theta) + \sum_{\alpha=1}^{m}\Upsilon^{\alpha}\left( \tilde{\bm C}, \Theta, \bm \Gamma^1, \cdots, \bm \Gamma^m \right),
\end{align}
The two terms $G^{\infty}_{\textup{vol}}$ and $G^{\infty}_{\textup{iso}}$ characterize the hyperelastic material behavior at the equilibrium state as time approaches infinity. Given the explicit structure of the Gibbs free energy \eqref{eq:Gibbs-free-energy-structure}, we may introduce the equilibrium and non-equilibrium parts of the fictitious second Piola-Kirchhoff stress as
\begin{align}
\label{eq:Piola-Kirchhoff-stress-original}
\tilde{\bm S} = \tilde{\bm S}^{\infty}_{\mathrm{iso}} + \sum_{\alpha=1}^{m} \tilde{\bm S}^{\alpha}_{\mathrm{neq}}, \quad \tilde{\bm S}^{\infty}_{\mathrm{iso}} := 2\frac{\partial G^{\infty}_{\mathrm{iso}}(\tilde{\bm C},\Theta)}{\partial \tilde{\bm C}}, \quad \tilde{\bm S}^{\alpha}_{\mathrm{neq}} := 2\frac{\partial \Upsilon^{\alpha}( \tilde{\bm C}, \Theta, \bm \Gamma^1, \cdots, \bm \Gamma^m)}{\partial \tilde{\bm C}}.
\end{align}
Correspondingly, the isochoric part of the second Piola-Kirchhoff stress can be expressed as
\begin{align*}
\bm S_{\mathrm{iso}} = \bm S^{\infty}_{\mathrm{iso}} + \sum_{\alpha=1}^{m} \bm S^{\alpha}_{\mathrm{neq}}, \quad \bm S^{\infty}_{\mathrm{iso}} := J^{-\frac23} \mathbb P : \tilde{\bm S}^{\infty}_{\mathrm{iso}}, \quad \bm S^{\alpha}_{\mathrm{neq}} := J^{-\frac23} \mathbb P : \tilde{\bm S}^{\alpha}_{\mathrm{neq}}.
\end{align*}
Due to \eqref{eq:Gibbs-free-energy-structure}, $\bm Q^{\alpha}$ can be defined in terms of the configurational free energy $\Upsilon^{\alpha}$ as
\begin{align}
\label{eq:definition_Q_2}
\bm Q^{\alpha} := - 2 \frac{\partial G(\tilde{\bm C}, P, \Theta, \bm \Gamma^1, \cdots, \bm \Gamma^m)}{\partial \bm \Gamma^{\alpha}} = -2 \frac{\partial \Upsilon^{\alpha}(\tilde{\bm C}, \Theta, \bm \Gamma^1, \cdots, \bm \Gamma^m)}{\partial \bm \Gamma^{\alpha}}.
\end{align}
Correspondingly, the evolution equations \eqref{eq:evolution_eqn_Q} can be reduced to
\begin{align}
\label{eq:evolution_eqn_Q_reduced}
\frac{d}{dt} \bm Q^{\alpha} - 2 \frac{\partial \bm Q^{\alpha}}{\partial \bm \Gamma^{\alpha}} : \left( \mathbb V^{\alpha} \right)^{-1} : \bm Q^{\alpha} = 2 \frac{\partial \bm Q^{\alpha}}{\partial \tilde{\bm C}} : \frac12 \frac{d}{dt} \tilde{\bm C} + \frac{\partial \bm Q^{\alpha}}{\partial \Theta} \frac{d}{dt}\Theta,
\end{align}
since $\bm Q^{\alpha}$ is independent of $P$ from \eqref{eq:definition_Q_2}. Here the evolution equations \eqref{eq:evolution_eqn_Q_reduced} together with \eqref{eq:definition_Q_2} can be viewed as a set of nonlinear rate equations characterizing the evolution of the internal state variables $\bm \Gamma^{\alpha}$ as well. Finally, to complete the thermomechanical theory, the entropy per unit reference volume adopts the form,
\begin{align*}
\mathfrak S = \mathfrak S^{\infty} + \sum_{\alpha=1}^{m} \mathfrak S^{\alpha}, \quad \mathfrak S^{\infty} := -\frac{\partial G^{\infty}_{\mathrm{vol}}}{\partial \Theta} -\frac{\partial G^{\infty}_{\mathrm{iso}}}{\partial \Theta}, \quad \mathfrak S^{\alpha} := -\frac{\partial \Upsilon^{\alpha}}{\partial \Theta}.
\end{align*}

\begin{remark}
In our derivation, we start from the classical definition of the Gibbs free energy \eqref{eq:def_Gibbs_free_energy} without assuming which variables it should depend on. Subsequently, the derived relation \eqref{eq:gibbs_truesdell_hint} together with the Truesdell's principle of equipresence suggests that it should be a function of $\tilde{\bm C}$, $P$, and $\Theta$. With the additive split of the free energy into isochoric and volumetric parts, we may observe from \eqref{eq:Piola-Kirchhoff-stress-original}$_2$ that the isochoric part of the Gibbs free energy plays the same role as the isochoric part of the conventional strain energy of the Helmholtz type \cite[Chapter~6]{Holzapfel2000}. Therefore, we may view that the Gibbs free energy is related to the Helmholtz free energy by performing the Legendre transformation on the volumetric part of the free energy only \cite[Section~2.4]{Liu2018}. Indeed, the volumetric energy is commonly convex with respect to $J$, unless there are multiple stable phases (e.g. the van der Waals model \cite{Liu2015,Sengers2003}), and this convexity property guarantees the validity of the Legendre transformation. We note that some adopt a different notion of the Gibbs free energy which depends on the whole stress \cite{Rajagopal2013,Surana2013}. That type of theory requires the convexity of the free energy with respect to the whole strain, which does not hold for many nonlinear materials \cite{Ball1976}.
\end{remark}

\begin{remark}
There is a requirement for the non-equilibrium stresses based on physical intuitions, that is $\bm S^{\alpha}_{\mathrm{neq}} = J^{-\frac23}\mathbb P : \tilde{\bm S}^{\alpha}_{\mathrm{neq}}$ should fully relax, or vanish, in the thermodynamic equilibrium limit. This condition apparently poses an additional constraint for the configurational free energy $\Upsilon^{\alpha}$, and we will revisit this point in Section \ref{sec:isotropic_response_quadratic_H}.
\end{remark}

\begin{remark}
It has recently been observed that the additive split of the energy \eqref{eq:G_first_additive_split} may lead to non-physical responses for anisotropic materials \cite{Sansour2008,Helfenstein2010,Gueltekin2019}. In this work, we restrict our discussion to isotropic materials, and the split \eqref{eq:G_first_additive_split} should be sound in both mathematics and physics \cite{Criscione2000,Sansour2008}.
\end{remark}


\begin{remark}
If we assume the viscosity tensors $\mathbb V^{\alpha}$ are isotropic, they adopt explicit forms as
\begin{align}
\label{eq:isotropic_V_alpha}
\mathbb V^{\alpha} = 2 \eta^{\alpha}_{D}(\Theta) \left( \mathbb I - \frac13 \bm I \otimes \bm I \right) + \frac23 \eta^{\alpha}_{V}(\Theta) \bm I \otimes \bm I,
\end{align}
where $\eta^{\alpha}_{D}(\Theta)$ and $\eta^{\alpha}_{V}(\Theta)$ are non-negative and represent the deviatoric and volumetric viscosities respectively \cite{Reese1998,Fung1994,Gurtin1982}.
The conjugate variables  $\bm Q^{\alpha}$ are related to the internal state variables $\bm \Gamma^{\alpha}$ by a linear relation,
\begin{align}
\label{eq:Q_eta_dot_Gamma_relationship}
\bm Q^{\alpha} = \eta^{\alpha}_{D}(\Theta) \left( \frac{d}{dt} \bm \Gamma^{\alpha} - \frac13 \mathrm{tr}[\frac{d}{dt} \bm \Gamma^{\alpha}] \bm I \right) + \frac13\eta^{\alpha}_{V}(\Theta) \mathrm{tr}[\frac{d}{dt} \bm \Gamma^{\alpha}] \bm I.
\end{align}
Consequently, the dissipation relation \eqref{eq:dissipation-relation} reduces to
\begin{align*}
\Theta \mathcal D = \bar{\kappa} |\nabla_{\bm X} \Theta |^2 + \sum_{\alpha=1}^{m} \left( \frac{\eta^{\alpha}_D(\Theta)}{2} \left\lvert \frac{d}{dt} \bm \Gamma^{\alpha} - \frac13 \mathrm{tr}[\frac{d}{dt} \bm \Gamma^{\alpha}] \bm I  \right\rvert^2 + \frac{\eta^{\alpha}_V(\Theta)}{6} \mathrm{tr}[\frac{d}{dt} \bm \Gamma^{\alpha}]^2 \right).
\end{align*}
\end{remark}

\subsection{Finite linear viscoelasticity}
\label{sec:linear-viscoelasticity}
In this section, we consider specifically the finite deformation linear viscoelasticity, the precise definition of which will be given momentarily. We start by restricting the configurational free energy to the following form,
\begin{align}
\label{eq:configuration-free-energy}
\Upsilon^{\alpha}\left(\tilde{\bm C}, \Theta, \bm \Gamma^1, \cdots, \bm \Gamma^m \right) = H^{\alpha}(\bm \Gamma^{\alpha}, \Theta) + \left( \hat{\bm S}^{\alpha}_{0} - 2 \frac{\partial G^{\alpha}(\tilde{\bm C}, \Theta)}{\partial \tilde{\bm C}} \right) : \frac{\bm \Gamma^{\alpha} - \bm I}{2} + F^{\alpha}( \tilde{\bm C}, \Theta ).
\end{align}
In the above choice, the $\alpha$-th configurational free energy $\Upsilon^{\alpha}$ depends on $\bm \Gamma^{\alpha}$ only; $H^{\alpha}$, $G^{\alpha}$, and $F^{\alpha}$ are scalar-valued functions to be provided for material modeling; $\hat{\bm S}^{\alpha}_{0}$ is a constant stress-like tensor whose physical significance will be revealed in Section \ref{sec:isotropic_response_quadratic_H}. A similar form of the configurational free energy was proposed in \cite{Holzapfel1996} with minor differences. To satisfy the normalization condition for $\Upsilon^{\alpha}$, we demand that
\begin{align*}
H^{\alpha}(\bm I, \Theta_0) = 0, \quad F^{\alpha}(\bm I, \Theta_0) = 0,
\end{align*}
for a given homogeneous temperature $\Theta_0 > 0$. It is worth pointing out that the essence of the particular configurational free energy is that the elastic strain $\tilde{\bm C}$ and the internal state variable $\bm \Gamma^{\alpha}$ are coupled through the second term in \eqref{eq:configuration-free-energy} only, where the internal state variables $\bm \Gamma^{\alpha}$ are present in linear forms. Also, recalling that $\bm \Gamma^{\alpha}$ are akin to the right Cauchy-Green strain tensor, the terms $\left( \bm \Gamma^{\alpha}- \bm I \right) / 2$ are analogous to the Green-Lagrange strain tensor. It will facilitate our subsequent discussion by introducing a stress term defined as
\begin{align*}
\tilde{\bm S}^{\alpha}_{\mathrm{iso}}(\tilde{\bm C}, \Theta) := 2 \frac{\partial G^{\alpha}(\tilde{\bm C}, \Theta)}{\partial \tilde{\bm C}},
\end{align*}
with which the configurational free energy \eqref{eq:configuration-free-energy} can be rewritten as
\begin{align*}
\Upsilon^{\alpha}\left(\tilde{\bm C}, \Theta, \bm \Gamma^1, \cdots, \bm \Gamma^m \right) = H^{\alpha}(\bm \Gamma^{\alpha}, \Theta) + \left( \hat{\bm S}^{\alpha}_{0} - \tilde{\bm S}^{\alpha}_{\mathrm{iso}}(\tilde{\bm C}, \Theta) \right) : \frac{\bm \Gamma^{\alpha} - \bm I}{2} + F^{\alpha}( \tilde{\bm C}, \Theta ).
\end{align*}
Due to the definition \eqref{eq:definition_Q_2} and the form of the configurational free energy \eqref{eq:configuration-free-energy}, we now have
\begin{align}
\label{eq:Q-relation-linear-viscoelasticity}
\bm Q^{\alpha}(\tilde{\bm C}, \Theta, \bm \Gamma^1, \cdots, \bm \Gamma^m) =& -2 \frac{\partial \Upsilon^{\alpha}(\tilde{\bm C}, \Theta, \bm \Gamma^1, \cdots, \bm \Gamma^m )}{\partial \bm \Gamma^{\alpha}}
= 2 \left( \frac{\partial G^{\alpha}(\tilde{\bm C}, \Theta)}{\partial \tilde{\bm C}} - \frac{\partial H^{\alpha}(\bm \Gamma^{\alpha},\Theta)}{\partial \bm \Gamma^{\alpha}}\right) - \hat{\bm S}^{\alpha}_{0} \nonumber \\
=& \tilde{\bm S}^{\alpha}_{\mathrm {iso}}(\tilde{\bm C},\Theta) - \hat{\bm S}^{\alpha}_{0} - 2 \frac{\partial H^{\alpha}(\bm \Gamma^{\alpha}, \Theta)}{\partial \bm \Gamma^{\alpha}}.
\end{align}
From the above relation, the material time derivative of $\tilde{\bm S}^{\alpha}_{\mathrm{iso}}$ can be expressed as
\begin{align*}
\frac{d}{dt}\tilde{\bm S}^{\alpha}_{\mathrm{iso}} = \frac{\partial \bm Q^{\alpha}}{\partial \tilde{\bm C}} : \frac{d}{dt} \tilde{\bm C} + \frac{\partial \bm Q^{\alpha}}{\partial \Theta} \frac{d}{dt}\Theta + 2 \frac{\partial^2 H^{\alpha}}{\partial \bm \Gamma^{\alpha}\partial \Theta} \frac{d}{dt}\Theta,
\end{align*}
with which the right-hand side of the evolution equations \eqref{eq:evolution_eqn_Q_reduced} can be written compactly as
\begin{align}
\label{eq:evolution-eqn-Upsilon-special}
\frac{d}{dt} \bm Q^{\alpha} - 2 \frac{\partial \bm Q^{\alpha}}{\partial \bm \Gamma^{\alpha}} : \left( \mathbb V^{\alpha} \right)^{-1} : \bm Q^{\alpha} = \frac{d}{dt}\tilde{\bm S}^{\alpha}_{\mathrm{iso}} - 2 \frac{\partial^2 H^{\alpha}}{\partial \bm \Gamma^{\alpha}\partial \Theta} \frac{d}{dt}\Theta.
\end{align}
The non-equilibrium stresses \eqref{eq:Piola-Kirchhoff-stress-original}$_{3}$ now can be represented as
\begin{align}
\label{eq:non_equilibrium_stress_2}
\tilde{\bm S}^{\alpha}_{\mathrm{neq}} = 2 \frac{\partial \Upsilon^{\alpha}}{\partial \tilde{\bm C}} = 2 \left( \frac{\partial F_{\alpha}(\tilde{\bm C}, \Theta)}{\partial \tilde{\bm C}} - \frac{\partial \tilde{\bm S}^{\alpha}_{\mathrm{iso}}}{\partial \tilde{\bm C}} : \frac{\bm \Gamma^{\alpha} - \bm I}{2}\right),
\end{align}
in which we have made use of the major symmetry of the fourth-order tensor $\partial \tilde{\bm S}^{\alpha}_{\mathrm{iso}}/\partial \tilde{\bm C}$.

\subsubsection{Definition}
\label{sec:isotropic_response_quadratic_H}
Based on the above discussion, we now provide the definition of finite linear viscoelasticity here.
\begin{definition}
\label{def:finite-linear-viscoelasticity}
A finite linear viscoelastic material is described by the configurational free energy \eqref{eq:configuration-free-energy} with $\mathbb V^{\alpha} = 2 \eta^{\alpha}(\Theta) \mathbb I$ and a quadratic form for $H^{\alpha}$ in terms of the internal strain variables,
\begin{align}
\label{eq:quadratic-H-energy}
H^{\alpha}(\bm \Gamma^{\alpha},\Theta) = \mu^{\alpha}(\Theta) \left\lvert \frac{\bm \Gamma^{\alpha} - \bm I}{2} \right\rvert^2 ,
\end{align}
where $\mu^{\alpha}(\Theta)$ is a temperature dependent shear modulus associated with the $\alpha$-th relaxation process.
\end{definition}
The form $\mathbb V^{\alpha} = 2 \eta^{\alpha}(\Theta) \mathbb I$ is a special case for an isotropic tensor of order four by taking $\eta^{\alpha} = \eta^{\alpha}_D = \eta^{\alpha}_V$ in \eqref{eq:isotropic_V_alpha}. According to \eqref{eq:Q-relation-linear-viscoelasticity}, one can show that the quadratic form \eqref{eq:quadratic-H-energy} implies
\begin{align}
\label{eq:quadratic_H_Q_def}
\bm Q^{\alpha} = \tilde{\bm S}^{\alpha}_{\mathrm{iso}} - \hat{\bm S}^{\alpha}_{0} - 2\mu^{\alpha}(\Theta) \frac{ \bm \Gamma^{\alpha} - \bm I }{2}, \quad \mbox{ and } \quad
\frac{\partial \bm Q^{\alpha}}{\partial \bm \Gamma^{\alpha}} = -\mu^{\alpha}(\Theta) \mathbb I.
\end{align}
Apparently, the choice of the quadratic form for $H^{\alpha}$ leads to the above linear relation between $\bm Q^{\alpha}$ and $\bm \Gamma^{\alpha}$, which further simplifies the evolution equations \eqref{eq:evolution-eqn-Upsilon-special} as
\begin{align}
\label{eq:evolution-eqn-Q-linear-viscoelasticity}
\frac{d}{dt} \bm Q^{\alpha} + \frac{ \bm Q^{\alpha} }{\tau^{\alpha}} = \frac{d}{dt}\tilde{\bm S}^{\alpha}_{\mathrm{iso}} - \frac{d\mu^{\alpha}}{dt} \left(\bm \Gamma^{\alpha}-\bm I\right) ,
\end{align}
with 
\begin{align*}
\tau^{\alpha}(\Theta) := \eta^{\alpha}(\Theta) / \mu^{\alpha}(\Theta).
\end{align*}
The term $\tau^{\alpha}$ has the dimension of time and is commonly referred to as the relaxation time for the $\alpha$-th process. For most polymers, the relaxation time $\tau^{\alpha}$ can be modeled by the Arrhenius equation, which characterizes faster viscoelastic effects with the increase of temperature \cite{Tobolsky1944}. The second term on the right-hand side of \eqref{eq:evolution-eqn-Q-linear-viscoelasticity} arises due to the temperature dependent material parameters \cite[p.~365]{Holzapfel2000}. Given an initial condition $\bm Q^{\alpha}_0 := \bm Q^{\alpha}|_{t=0}$, the solution of \eqref{eq:evolution-eqn-Q-linear-viscoelasticity} can be obtained in a hereditary integral form,
\begin{align}
\label{eq:linear_hereditary_integral}
\bm Q^{\alpha} = \exp\left( - t / \tau^{\alpha} \right) \bm Q^{\alpha}_0 + \int^t_{0^+} \exp\left( -(t-s)/\tau^{\alpha} \right) \left( \frac{d}{ds} \tilde{\bm S}^{\alpha}_{\mathrm{iso}} - \left(\bm \Gamma^{\alpha} - \bm I \right) \frac{d\mu^{\alpha}}{ds}\right) ds.
\end{align}
We recall that the general evolution equations \eqref{eq:evolution-eqn-Upsilon-special} is nonlinear for $\bm Q^{\alpha}$. For finite \textit{linear} viscoelastic materials, the evolution equations reduce to a linear system, whose solution can be conveniently represented in terms of the hereditary integral form \eqref{eq:linear_hereditary_integral}. Within this work, we assume that $\bm \Gamma^{\alpha} = \bm I$ for $\alpha =1,\cdots, m$ at time $t=0$. Consequently, the relation \eqref{eq:quadratic_H_Q_def}$_{1}$ implies
\begin{align*}
\bm Q^{\alpha}_0 = \tilde{\bm S}^{\alpha}_{\mathrm{iso} \: 0} - \hat{\bm S}^{\alpha}_0, \quad \mbox{ with } \tilde{\bm S}^{\alpha}_{\mathrm{iso} \: 0} := \tilde{\bm S}^{\alpha}_{\mathrm{iso}}|_{t=0}.
\end{align*}
Therefore, the constant tensor $\hat{\bm S}^{\alpha}_{0}$ is determined by the initial values of the stresses,
\begin{align}
\label{eq:def_hat_S_alpha_0}
\hat{\bm S}^{\alpha}_0 := \tilde{\bm S}^{\alpha}_{\mathrm{iso} \: 0} - \bm Q^{\alpha}_{0},
\end{align}
which ensures the consistency of the initial condition. With \eqref{eq:non_equilibrium_stress_2} and \eqref{eq:quadratic_H_Q_def}, the non-equilibrium stresses for finite linear viscoelasticity are given by
\begin{align*}
\tilde{\bm S}^{\alpha}_{\mathrm{neq}} = 2 \left( \frac{\partial F_{\alpha}(\tilde{\bm C}, \Theta)}{\partial \tilde{\bm C}} - \frac{\partial \tilde{\bm S}^{\alpha}_{\mathrm{iso}}}{\partial \tilde{\bm C}} : \frac{\bm \Gamma^{\alpha} - \bm I}{2}\right) = 2 \frac{\partial F_{\alpha}(\tilde{\bm C}, \Theta)}{\partial \tilde{\bm C}} - \frac{1}{\mu^{\alpha}(\Theta)} \frac{\partial \tilde{\bm S}^{\alpha}_{\mathrm{iso}}}{\partial \tilde{\bm C}} : \left( \tilde{\bm S}^{\alpha}_{\mathrm{iso}} - \hat{\bm S}^{\alpha}_0 - \bm Q^{\alpha} \right).
\end{align*}
\begin{remark}
The evolution equations \eqref{eq:evolution-eqn-Q-linear-viscoelasticity} derived here are similar to those introduced and used in \cite{Simo1987,Holzapfel1996a,Holzapfel2001,Simo2006,Holzapfel1996,Gueltekin2016a}, which can be expressed as
\begin{align*}
\frac{d}{dt} \bm Q^{\alpha} + \frac{ \bm Q^{\alpha} }{\tau^{\alpha}} = \frac{d}{dt} \left( J^{-\frac23} \mathbb P : \tilde{\bm S}^{\alpha}_{\mathrm{iso}} \right) - 2 \frac{d\mu^{\alpha}}{dt} \bm \Gamma^{\alpha}.
\end{align*}
We emphasize that the evolution equations in those works were proposed based on a purely heuristic argument with inspiration coming from the standard linear solid model \cite{Valanis1972}. The right-hand side terms of the above evolution equations are different from those in \eqref{eq:evolution-eqn-Q-linear-viscoelasticity}, in which a deviatoric projected stress from $\tilde{\bm S}^{\alpha}_{\mathrm{iso}}$ was utilized to drive the evolution of $\bm Q^{\alpha}$. 
\end{remark}

\subsubsection{A model based on the identical polymer chain assumption}
\label{sec:IPC-model}
Now for a finite linear viscoelasticity model, considering the form for $H^{\alpha}$ given in Definition \ref{def:finite-linear-viscoelasticity}, the modeling work reduces to designing the form of $G^{\alpha}$ and $F^{\alpha}$ in \eqref{eq:configuration-free-energy}. In \cite{Govindjee1992,Holzapfel1996}, a model was proposed by taking
\begin{align*}
G^{\alpha}(\tilde{\bm C}, \Theta) = F^{\alpha}(\tilde{\bm C}, \Theta) = \beta^{\alpha}_{\infty} G^{\infty}_{\mathrm{iso}}(\tilde{\bm C}, \Theta),
\end{align*}
with $\beta^{\infty}_{\alpha} \in (0,\infty)$ being non-dimensional constants. This model is motivated by the observation that the phenomenological viscoelastic behavior is induced by a medium composed of identical polymer chains. With this assumption, the viscoelastic material is completely characterized by the potential $G^{\infty}_{\mathrm{iso}}$ together with non-dimensional parameters $\beta^{\infty}_{\alpha}$. The configurational free energy for this model can be rewritten as
\begin{align}
\label{eq:configuration-free-energy-IPC}
\Upsilon^{\alpha}\left(\tilde{\bm C}, \Theta, \bm \Gamma^1, \cdots, \bm \Gamma^m \right) = \mu^{\alpha}(\Theta) \left\lvert \frac{\bm \Gamma^{\alpha} - \bm I}{2} \right\rvert^2  + \left( \hat{\bm S}^{\alpha}_{0} - \beta^{\alpha}_{\infty} \tilde{\bm S}^{\infty}_{\mathrm{iso}} \right) : \frac{\bm \Gamma^{\alpha} - \bm I}{2} + \beta^{\alpha}_{\infty}G^{\infty}_{\mathrm{iso}}( \tilde{\bm C}, \Theta ).
\end{align}
The evolution equations can be expressed as
\begin{align}
\label{eq:evolution-eqn-Q-linear-viscoelasticity-IPC}
\frac{d}{dt} \bm Q^{\alpha} + \frac{ \bm Q^{\alpha} }{\tau^{\alpha}} = \beta^{\alpha}_{\infty}\frac{d}{dt}\tilde{\bm S}^{\infty}_{\mathrm{iso}} - \frac{d\mu^{\alpha}}{dt} \left(\bm \Gamma^{\alpha}-\bm I\right).
\end{align}
With $\bm Q^{\alpha}$ calculated, one may obtain the non-equilibrium stresses as
\begin{align}
\label{eq:S_neq_IPC_1}
\tilde{\bm S}^{\alpha}_{\mathrm{neq}} = \beta^{\alpha}_{\infty} \tilde{\bm S}^{\infty}_{\mathrm{iso}} - \frac{\beta^{\alpha}_{\infty}}{\mu^{\alpha}(\Theta)} \frac{\partial \tilde{\bm S}^{\infty}_{\mathrm{iso}}}{\partial \tilde{\bm C}} : \left( \beta^{\alpha}_{\infty} \tilde{\bm S}^{\infty}_{\mathrm{iso}} - \hat{\bm S}^{\alpha}_0 - \bm Q^{\alpha} \right),
\end{align}
with $\hat{\bm S}^{\alpha}_0 = \beta^{\alpha}_{\infty}\tilde{\bm S}^{\infty}_{\mathrm{iso} \: 0} - \bm Q^{\alpha}_{0}$ and $\tilde{\bm S}^{\infty}_{\mathrm{iso} \: 0} := \tilde{\bm S}^{\infty}_{\mathrm{iso}}|_{t=0}$. For convenience, we introduce the fictitious elasticity tensor $\tilde{\mathbb C}_{\mathrm{iso}}^{\infty}$ defined as
\begin{align}
\label{eq:def-tilde-mathbb-C-iso-infty}
\tilde{\mathbb C}_{\mathrm{iso}}^{\infty} := 2J^{-\frac43}\frac{\partial \tilde{\bm S}^{\infty}_{\mathrm{iso}}}{\partial \tilde{\bm C}} = 4 J^{-\frac43}\frac{\partial^2 G^{\infty}_{\mathrm{iso}}(\tilde{\bm C},\Theta)}{\partial \tilde{\bm C} \partial \tilde{\bm C}},
\end{align}
with which the non-equilibrium stresses \eqref{eq:S_neq_IPC_1} can be expressed as
\begin{align*}
\tilde{\bm S}^{\alpha}_{\mathrm{neq}} = \beta^{\alpha}_{\infty} \tilde{\bm S}^{\infty}_{\mathrm{iso}} - \frac{\beta^{\alpha}_{\infty}}{2\mu^{\alpha}(\Theta)} J^{\frac43} \tilde{\mathbb C}^{\infty}_{\mathrm{iso}} : \left( \beta^{\alpha}_{\infty} \tilde{\bm S}^{\infty}_{\mathrm{iso}} - \hat{\bm S}^{\alpha}_0 - \bm Q^{\alpha} \right).
\end{align*}
As will be revealed in Section \ref{subsubsec:condition-vanish-non-equilibrium-stress}, in the thermodynamic equilibrium limit, the non-equilibrium stresses given by this model will \textit{not} vanish in general. This behavior is counter-intuitive as the stress from the dissipative potential is expected to be fully relaxed in the limit from the material modeling perspective. In our numerical experiences, a non-vanishing $\bm S^{\alpha}_{\mathrm{neq}}$ often leads the body to deform to an unexpected state and sometimes gives unstable material behavior (see Figures \ref{fig:compare-IPC-MIPC} and \ref{fig:compare-IPC-MPIC-energy}). It is therefore necessary to discuss the condition characterizing the relaxation of the non-equilibrium stresses.

\subsubsection{Condition for the vanishment of the non-equilibrium stress $\bm S^{\alpha}_{\mathrm{neq}}$ in the thermodynamic equilibrium limit}
\label{subsubsec:condition-vanish-non-equilibrium-stress}
Intuitively, the non-equilibrium stress $\bm S^{\alpha}_{\mathrm{neq}}$ \eqref{eq:non_equilibrium_stress_2} shall vanish in the thermodynamic equilibrium limit. This property can be conveniently achieved in the multiplicative viscoelasticity theory \cite{Reese1998}. Without having the multiplicative decomposition of the deformation gradient, we need to characterize this property by analyzing necessary and sufficient conditions for the vanishment of $\bm S^{\alpha}_{\mathrm{neq}}$ in the equilibrium limit. First, the relaxation of the non-equilibrium stresses in the limit is given by
\begin{align*}
\bm O = \bm S^{\alpha}_{\mathrm{neq}} \Big|_{\mathrm{eq}} = \mathbb P : \tilde{\bm S}^{\alpha}_{\mathrm{neq}}\Big|_{\mathrm{eq}}.
\end{align*}
It can be proved that the above relation is equivalent to the vanishment of the fictitious stresses $\tilde{\bm S}^{\alpha}_{\mathrm{neq}}$,
\begin{align*}
\bm O = \tilde{\bm S}^{\alpha}_{\mathrm{neq}}\Big|_{\mathrm{eq}},
\end{align*}
with the proof detailed in \ref{appendix:an-analysis-of-the-null-space-of-P}. Recall that the equilibrium \cite{Simo2006} is characterized by,
\begin{align*}
\bm Q^{\alpha} \Big|_{\mathrm{eq}} = \bm O, \quad \mbox{or equivalently } \frac{d}{dt} \bm \Gamma^{\alpha} \Big|_{\mathrm{eq}} = \bm O, \quad \mbox{for } \alpha=1, \cdots, m.
\end{align*}
For finite linear viscoelastic materials, the above conditions imply that
\begin{align*}
\bm O = \bm Q^{\alpha} \Big|_{\mathrm{eq}} = \tilde{\bm S}^{\alpha}_{\mathrm{iso}} \Big|_{\mathrm{eq}} - \hat{\bm S}^{\alpha}_{0} - 2\mu^{\alpha}(\Theta) \frac{\bm \Gamma^{\alpha} - \bm I}{2} \Big|_{\mathrm{eq}},
\end{align*}
which is equivalent to
\begin{align*}
\frac{\bm \Gamma^{\alpha} - \bm I}{2} \Big|_{\mathrm{eq}} = \frac{1}{2\mu^{\alpha}(\Theta)} \left( \tilde{\bm S}^{\alpha}_{\mathrm{iso}} \Big|_{\mathrm{eq}} - \hat{\bm S}^{\alpha}_{0} \right).
\end{align*}
From \eqref{eq:non_equilibrium_stress_2} and the above relation, we know that in the equilibrium limit,
\begin{align*}
\tilde{\bm S}^{\alpha}_{\mathrm{neq}} \Big|_{\mathrm{eq}} &= 2 \left( \frac{\partial F^{\alpha}(\tilde{\bm C}, \Theta)}{\partial \tilde{\bm C}} \Big|_{\mathrm{eq}} - \frac{\partial \tilde{\bm S}^{\alpha}_{\mathrm{iso}}}{\partial \tilde{\bm C}} \Big|_{\mathrm{eq}} : \frac{\bm \Gamma^{\alpha} - \bm I}{2}\Big|_{\mathrm{eq}} \right) \\
&= 2 \left( \frac{\partial F^{\alpha}(\tilde{\bm C}, \Theta)}{\partial \tilde{\bm C}} \Big|_{\mathrm{eq}} - \frac{1}{2\mu^{\alpha}(\Theta)} \frac{\partial \tilde{\bm S}^{\alpha}_{\mathrm{iso}}}{\partial \tilde{\bm C}} \Big|_{\mathrm{eq}} : \left( \tilde{\bm S}^{\alpha}_{\mathrm{iso}} - \hat{\bm S}^{\alpha}_{0} \right)  \Big|_{\mathrm{eq}} \right) \\
&= 2 \frac{\partial}{\partial \tilde{\bm C}} \left( F^{\alpha}(\tilde{\bm C}, \Theta) - \frac{1}{4\mu^{\alpha}(\Theta)} \left\lvert \tilde{\bm S}^{\alpha}_{\mathrm{iso}} - \hat{\bm S}^{\alpha}_0 \right\lvert^2 \right) \Big|_{\mathrm{eq}}.
\end{align*}
The last equality in the above derivation makes use of the symmetry property of $\partial \tilde{\bm S}^{\alpha}_{\mathrm{iso}}/\partial \tilde{\bm C}$. Therefore, the \textit{necessary} condition for ensuring $\tilde{\bm S}^{\alpha}_{\mathrm{neq}} \Big|_{\mathrm{eq}} = \bm O$ is
\begin{align*}
\bm O = \frac{\partial}{\partial \tilde{\bm C}} \left( F^{\alpha}(\tilde{\bm C}, \Theta) - \frac{1}{4\mu^{\alpha}(\Theta)} \left\lvert \tilde{\bm S}^{\alpha}_{\mathrm{iso}} - \hat{\bm S}^{\alpha}_0 \right\lvert^2 \right) \Big|_{\mathrm{eq}},
\end{align*} 
or equivalently, after a partial integration with respect to $\tilde{\bm C}$,
\begin{align*}
F^{\alpha}(\tilde{\bm C}, \Theta) \Big|_{\mathrm{eq}} = \left( \frac{1}{4\mu^{\alpha}(\Theta)} \left\lvert \tilde{\bm S}^{\alpha}_{\mathrm{iso}} - \hat{\bm S}^{\alpha}_0 \right\lvert^2  + T^{\alpha}(\Theta) \right)\Big|_{\mathrm{eq}},
\end{align*}
wherein $T^{\alpha}$ is an arbitrary function of the temperature. The configurational free energy $\Upsilon^{\alpha}$ \eqref{eq:configuration-free-energy} for finite linear viscoelasticity in the equilibrium limit is
\begin{align}
\label{eq:configuration-free-energy-thermo-limit}
\Upsilon^{\alpha}\left(\tilde{\bm C}, \Theta, \bm \Gamma^1, \cdots, \bm \Gamma^m \right)\Big|_{\mathrm{eq}} =& H^{\alpha}(\bm \Gamma^{\alpha}, \Theta)\Big|_{\mathrm{eq}} + \left( \hat{\bm S}^{\alpha}_{0} - \tilde{\bm S}^{\alpha}_{\mathrm{iso}} \right) : \frac{\bm \Gamma^{\alpha} - \bm I}{2}\Big|_{\mathrm{eq}} + F^{\alpha}( \tilde{\bm C}, \Theta )\Big|_{\mathrm{eq}} \nonumber \displaybreak[2] \\
=& \left( \mu^{\alpha}(\Theta) \left\lvert \frac{\bm \Gamma^{\alpha} - \bm I}{2} \right\rvert^2 - 2\mu^{\alpha}(\Theta) \left\lvert \frac{\bm \Gamma^{\alpha} - \bm I}{2} \right\rvert^2 \right) \Big|_{\mathrm{eq}} \nonumber \displaybreak[2] \\
& + \left( \frac{1}{4\mu^{\alpha}(\Theta)} \left\lvert \tilde{\bm S}^{\alpha}_{\mathrm{iso}} - \hat{\bm S}^{\alpha}_0 \right\lvert^2  + T^{\alpha}(\Theta) \right)\Big|_{\mathrm{eq}} \nonumber \displaybreak[2] \\
=& -\mu^{\alpha}(\Theta) \left\lvert \frac{\bm \Gamma^{\alpha} - \bm I}{2} \right\rvert^2 \Big|_{\mathrm{eq}} + \mu^{\alpha}(\Theta) \left\lvert \frac{\bm \Gamma^{\alpha} - \bm I}{2} \right\rvert^2 \Big|_{\mathrm{eq}}  + T^{\alpha}(\Theta) \Big|_{\mathrm{eq}}  \nonumber \displaybreak[2] \\
=& T^{\alpha}(\Theta) \Big|_{\mathrm{eq}}.
\end{align}
This suggests that, for the finite linear viscoelastic model considered, the mechanical part of the configurational free energy will eventually vanish in the thermodynamic equilibrium limit, if we demand the non-equilibrium stresses to vanish in the limit. Based on the above discussion, a convenient modeling choice can be made by demanding that $F^{\alpha}$ always satisfies the following relation,
\begin{align}
\label{eq:constraint-F-alpha}
F^{\alpha}(\tilde{\bm C}, \Theta) = \frac{1}{4\mu^{\alpha}(\Theta)} \left\lvert \tilde{\bm S}^{\alpha}_{\mathrm{iso}} - \hat{\bm S}^{\alpha}_0 \right\lvert^2  + T^{\alpha}(\Theta),
\end{align}
from which we readily have
\begin{align}
\label{eq:tilde-S-alpha-neq-form}
\tilde{\bm S}^{\alpha}_{\mathrm{neq}} = 2 \frac{\partial \tilde{\bm S}^{\alpha}_{\mathrm{iso}}}{\partial \tilde{\bm C}} : \left( \frac{1}{2\mu^{\alpha}(\Theta)} \left( \tilde{\bm S}^{\alpha}_{\mathrm{iso}} - \hat{\bm S}^{\alpha}_{0} \right) - \frac{\bm \Gamma^{\alpha} - \bm I}{2} \right) = \frac{1}{\mu^{\alpha}(\Theta)} \frac{\partial \tilde{\bm S}^{\alpha}_{\mathrm{iso}}}{\partial \tilde{\bm C}} : \bm Q^{\alpha}.
\end{align}
Therefore, the non-equilibrium stresses $\bm S^{\alpha}_{\mathrm{neq}}$ automatically vanish when reaching the equilibrium limit, and the relation \eqref{eq:constraint-F-alpha} serves as a \textit{sufficient} condition for the configurational free energy that guarantees the stress relaxation. In this work, we adopt \eqref{eq:constraint-F-alpha} for the analytic form of $F^{\alpha}$, thereby leaving $G^{\alpha}$ as the only undetermined part of the configurational free energy for modelers. 



\begin{remark}
In the literature, the configurational free energy is often proposed with $F^{\alpha} = G^{\alpha}$ in \eqref{eq:configuration-free-energy} \cite{Simo2006,Holzapfel1996}. Based on the above analysis, the non-equilibrium stress arising from that configurational free energy will not vanish in the thermodynamic equilibrium limit, except for one special case to be discussed in the next section.
\end{remark}

\subsubsection{The Holzapfel-Simo-Saint Venant-Kirchhoff model}
\label{sec:HS-model}
For finite linear viscoelastic materials that satisfy the condition \eqref{eq:constraint-F-alpha}, the stresses are completely determined in terms of $G^{\alpha}$. Here, as an illustrative example, we consider a special form of $G^{\alpha}$,
\begin{align}
\label{eq:HS-model-G-alpha-form}
G^{\alpha}(\tilde{\bm C}, \Theta) = \mu^{\alpha}(\Theta) \left\lvert \frac{\tilde{\bm C} - \bm I}{2} \right\rvert^2 .
\end{align}
Based on \eqref{eq:Q-relation-linear-viscoelasticity}, one readily has
\begin{align*}
\tilde{\bm S}^{\alpha}_{\mathrm{iso}} = 2 \frac{\partial G^{\alpha}(\tilde{\bm C}, \Theta)}{\partial \tilde{\bm C}} = \mu^{\alpha}(\Theta) \left( \tilde{\bm C} - \bm I \right), \mbox{ and }
\bm Q^{\alpha} =  \tilde{\bm S}^{\alpha}_{\mathrm{iso}} - \hat{\bm S}^{\alpha}_{0} - \mu^{\alpha}(\Theta) \left( \bm \Gamma^{\alpha} - \bm I \right) = \mu^{\alpha}(\Theta) \left( \tilde{\bm C} - \bm \Gamma^{\alpha} \right) + \bm Q^{\alpha}_{0}.
\end{align*}
The evolution eqation \eqref{eq:evolution-eqn-Q-linear-viscoelasticity} can be written as
\begin{align*}
\frac{d}{dt} \bm Q^{\alpha} + \frac{ \bm Q^{\alpha} }{\tau^{\alpha}} =& \frac{d}{dt}\tilde{\bm S}^{\alpha}_{\mathrm{iso}} - \frac{d\mu^{\alpha}(\Theta)}{dt} \left( \bm \Gamma^{\alpha} - \bm I \right) = \mu^{\alpha}(\Theta) \frac{d}{dt} \tilde{\bm C} + \frac{d\mu^{\alpha}(\Theta)}{dt} \left( \tilde{\bm C} - \bm I \right)  - \frac{d\mu^{\alpha}(\Theta)}{dt} \left( \bm \Gamma^{\alpha} - \bm I \right) \\
=& \mu^{\alpha}(\Theta) \frac{d}{dt}\tilde{\bm C} + \frac{d\mu^{\alpha}(\Theta)}{dt} \left( \tilde{\bm C} - \bm \Gamma^{\alpha} \right).
\end{align*}
The fictitious non-equilibrium stress $\tilde{\bm S}^{\alpha}_{\mathrm{neq}}$ can be represented as
\begin{align*}
\tilde{\bm S}^{\alpha}_{\mathrm{neq}} = \frac{1}{\mu^{\alpha}(\Theta)} \frac{\partial \tilde{\bm S}^{\alpha}_{\mathrm{iso}}}{\partial \tilde{\bm C}} : \bm Q^{\alpha} = \frac{1}{\mu^{\alpha}(\Theta)} \mu^{\alpha}(\Theta) \bm Q^{\alpha} = \bm Q^{\alpha}.
\end{align*}
Consequently, we have the fictitious second Piola-Kirchhoff stress represented as
\begin{align*}
\tilde{\bm S} = \tilde{\bm S}^{\infty}_{\mathrm{iso}} + \sum_{\alpha=1}^{m} \tilde{\bm S}^{\alpha}_{\mathrm{neq}} = \tilde{\bm S}^{\infty}_{\mathrm{iso}} + \sum_{\alpha=1}^{m} \bm Q^{\alpha}.
\end{align*}
This recovers one specific model proposed by Holzapfel and Simo in \cite[Section~4.2]{Holzapfel1996}, which can be viewed as a generalization of the Saint Venant-Kirchhoff model to the viscous regime. This is the reason for its name used here. We should note that directly using $\bm Q^{\alpha}$ in place of $\tilde{\bm S}^{\alpha}_{\mathrm{neq}}$ is fairly common in the literature \cite{Holzapfel1996a,Gueltekin2016,Simo2006,Holzapfel2000}, and the above analysis shows that this choice implicitly implies a special form of the configurational free energy. We note that, to ensure $\tilde{\bm S}^{\alpha}_{\mathrm{neq}} = \bm Q^{\alpha}$, one needs
\begin{align*}
\frac{\partial \tilde{\bm S}^{\alpha}_{\mathrm{iso}}}{\partial \tilde{\bm C}} = 2 \frac{\partial G^{\alpha}(\tilde{\bm C},\Theta)}{\partial \tilde{\bm C} \otimes \tilde{\bm C}} = \mu^{\alpha}(\Theta) \mathbb I,
\end{align*}
which implies a quadratic form of $G^{\alpha}$ in terms of $\tilde{\bm C}$. Therefore, enforcing $\tilde{\bm S}^{\alpha}_{\mathrm{neq}} = \bm Q^{\alpha}$ implicitly requires $G^{\alpha}$ to be quadratic in terms of $\tilde{\bm C}$. This also suggests a linear relationship between $\tilde{\bm S}^{\alpha}_{\mathrm{iso}}$ and $\tilde{\bm C}$. For general nonlinear models, apparently one should distinguish $\bm Q^{\alpha}$ from $\tilde{\bm S}^{\alpha}_{\mathrm{neq}}$. Furthermore, according to \eqref{eq:constraint-F-alpha}, we have
\begin{align*}
F^{\alpha}(\tilde{\bm C}, \Theta) = \frac{1}{4\mu^{\alpha}(\Theta)} \left\lvert \tilde{\bm S}^{\alpha}_{\mathrm{iso}} - \hat{\bm S}^{\alpha}_0 \right\lvert^2  + T^{\alpha}(\Theta) = \frac{1}{4\mu^{\alpha}(\Theta)} \left\lvert \mu^{\alpha}(\Theta) \left( \tilde{\bm C} - \bm I \right) - \hat{\bm S}^{\alpha}_0 \right\lvert^2  + T^{\alpha}(\Theta).
\end{align*}
The configurational free energy \eqref{eq:configuration-free-energy} can be written as
\begin{align}
\label{eq:HS-model-configurational-energy-additive-split}
\Upsilon^{\alpha}\left(\tilde{\bm C}, \Theta, \bm \Gamma^1, \cdots, \bm \Gamma^m \right) =& H^{\alpha}(\bm \Gamma^{\alpha},\Theta) + \left( \hat{\bm S}^{\alpha}_0 - 2 \frac{\partial G^{\alpha}(\tilde{\bm C}, \Theta)}{\partial \tilde{\bm C}} \right) : \frac{\bm \Gamma^{\alpha} - \bm I}{2} + F^{\alpha}( \tilde{\bm C}, \Theta )  \nonumber \displaybreak[2] \\
=& \mu^{\alpha}(\Theta) \left\lvert \frac{\bm \Gamma^{\alpha} - \bm I}{2} \right\rvert^2 + \left( \hat{\bm S}^{\alpha}_0 - 2 \mu^{\alpha}(\Theta) \frac{\tilde{\bm C} - \bm I}{2} \right) : \frac{\bm \Gamma^{\alpha} - \bm I}{2}  \nonumber \displaybreak[2] \\
& + \frac{1}{4\mu^{\alpha}(\Theta)} \left\lvert \mu^{\alpha}(\Theta) \left( \tilde{\bm C} - \bm I \right) - \hat{\bm S}^{\alpha}_0 \right\lvert^2 + T^{\alpha}(\Theta) \nonumber \displaybreak[2] \\
=& \mu^{\alpha}(\Theta) \left\lvert \frac{\tilde{\bm C} - \bm \Gamma^{\alpha}}{2} \right\rvert^2 - \hat{\bm S}^{\alpha}_0 : \frac{ \tilde{\bm C} - \bm \Gamma^{\alpha}}{2} + \frac{1}{4\mu^{\alpha}(\Theta)} \left\lvert \hat{\bm S}^{\alpha}_0 \right\rvert^2 + T^{\alpha}(\Theta) \nonumber \displaybreak[2] \\
=& \frac{1}{4\mu^{\alpha}(\Theta)} \left\lvert \mu^{\alpha}(\Theta)\left( \tilde{\bm C} - \Gamma^{\alpha} \right) - \hat{\bm S}^{\alpha}_0 \right\rvert^2 + T^{\alpha}(\Theta).
\end{align}
One can see that for the model considered above, the configurational free energy is a function of $(\tilde{\bm C} - \bm \Gamma^{\alpha})/2$, which can be interpreted as the elastic strain measure in the standard solid model \cite[p.~286]{Holzapfel2000}, under an additive split of the strain. Although is less frequently used in inelasticity when compared with the multiplicative decomposition \cite{Reese1998,Lee1969}, the additive split has been shown to be supported by thermodynamic principles \cite{Green1965,Green1971} and has been utilized to model elastoplasticity \cite{Meng2002}. We may regard the Holzapfel-Simo-Saint Venant-Kirchhoff model as a viscoelasticity theory based on a similar additive strain split concept (see also \cite{Eidel2011}), in contrast to the multiplicative viscoelasticity theory \cite{Reese1998,Latorre2015,Latorre2016,Bonet2001,Liu2019c}.
\begin{remark}
It is convenient to make a rheological interpretation of this model. It can be seen that in this model $\tilde{\bm S}^{\alpha}_{\mathrm{iso},0} = \bm O$ as $\tilde{\bm C}|_{t=0} = \bm I$. If we choose $\bm Q^{\alpha}_0 = \bm O$, then $\hat{\bm S}^{\alpha}_{0} = \bm O$ according to \eqref{eq:def_hat_S_alpha_0}. We may introduce an ininitesimal strain $\bm \varepsilon$ as an approximation of $\left(\tilde{\bm C} - \bm I \right)/2$,
\begin{align*}
\bm \varepsilon \approx \frac{\tilde{\bm C} - \bm I}{2}.
\end{align*}
Similarily, there is a small-strain approximation of $\left(\bm \Gamma^{\alpha} - \bm I \right)/2$,
\begin{align*}
\bm \varepsilon^{\alpha}_{\mathrm{v}} \approx \frac{\tilde{\bm \Gamma}^{\alpha} - \bm I}{2}.
\end{align*}
If we denote $\bm \varepsilon^{\alpha}_e := \bm \varepsilon - \bm \varepsilon^{\alpha}_{\mathrm{v}}$, it can be shown that under the infinitesimal strain approximation, the configurational free energy \eqref{eq:HS-model-configurational-energy-additive-split} can be approximated by
\begin{align*}
\Upsilon^{\alpha} \approx \mu^{\alpha}(\Theta) \left\lvert \bm \varepsilon^{\alpha}_e \right\rvert^2 + T^{\alpha}(\Theta).
\end{align*}
This can be interpreted as the elastic energy of the springs within the Maxwell element of the standard Zener model.
\end{remark}

\subsubsection{A modified model based on the identical-polymer-chain assumption}
\label{sec:MIPC-model}
From the analysis in Section \ref{subsubsec:condition-vanish-non-equilibrium-stress}, we can see that the identical polymer chain model stated in Section \ref{sec:IPC-model} will engender counter-intuitive behavior of the non-equilibrium stress. In this section, we revisit the identical polymer chain model, by adopting the form \eqref{eq:constraint-F-alpha} for $F^{\alpha}$. The remaining component for the configurational free energy is given by
\begin{align}
\label{eq:G-alpha-beta-G-infty}
G^{\alpha}(\tilde{\bm C},\Theta) = \beta^{\infty}_{\alpha} G^{\infty}_{\mathrm{iso}}(\tilde{\bm C},\Theta),
\end{align}
with $\beta^{\infty}_{\alpha} \in (0, \infty)$ being non-dimensional constants \cite{Simo1987,Govindjee1992}. With this modeling assumption, the material behavior is completely characterized by the form of $G^{\infty}_{\mathrm{iso}}$, and we have
\begin{align*}
\tilde{\bm S}^{\alpha}_{\mathrm{iso}} = 2 \frac{\partial G^{\alpha}(\tilde{\bm C},\Theta)}{\partial \tilde{\bm C}} = \beta^{\infty}_{\alpha} 2 \frac{\partial G^{\infty}_{\mathrm{iso}}(\tilde{\bm C},\Theta)}{\partial \tilde{\bm C}} = \beta^{\infty}_{\alpha} \tilde{\bm S}^{\infty}_{\mathrm{iso}}.
\end{align*}
With this elasticity tensor, the non-equilibrium stress \eqref{eq:tilde-S-alpha-neq-form} can be represented as
\begin{align*}
\tilde{\bm S}^{\alpha}_{\mathrm{neq}} = \frac{1}{\mu^{\alpha}} \frac{\partial \tilde{\bm S}^{\alpha}_{\mathrm{iso}}}{\partial \tilde{\bm C}} : \bm Q^{\alpha} = \frac{J^{\frac43} \beta^{\infty}_{\alpha}}{2\mu^{\alpha}} \tilde{\mathbb C}^{\infty}_{\mathrm{iso}} : \bm Q^{\alpha} = \frac{2\beta^{\infty}_{\alpha}}{\mu^{\alpha}} \frac{\partial^2 G^{\infty}_{\mathrm{iso}}(\tilde{\bm C}, \Theta)}{\partial \tilde{\bm C} \partial \tilde{\bm C}} : \bm Q^{\alpha}.
\end{align*}

\begin{remark}
In the literature, the terminology ``ground-stress" is often used to refer the fictitious stress $\bm S^{\infty}_{\mathrm{iso}}$ \cite{Miehe2000,Miehe2005}; ``over-stress" is used to refer $\bm Q^{\alpha}$ \cite{Reese1998,Gasser2011} or $\bm S^{\alpha}_{\mathrm{neq}}$ \cite{Miehe2005}. The prevailing custom is to use $\bm Q^{\alpha}$ and $\tilde{\bm S}^{\alpha}_{\mathrm{neq}}$ interchangeably \cite{Simo1987,Govindjee1992,Holzapfel2000}. However, the above analysis indicates that, for a general nonlinear form of $G^{\alpha}$ like \eqref{eq:G-alpha-beta-G-infty}, $\tilde{\bm S}^{\alpha}_{\mathrm{neq}} \neq \bm Q^{\alpha}$.
\end{remark}

\begin{remark}
Interestingly, if the functional form of $G^{\infty}_{\mathrm{iso}}$ is given by the Neo-Hookean model, we have
\begin{align*}
\frac{\partial^2 G^{\infty}_{\mathrm{iso}}(\tilde{\bm C}, \Theta)}{\partial \tilde{\bm C} \partial \tilde{\bm C}} = \mathbb O,
\end{align*}
which indicates that $\bm S^{\alpha}_{\mathrm{neq}} = \bm O$. In other words, under the identical-polymer-chain assumption, the Neo-Hookean material cannot have viscous stress in finite linear viscoelasticity.
\end{remark}

We may summarize the thermodynamic potential and the resulting constitutive relations for finite linear thermal-visco-hyperelastic materials discussed in Section \ref{sec:continuum} as follows.

\vspace{3mm}

\begin{myenv}{Constitutive laws for finite linear thermal-visco-elasticity}
Gibbs free energy:
\begin{align*}
G(\tilde{\bm C}, P, \Theta, \bm \Gamma^1, \cdots, \bm \Gamma^m) =& G^{\infty}_{\textup{vol}}(P,\Theta) + G^{\infty}_{\textup{iso}}(\tilde{\bm C},\Theta) + \sum_{\alpha=1}^{m}\Upsilon^{\alpha}\left( \tilde{\bm C}, \Theta, \bm \Gamma^1, \cdots, \bm \Gamma^m \right), \\
\Upsilon^{\alpha}\left(\tilde{\bm C}, \Theta, \bm \Gamma^1, \cdots, \bm \Gamma^m \right) =& H^{\alpha}(\bm \Gamma^{\alpha},\Theta) - 2 \frac{\partial G^{\alpha}(\tilde{\bm C}, \Theta)}{\partial \tilde{\bm C}} : \frac{\bm \Gamma^{\alpha} - \bm I}{2} + \hat{\bm S}^{\alpha}_0 : \frac{\bm \Gamma^{\alpha} - \bm I}{2} + F^{\alpha}( \tilde{\bm C}, \Theta ), \\
H^{\alpha}(\bm \Gamma^{\alpha},\Theta) =& \mu^{\alpha}(\Theta) \left\lvert \frac{\bm \Gamma^{\alpha} - \bm I}{2} \right\rvert^2, \quad F^{\alpha}(\tilde{\bm C}, \Theta) = \frac{1}{4\mu^{\alpha}(\Theta)} \left\lvert \tilde{\bm S}^{\alpha}_{\mathrm{iso}} - \hat{\bm S}^{\alpha}_0 \right\lvert^2  + T^{\alpha}(\Theta).
\end{align*}
Second Piola-Kirchhoff stress:
\begin{align*}
& \bm S = \bm S_{\mathrm{iso}} + \bm S_{\mathrm{vol}}, \quad \bm S_{\mathrm{iso}} = J^{-\frac23} \mathbb P : \tilde{\bm S}, \quad \bm S_{\mathrm{vol}} = -J P \bm C^{-1}, \quad \tilde{\bm S} = \tilde{\bm S}_{\mathrm{iso}}^{\infty} + \sum_{\alpha =1}^{m} \tilde{\bm S}^{\alpha}_{\mathrm{neq}}, \displaybreak[2]  \\
& \tilde{\bm S}_{\mathrm{iso}}^{\infty} = 2 \frac{\partial G^{\infty}_{\mathrm{iso}}}{\partial \tilde{\bm C}}, \quad \tilde{\bm S}^{\alpha}_{\mathrm{neq}} = \frac{1}{\mu^{\alpha}(\Theta)} \frac{\partial \tilde{\bm S}^{\alpha}_{\mathrm{iso}}}{\partial \tilde{\bm C}} : \bm Q^{\alpha}, \quad \tilde{\bm S}^{\alpha}_{\mathrm{iso}} = 2 \frac{\partial G^{\alpha}}{\partial \tilde{\bm C}}, \quad \tilde{\mathbb C}_{\mathrm{iso}}^{\infty} = 4 J^{-\frac43}\frac{\partial^2 G^{\infty}_{\mathrm{iso}}}{\partial \tilde{\bm C} \partial \tilde{\bm C}}.
\end{align*}
Cauchy stress:
\begin{align*}
\bm \sigma = \bm \sigma_{\mathrm{dev}} - p \bm I, \quad \bm \sigma_{\mathrm{dev}} = \frac{1}{J} \bm F \bm S_{\mathrm{iso}} \bm F^{T}.
\end{align*}
Entropy:
\begin{align*}
\mathfrak S = \mathfrak S^{\infty} + \sum_{\alpha=1}^{m} \mathfrak S^{\alpha}, \quad \mathfrak S^{\infty} := -\frac{\partial G^{\infty}_{\mathrm{vol}}}{\partial \Theta} -\frac{\partial G^{\infty}_{\mathrm{iso}}}{\partial \Theta}, \quad \mathfrak S^{\alpha} := -\frac{\partial \Upsilon^{\alpha}}{\partial \Theta}.
\end{align*}
Dissipation:
\begin{align*}
\mathcal D = \frac{\bar{\kappa}}{\Theta} |\nabla_{\bm X} \Theta |^2 + \sum_{\alpha=1}^{m} \frac{\eta^{\alpha}}{2\Theta} \left\lvert \frac{d}{dt}\bm \Gamma^{\alpha}\right\rvert^2 = \frac{\bar{\kappa}}{\Theta} |\nabla_{\bm X} \Theta |^2 + \sum_{\alpha=1}^{m} \frac{1}{2\Theta \eta^{\alpha}} \left\lvert \bm Q^{\alpha}\right\rvert^2 .
\end{align*}
Evolution equations for $\bm Q^{\alpha}$:
\begin{align*}
\begin{dcases}
& \frac{d}{dt} \bm Q^{\alpha} + \frac{1}{\tau^{\alpha}} \bm Q^{\alpha} = \frac{d}{dt} \tilde{\bm S}^{\alpha}_{\mathrm{iso}} - \frac{d\mu^{\alpha}(\Theta)}{dt}\left(\bm \Gamma^{\alpha} - \bm I \right), \\
& \bm Q^{\alpha}|_{t=0} = \bm Q^{\alpha}_0.
\end{dcases}
\end{align*}
Hereditary integral form for $\bm Q^{\alpha}$:
\begin{align*}
\bm Q^{\alpha} =& \exp\left( - t / \tau^{\alpha} \right) \bm Q^{\alpha}_0 + \int^t_{0^+} \exp\left( -(t-s)/\tau^{\alpha} \right) \left( \frac{d}{ds} \tilde{\bm S}^{\alpha}_{\mathrm{iso}} - \left(\bm \Gamma^{\alpha} - \bm I \right) \frac{d\mu^{\alpha}(\Theta)}{ds} \right) ds.
\end{align*}
\end{myenv}

\section{Numerical formulation}
\label{sec:numerical_formulation}
In this section, we develop a suite of discretization methods for the material model developed in Section \ref{sec:continuum} by restricting our discussion to \textit{fully incompressible} materials under the \textit{isothermal} condition. We will prove that the proposed numerical formulation is stable in energy. Henceforth, we use the acronyms IPC, HS, and MIPC models to stand for the identical polymer chain model stated in Section \ref{sec:IPC-model}, the Holzapfel-Simo-Saint Venant-Kirchhoff model stated in Section \ref{sec:HS-model}, and the modified identical polymer chain model stated in Section \ref{sec:MIPC-model}, respectively.

\subsection{Strong-form problem}
Under isothermal conditions, the energy equation is decoupled from the mechanical system, and the motion of an incompressible continuum body is governed by the following system of equations,
\begin{align}
\label{eq:strong-form-current-kinematic}
& \bm 0 = \frac{d\bm u}{dt} - \bm v, && \mbox{ in } \Omega_{\bm x}^t, \\
& 0 = \nabla_{\bm x} \cdot \bm v, && \mbox{ in } \Omega_{\bm x}^t, \\
\label{eq:strong-form-current-momentum}
& \bm 0 = \rho(J) \frac{d\bm v}{dt} - \nabla_{\bm x} \cdot \bm \sigma_{\mathrm{dev}} + \nabla_{\bm x} p - \rho(J) \bm b, && \mbox{ in } \Omega_{\bm x}^t.
\end{align}
The boundary $\Gamma_{\bm x}^t := \partial \Omega^t_{\bm x}$ enjoys a non-overlapping subdivision: $\Gamma_{\bm x}^t = \Gamma_{\bm x}^{g,t} \cup \Gamma_{\bm x}^{h,t}$, wherein $\Gamma_{\bm x}^{g,t}$ and $\Gamma_{\bm x}^{h,t}$ are the Dirichlet and Neumann parts of the boundary, respectively. The boundary conditions can be stated as
\begin{align}
\label{eq:strong-form-current-bc}
& \bm u = \bm g, \mbox{ on } \Gamma_{\bm x}^{g,t}, \qquad
\bm v = \frac{d\bm g}{dt}, \mbox{ on } \Gamma_{\bm x}^{g,t}, \qquad (\bm \sigma_{\mathrm{dev}} - p\bm I) \bm n = \bm h, \mbox{ on } \Gamma_{\bm x}^{h,t}. 
\end{align}
Given the initial data $\bm u_0$, $p_0$, and $\bm v_0$, the initial-boundary value problem is to solve for $\bm u$, $p$, and $\bm v$ that satisfy \eqref{eq:strong-form-current-kinematic}-\eqref{eq:strong-form-current-bc}, and
\begin{align}
\label{eq:initial_condition_v}
\bm u(\bm x, 0) = \bm u_0(\bm x), \qquad p(\bm x, 0) = p_0(\bm x), \qquad \bm v(\bm x, 0) = \bm v_0(\bm x).
\end{align}
The equations \eqref{eq:strong-form-current-kinematic}-\eqref{eq:initial_condition_v} constitute an initial-boundary value problem. The constitutive relations for the material has been given in Section \ref{sec:continuum}. For fully incompressible materials, it can be shown that the volumetric part of the free energy adopts the form $G^{\infty}_{\mathrm{vol}} = P / \rho_0$ \cite{Liu2018}, which leads to $\rho = \rho_0$. In the formulation, we adopt a modified constitutive relation for the density, that is $\rho(J)= \rho_0 / J$ \cite{Liu2019a}. Apparently, at the continuum level, this relation is equivalent to $\rho(J) = \rho_0$, as the divergence-free condition for the velocity guarantees $J=1$. Yet, we note that at the discrete level, the condition $J=1$ in general rarely holds in the pointwise sense \cite{Auricchio2013}. The adoption of the relation $\rho(J)=\rho_0 / J$ facilitates our discussion of the nonlinear numerical stability, as will be shown in the next section.

\subsection{Spline spaces}
Before introducing the semi-discrete formulation, we state the construction of B-splines and NURBS basis functions, which are utilized to build the discrete function spaces. Given the polynomial degree $\mathsf p$ and the dimensionality of the B-spline space $\mathsf n$, the knot vector can be represented by $\Xi := \left\{\xi_1, \xi_2 \cdots, \xi_{\mathsf n + \mathsf p + 1} \right\}$, wherein $0=\xi_1 \leq \xi_2 \leq \cdots \leq \xi_{\mathsf n + \mathsf p + 1}=1$.
The B-spline basis functions of degree $\mathsf p$, denoted as $\mathsf N_{i}^{\mathsf p}(\xi)$, for $i=1,\cdots, \mathsf n$, can then be defined recursively from the knot vector using the Cox-de Boor recursion formula \cite[Chapter~2]{Hughes2005}. The definition starts with the case of $\mathsf p=0$, where the basis functions are defined as piecewise constants,
\begin{align*}
\mathsf N_{i}^0(\xi) = 
\begin{cases}
1 & \mbox{ if } \xi_{i} \leq \xi < \xi_{i+1}, \\
0 & \mbox{ otherwise}.
\end{cases}
\end{align*}
For $\mathsf p\geq 1$, the basis functions are defined recursively as
\begin{align*}
\mathsf N_{i}^{\mathsf p}(\xi)= \frac{\xi - \xi_{i}}{\xi_{i+ \mathsf p} - \xi_{i}} \mathsf N_{i}^{\mathsf p-1}(\xi) + \frac{\xi_{i+\mathsf p+1} - \xi}{\xi_{i+ \mathsf p+1} - \xi_{i+1}} \mathsf N_{i+1}^{\mathsf p-1}(\xi).
\end{align*} Given a set of weights $\{\mathsf w_1, \mathsf w_2, \cdots , \mathsf w_{\mathsf n} \}$, the NURBS basis functions of degree $\mathsf p$ can be defined as
\begin{align*}
\mathsf R_i^{\mathsf p}(\xi) := \frac{\mathsf w_i \mathsf N^{\mathsf p}_i(\xi)}{\mathsf W(\xi)}, \quad \mbox{ and } \quad \mathsf W(\xi) := \sum\limits_{j=1}^{\mathsf n}\mathsf w_j \mathsf N^{\mathsf p}_j(\xi).
\end{align*}
Importantly, the knots can also be represented with two vectors, one of the unique knots $\{\zeta_1, \zeta_2, \cdots, \zeta_{\mathsf m} \}$ and another of the corresponding knot multiplicities $\{r_1, r_2, \cdots, r_{\mathsf m}\}$. As is standard in the literature of computer-aided design, we consider open knot vectors in this work, meaning that $r_1 = r_{\mathsf m} = \mathsf p+1$. We further assume that $r_i \leq \mathsf p$ for $i=2,\cdots, \mathsf m-1$. Across any given knot $\zeta_i$, the B-spline basis functions have $\alpha_i := \mathsf p - r_i$ continuous derivatives. The vector $\bm \alpha := \{ \alpha_1, \alpha_2, \cdots, \alpha_{\mathsf m-1}, \alpha_m \} = \{ -1, \alpha_2, \cdots, \alpha_{\mathsf m-1}, -1 \}$ is referred to as the regularity vector. A value of $-1$ for $\alpha_i$ indicates discontinuity of the basis functions at $\zeta_i$. We introduce the function space $\mathcal R^{\mathsf p}_{\bm \alpha} := \textup{span}\{\mathsf R_{i}^{\mathsf p}\}_{i=1}^n$, where the notation $\mathcal R^{\mathsf p}_{\alpha}$ is used to indicate that $\alpha_i = \alpha$ for $i=2,\cdots, \mathsf m-1$, suggesting continuity $C^{\alpha}$ for the spline function spaces. The construction of multivariate B-spline and NURBS basis functions follows a tensor-product manner. For $l=1, 2, 3$, given $\mathsf p_{l}$, $\mathsf n_{l}$, and the knot vectors $\Xi_{l} = \{\xi_{1,l}, \xi_{2,l}, \cdots , \xi_{\mathsf n_{l}+\mathsf p_{l}+1, l}\}$, the univariate B-spline basis functions $\mathsf N^{\mathsf p_l}_{i_l,l}$ are well-defined. Consequently, the multivariate B-spline basis functions can be defined by exploiting the tensor product structure,
\begin{align*}
\mathsf N^{\mathsf p_1, \mathsf p_2 \mathsf p_3}_{i_1, i_2, i_3}(\xi_1,\xi_2 ,\xi_3) := \mathsf N^{\mathsf p_1}_{i_1, 1}(\xi_1) \otimes \mathsf N^{\mathsf p_2}_{i_2, 2}(\xi_2) \otimes \mathsf N^{\mathsf p_{3}}_{i_{3}, 3}(\xi_3), \mbox{ for } i_l = 1,2, \cdots, \mathsf n_l \mbox{ and } l = 1, 2, 3.
\end{align*}
Given a set of weights $\{\mathsf w_{i_{1}, i_{2}, i_{3}}\}$, the NURBS basis functions are defined by
\begin{align*}
\mathsf R^{\mathsf p_1, \mathsf p_2, \mathsf p_3}_{i_1, i_2, i_3}(\xi_1,\xi_2 ,\xi_3) := \frac{\mathsf w_{i_1, i_2, i_3} \mathsf N^{\mathsf p_1, \mathsf p_2 \mathsf p_3}_{i_1, i_2, i_3}(\xi_1,\xi_2 ,\xi_3)}{\mathsf W(\xi_1,\xi_2 ,\xi_3)}, \quad \mathsf W(\xi_1, \xi_2, \xi_3) := \sum_{l=1}^3 \sum_{i_l=1}^{\mathsf n_l} \mathsf w_{i_1, i_2, i_3} \mathsf N^{\mathsf p_1, \mathsf p_2 \mathsf p_3}_{i_1, i_2, i_3}(\xi_1,\xi_2 ,\xi_3).
\end{align*}
Correspondingly, the NURBS function space is defined as
\begin{align*}
\mathcal R^{\mathsf p_1, \mathsf p_2, \mathsf p_3}_{\bm \alpha_1, \bm \alpha_2, \bm \alpha_3} :=\textup{span}\{ \mathsf R^{\mathsf p_1, \mathsf p_2, \mathsf p_3}_{i_1, i_2, i_3} \}_{i_1=1, i_2=1, i_3 =1}^{\mathsf n_1, \mathsf n_2, \mathsf n_3}.
\end{align*}

\begin{figure}
\begin{center}
	\begin{tabular}{c}
		\includegraphics[angle=0, trim=180 220 130 200, clip=true, scale = 0.38]{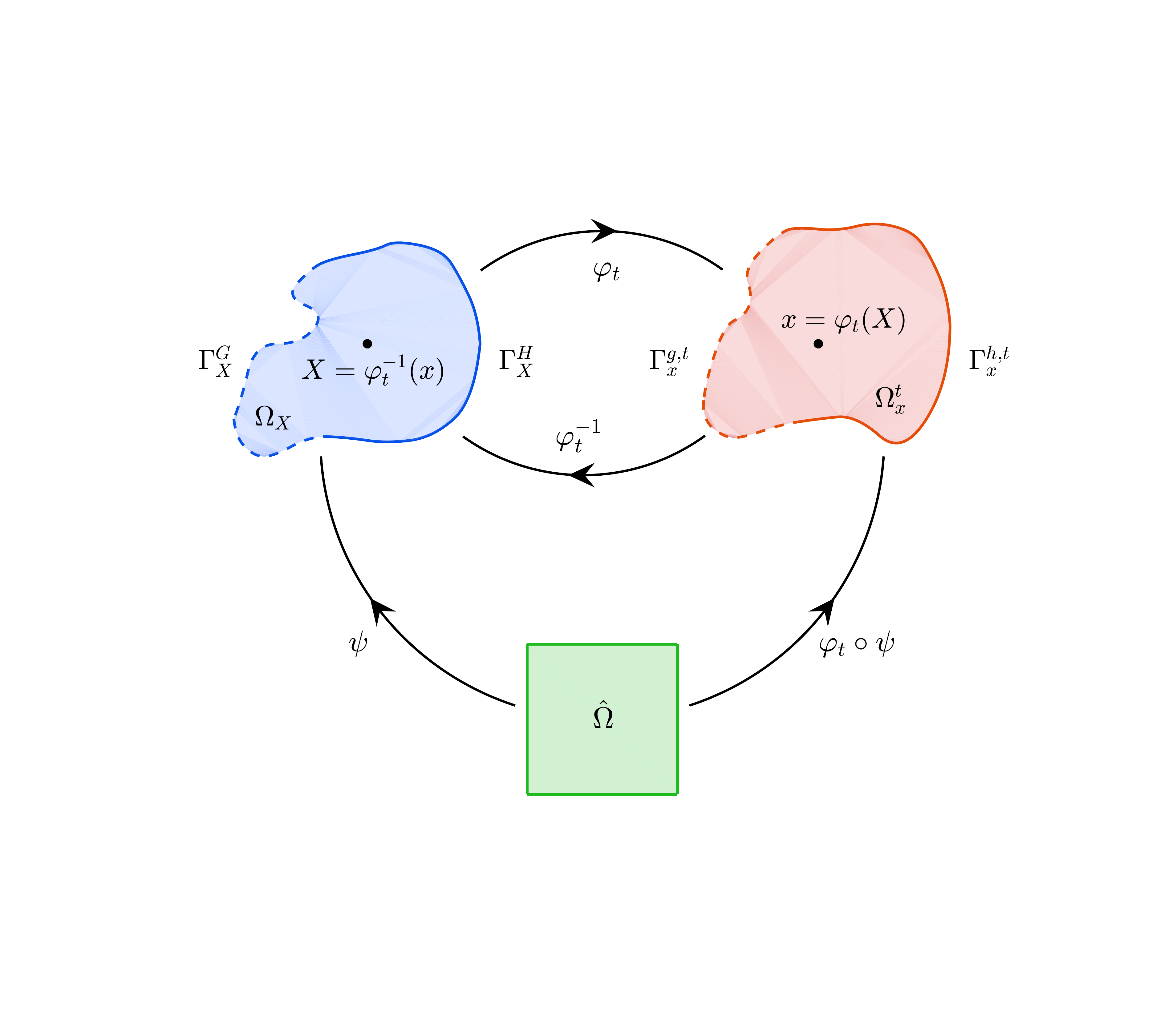} 
	\end{tabular}
	\caption{Illustration of the referential and current configurations with the boundary subdivisions. The spline functions are defined on $\hat{\Omega}$ and are pushed forward to the two configurations via the mapping $\psi$ and $\varphi_t \circ \psi$.} 
	\label{fig:continuum-problem-mapping}
\end{center}
\end{figure}

\subsection{Semi-discrete formulation}
\label{subsec:semi_discrete_formulation}
We first define two discrete function spaces on $\hat{\Omega}:=(0,1)^3$,
\begin{align*}
\hat{\mathcal S}_h :=
\mathcal R^{\mathsf p+\mathsf a, \mathsf p+\mathsf a, \mathsf p+\mathsf a}_{\bm \alpha_1+\mathsf b, \bm \alpha_2+\mathsf b, \bm \alpha_3+\mathsf b} \times \mathcal R^{\mathsf p+\mathsf a, \mathsf p+\mathsf a, \mathsf p+\mathsf a}_{\bm \alpha_1+\mathsf b, \bm \alpha_2+\mathsf b, \bm \alpha_3+\mathsf b} \times \mathcal R^{\mathsf p+\mathsf a, \mathsf p+\mathsf a, \mathsf p+\mathsf a}_{\bm \alpha_1+\mathsf b, \bm \alpha_2+\mathsf b, \bm \alpha_3+\mathsf b}, \quad
\hat{\mathcal P}_h := \mathcal R^{\mathsf p, \mathsf p, \mathsf p}_{\bm \alpha_1, \bm \alpha_2, \bm \alpha_3},
\end{align*}
with integer parameters $1 \leq \mathsf a $ and $0 \leq \mathsf b \leq \mathsf a$. Assuming the referential configuration of the body can be parametrized by a geometrical mapping $\bm \psi : \hat{\Omega} \rightarrow \Omega_{\bm X}$. The boundary of $\Omega_{\bm X}$ can be partitioned into two non-overlapping subdivisions as $\partial \Omega_{\bm x} := \Gamma^{G}_{\bm X} \cup \Gamma^{H}_{\bm X}$, and the two subdivisions satisfy $\Gamma^{G}_{\bm X} = \bm \varphi_t^{-1}\left(\Gamma^{g,t}_{\bm x} \right)$ and $\Gamma^{H}_{\bm X} = \bm \varphi_t^{-1}\left(\Gamma^{h,t}_{\bm x} \right)$. The relation between the two configurations and the boundary subdivisions are illustrated in Figure \ref{fig:continuum-problem-mapping}. The discrete function spaces on $\Omega_{\bm X}$ can be defined through the pull-back operation,
\begin{align*}
& \mathcal S_h :=  \{ \bm w : \bm w \circ \bm \psi \in \hat{\mathcal S}_h \}, \quad \mathcal P_h := \{ q : q \circ \bm \psi \in \hat{\mathcal P}_h \}.
\end{align*}
With the spaces $\mathcal S_h$ and $\mathcal P_h$ defined, we may specify the trial solution spaces on the referential configuration as
\begin{align*}
\mathcal S_{\bm U_h} &= \Big\lbrace \bm U_h : \bm U_h(\cdot, t) \in \mathcal S_h, t \in [0,T], \quad \bm U_h(\cdot,t) = \bm G \mbox{ on } \Gamma_{\bm X}^{G}  \Big\rbrace , \displaybreak[2] \\
\mathcal S_{P_h} &= \Big\lbrace P_h : P_h(\cdot, t) \in \mathcal P_h, t \in [0,T] \Big\rbrace , \displaybreak[2] \\
 \mathcal S_{\bm V_h} &= \left\lbrace \bm V_h : \bm V_h(\cdot, t) \in \mathcal S_h, t \in [0,T], \quad \bm V_h(\cdot,t) = \frac{d\bm G}{dt} \mbox{ on } \Gamma_{\bm X}^{G}  \right\rbrace ,
\end{align*}
and the corresponding test function spaces are defined as
\begin{align*}
\mathcal V_{P_h} &= \Big\lbrace Q_h : Q_h(\cdot, t) \in \mathcal P_h, t \in [0,T] \Big\rbrace, \qquad
\mathcal V_{\bm V_h} = \left\lbrace \bm W_h : \bm W_h(\cdot, t) \in \mathcal S_h, t \in [0,T], \quad \bm W_h(\cdot,t) = \bm 0 \mbox{ on } \Gamma_{\bm X}^{G}  \right\rbrace.
\end{align*}
Given the displacement $\bm U_h$, the placement field is given by $\bm \varphi_h = \bm U_h + \bm X$. Consequently, we may also state the trial solution space defined on the current configuration as
\begin{align*}
\mathcal S_{\bm u_h} &= \Big\lbrace \bm u_h : \bm u_h \circ \bm \varphi_h \in \mathcal S_h, t \in [0,T], \bm u_h(\cdot,t) = \bm g \mbox{ on } \Gamma_{\bm x}^{g,t}  \Big\rbrace , \displaybreak[2] \\
\mathcal S_{p_h} &= \Big\lbrace p_h : p_h \circ \bm \varphi_h \in \mathcal P_h, t \in [0,T] \Big\rbrace , \displaybreak[2] \\
 \mathcal S_{\bm v_h} &= \left\lbrace \bm v_h : \bm v_h \circ \bm \varphi_h \in \mathcal S_h, t \in [0,T], \bm v_h(\cdot,t) = \frac{d\bm g}{dt} \mbox{ on } \Gamma_{\bm x}^{g,t}  \right\rbrace ,
\end{align*}
and the test function spaces are defined as
\begin{align*}
\mathcal V_{p_h} &= \Big\lbrace q_h : q_h \circ \bm \varphi_h \in \mathcal P_h, t \in [0,T]  \Big\rbrace , \quad  \mathcal V_{\bm v_h} = \Big\lbrace \bm w_h : \bm w_h \circ \bm \varphi_h \in \mathcal S_h, t \in [0,T], \bm w_h(\cdot,t) = \bm 0 \mbox{ on } \Gamma_{\bm x}^{g,t} \Big\rbrace .
\end{align*}
With the discrete function defined above, the semi-discrete formulation on the current configuration can be stated as follows. Find $\bm y_h(t) := \left\lbrace  \bm u_h(t), p_h(t), \bm v_h(t)\right\rbrace^T \in \mathcal S_{\bm u_h} \times \mathcal S_{p_h} \times \mathcal S_{\bm v_h}$ such that for $t\in [0, T]$,
\begin{align}
\label{eq:kinematics_current}
& \bm 0 = \mathbf B^k\left( \dot{\bm y}_h, \bm y_h  \right) := \frac{d\bm u_h}{dt} - \bm v_h, \displaybreak[2]\\
\label{eq:mass_current}
& 0 = \mathbf B^p\left( q_h; \dot{\bm y}_h, \bm y_h  \right) := \int_{\Omega_{\bm x}^t} q_h \nabla_{\bm x} \cdot \bm v_h d\Omega_{\bm x}, \displaybreak[2] \\
\label{eq:momentum_current}
& 0 = \mathbf B^m\left( \bm w_{h}; \dot{\bm y}_h, \bm y_h  \right) := \int_{\Omega_{\bm x}^t} \bm w_{h} \cdot \rho(J_h) \frac{d\bm v_h}{dt} + \nabla_{\bm x} \bm w_{h} : \bm \sigma_{\mathrm{dev}} - \nabla_{\bm x} \cdot \bm w_{h} p_h - \bm w_{h} \cdot \rho(J_h)  \bm b d\Omega_{\bm x}  \nonumber \\
& \hspace{3.7cm}  - \int_{\Gamma_{\bm x}^{h,t}} \bm w_{h} \cdot \bm h d\Gamma_{\bm x}, 
\end{align}
for $\forall \left\lbrace q_h, \bm w_{h}\right\rbrace \in \mathcal V_{p_h} \times \mathcal V_{\bm v_h}$, with $\bm y_h(0) := \left\lbrace \bm u_{h0}, p_{h0}, \bm v_{h0} \right\rbrace^T$. Here $\bm u_{h0}$, $p_{h0}$, and $\bm v_{h0}$ are the $\mathcal L^2$ projections of the initial data onto the finite dimensional trial solution spaces. Alternatively, the semi-discrete formulation can be pulled back to the referential configuration, which can be stated as follows. Find $\bm Y_h(t) := \left\lbrace \bm U_h(t), P_h(t), \bm V_h(t)\right\rbrace^T \in \mathcal S_{\bm U_h} \times \mathcal S_{P_h} \times \mathcal S_{\bm V_h}$ such that for $t\in [0, T]$,
\begin{align}
\label{eq:mix_solids_kinematics_ref}
& \bm 0 = \mathbf B^k\left( \dot{\bm Y}_h, \bm Y_h \right) :=  \frac{d\bm U_h}{dt} - \bm V_h, \displaybreak[2]\\
\label{eq:mix_solids_mass_ref}
& 0 = \mathbf B^p\left( Q_h; \dot{\bm Y}_h, \bm Y_h \right) := \int_{\Omega_{\bm X}} Q_h J_h \nabla_{\bm X} \bm V_h : \bm F^{-T}_h d\Omega_{\bm X}, \displaybreak[2] \\
\label{eq:mix_solids_momentum_ref}
& 0 = \mathbf B^m\left( \bm W_h; \dot{\bm Y}_h, \bm Y_h \right) := \int_{\Omega_{\bm X}} \Big( \bm W_h \cdot \rho_0 \frac{d\bm V_h}{dt} + \nabla_{\bm X} \bm W_h : \left( J_h \bm \sigma_{\mathrm{dev}} \bm F^{-T}_h\right) - J_h P_h \nabla_{\bm X} \bm W_h : \bm F^{-T}_h  \displaybreak[2] \nonumber \\
& \hspace{3.9cm} - \bm W_h \cdot \rho_0  \bm B \Big)  d\Omega_{\bm X} - \int_{\Gamma_{\bm X}^{H}} \bm W_h \cdot \bm H d\Gamma_{\bm X}, 
\end{align}
for $\forall \left\lbrace  Q_h, \bm W_h \right\rbrace \in \mathcal V_{P_h} \times \mathcal V_{\bm V_h}$, with $\bm Y_h(0) = \left\lbrace \bm U_{h0}, P_{h0}, \bm V_{h0} \right\rbrace^T$. Here the initial data are related by $\bm U_{h0} = \bm u_{h0} \circ \bm \varphi_t$, $P_{h0} = p_{h0} \circ \bm \varphi_t$, and $\bm V_{h0} = \bm v_{h0} \circ \bm \varphi_t$. It can be shown that the above semi-discrete formulation inherits the dissipation property from the continuum model from the following theorem.

\begin{theorem}[A priori energy stability]
\label{the:energy-stability}
Assuming the Dirichlet boundary data $\bm G$ is time independent, the solutions of the semi-discrete problem \eqref{eq:mix_solids_kinematics_ref}-\eqref{eq:mix_solids_momentum_ref} satisfy
\begin{align}
\label{eq:weak-form-energy-stability}
\frac{d}{dt}\int_{\Omega_{\bm X}} \frac12 \rho_0 \|\bm V_h\|^2 +  G^{\infty}_{\textup{iso}} ( \tilde{\bm C}_h ) + \displaystyle \sum_{\alpha=1}^{m}\Upsilon^{\alpha}(\tilde{\bm C}_h, \bm \Gamma^{\alpha}_h)  d\Omega_{\bm X} =& \int_{\Omega_{\bm X}} \rho_0 \bm V_h \cdot \bm B d\Omega_{\bm X} + \int_{\Gamma_{\bm X}} \bm V_h \cdot \bm H d\Gamma_{\bm X} \displaybreak[2] \nonumber \\
& - \sum_{\alpha=1}^{m} \int_{\Omega_{\bm X}}\frac14 \left(\frac{d}{dt}\bm \Gamma^{\alpha}_h \right) : \mathbb V^{\alpha} : \left(\frac{d}{dt}\bm \Gamma^{\alpha}_h \right) d\Omega_{\bm X}.
\end{align}
\end{theorem}

\begin{proof}
Since the Dirichlet boundary data $\bm G$ is independent of time, one is allowed to choose $Q_h = P_h$ and $\bm W_h = \bm V_h$, which leads to
\begin{align*}
0 =& \mathbf B^p\left( P_h; \dot{\bm Y}_h, \bm Y_h \right) + \mathbf B^m\left( \bm V_h; \dot{\bm Y}_h, \bm Y_h \right) \nonumber \displaybreak[2] \\
=& \int_{\Omega_{\bm X}} P_h J_h \nabla_{\bm X} \bm V_h : \bm F^{-T}_h d\Omega_{\bm X} + \int_{\Omega_{\bm X}} \bm V_h \cdot \rho_0 \frac{d\bm V_h}{dt} + \nabla_{\bm X} \bm V_h : \left( J_h \bm \sigma_{\textup{dev}} \bm F_h^{-T} \right) - J_h P_h\nabla_{\bm X} \bm V_h : \bm F^{-T}_h  \nonumber \displaybreak[2] \\
&  - \bm V_h \cdot \rho_0  \bm B d\Omega_{\bm x} - \int_{\Gamma_{\bm X}^{H}} \bm V_h \cdot \bm H d\Gamma_{\bm x} \displaybreak[2] \\
=& \frac{d}{dt}\int_{\Omega_{\bm X}} \frac12 \rho_0 \|\bm V_h\|^2 d\Omega_{\bm X} + \int_{\Omega_{\bm X}} \frac{d}{dt}\bm F_h : \frac{\partial \left( G^{\infty}_{\textup{iso}} ( \tilde{\bm C}_h ) + \displaystyle \sum_{\alpha=1}^{m}\Upsilon^{\alpha}(\tilde{\bm C}_h, \bm \Gamma^{\alpha}_h) \right)}{\partial \bm F_h}- \bm V_h \cdot \rho_0  \bm B d\Omega_{\bm X} \nonumber \displaybreak[2]\\
& - \int_{\Gamma_{\bm X}} \bm V_h \cdot \bm H d\Gamma_{\bm X} \displaybreak[2] \\
=&  \frac{d}{dt}\int_{\Omega_{\bm X}} \frac12 \rho_0 \|\bm V_h\|^2 d\Omega_{\bm X} + \frac{d}{dt} \int_{\Omega_{\bm X}} \left( G^{\infty}_{\textup{iso}} ( \tilde{\bm C}_h ) + \displaystyle \sum_{\alpha=1}^{m}\Upsilon^{\alpha}(\tilde{\bm C}_h, \bm \Gamma^{\alpha}_h) \right) d\Omega_{\bm X} + \int_{\Omega_{\bm X}} \sum_{\alpha=1}^{m} \bm Q^{\alpha}_h : \frac12 \frac{d}{dt}\bm \Gamma^{\alpha}_h \displaybreak[2] \\
& - \bm V_h \cdot \rho_0  \bm B d\Omega_{\bm X} - \int_{\Gamma_{\bm X}} \bm V_h \cdot \bm H d\Gamma_{\bm X} \displaybreak[2] \\
=&  \frac{d}{dt}\int_{\Omega_{\bm X}} \frac12 \rho_0 \|\bm V_h\|^2 +  G^{\infty}_{\textup{iso}} ( \tilde{\bm C}_h ) + \displaystyle \sum_{\alpha=1}^{m}\Upsilon^{\alpha}(\tilde{\bm C}_h, \bm \Gamma^{\alpha}_h)  d\Omega_{\bm X} + \int_{\Omega_{\bm X}}  \sum_{\alpha=1}^{m} \frac14 \left(\frac{d}{dt}\bm \Gamma^{\alpha}_h \right) : \mathbb V^{\alpha} : \left(\frac{d}{dt}\bm \Gamma^{\alpha}_h \right) \displaybreak[2] \\
& - \bm V_h \cdot \rho_0  \bm B d\Omega_{\bm X} - \int_{\Gamma_{\bm X}} \bm V_h \cdot \bm H d\Gamma_{\bm X}.
\end{align*}
Rearranging terms in the above equality results in
\begin{align}
\label{eq:dissipation_alternative_form}
\frac{d}{dt}\int_{\Omega_{\bm X}} \frac12 \rho_0 \|\bm V_h \|^2 +  G^{\infty}_{\textup{iso}} ( \tilde{\bm C}_h ) + \displaystyle \sum_{\alpha=1}^{m}\Upsilon^{\alpha}(\tilde{\bm C}_h, \bm \Gamma^{\alpha}_h )  d\Omega_{\bm X} =& \int_{\Omega_{\bm X}} \rho_0 \bm V_h \cdot \bm B d\Omega_{\bm X} + \int_{\Gamma_{\bm X}} \bm V_h \cdot \bm H d\Gamma_{\bm X} \displaybreak[2] \\
& - \sum_{\alpha=1}^{m} \int_{\Omega_{\bm X}}\frac14 \left(\frac{d}{dt}\bm \Gamma^{\alpha}_h \right) : \mathbb V^{\alpha} : \left(\frac{d}{dt}\bm \Gamma^{\alpha}_h \right) d\Omega_{\bm X},
\end{align}
which completes the proof.
\end{proof}
This numerical stability property guarantees that the proposed modeling and computational framework preserve critical structures of the physical system. It is worth emphasizing that the energy stable scheme is constructed based on the thermodynamically consistent continuum model derived in Section \ref{sec:continuum}. To the best of our knowledge, we are unaware of any other numerical scheme with a priori energy stability proved for incompressible viscoelasticity. A straightforward consequence of this stability property is that the energy will be monotonically decreasing for unforced mechanical systems (i.e. $\bm B= \bm 0$ and $\bm H = \bm 0$).

For compressible materials, one may show that a pressure-squared term enters into the definition of the energy, providing boundedness of the pressure field. This justifies the use of equal-order interpolation for compressible materials. For fully incompressible materials, the pressure force has no contribution to the energy, and its stability comes from the inf-sup condition \cite[Section~3.2]{Liu2019a}. The inf-sup stability of the discrete spaces given in Section \ref{subsec:semi_discrete_formulation} has been numerically examined in \cite[Section~4.1]{Liu2019a}, where it was found that $\mathsf b + 1 \leq \mathsf a$ guarantees the element pair to be inf-sup stable. In this work, we use $\mathsf a = 1$ and $\mathsf b = 0$ for the discrete function spaces.

In addition to the energy stability, one may also conveniently show the momentum conservation of the semidiscrete formulatoin, which is stated in the following theorem.
\begin{theorem}[Semidiscrete momentum conservation]
\label{the:momentum-conservation}
If $\Gamma^t_{\bm x} = \Gamma^{h,t}_{\bm x}$, the following conservation properties hold for the semidiscrete formulation \eqref{eq:kinematics_current}-\eqref{eq:momentum_current},
\begin{align*}
\frac{d}{dt}\int_{\Omega_{\bm X}} \rho_0 \bm V_h d\Omega_{\bm X} =& \int_{\Omega_{\bm X}} \rho_0 \bm B d\Omega_{\bm X} + \int_{\Gamma_{\bm X}} \bm H d\Gamma_{\bm X}, \\
\frac{d}{dt}\int_{\Omega_{\bm X}} \rho_0 \bm \varphi_h \times \bm V_h d\Omega_{\bm X} =& \int_{\Omega_{\bm X}} \rho_0 \bm \varphi_h \times \bm B d\Omega_{\bm X} + \int_{\Gamma_{\bm X}} \bm \varphi_h \times \bm H d\Gamma_{\bm X}.
\end{align*}
\end{theorem}
\begin{proof}
The above conservation properties are direct consequences of choosing $\bm w_{h} = \bm e_i$ and $\bm w_{h} = \bm e_i \times \bm \varphi_h$ respectively in \eqref{eq:momentum_current}, where $\bm e_i$ is a unit vector in the $i$-th direction.
\end{proof}

\begin{remark}
It can be shown that the dissipation term in the stability estimate \eqref{eq:weak-form-energy-stability} can be equivalently written as
\begin{align*}
\sum_{\alpha=1}^{m} \int_{\Omega_{\bm X}}\frac14 \left(\frac{d}{dt}\bm \Gamma^{\alpha}_h \right) : \mathbb V^{\alpha} : \left(\frac{d}{dt}\bm \Gamma^{\alpha}_h \right) d\Omega_{\bm X} = \sum_{\alpha=1}^{m} \int_{\Omega_{\bm X}} \bm Q^{\alpha} : \left( \mathbb V^{\alpha} \right)^{-1} : \bm Q^{\alpha} d\Omega_{\bm X}.
\end{align*}
It should be pointed out that the stability is analyzed for the semi-discrete scheme. It remains an intriguing topic to further extend this estimate to the fully discrete regime. The energy-momentum scheme \cite{Simo1992a,Romero2012,Krueger2016} is a promising candidate for this goal.
\end{remark}


\subsection{Time integration algorithm}
In this section, we first introduce a discrete algorithm for updating the stresses. Following that, we state the fully discrete scheme using the generalized-$\alpha$ method.
\subsubsection{Stress update algorithm}
To obtain the stress, we need to perform time integration for the constitutive laws. Let the time interval $(0^+, T]$ be divided into $N$ subintervals of size $\Delta t_n := t_{n+1} - t_n$ delimited by a discrete time vector $\left\lbrace t_n \right\rbrace_{n=0}^{N}$. The approximations to the velocity, pressure, and displacement and their first time derivatives at time $t_n$ are denoted as 
\begin{align*}
\bm Y_n := \left\lbrace \bm V_n, P_n, \bm U_n \right\rbrace^T \quad \mbox{ and } \quad \dot{\bm Y}_n := \left\lbrace \dot{\bm V}_n, \dot{P}_n, \dot{\bm U}_n \right\rbrace^T,
\end{align*}
respectively. Correspondingly, the approximations to the deformation gradient and strain measures at time $t_n$ are represented as
\begin{align*}
\bm F_n = \bm I + \nabla_{\bm X} \bm U_n , \quad J_n = \mathrm{det}(\bm F_n), \quad \bm C_n = \bm F^T_n \bm F_n, \quad \tilde{\bm C}_n = J^{-2/3}_n \bm C_n.
\end{align*}
The approximated projection tensor and elasticity tensor are given by
\begin{align*}
\mathbb P_{n+1} = \mathbb I - \frac13 \bm C^{-1}_{n+1} \otimes \bm C_{n+1}, \quad \tilde{\mathbb C}^{\infty}_{\mathrm{iso} \: n+1} = 4 J^{-\frac43}_{n+1} \frac{\partial^2 G^{\infty}_{\mathrm{iso}}(\tilde{\bm C}_{n+1})}{\partial \tilde{\bm C} \partial \tilde{\bm C}}
\end{align*}
The algorithmic stresses at time $t_{n+1}$ read as
\begin{align*}
& \bm S_{n+1} = \bm S_{\mathrm{iso} \: n+1} + \bm S_{\mathrm{vol} \: n+1}, \quad \bm S_{\mathrm{iso} \: n+1} = J^{-\frac23}_{n+1} \mathbb P_{n+1} : \tilde{\bm S}_{n+1}, \quad \bm S_{\mathrm{vol} \: n+1} = -J_{n+1} P_{n+1} \bm C^{-1}_{n+1}, \displaybreak[2]  \\
& \tilde{\bm S}_{n+1} = \tilde{\bm S}_{\mathrm{iso} \: n+1}^{\infty} + \sum_{\alpha=1}^{m} \tilde{\bm S}^{\alpha}_{\mathrm{neq} \: n+1}, \quad \tilde{\bm S}_{\mathrm{iso} \: n+1}^{\infty} = 2 \left( \frac{\partial G^{\infty}_{\mathrm{iso}}}{\partial \tilde{\bm C}} \right)_{n+1}, \quad \tilde{\bm S}^{\alpha}_{\mathrm{neq} \: n+1} = \frac{J^{\frac43}_{n+1} \beta^{\infty}_{\alpha}}{2\mu^{\alpha}} \tilde{\mathbb C}^{\infty}_{\mathrm{iso} \: n+1} : \bm Q^{\alpha}_{n+1}.
\end{align*}
To evaluate the stresses, we need to provide an algorithmic way to evaluate $\bm Q^{\alpha}_{n+1}$ based on the hereditary integral \eqref{eq:linear_hereditary_integral}. Following the notation introduced in \cite{Simo2006,Holzapfel2000}, we first introduce a dimensionless parameter $\xi^{\alpha}:= -\Delta t_n / 2\tau^{\alpha}$. The approximation to the variable $\bm Q^{\alpha}$ at time $t_{n+1}$ is given by
\begin{align}
\label{eq:Q-alpha-np1-recursive-formula}
\bm Q^{\alpha}_{n+1} =& \exp\left( - t_{n+1} / \tau^{\alpha} \right) \bm Q^{\alpha}_0 + \int^{t_{n+1}}_{0^+} \exp\left( -(t_{n+1}-s)/\tau^{\alpha} \right) \beta_{\alpha}^{\infty} \frac{d}{ds} \tilde{\bm S}^{\infty}_{\mathrm{iso}} ds \displaybreak[2] \nonumber \\
=& \exp(2\xi^{\alpha})\exp(-t_n/\tau^{\alpha}) \bm Q^{\alpha}_0 + \exp(2\xi^{\alpha}) \int^{t_{n}}_{0^+} \exp\left( -(t_n-s)/\tau^{\alpha} \right) \beta_{\alpha}^{\infty} \frac{d}{ds} \tilde{\bm S}^{\infty}_{\mathrm{iso}} ds \nonumber \\
& + \int^{t_{n+1}}_{t_n} \exp\left( -(t_{n+1}-s)/\tau^{\alpha} \right) \beta_{\alpha}^{\infty} \frac{d}{ds} \tilde{\bm S}^{\infty}_{\mathrm{iso}} ds \displaybreak[2] \nonumber \\
= & \exp(2\xi^{\alpha}) \bm Q^{\alpha}_n + \int^{t_{n+1}}_{t_n} \exp\left( -(t_{n+1}-s)/\tau^{\alpha} \right) \beta_{\alpha}^{\infty} \frac{d}{ds} \tilde{\bm S}^{\infty}_{\mathrm{iso}} ds \displaybreak[2] \nonumber \\
\approx & \exp(2\xi^{\alpha}) \bm Q^{\alpha}_n + \exp(\xi^{\alpha}) \beta^{\infty}_{\alpha} \int^{t_{n+1}}_{t_n} \frac{d}{ds} \tilde{\bm S}^{\infty}_{\mathrm{iso}} ds \displaybreak[2] \nonumber \\
=& \beta^{\infty}_{\alpha} \exp(\xi^{\alpha}) \tilde{\bm S}^{\infty}_{\mathrm{iso} \: n+1} + \exp(\xi^{\alpha}) \left(\exp(\xi^{\alpha}) \bm Q^{\alpha}_n - \beta^{\infty}_{\alpha} \tilde{\bm S}^{\infty}_{\mathrm{iso} \: n}\right).
\end{align}
In the second-to-last step of \eqref{eq:Q-alpha-np1-recursive-formula}, a mid-point rule is applied for the exponential term in the time integral to obtain an approximation \cite{Simo2006,Holzapfel2000}. Error analysis shows that the above approximation is second-order accurate, and the relation \eqref{eq:Q-alpha-np1-recursive-formula} is often referred to as the recurrence update formula for $\bm Q^{\alpha}$. There exists an alternate second-order accurate recurrence update formula by applying the mid-point rule to the stress rate, rather than the exponential kernel term, in the time integral. Interested readers are referred to \cite[p.~355]{Simo2006} for details.

\begin{remark}
We mention that the recursive formula \eqref{eq:Q-alpha-np1-recursive-formula} is obtained by making use of the semigroup property of the kernel of the hereditary integral, and there exist other recursive formulas \cite[Chapter~10]{Simo2006}. The recursive formula \eqref{eq:Q-alpha-np1-recursive-formula} is second-order accurate, and it is feasible to achieve higher-order accuracy with more involved update formulas for the internal state variables \cite{Eidel2011}. In fractional-order viscoelasticity models, however, a different recursive formula can be derived by invoking the fast convolution method \cite{Yu2016}.
\end{remark}

\begin{remark}
In the analysis calculations, the values of $\bm Q^{\alpha}_{n+1}$ can be conveniently initialized and stored at quadrature points. In the constitutive routine, their values at the quadrature points are updated by  \eqref{eq:Q-alpha-np1-recursive-formula}, which are utilized as the input for the calculation of stresses at the quadrature points. Yet, in the postprocessing (such as visualization), the stresses are typically not sampled or interpolated at the quadrature points used for analysis. The global smoothing procedure \cite{Hinton1974,Oden1971} can be invoked to recover the values of $\bm Q^{\alpha}_{n+1}$ and subsequently the values of stresses.
\end{remark}

\subsubsection{Fully discrete scheme}
With the time discrete stress given above, we may state the fully discrete algorithm by invoking the generalized-$\alpha$ method \cite{Chung1993,Jansen2000}. At time $t_n$, given $\bm Y_n$ and $\dot{\bm Y}_n$, the time step size $\Delta t_n$, and the parameters $\alpha_m$, $\alpha_f$, and $\gamma$, find $\bm Y_{n+1}$ and $\dot{\bm Y}_{n+1}$, such that for $\forall \left\lbrace Q_h, \bm W_h \right\rbrace \in \mathcal V_{P_h} \times \mathcal V_{\bm V_h}$,
\begin{align}
\label{eq:gen_a_solids_kinematics_ref}
& \bm 0 = \mathbf B^k\left( \dot{\bm Y}_{n+\alpha_m}, \bm Y_{n+\alpha_f} \right), \displaybreak[2]\\
\label{eq:gen_a_solids_mass_ref}
& 0 = \mathbf B^p\left( Q_h; \dot{\bm Y}_{n+\alpha_m}, \bm Y_{n+\alpha_f} \right), \displaybreak[2] \\
\label{eq:gen_a_solids_momentum_ref}
& 0 = \mathbf B^m\left( \bm W_h; \dot{\bm Y}_{n+\alpha_m}, \bm Y_{n+\alpha_f} \right), \displaybreak[2] \\
& \bm Y_{n+1} = \bm Y_n + \Delta t_n \dot{\bm Y}_n + \gamma \Delta t_n \left( \dot{\bm Y}_{n+1} - \dot{\bm Y}_n \right), \displaybreak[2] \\
& \dot{\bm Y}_{n+\alpha_m} = \dot{\bm Y}_n + \alpha_m \left( \dot{\bm Y}_{n+1} - \dot{\bm Y}_n \right), \displaybreak[2] \\
\label{eq:gen_a_Y_n_plus_alpha_f}
& \bm Y_{n+\alpha_f} = \bm Y_n + \alpha_f \left( \bm Y_{n+1} - \bm Y_n \right).
\end{align}
Let $\varrho_{\infty} \in [0,1]$ denote the spectral radius of the amplification matrix at the highest mode. The following choice of the parameters ensures second-order accuracy, unconditional stability, and controllable high-frequency dissipation for linear first-order ordinary differential equations \cite{Jansen2000},
\begin{align*}
\alpha_m = \frac12 \left( \frac{3-\varrho_{\infty}}{1+\varrho_{\infty}} \right), \quad \alpha_f = \frac{1}{1+\varrho_{\infty}}, \quad \gamma = \frac{1}{1+\varrho_{\infty}}.
\end{align*}

\begin{remark}
It is known that the many different temporal schemes may be recovered by the generalized-$\alpha$ scheme via distinct choices of the parameters. For example, choosing $\varrho_{\infty} = 0.0$ renders a scheme that is spectrally equivalent to the second-order backward difference method \cite{Jansen2000}; choosing $\varrho_{\infty} = 1.0$ recovers the mid-point rule. It is also worth pointing out that the generalized-$\alpha$ scheme has been conventionally applied to second-order structural dynamics. Recent work shows that the generalized-$\alpha$ method applied to a first-order structural dynamic system does not suffer from the `overshoot' phenomenon \cite{Kadapa2017}, and thus possesses many desirable properties of implicit schemes, as noted by Hilber and Hughes \cite{Hilber1978}. Writing the structural dynamics problem as a first-order system introduces three additional velocity degrees of freedom per node. It can be shown that these additional degrees of freedom can be solved in a segregated manner in a consistent Newton-Raphson algorithm \cite{Liu2018,Liu2019,Liu2019a,Scovazzi2016}. The additional cost induced by the velocity degrees of freedom is merely the memory for storing them and an explicit update formula, which are thus marginal.
\end{remark}

\begin{remark}
The stability of the generalized-$\alpha$ scheme was analyzed based on linear problems \cite{Chung1993,Jansen2000} and remains unclear for nonlinear problems. In fact, it is known that the mid-point rule, an instantiation of the generalized-$\alpha$ scheme with $\varrho_{\infty}=1.0$, is often energetically unstable for nonlinear structural dynamics \cite{Ortigosa2018,Betsch2016}. It remains an interesting topic to construct fully discrete schemes that are provably stable in energy, and the family of energy-momentum methods serves as a promising candidate in this role \cite{Simo1992a,Ortigosa2018,Betsch2016,Kuhl1999}
\end{remark}

\begin{remark}
It is worth pointing out that the generalized-$\alpha$ scheme may achieve the claimed second-order accuracy if all unknowns are collocated at the intermediate time step, as is done in \eqref{eq:gen_a_solids_kinematics_ref}-\eqref{eq:gen_a_Y_n_plus_alpha_f}. In a very popular approach, the pressure is collocated at the time step $t_{n+1}$ with the rest variables collocated at the intermediate time step following the rule of the generalized-$\alpha$ scheme. It was found recently that the claimed second-order accuracy is lost in that approach \cite{Liu2021}.
\end{remark}

\subsection{Elasticity tensor}
Here we provide the isochoric elasticity tensors that are used in the consistent linearization of the model. We first define the elasticity tensors approximated at time $t_{n+1}$ as
\begin{align*}
\mathbb C_{\mathrm{iso} \: n+1} := \left( 2\frac{\partial \bm S_{\mathrm{iso}}}{\partial \bm C}\right)_{n+1} , \quad \mathbb C_{\mathrm{iso} \: n+1}^{\infty} := \left( 2 \frac{\partial \bm S^{\infty}_{\mathrm{iso}}}{\partial \bm C} \right)_{n+1} , \quad \mbox{and} \quad \mathbb C_{\mathrm{neq} \: n+1}^{\alpha} := \left( 2\frac{\partial \bm S^{\alpha}_{\mathrm{neq}}}{\partial \bm C} \right)_{n+1}.
\end{align*}
Based on the additive split structure of the stress, the isochoric elasticity tensor $\mathbb C_{\mathrm{iso} \: n+1}$ can be expressed as
\begin{align*}
\mathbb C_{\mathrm{iso} \: n+1} = \left( 2 \frac{\partial \bm S^{\infty}_{\mathrm{iso}}}{\partial \bm C} \right)_{n+1} + \sum_{\alpha=1}^{m} \left(2\frac{\partial \bm S^{\alpha}_{\mathrm{neq}}}{\partial \bm C} \right)_{n+1} = \mathbb C^{\infty}_{\mathrm{iso} \: n+1} + \sum_{\alpha=1}^{m} \mathbb C^{\alpha}_{\mathrm{neq} \: n+1}.
\end{align*}
It can be shown that 
\begin{align*}
& \mathbb C^{\infty}_{\mathrm{iso} \: n+1} = \mathbb P_{n+1} : \tilde{\mathbb C}^{\infty}_{\mathrm{iso} \: n+1} : \mathbb P^T_{n+1} + \frac23 \mathrm{Tr}\left( J^{-\frac23}_{n+1} \tilde{\bm S}^{\infty}_{\mathrm{iso} \: n+1} \right) \tilde{\mathbb P}_{n+1} - \frac23 \left( \bm C^{-1}_{n+1} \otimes \bm S^{\infty}_{\mathrm{iso} \: n+1} + \bm S^{\infty}_{\mathrm{iso} \: n+1} \otimes \bm C^{-1}_{n+1} \right), \\
& \tilde{\mathbb C}^{\infty}_{\mathrm{iso} \: n+1} := 4 J^{-\frac43}_{n+1} \left( \frac{\partial^2 G^{\infty}_{\mathrm{iso}}}{\partial \tilde{\bm C} \partial \tilde{\bm C}} \right)_{n+1}, \quad \mathrm{Tr}\left( \cdot \right) = \left( \cdot \right) : \bm C_{n+1}, \quad \tilde{\mathbb P}_{n+1} := \bm C^{-1}_{n+1} \odot \bm C^{-1}_{n+1} - \frac13 \bm C^{-1}_{n+1} \otimes \bm C^{-1}_{n+1}.
\end{align*}
The derivation of the formula can be found in \cite[p.~255]{Holzapfel2000}. Following similar steps, we have
\begin{align*}
\mathbb C^{\alpha}_{\mathrm{neq} \: n+1} = \mathbb P_{n+1} : \tilde{\mathbb C}^{\alpha}_{\mathrm{neq} \: n+1} : \mathbb P^{T}_{n+1} + \frac23 \mathrm{Tr}\left( J^{-\frac23} \tilde{\bm S}^{\alpha}_{\mathrm{neq} \: n+1} \right) \tilde{\mathbb P} - \frac23 \left( \bm C^{-1}_{n+1} \otimes \bm S^{\alpha}_{\mathrm{neq} \: n+1} + \bm S^{\alpha}_{\mathrm{neq} \: n+1} \otimes \bm C^{-1}_{n+1} \right),
\end{align*}
where
\begin{align*}
\tilde{\mathbb C}^{\alpha}_{\mathrm{neq} \: n+1} := 2 J^{-\frac43}_{n+1} \left( \frac{\partial \tilde{\bm S}^{\alpha}_{\mathrm{neq}}}{\partial \tilde{\bm C}} \right)_{n+1}.
\end{align*}
Based on \eqref{eq:Q-alpha-np1-recursive-formula}, we have
\begin{align*}
2 \left( \frac{\partial \bm Q^{\alpha}}{\partial \bm C} \right)_{n+1} = \delta_{\alpha} J^{\frac43}_{n+1} \tilde{\mathbb C}^{\infty}_{\mathrm{iso} \: n+1},
\end{align*}
with $\delta_{\alpha} := \beta_{\alpha}^{\infty} \exp(\xi^{\alpha})$. We may now express $\tilde{\mathbb C}^{\alpha}_{\mathrm{neq} \: n+1}$ as
\begin{align*}
\tilde{\mathbb C}^{\alpha}_{\mathrm{neq} \: n+1} =& \frac{4\beta^{\infty}_{\alpha} J^{-\frac43}_{n+1}}{\mu^{\alpha}} \left(\frac{\partial^3 G^{\infty}_{\mathrm{iso}}}{\partial \tilde{\bm C} \partial \tilde{\bm C} \partial \tilde{\bm C}} \right)_{n+1} : \bm Q^{\alpha}_{n+1} + \frac{2 \delta_{\alpha} \beta^{\infty}_{\alpha} }{\mu^{\alpha}} \left( \frac{\partial^2 G^{\infty}_{\mathrm{iso}}}{\partial \tilde{\bm C} \partial \tilde{\bm C} } \right)_{n+1} : \tilde{\mathbb C}^{\infty}_{\mathrm{iso} \: n+1} \\
=& \frac{4\beta^{\infty}_{\alpha} J^{-\frac43}_{n+1}}{\mu^{\alpha}} \left(\frac{\partial^3 G^{\infty}_{\mathrm{iso}}}{\partial \tilde{\bm C} \partial \tilde{\bm C} \partial \tilde{\bm C}} \right)_{n+1} : \bm Q^{\alpha}_{n+1} + \frac{\delta_{\alpha} \beta^{\infty}_{\alpha} J^{\frac43}_{n+1} }{2\mu^{\alpha}} \tilde{\mathbb C}^{\infty}_{\mathrm{iso} \: n+1} : \tilde{\mathbb C}^{\infty}_{\mathrm{iso} \: n+1}.
\end{align*}
The algorithms for calculating the stresses and elasticity tensors for the IPC, HS, and MIPC models are documented in \ref{sec:algorithm-IPC}, \ref{sec:algorithm-HS}, and \ref{sec:algorithm-MIPC}, respectively.

\section{Numerical results}
\label{sec:numerical_results}
In this section, we investigate the proposed viscoelastic model by a suite of numerical examples using the numerical scheme proposed in the previous section. Unless otherwise specified, we use $\mathsf p+\mathsf a+1$ Gauss quadrature points in each direction. Also recall that we have fixed $\mathsf a = 1$ and $\mathsf b = 0$ in the construction of discrete function spaces. We take $\bm Q^{\alpha}_0 = \bm O$ in all numerical studies, which implies $\hat{\bm S}^{\alpha}_{0} = \tilde{\bm S}^{\alpha}_{\mathrm{iso} \: 0}$.

\begin{table}[h]
  \centering
  \begin{tabular}[t]{ m{.35\textwidth}   m{.45\textwidth} }
    \hline
    \begin{minipage}{.35\textwidth}
      \includegraphics[angle=0, trim=145 270 270 190, clip=true, scale = 0.6]{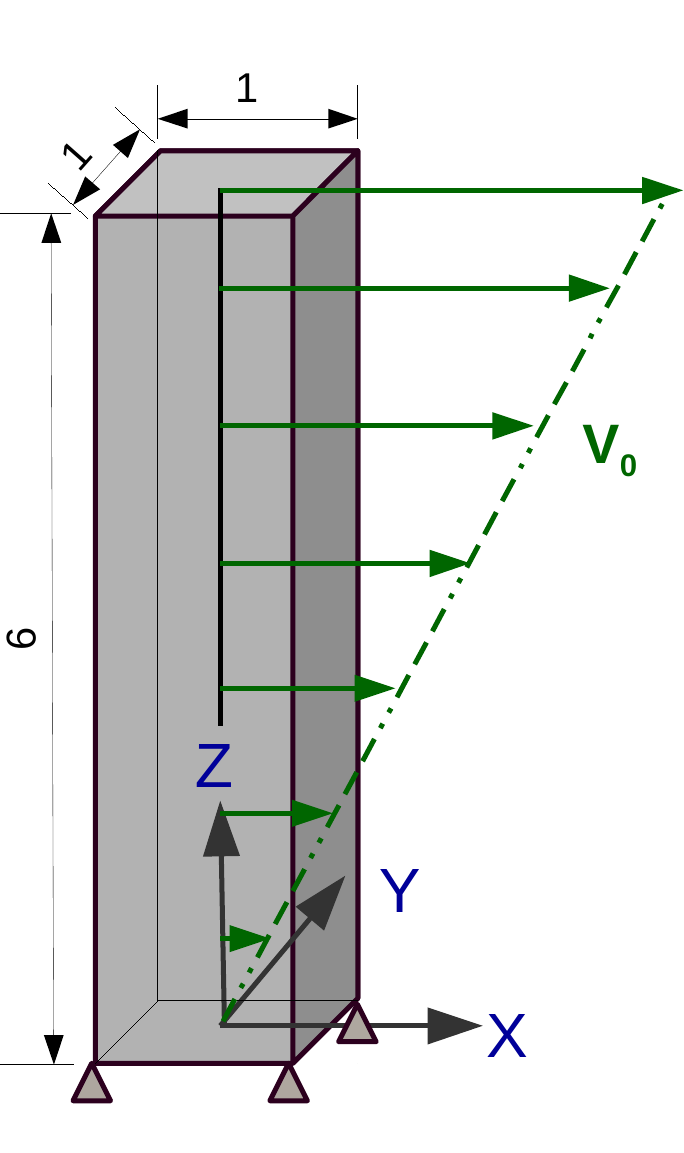}
    \end{minipage}
&    
    \begin{minipage}{.45\textwidth}
      \begin{itemize}
      \item[] Material properties:
      \item[] $G^{\infty}_{\mathrm{iso}} = \frac{c_1}{2} \left( \tilde{I}_1 - 3 \right) + \frac{c_2}{2} \left( \tilde{I}_2 - 3 \right)$,    
\item[] $\rho_0 = 1.1 \times 10^3$ kg/m$^3$,  
      \item[] $E = 1.7\times 10^7$ Pa, $c_1 = c_2 = E/6$,
      \item[] $\beta_1^{\infty} = 1.0$, $\mu^1 = 10c_1$, $\tau^1 = 1$ s.
      \item[] Reference scales: 
      \item[] $L_0 = 1$ m, $M_0 = 1$ kg, $T_0 = 1$ s. 
      \end{itemize}
    \end{minipage}   
    \\ 
    \hline
  \end{tabular}
  \caption{Three-dimensional beam bending: problem setting, boundary conditions, initial conditions, and material properties. Notice that the parameter $\beta_1^{\infty}$ is only used in the IPC and MIPC models.} 
\label{table:3d_beam_bending_benchmark_geometry}
\end{table}

\begin{figure}[htbp]
\begin{center}
\begin{tabular}{c}
\includegraphics[angle=0, trim=85 85 120 80, clip=true, scale = 0.39]{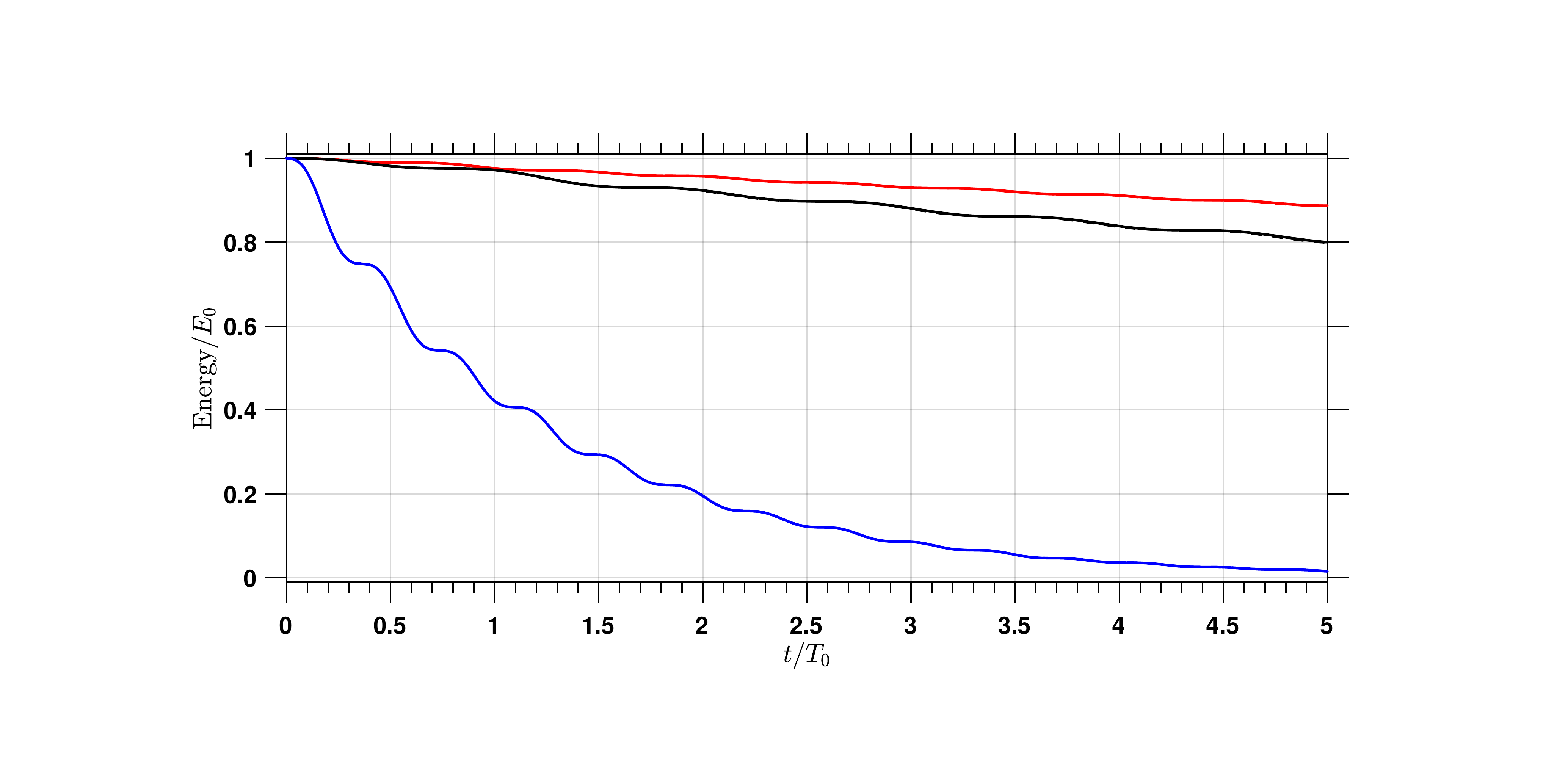}\\
(a) Total energies \\
\includegraphics[angle=0, trim=85 85 120 80, clip=true, scale = 0.39]{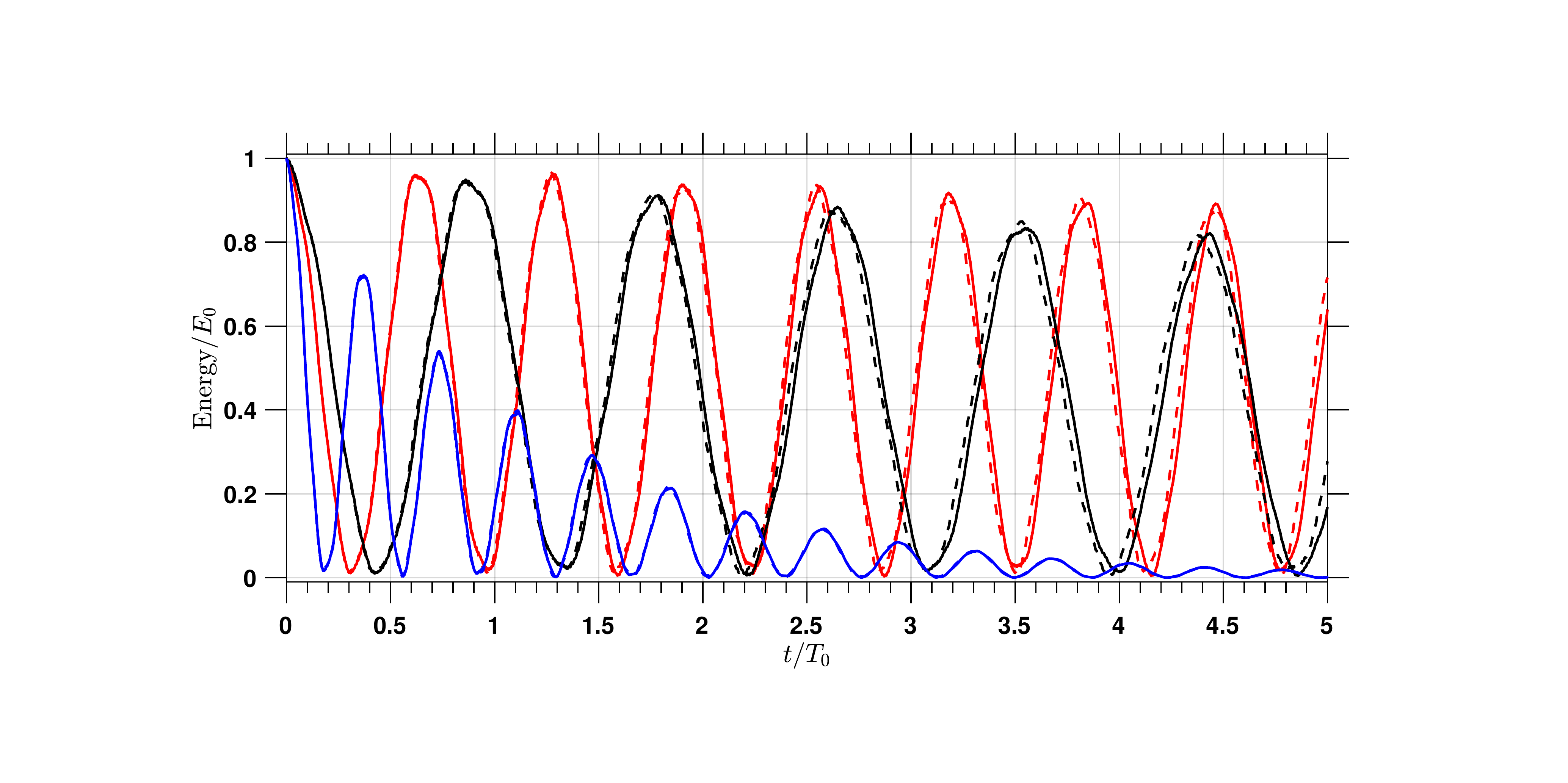} \\
(b) Kinetic energies \\
\includegraphics[angle=0, trim=85 85 120 80, clip=true, scale = 0.39]{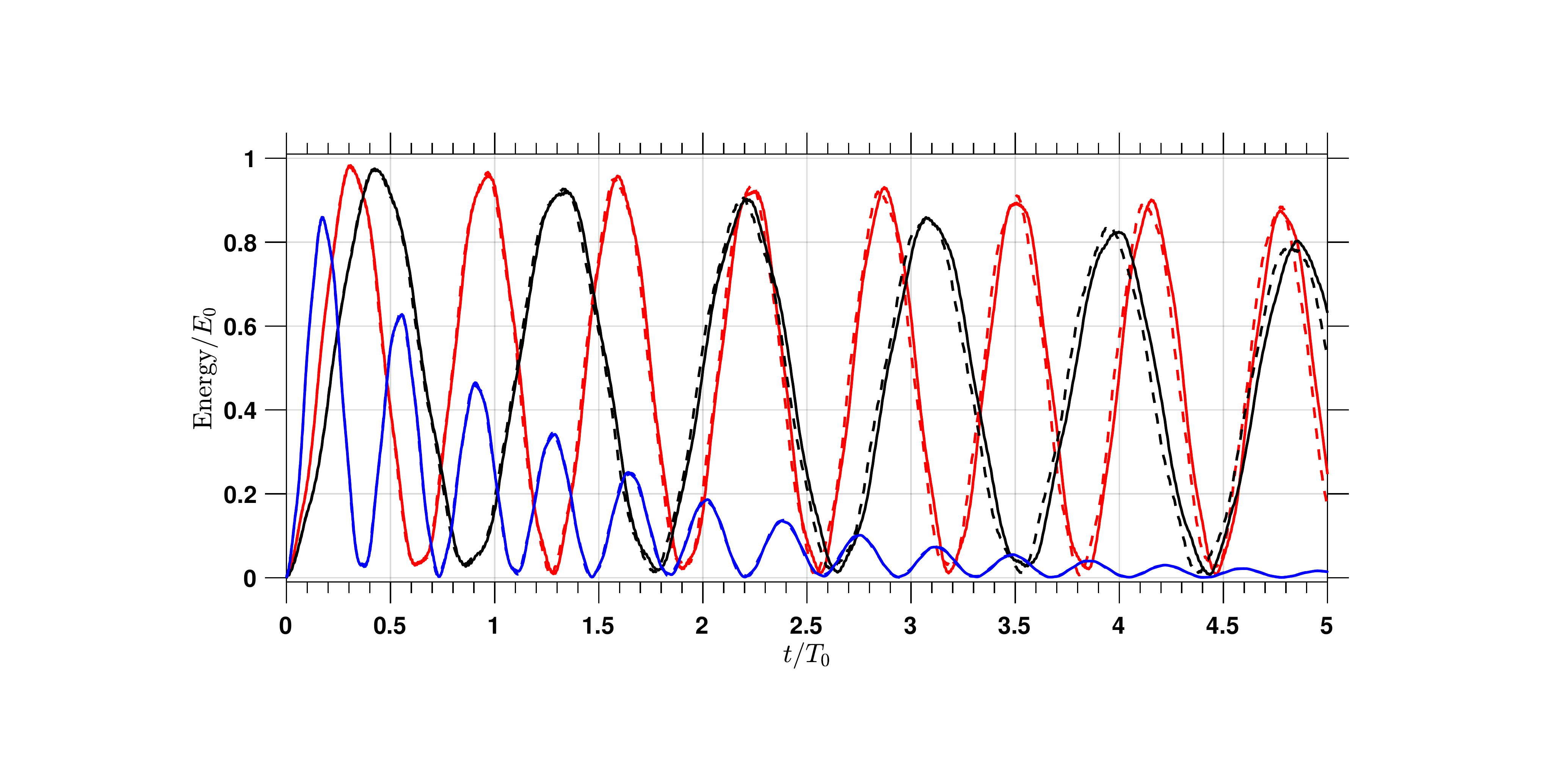}\\
(c) Potential energies
\end{tabular}
\caption{The total, kinetic, and potential energies (i.e., $G^{\infty}_{\mathrm{iso}}(\tilde{\bm C}_h) + \Upsilon^1(\tilde{\bm C}_h, \bm \Gamma^1_h)$) of the IPC (red), HS (blue), and MIPC (black) models over time. The simulations are performed with a fixed time step size $\Delta t/T_0 = 1\times 10^{-3}$. The solid lines illustrate results obtained from a spatial mesh with $\mathsf p = 2$, $\mathsf a=1$, $\mathsf b = 0$, and $5 \times 5 \times 30$ elements; the dashed lines illustrate results obtained from a spatial mesh with $\mathsf p = 1$, $\mathsf a=1$, $\mathsf b = 0$, and $1 \times 1 \times 6$ elements. The reference value of the total energy $E_0$ is chosen to be the total energy at time $t=0$, which is $1.1\times 10^5$ kg m$^2$/s$^2$.} 
\label{fig:beam_energy}
\end{center}
\end{figure}

\begin{figure}[!htbp]
\begin{tabular}{ c c c c c c }
\multicolumn{6}{c}{
\includegraphics[angle=0, trim=0 0 280 750, clip=true, scale = 0.26]{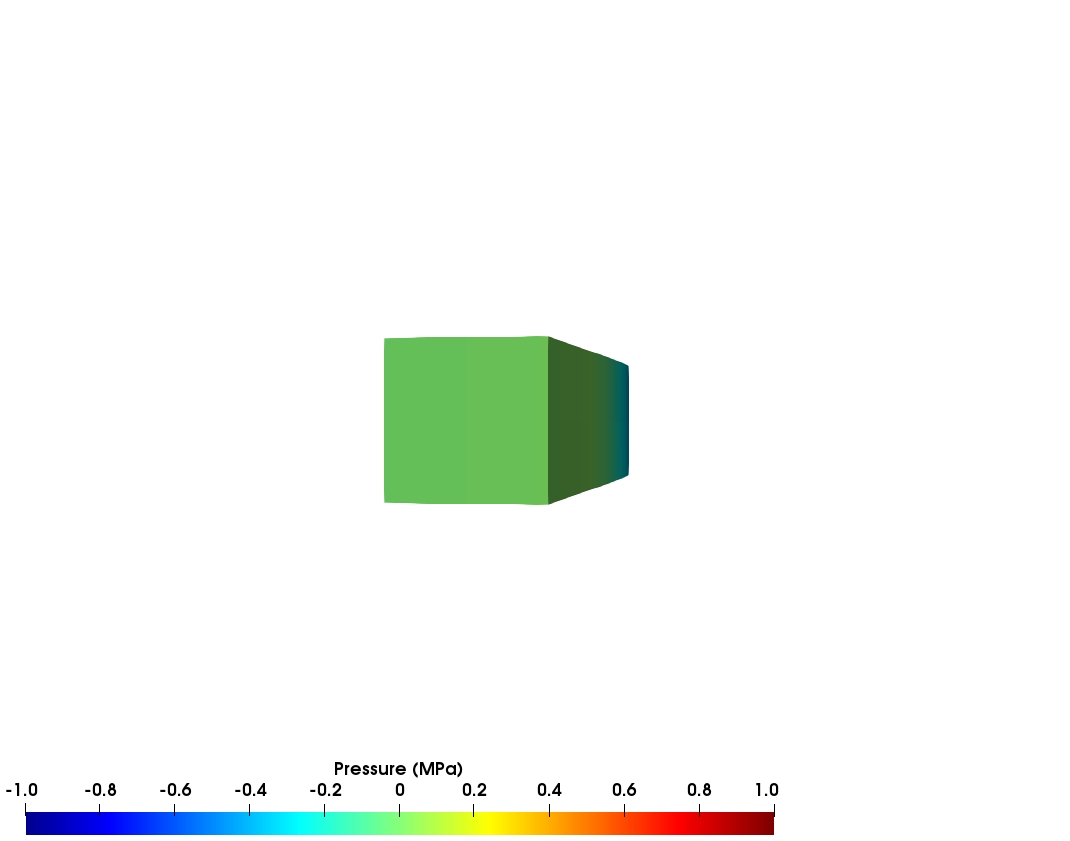}
} \\
\includegraphics[angle=0, trim=320 0 450 0, clip=true, scale = 0.21]{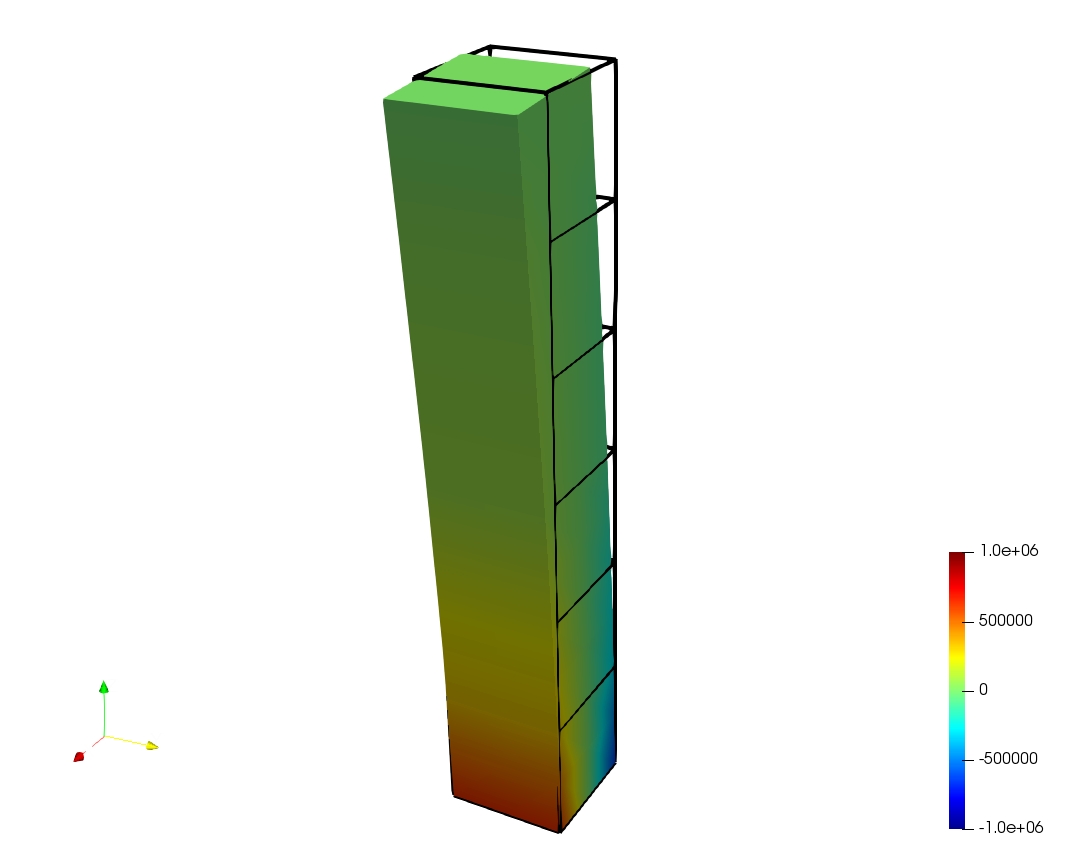} &
\includegraphics[angle=0, trim=320 0 450 0, clip=true, scale = 0.21]{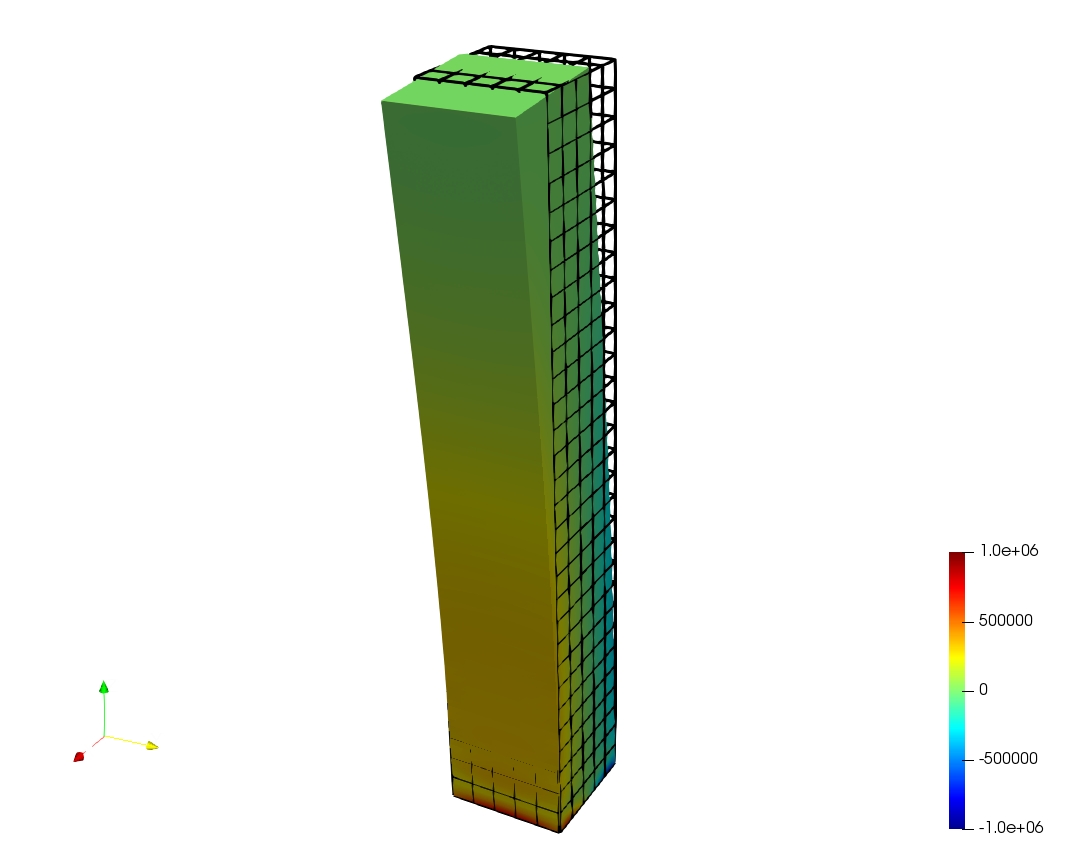} &
\includegraphics[angle=0, trim=280 0 450 0, clip=true, scale = 0.21]{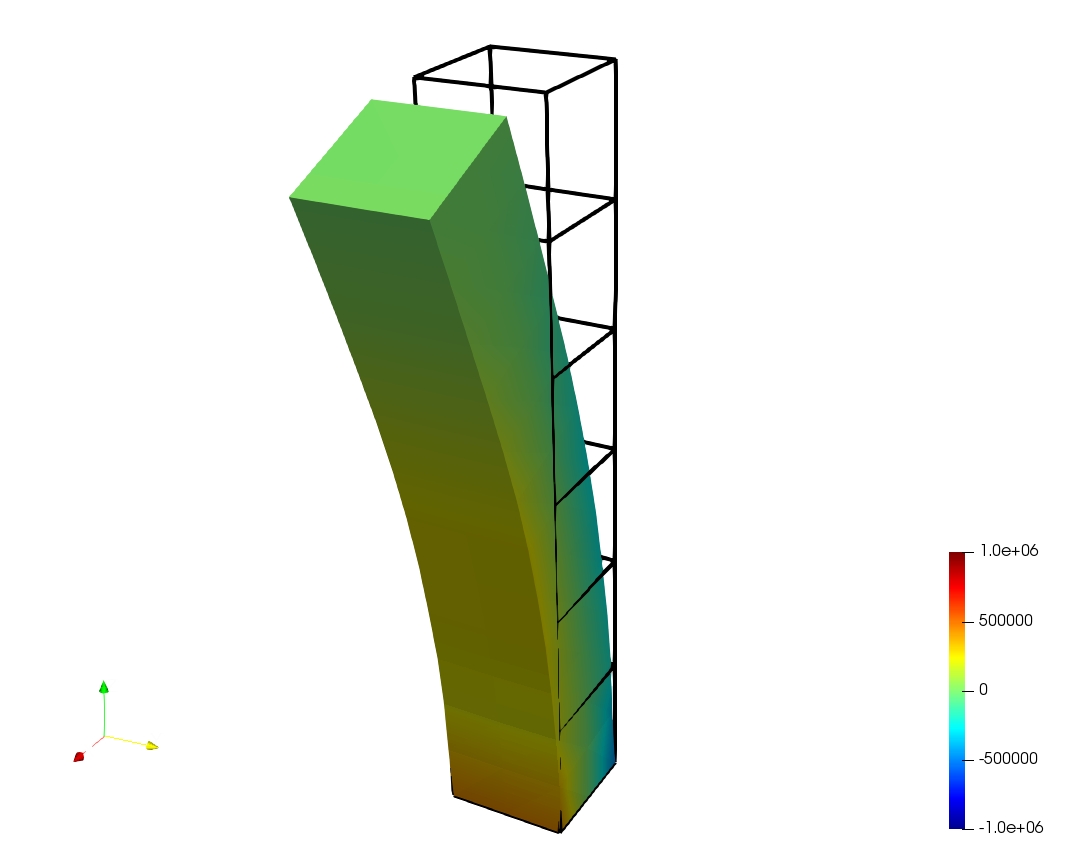} &
\includegraphics[angle=0, trim=280 0 450 0, clip=true, scale = 0.21]{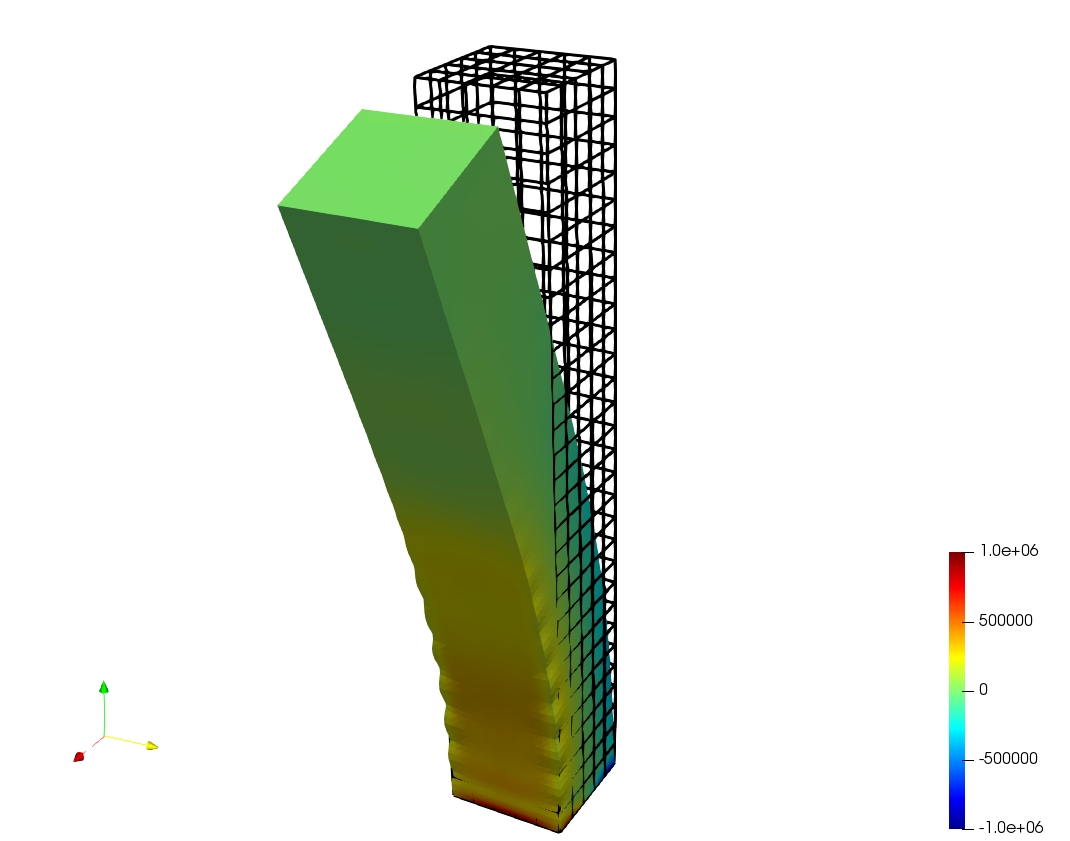} &
\includegraphics[angle=0, trim=360 0 420 0, clip=true, scale = 0.21]{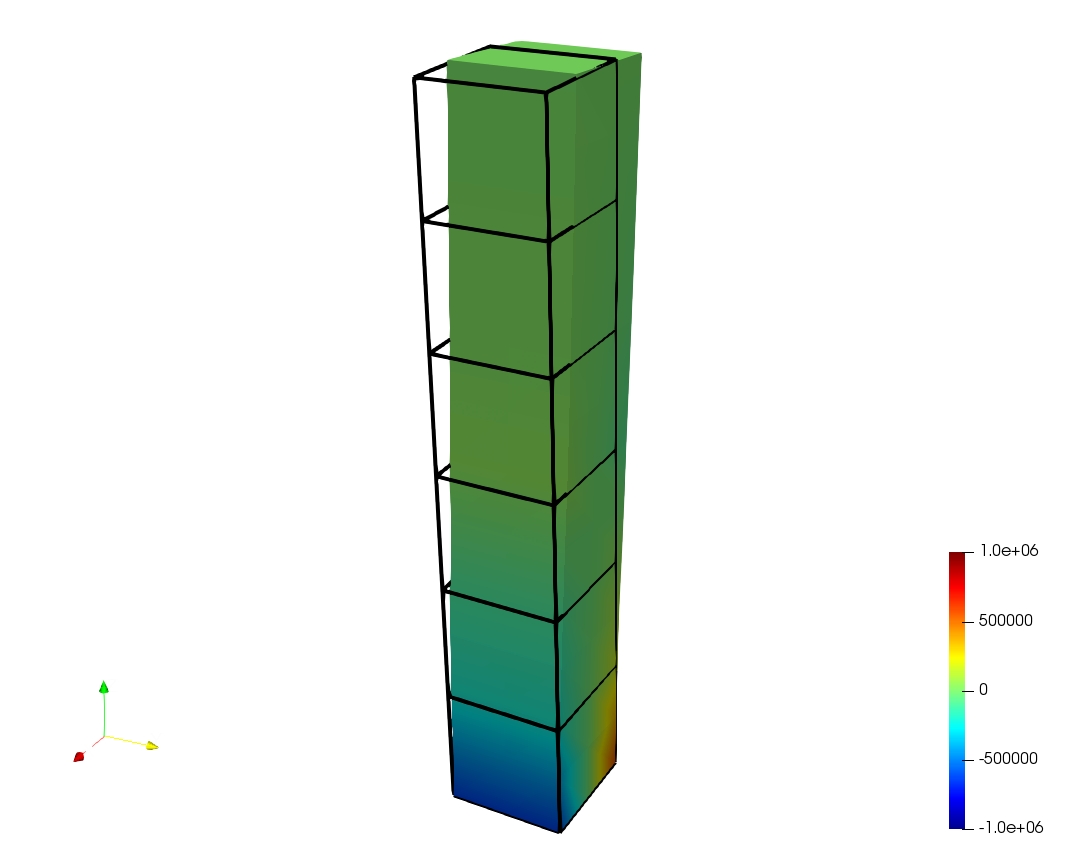} &
\includegraphics[angle=0, trim=360 0 420 0, clip=true, scale = 0.21]{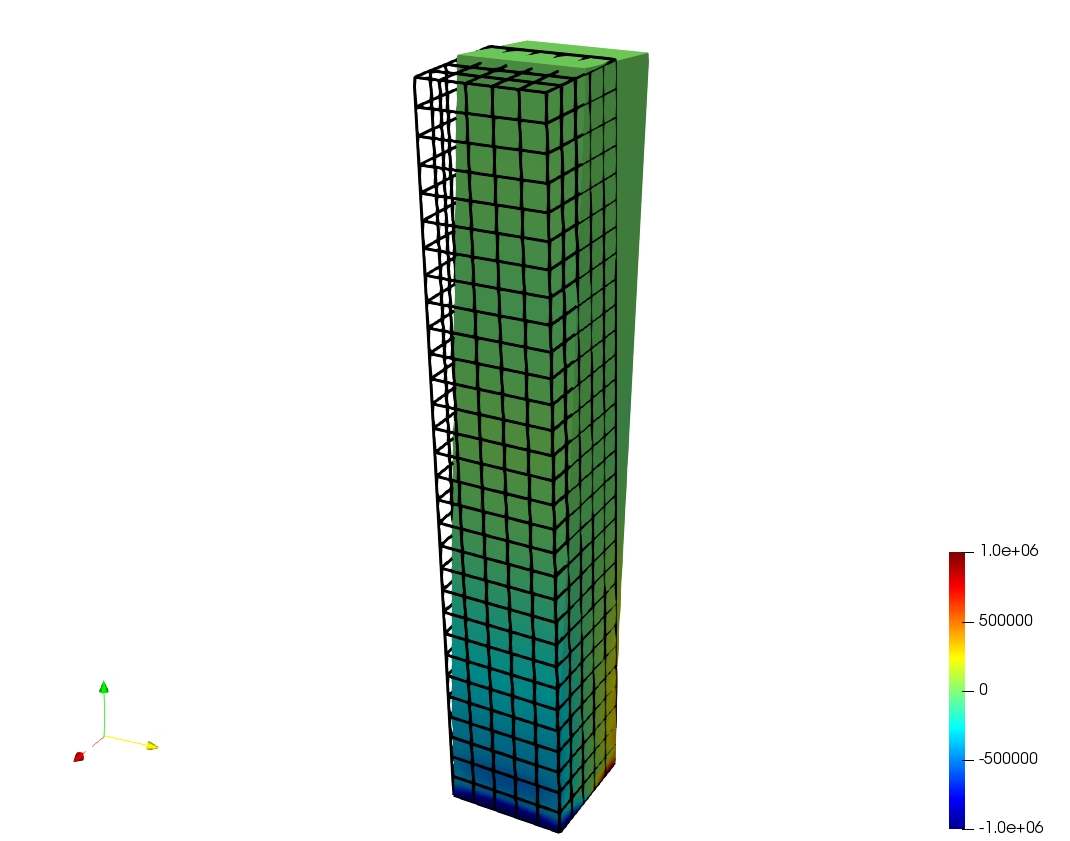} \\
\multicolumn{2}{c}{HS} & \multicolumn{2}{c}{MIPC} & \multicolumn{2}{c}{IPC}
\end{tabular}
\caption{The snapshots of the pressure fields at time $t/T_0 = 2.5$ of the three models with a fixed time step size $\Delta t/T_0 = 1 \times 10^{-3}$ plotted on the deformed configuration. Two sets of spatial meshes are used: a spatial mesh with $\mathsf p = 2$, $\mathsf a=1$, $\mathsf b = 0$, and $5 \times 5 \times 30$ elements and a spatial mesh with $\mathsf p = 1$, $\mathsf a=1$, $\mathsf b = 0$, and $1 \times 1 \times 6$ elements. The meshes at the initial time are plotted as the black grid.} 
\label{fig:beam_pressure_snapshot}
\end{figure}

\subsection{Beam bending}
In this example, we consider a three-dimensional beam vibration problem with bending dominated deformation. The problem setting and the material properties are defined in Table \ref{table:3d_beam_bending_benchmark_geometry}. In particular, the background elastic material is characterized by the Mooney-Rivlin model. The bottom surface of the beam is fully clamped while the rest boundary surfaces are specified with zero traction boundary conditions. The body is initially in a stress-free condition with zero displacements. The initial velocity is given by
\begin{align*}
\bm V(\bm X, 0) = \left( V_0 \frac{Z}{L_0}, 0, 0 \right)^T, \quad V_0 = \frac53 \textup{m}/\textup{s},
\end{align*}
which initiates the vibration. For this specific problem, the energy stability given by Theorem \ref{the:energy-stability} suggests that the total energy monotonically decreases with respect to time, i.e.,
\begin{align*}
\frac{d}{dt}\int_{\Omega_{\bm X}} \frac12 \rho_0 \|\bm V_h\|^2 +  G^{\infty}_{\textup{iso}} ( \tilde{\bm C}_h ) + \Upsilon^{1}(\tilde{\bm C}_h, \bm \Gamma^{1}_h)  d\Omega_{\bm X} =& -  \int_{\Omega_{\bm X}}\frac14 \left(\frac{d}{dt}\bm \Gamma^{1}_h \right) : \mathbb V^{1} : \left(\frac{d}{dt}\bm \Gamma^{1}_h \right) d\Omega_{\bm X} \nonumber \\
=& - \int_{\Omega_{\bm X}} \bm Q^{1} : \left( \mathbb V^{1} \right)^{-1} : \bm Q^{1} d\Omega_{\bm X} \leq 0.
\end{align*}
In the first set of simulations, we investigate the problem with two spatial meshes and integrate in time untill $5T_0$. The generalized-$\alpha$ method is utilized with a fixed time step size and $\varrho_{\infty} = 1.0$, which recovers the mid-point rule. The evolutions of the total, kinetic, and potential energies are illustrated in Figure \ref{fig:beam_energy}. We can observe the energy decay of all three material models with different dissipation rates. The HS model leads to the fastest dissipation of the energy, while the energies of the IPC and MIPC models dissipate relatively slower. In Figure \ref{fig:beam_pressure_snapshot}, the snapshots of the pressure field on the deformed configuration at time $t/T_0 = 2.5$ are depicted. It is worth pointing out that the pressure field calculated based on the coarse mesh is very close to that of the finer mesh, suggesting the spline-based technology can accurately capture the stress with a relatively coarse mesh. Yet, we also observe excessive oscillations for higher-order methods, as can be seen in the results of MIPC model shown in Figure \ref{fig:beam_pressure_snapshot}. In Figure \ref{fig:beam_energy_gen_alpha_parameter}, the impact of the numerical dissipation from the generalized-$\alpha$ scheme is illustrated. The generalized-$\alpha$ schemes with $\Delta t/T_0 =1.0 \times 10^{-3}$ and $\varrho_{\infty} = 0.0$, $0.5$, and $1.0$ are compared against the results of the mid-point rule with $\Delta t/T_0 = 1.0 \times 10^{-4}$. The numerical dissipation introduced by the time-stepping algorithm is negligible in comparison with the physical dissipation generated by the model from the energy evolution. In the detailed view shown in the three bottom figures of Figure \ref{fig:beam_energy_gen_alpha_parameter}, we see that, as expected, choosing a smaller value of $\varrho_{\infty}$ induces more numerical dissipation since this parameter dictates the dissipation on the high-frequency modes according to the analysis of linear problems. The differences between the two simulations using $\varrho_{\infty} = 1.0$ with different time step sizes are indistinguishable. The HS model leads to the fastest dissipation of the energy.

\begin{figure}[!htbp]
\begin{tabular}{ c c c }
\multicolumn{3}{c}{
\includegraphics[angle=0, trim=150 80 150 100, clip=true, scale = 0.45]{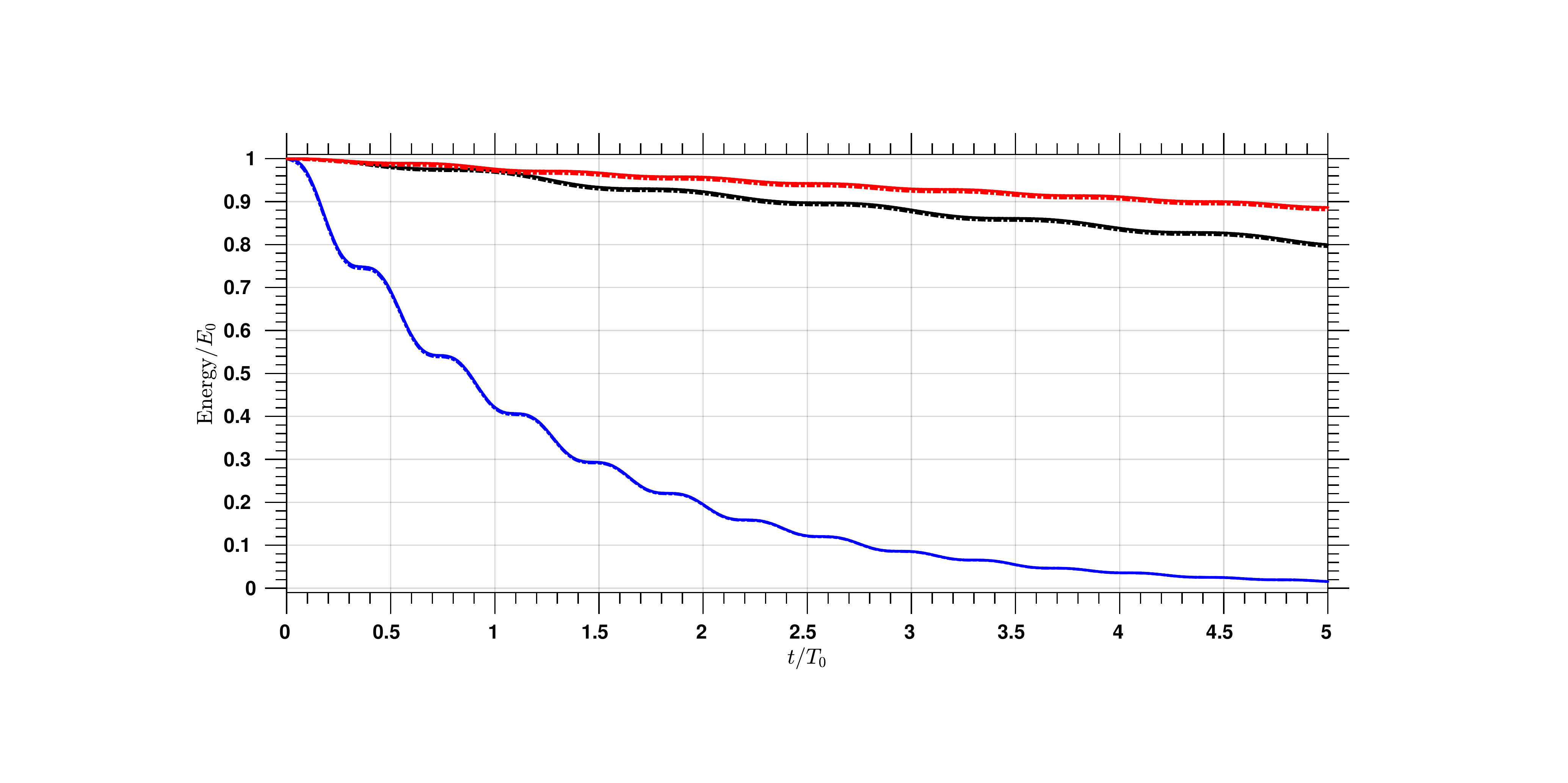} } \\
\includegraphics[angle=0, trim=80 80 150 100, clip=true, scale = 0.21]{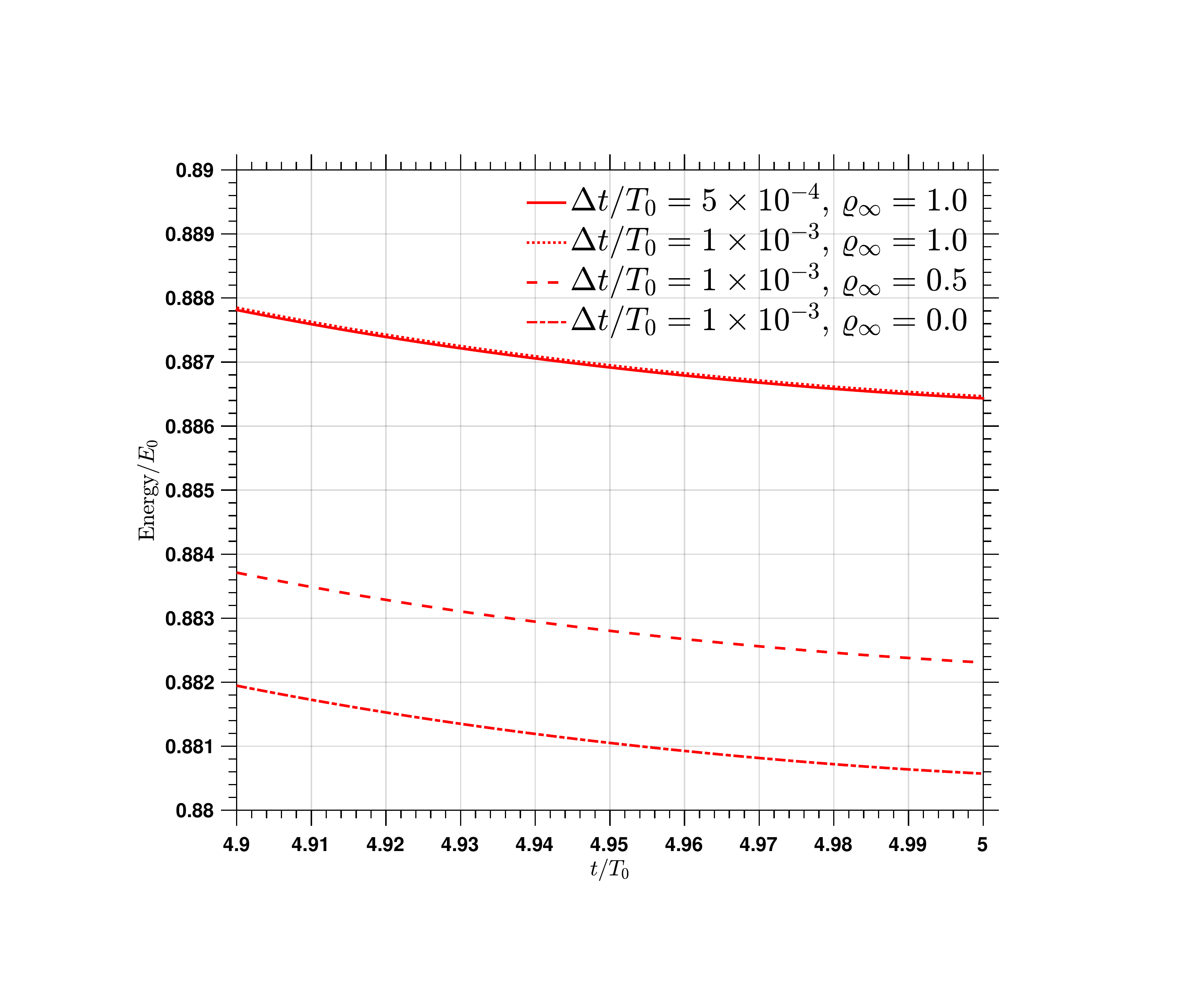} &
\includegraphics[angle=0, trim=80 80 150 100, clip=true, scale = 0.21]{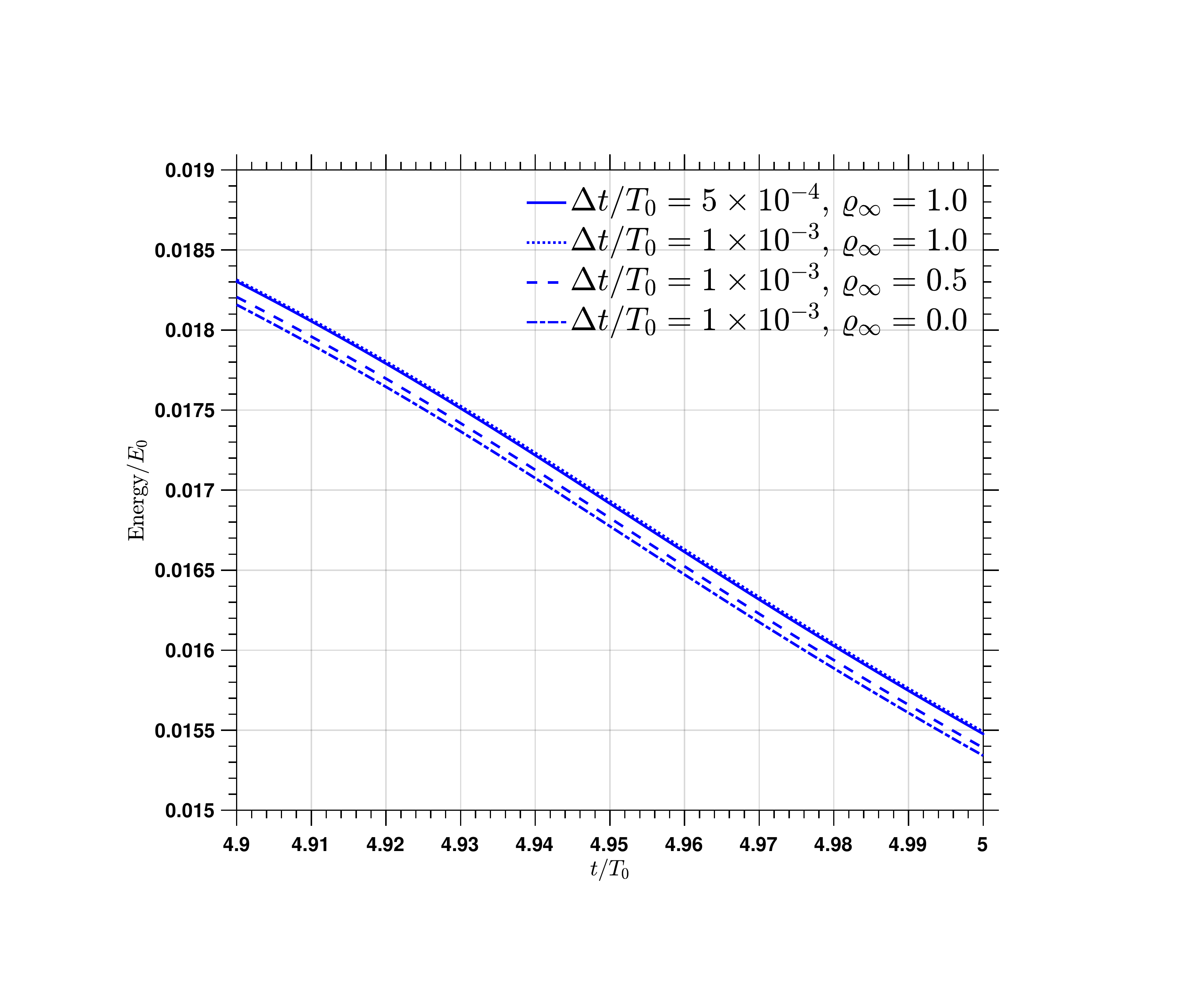} &
\includegraphics[angle=0, trim=80 80 150 100, clip=true, scale = 0.21]{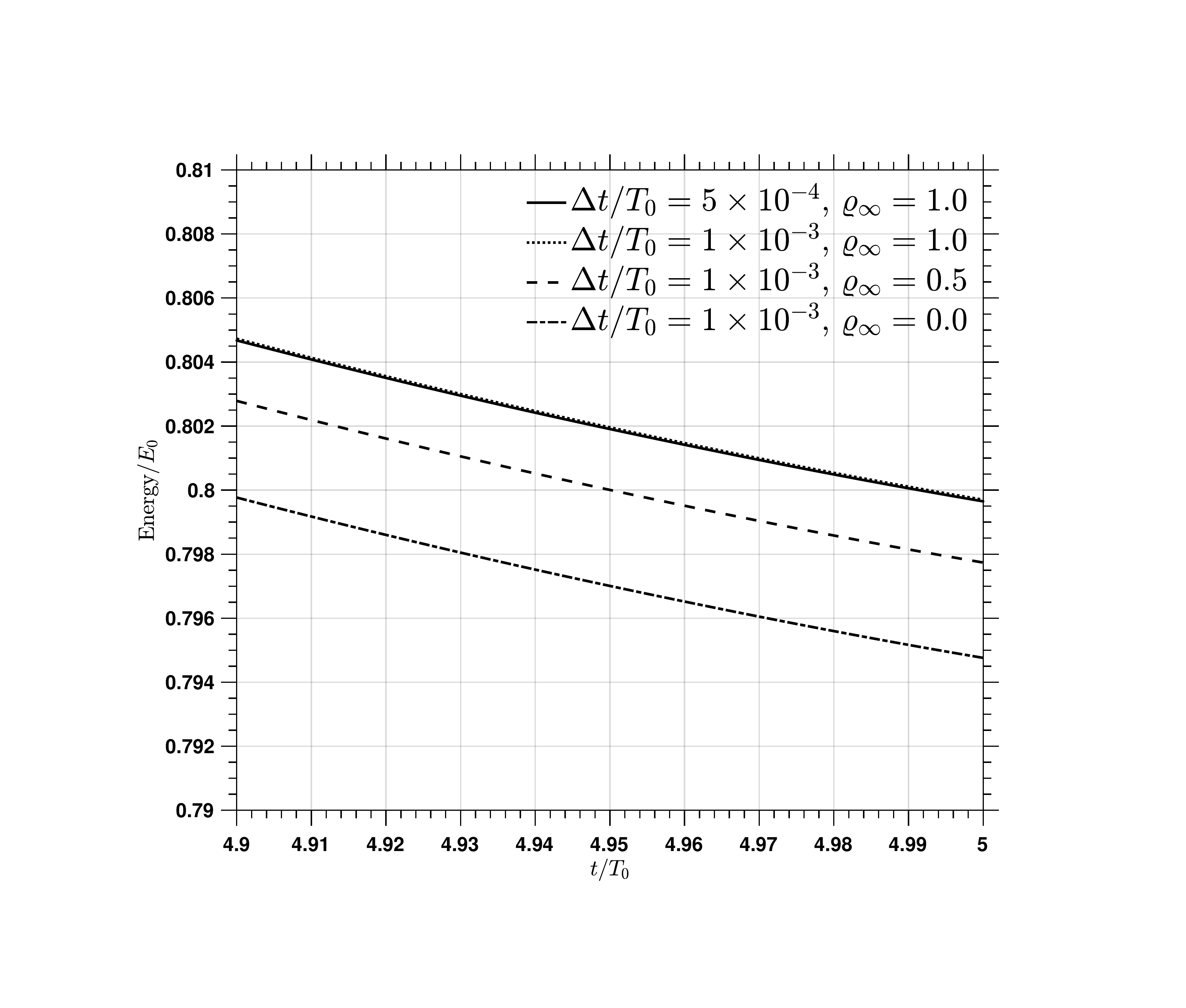} \\
IPC & HS & MIPC
\end{tabular}
\caption{The total energies of the IPC (red), HS (blue), and MIPC (black) models over time. The simulations are performed with a spatial mesh with $\mathsf p = 2$, $\mathsf a=1$, $\mathsf b = 0$, and $5 \times 5 \times 30$ elements. The reference value $E_0$ is chosen to be the total energy at time $t/T_0=0$, which is $1.1\times 10^5$ kg m$^2$/s$^2$. Detailed view of the energies in the vicinity of $t/T_0=5$ is depicted in the bottom.} 
\label{fig:beam_energy_gen_alpha_parameter}
\end{figure}

\begin{figure}[!htbp]
\begin{tabular}{ c c c c c c c c c }
\multicolumn{9}{c}{
\includegraphics[angle=0, trim=0 0 280 750, clip=true, scale = 0.26]{./figures/t2d5-pressure-color-bar-horizontal.jpeg}
} \\
\includegraphics[angle=0, trim=530 0 600 0, clip=true, scale = 0.11]{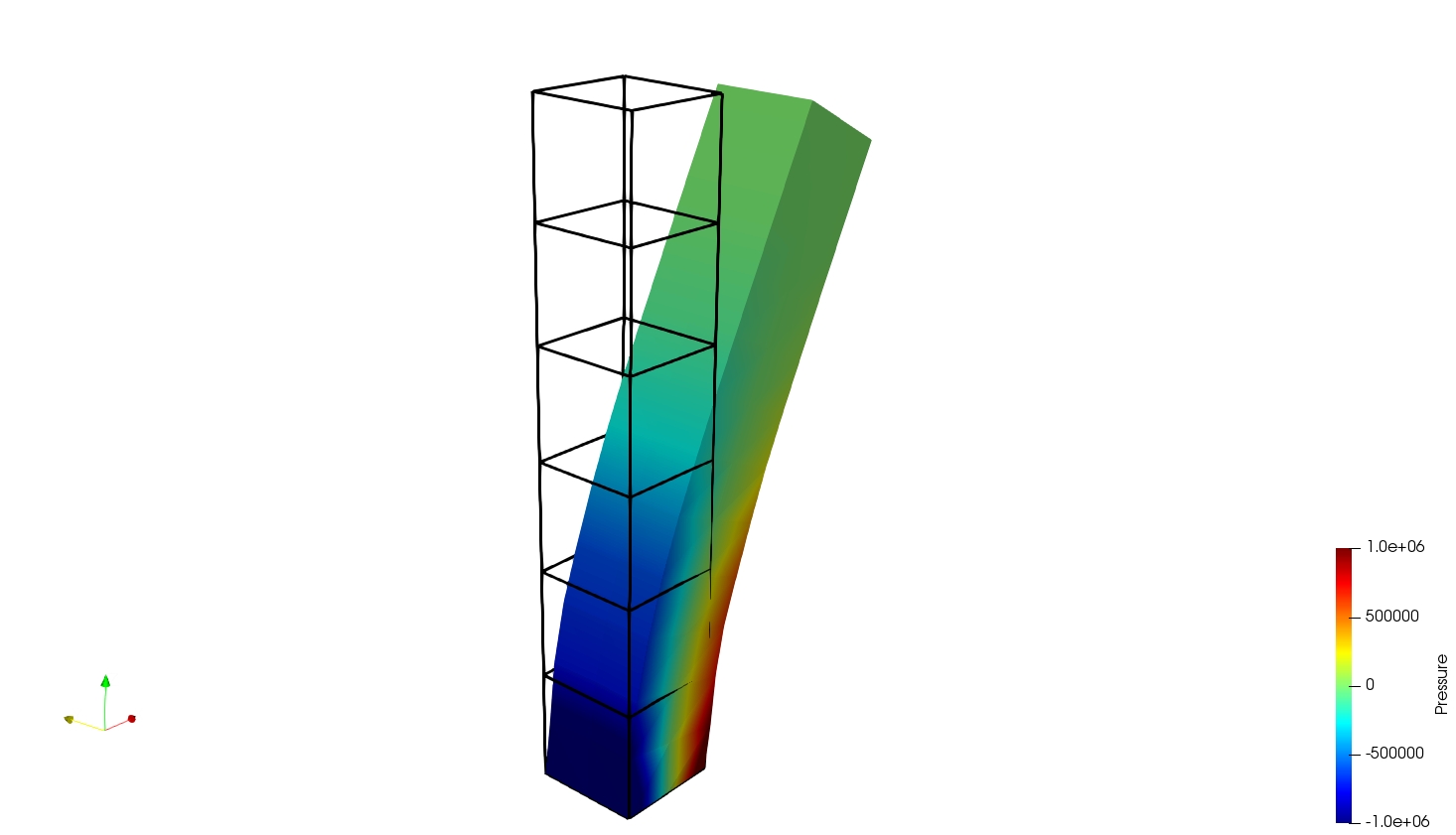} &
\includegraphics[angle=0, trim=530 0 600 0, clip=true, scale = 0.11]{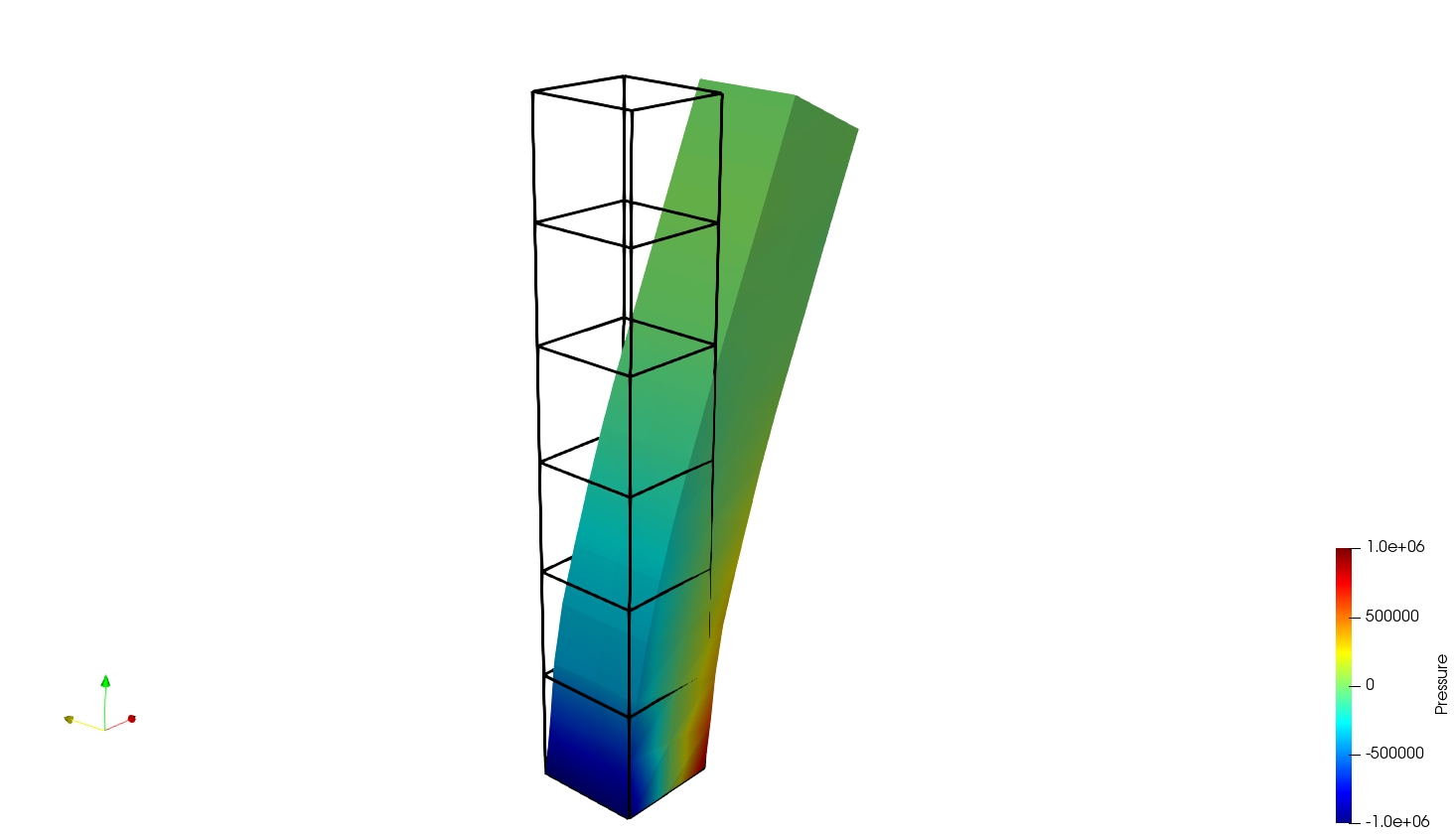} &
\includegraphics[angle=0, trim=530 0 600 0, clip=true, scale = 0.11]{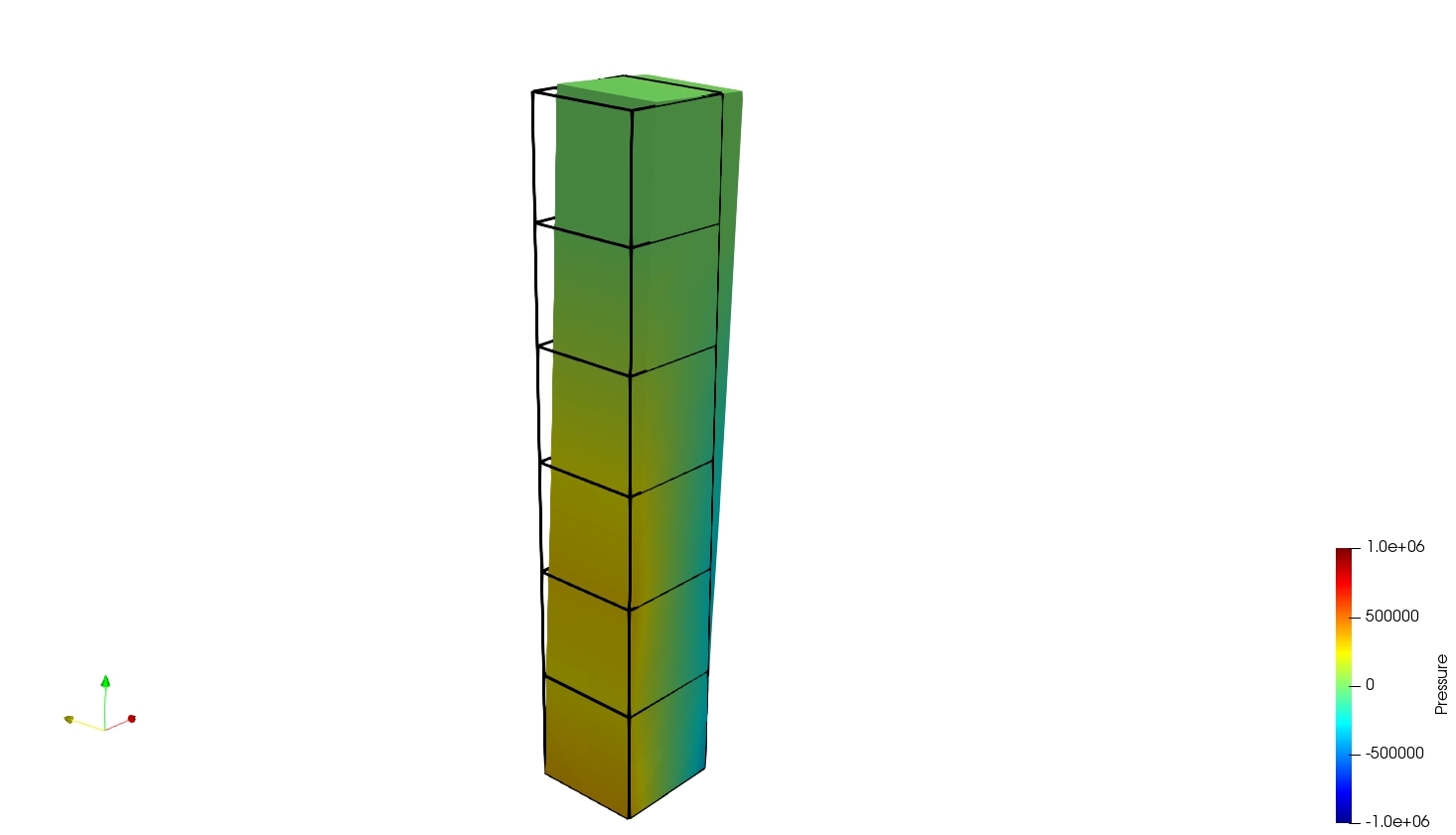} &
\includegraphics[angle=0, trim=470 0 680 0, clip=true, scale = 0.11]{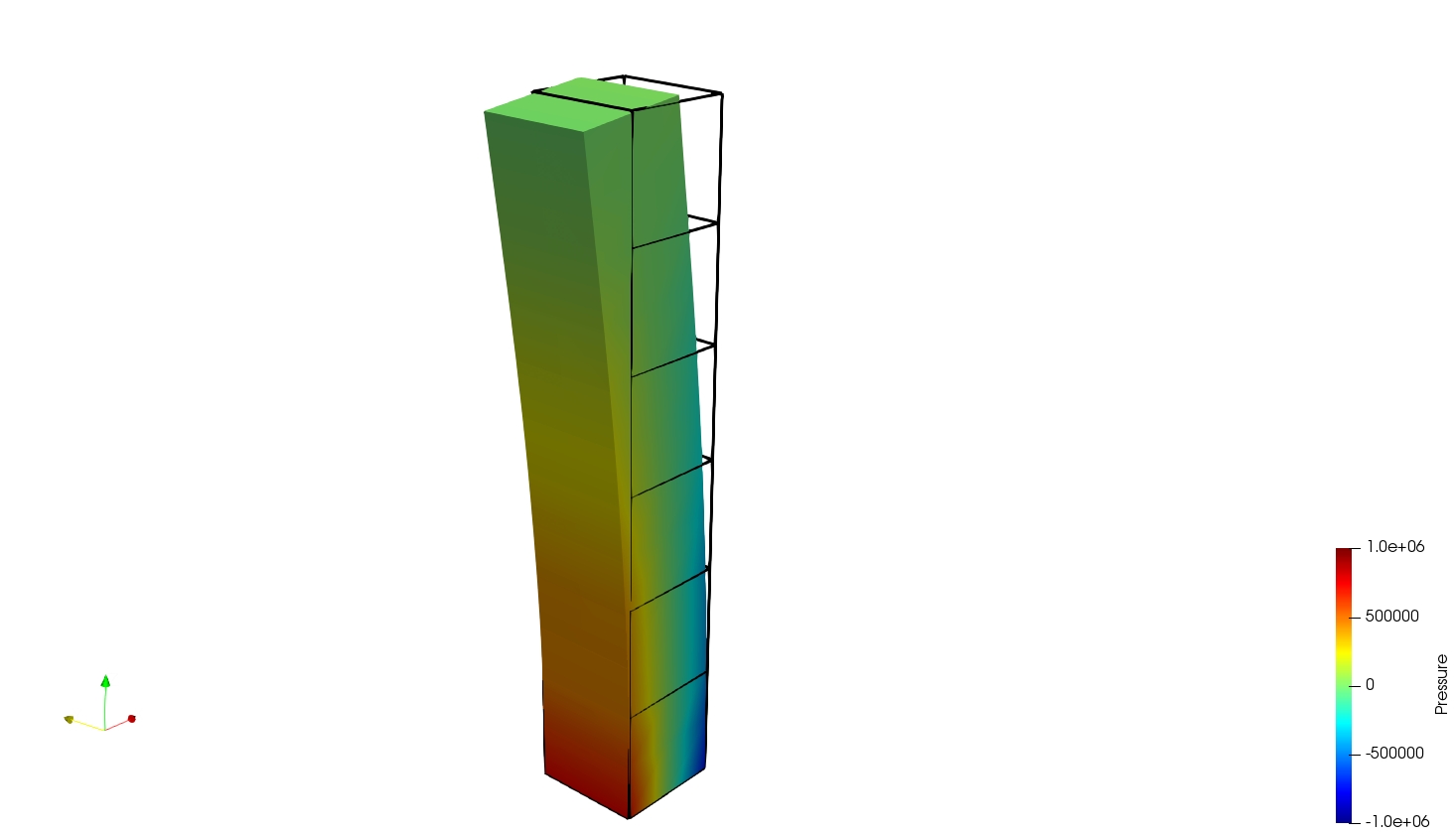} &
\includegraphics[angle=0, trim=530 0 600 0, clip=true, scale = 0.11]{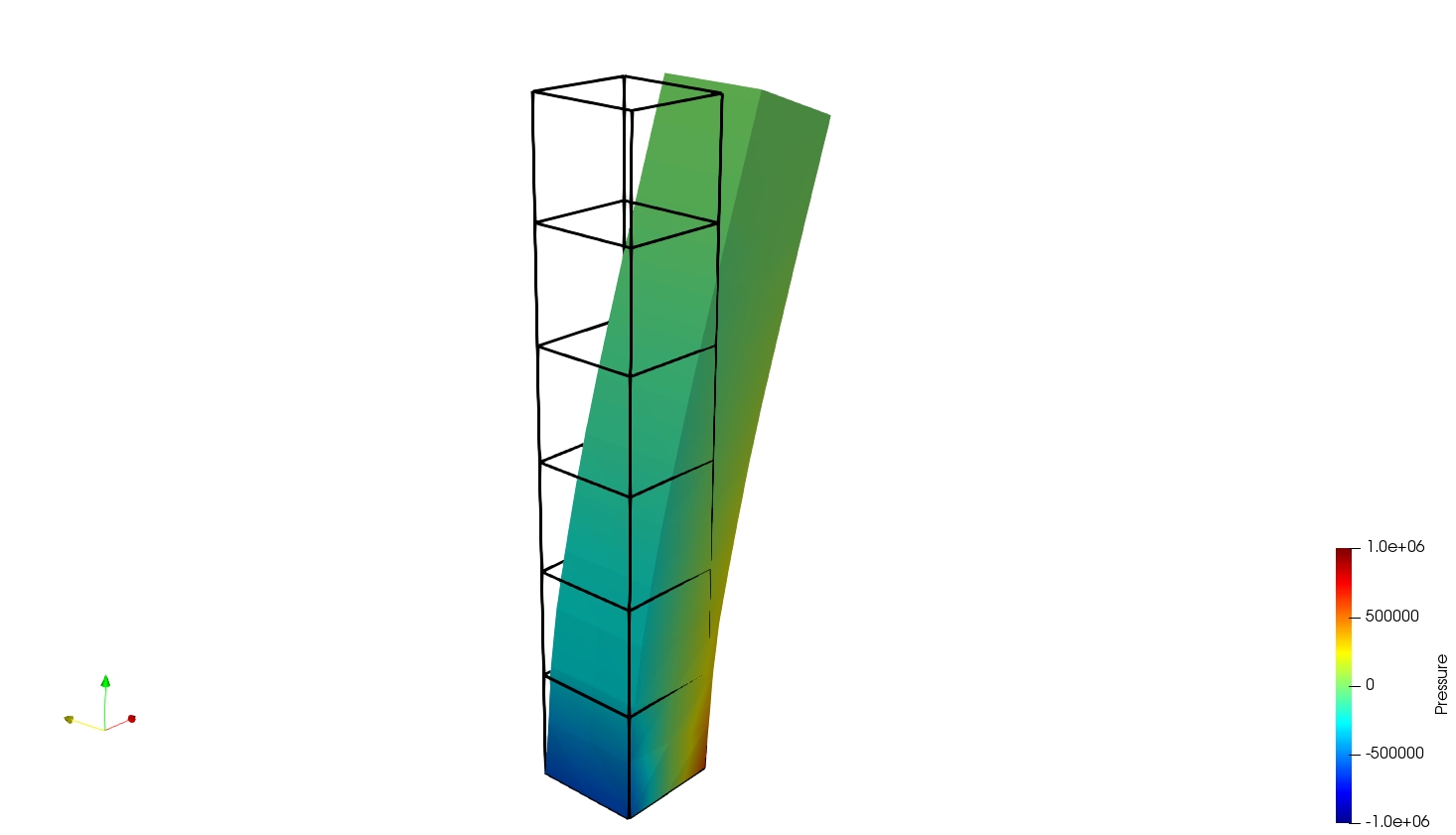} &
\includegraphics[angle=0, trim=530 0 600 0, clip=true, scale = 0.11]{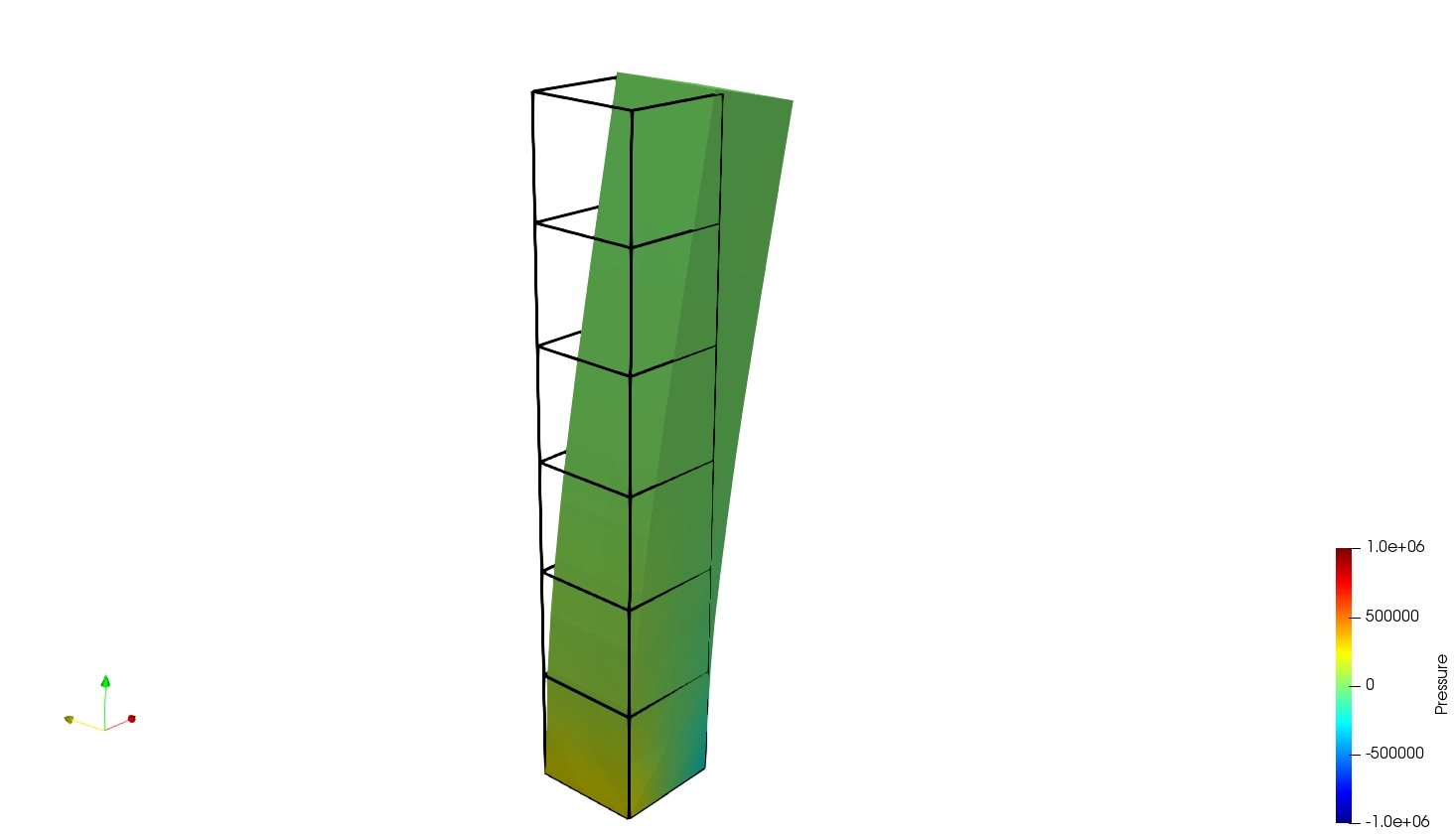} &
\includegraphics[angle=0, trim=530 0 600 0, clip=true, scale = 0.11]{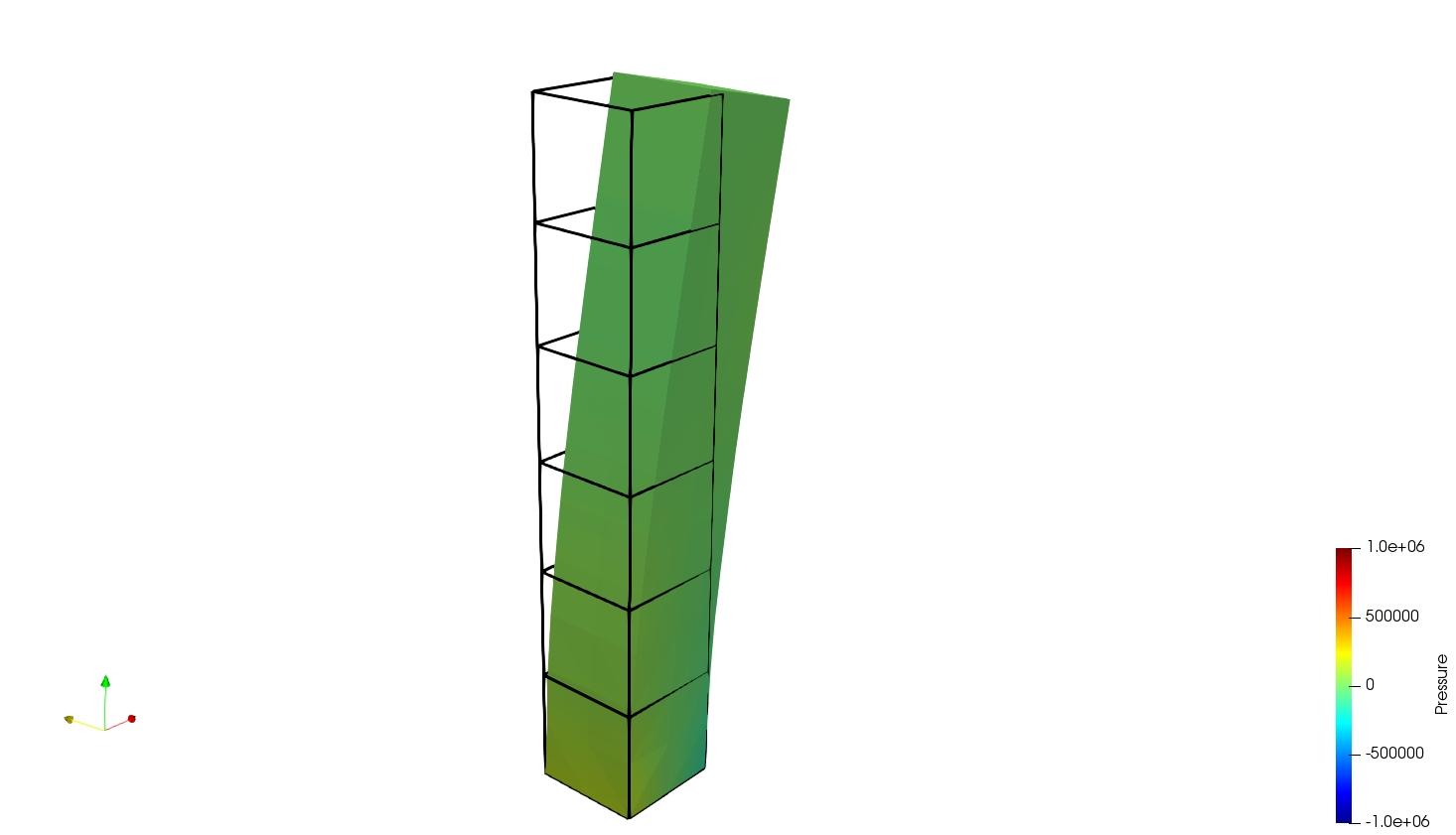} &
\includegraphics[angle=0, trim=530 0 600 0, clip=true, scale = 0.11]{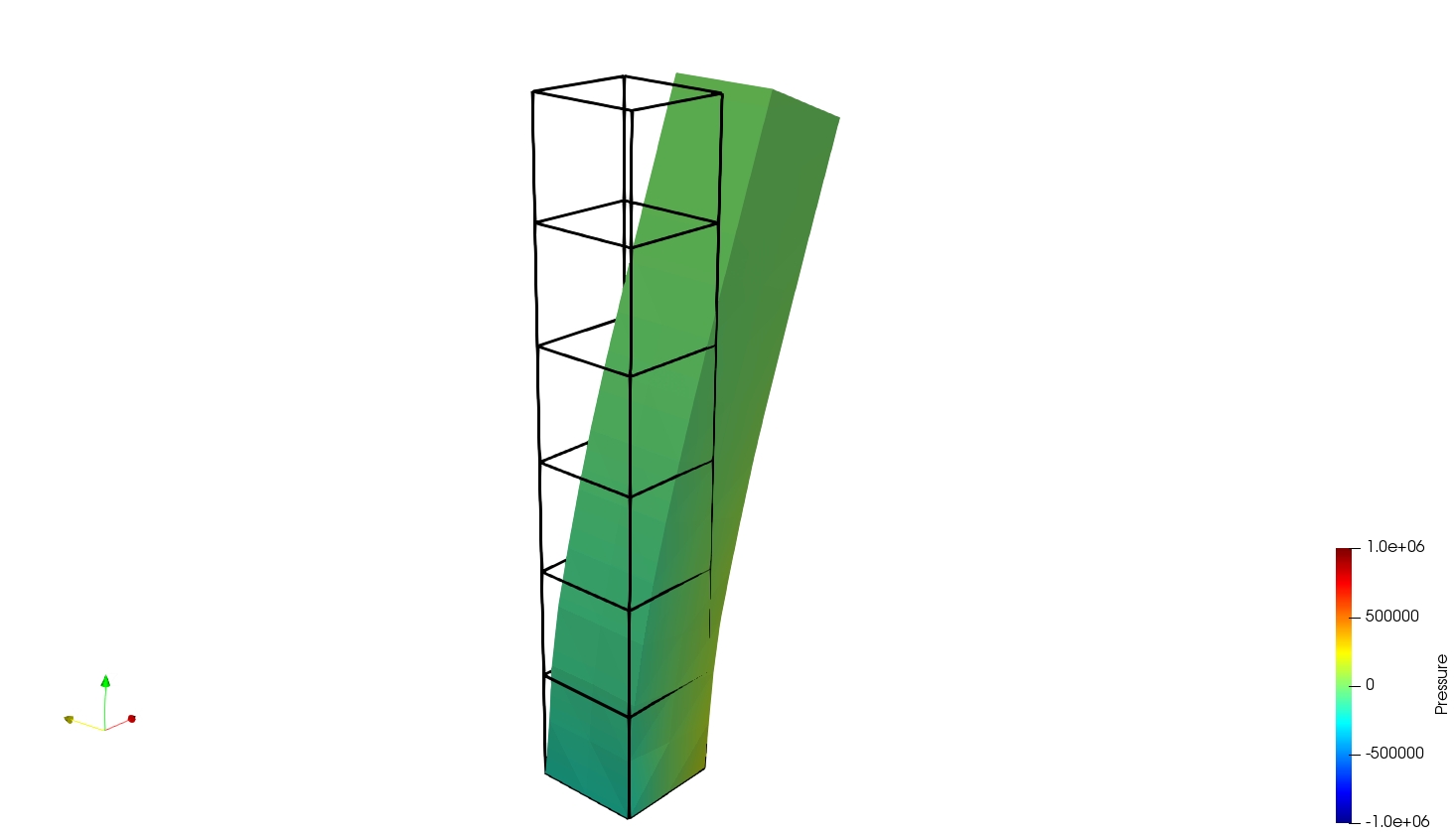} &
\includegraphics[angle=0, trim=530 0 600 0, clip=true, scale = 0.11]{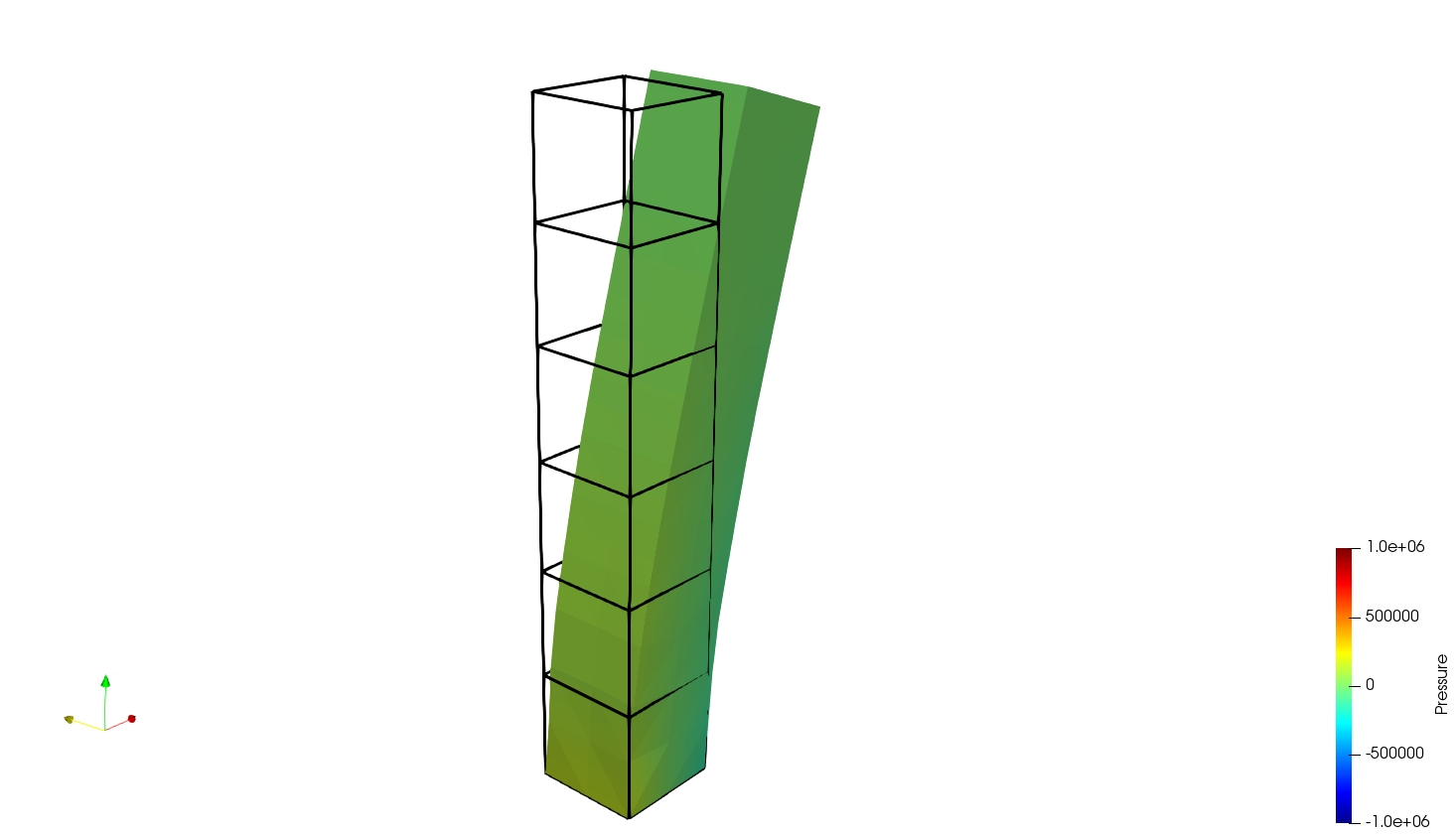} \\
\includegraphics[angle=0, trim=530 0 600 0, clip=true, scale = 0.11]{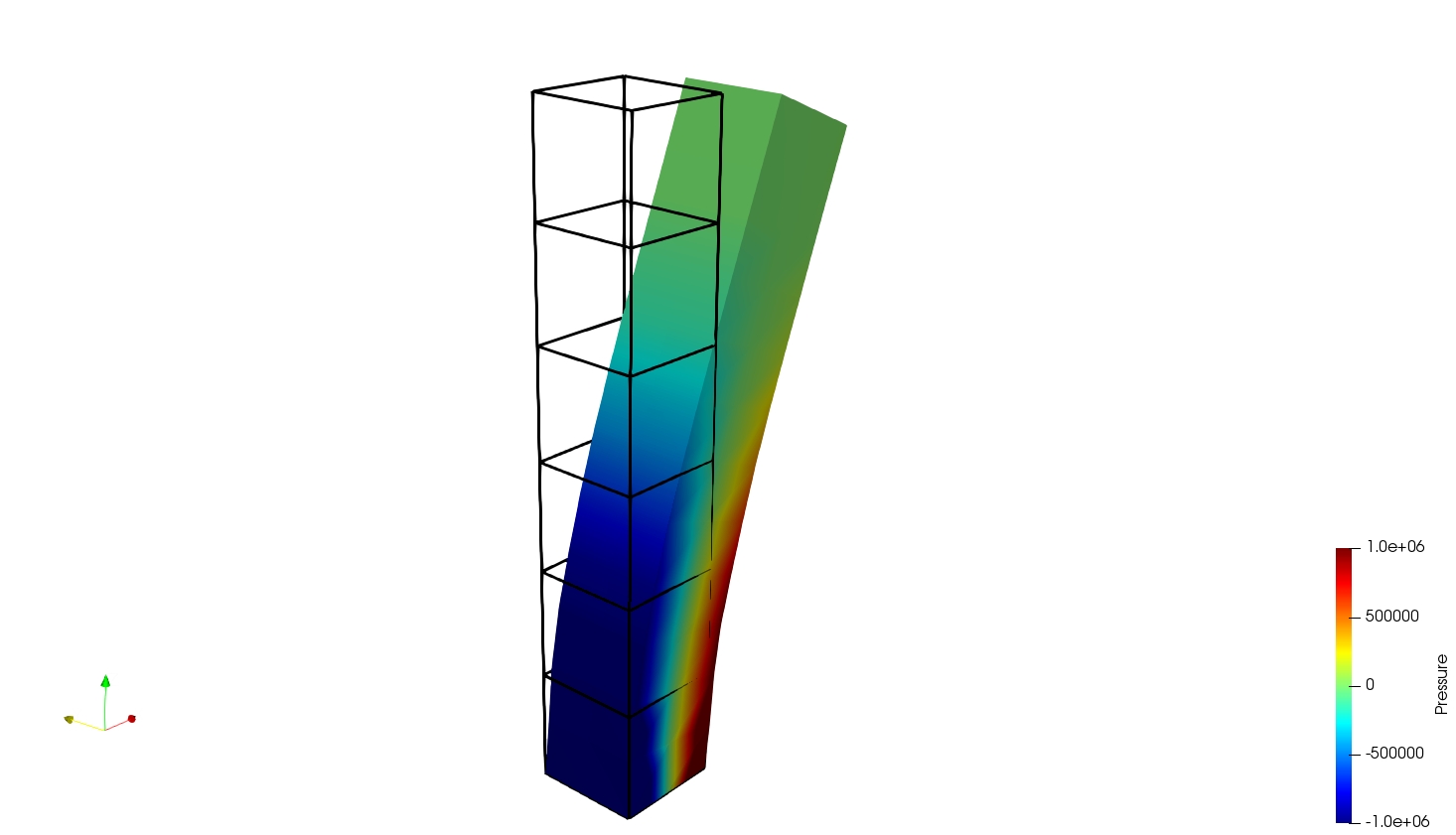} &
\includegraphics[angle=0, trim=530 0 720 0, clip=true, scale = 0.11]{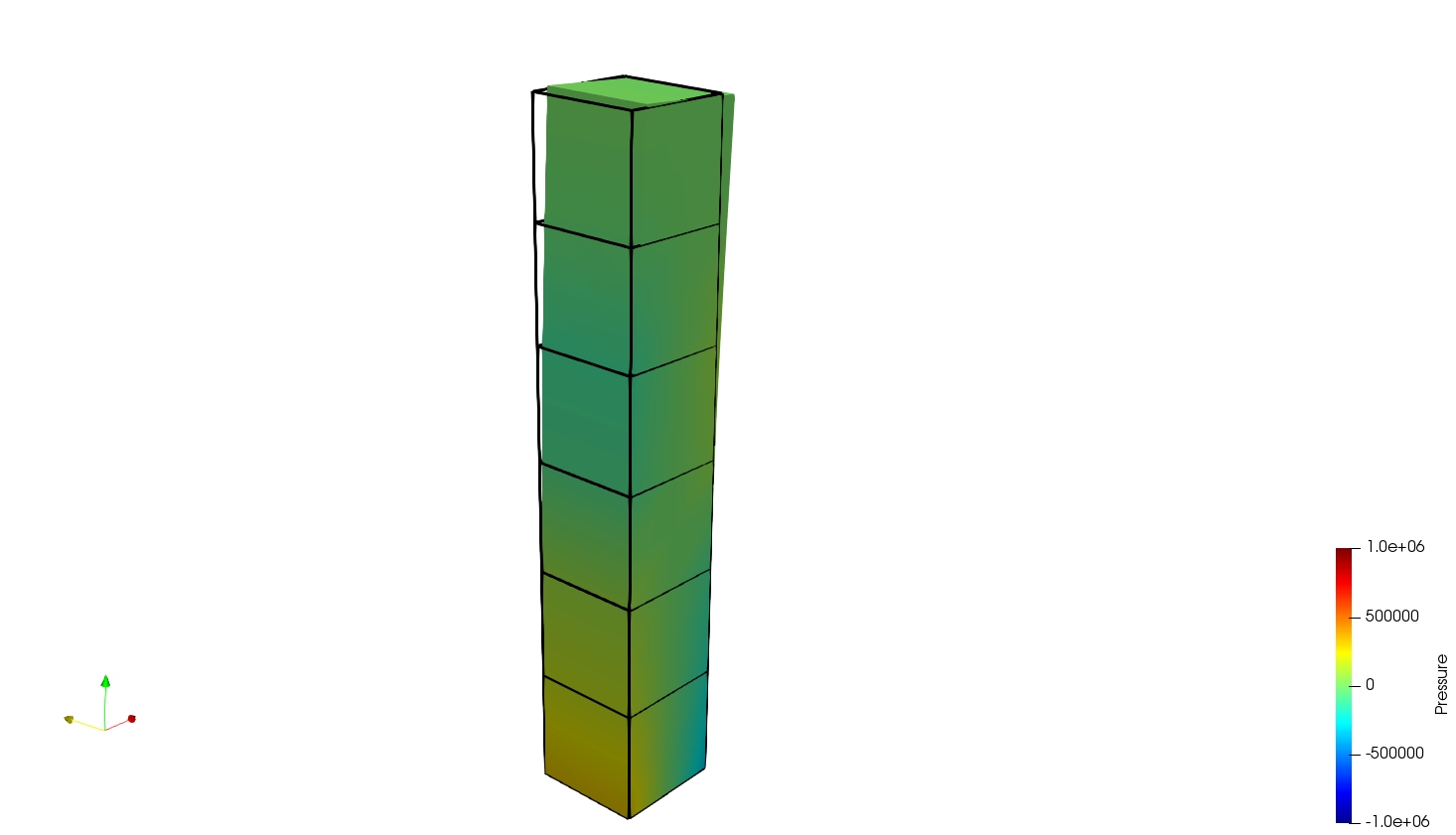} &
\includegraphics[angle=0, trim=400 0 720 0, clip=true, scale = 0.11]{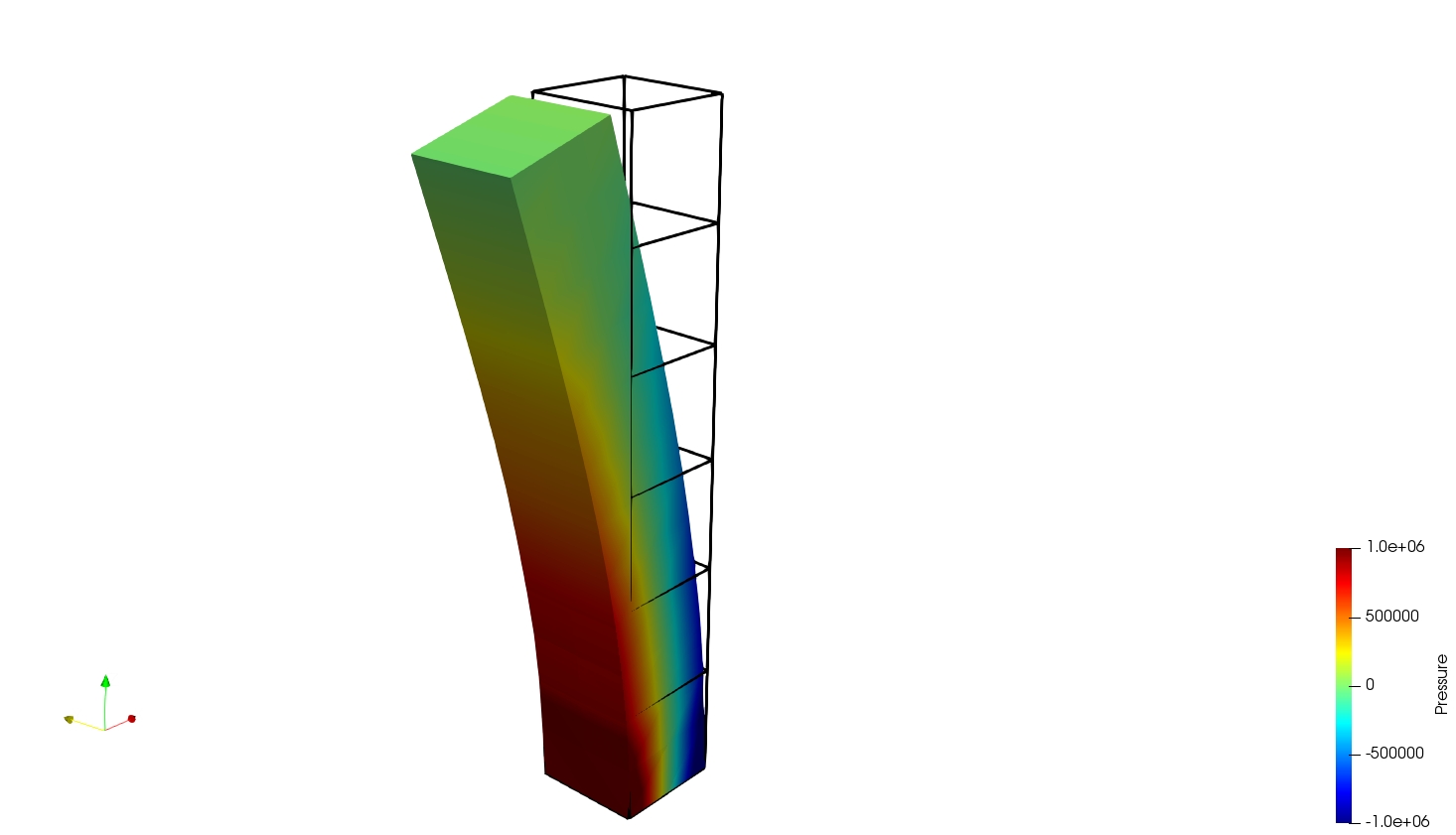} &
\includegraphics[angle=0, trim=530 0 680 0, clip=true, scale = 0.11]{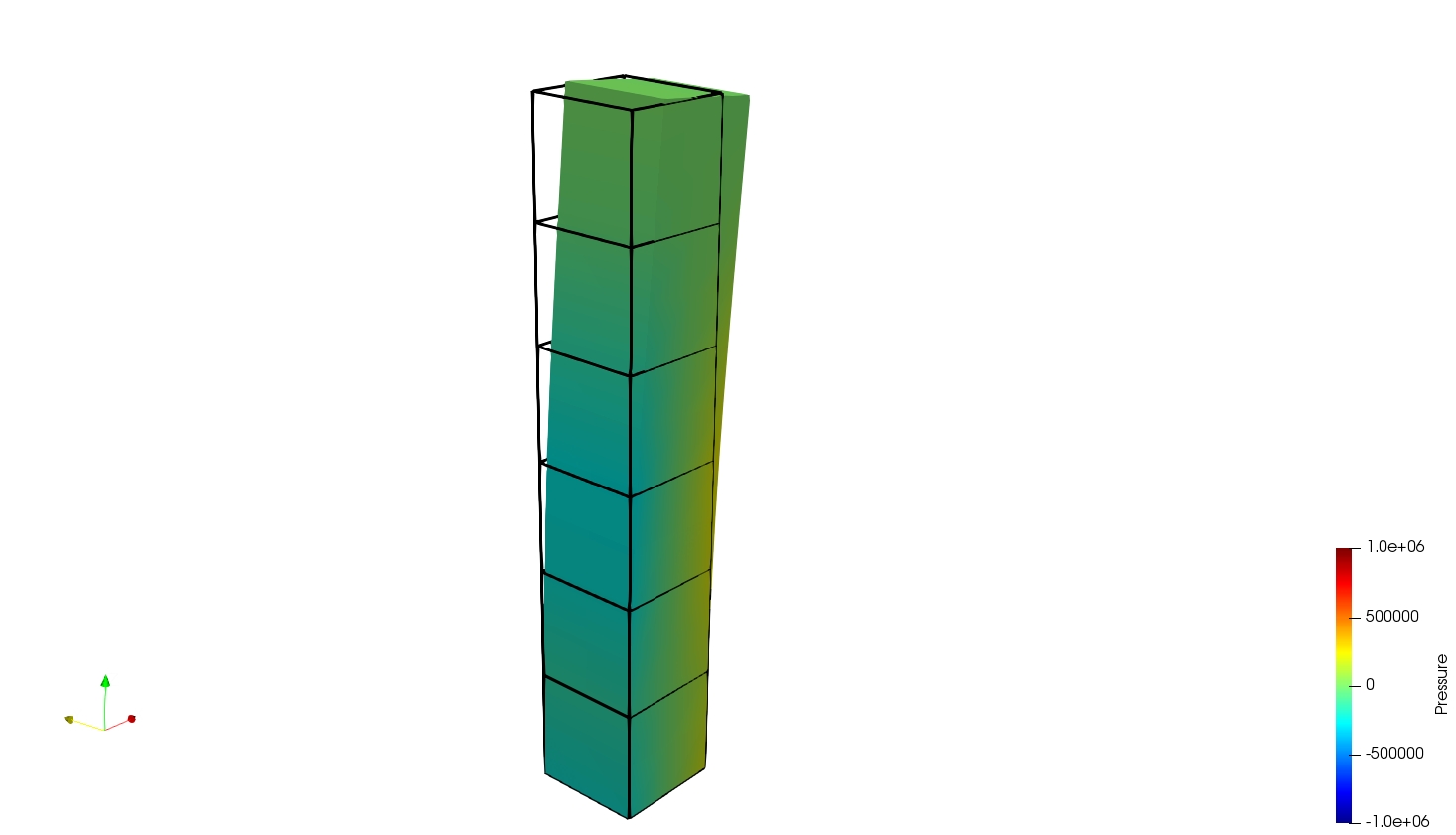} &
\includegraphics[angle=0, trim=530 0 700 0, clip=true, scale = 0.11]{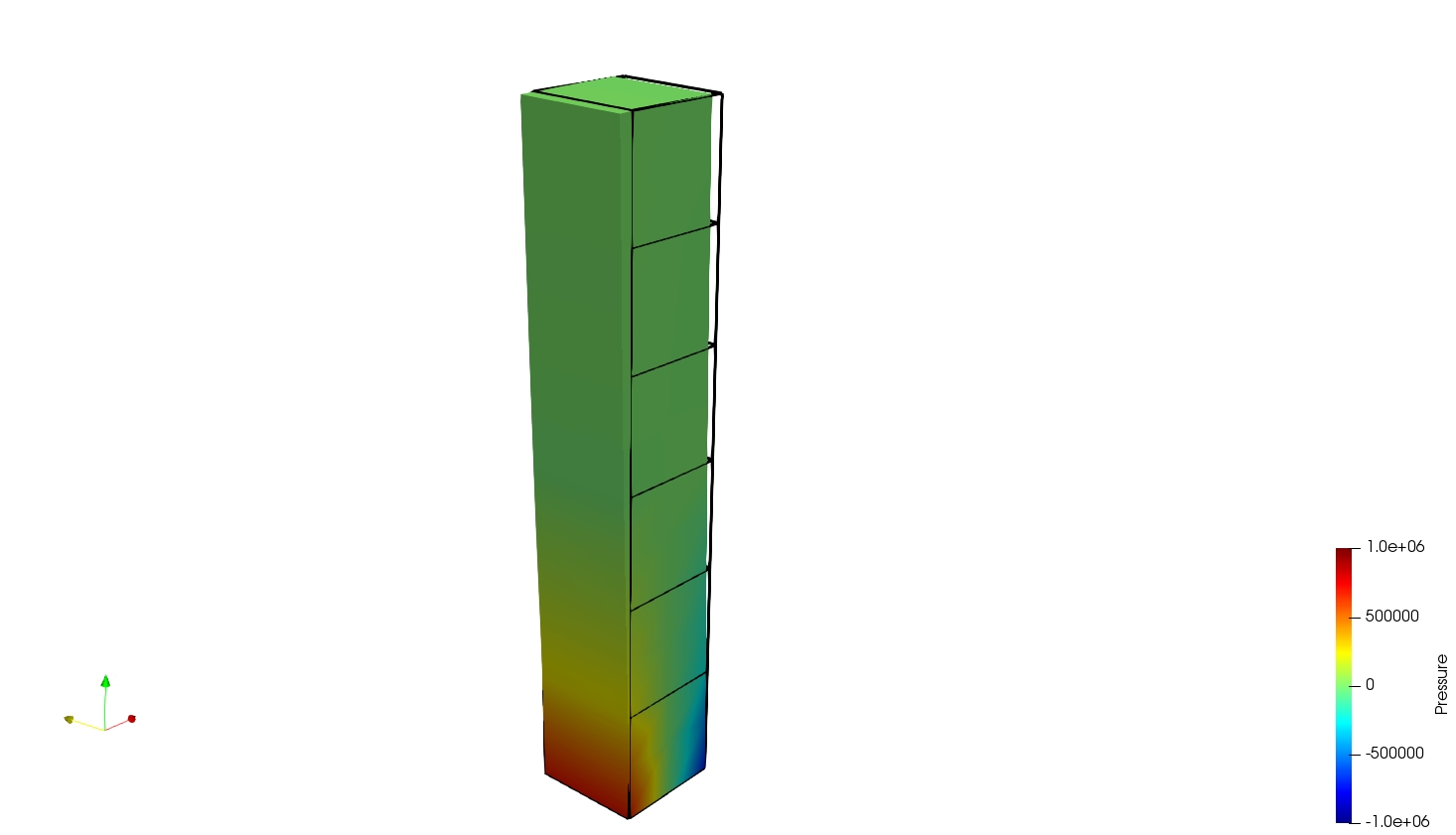} &
\includegraphics[angle=0, trim=530 0 680 0, clip=true, scale = 0.11]{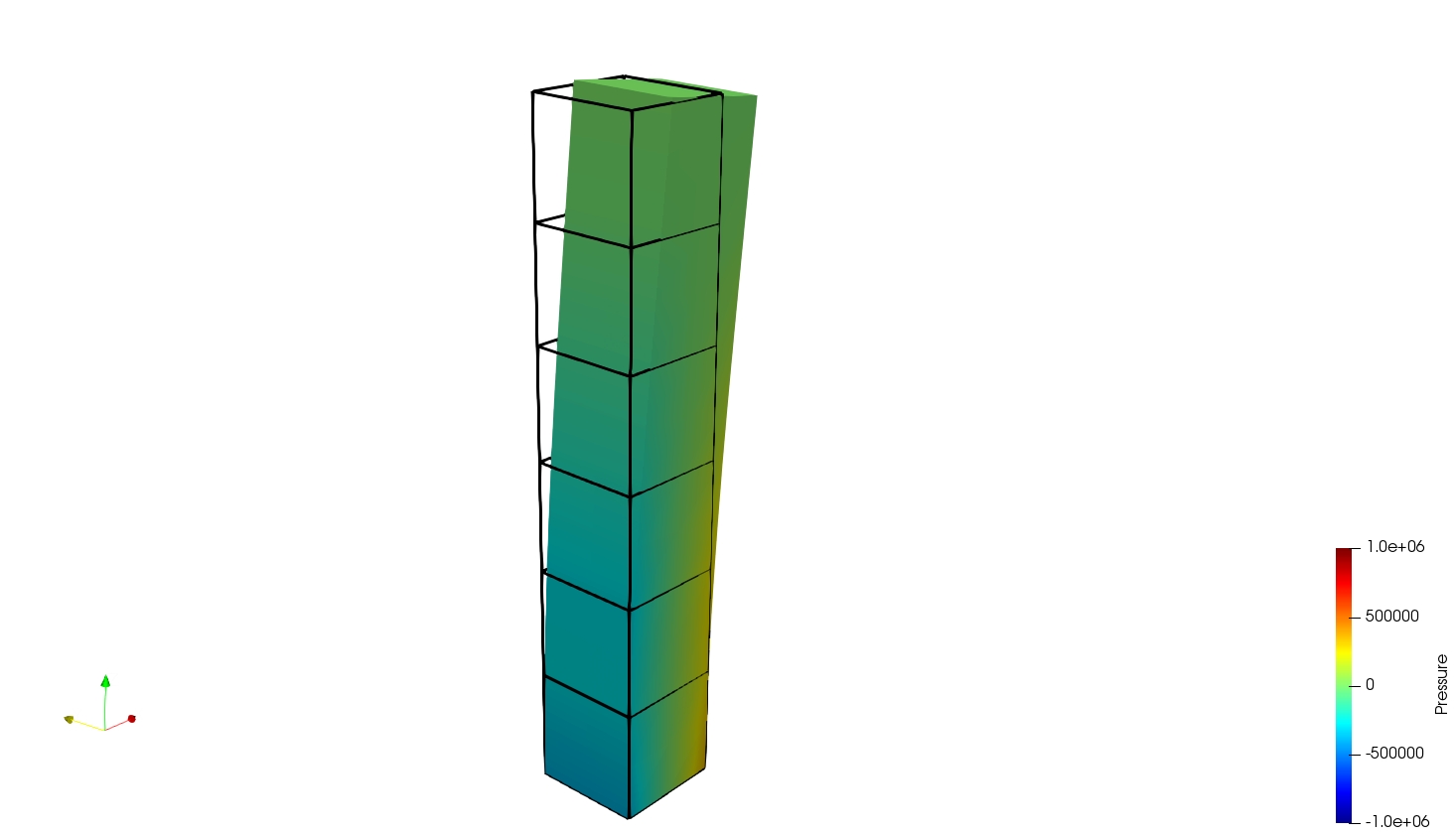} &
\includegraphics[angle=0, trim=480 0 680 0, clip=true, scale = 0.11]{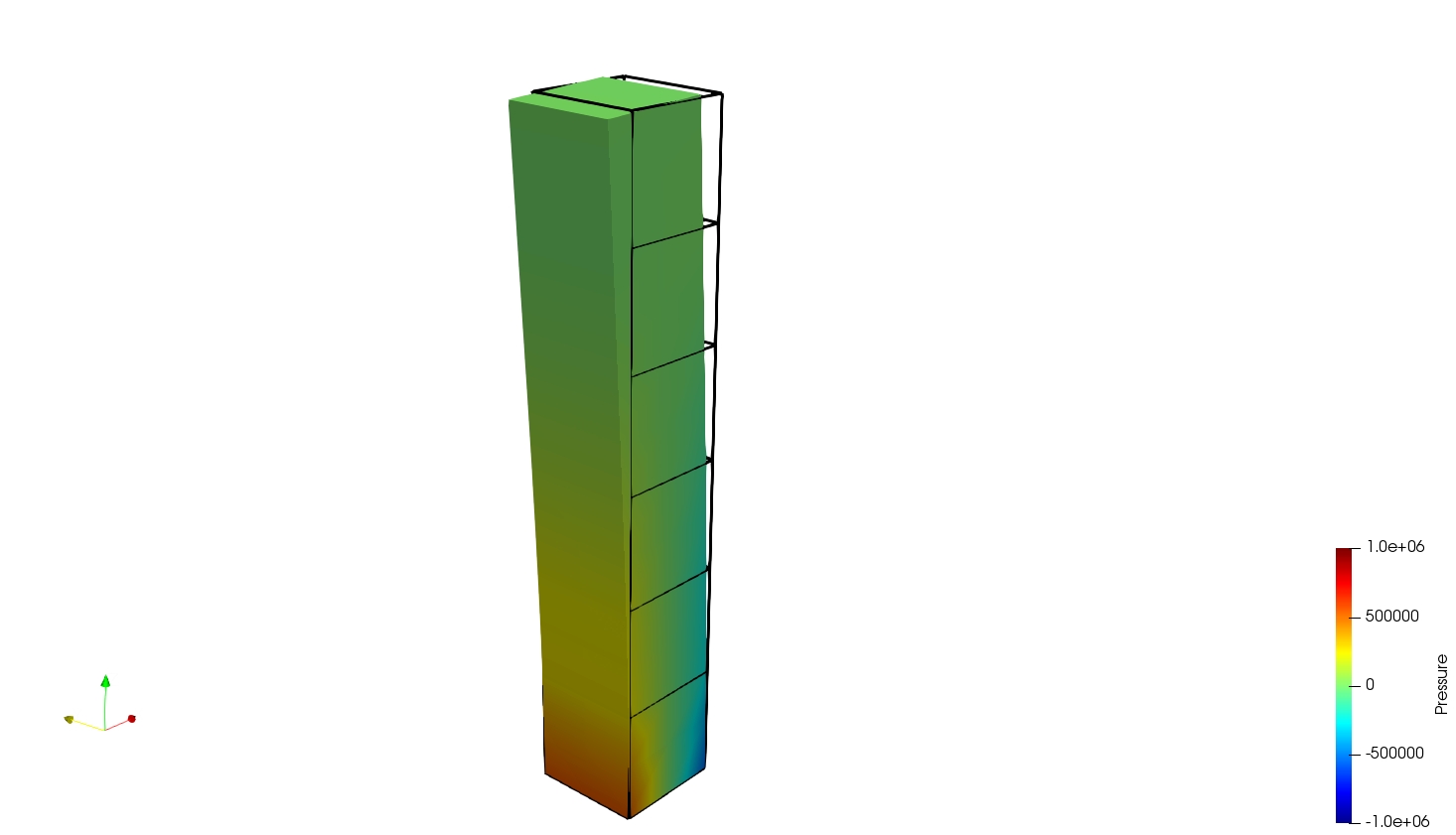} &
\includegraphics[angle=0, trim=530 0 680 0, clip=true, scale = 0.11]{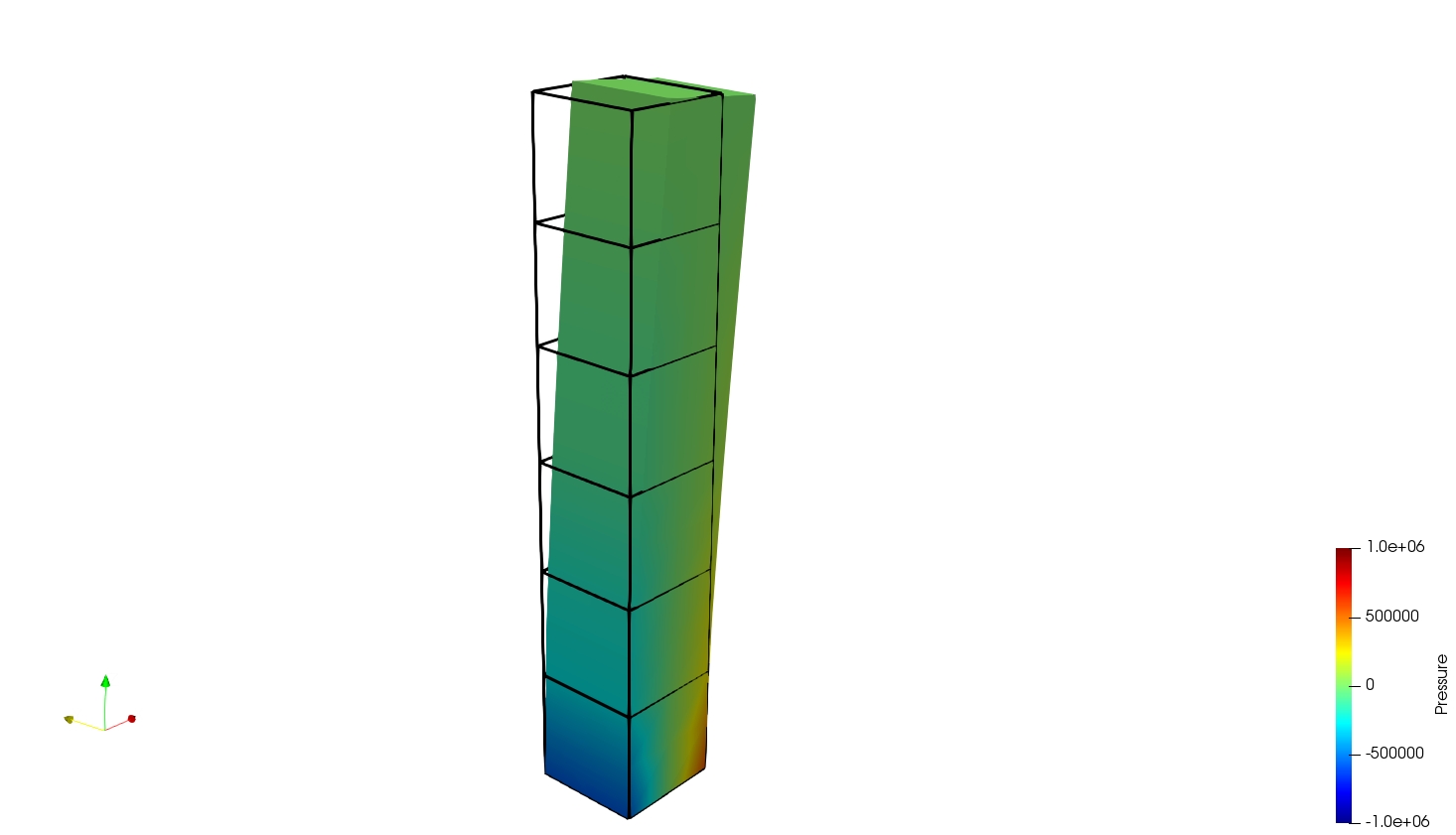} &
\includegraphics[angle=0, trim=480 0 680 0, clip=true, scale = 0.11]{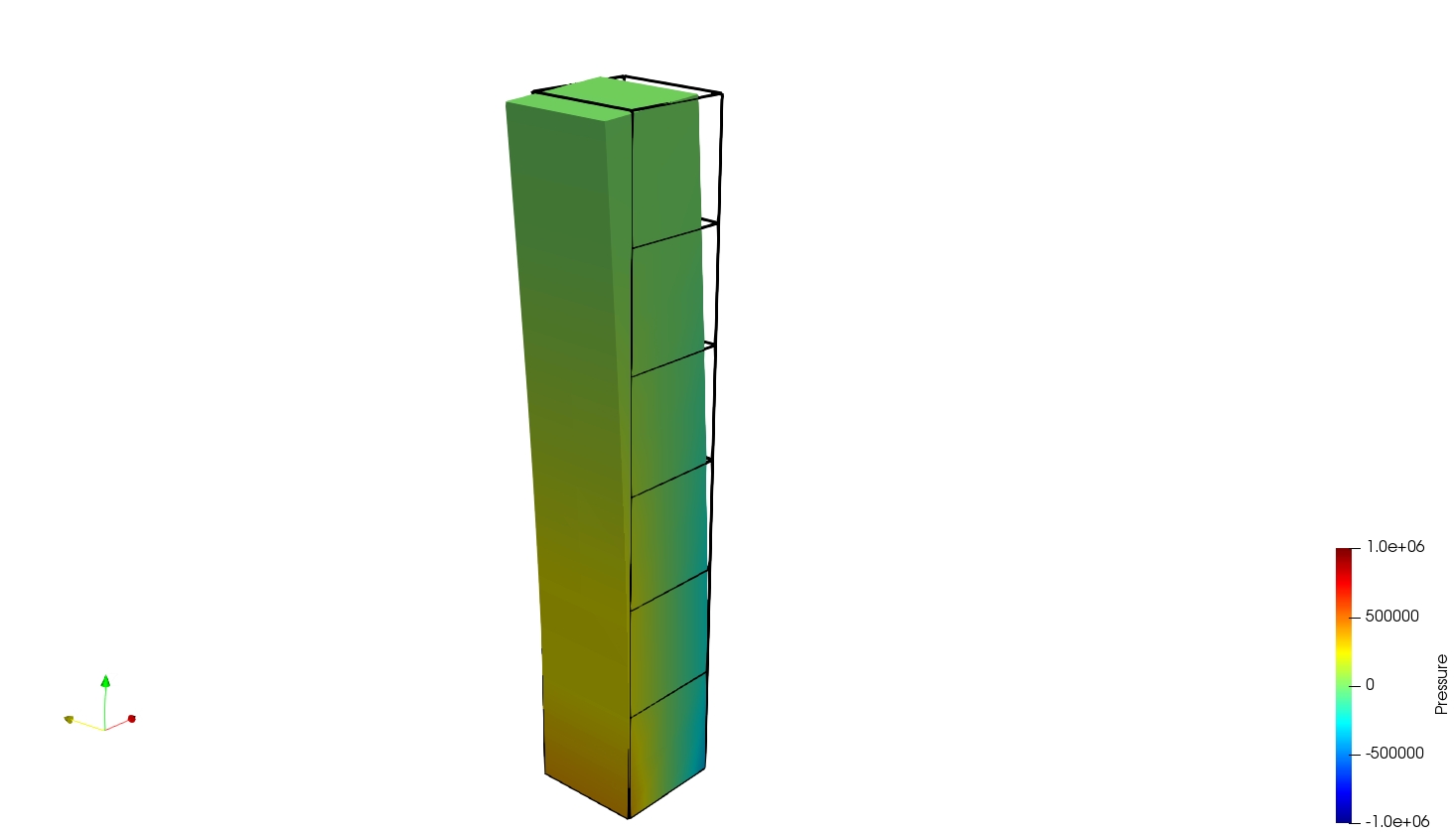} \\
$\frac{t}{T_0}=0.25$ & $0.50$ & $0.75$ & $1.00$ & $1.50$ & $2.00$ & $2.50$ & $3.00$ & $3.50$
\end{tabular}
\caption{The snapshots of the pressure fields at nice time instances of the IPC (top) and MIPC (bottom) models with a fixed time step size $\Delta t/T_0 = 1 \times 10^{-3}$ plotted on the deformed configuration. The material parameters are the same as those reported in Table \ref{table:3d_beam_bending_benchmark_geometry}, except $\mu^1 = 0.2 c_1$ here. The simulations are performed using a spatial mesh with $\mathsf p = 1$, $\mathsf a=1$, $\mathsf b = 0$, and $1 \times 1 \times 6$ elements. The meshes at the initial time are plotted as the black grid.} 
\label{fig:compare-IPC-MIPC}
\end{figure}

\begin{figure}[!htbp]
\begin{center}
\begin{tabular}{c}
\includegraphics[angle=0, trim=85 90 130 80, clip=true, scale = 0.45]{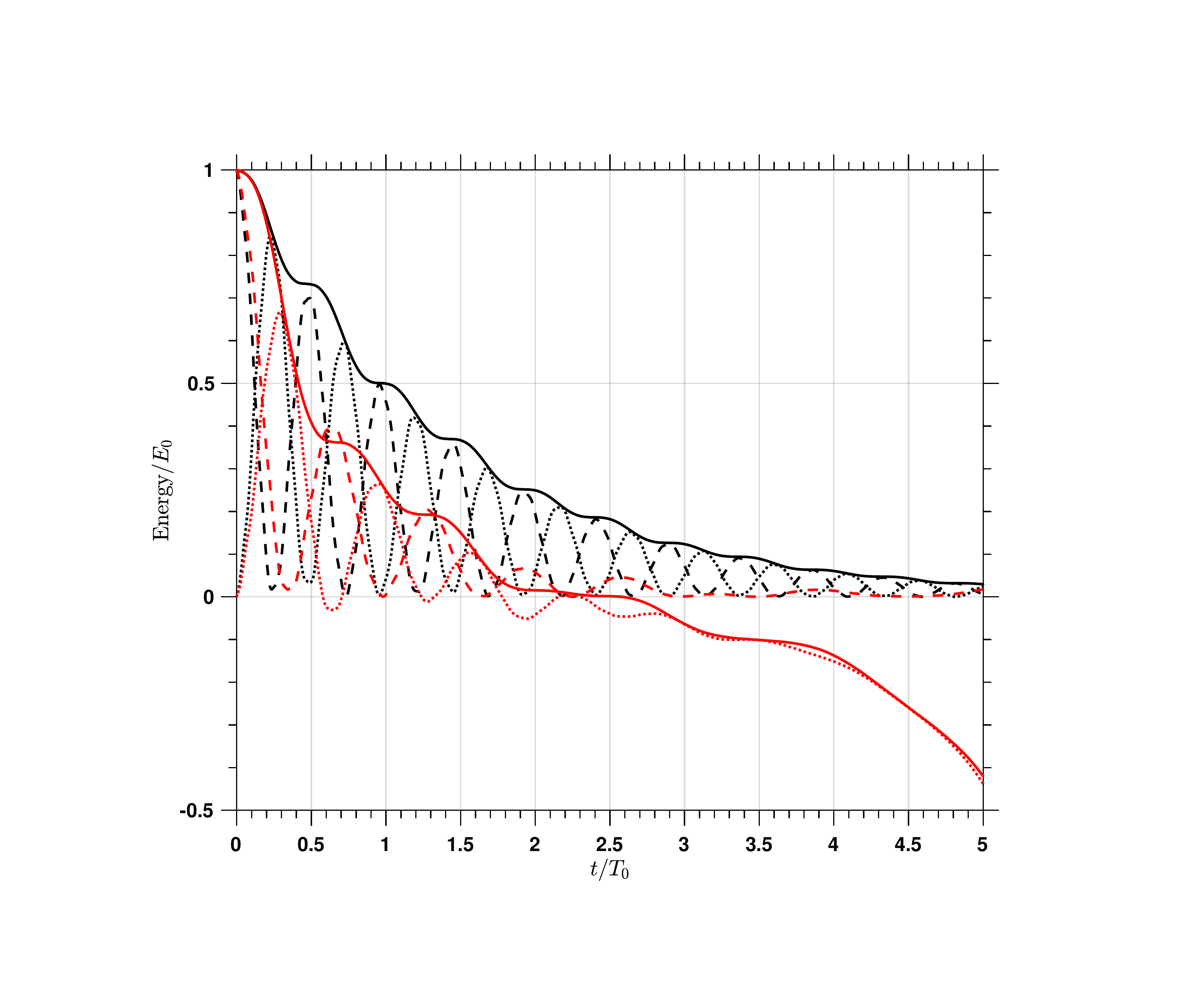}
\end{tabular}
\caption{The total, kinetic, and potential energies of the IPC (red) and MIPC (black) models over time. The simulations are performed with $\mu^1 = 0.2 c_1$, a fixed time step size $\Delta t/T_0 = 1\times 10^{-3}$, and a spatial mesh with $\mathsf p = 1$, $\mathsf a=1$, $\mathsf b = 0$, and $1 \times 1 \times 6$ elements. The solid, dashed, and dotted lines represent the total, kinetic, and potential energies, respectively. The reference value of the total energy $E_0$ is chosen to be the total energy at time $t=0$, which is $1.1\times 10^5$ kg m$^2$/s$^2$.} 
\label{fig:compare-IPC-MPIC-energy}
\end{center}
\end{figure}

We also report a suite of simulation results of the IPC and MIPC models with the modulus $\mu^1 = 0.2 c_1$, with all other settings equal to those in Table \ref{table:3d_beam_bending_benchmark_geometry}. The snapshots of the deformation states and the pressure field are depicted in Figure \ref{fig:compare-IPC-MIPC}. For the IPC model, the beam does not return to the original state as it reaches equilibrium. Instead, the beam is bent after $t/T_0 = 1.5$ and vibrates slightly around that bent configuration. The bent configuration is an unstable equilibrium configuration and eventually the simulation diverged. For the MIPC model, the beam eventually returns to the initial configuration with the energy dissipated due to the viscous effects. In Figure \ref{fig:compare-IPC-MPIC-energy}, the evolution of the energies is illustrated. The total energy of the IPC model reaches zero at $t/T_0 = 2.2$, after which the potential and total energies quickly become negative. Indeed, the IPC model cannot guarantee the boundedness of the configurational free energy (see, e.g. \eqref{eq:configuration-free-energy-thermo-limit}). Thus, the evolution to unbounded negative energy is indeed possible for this very model, as is discovered in this case. This unstable behavior of the IPC model again suggests the original IPC model, which results in non-vanishing non-equilibrium stresses, is materially unstable and may produce non-physical results.

\begin{table}[!htbp]
  \centering
  \begin{tabular}[t]{ m{.35\textwidth}   m{.45\textwidth} }
    \hline
    \begin{minipage}{.35\textwidth}
      \includegraphics[angle=0, trim=110 270 270 190, clip=true, scale = 0.6]{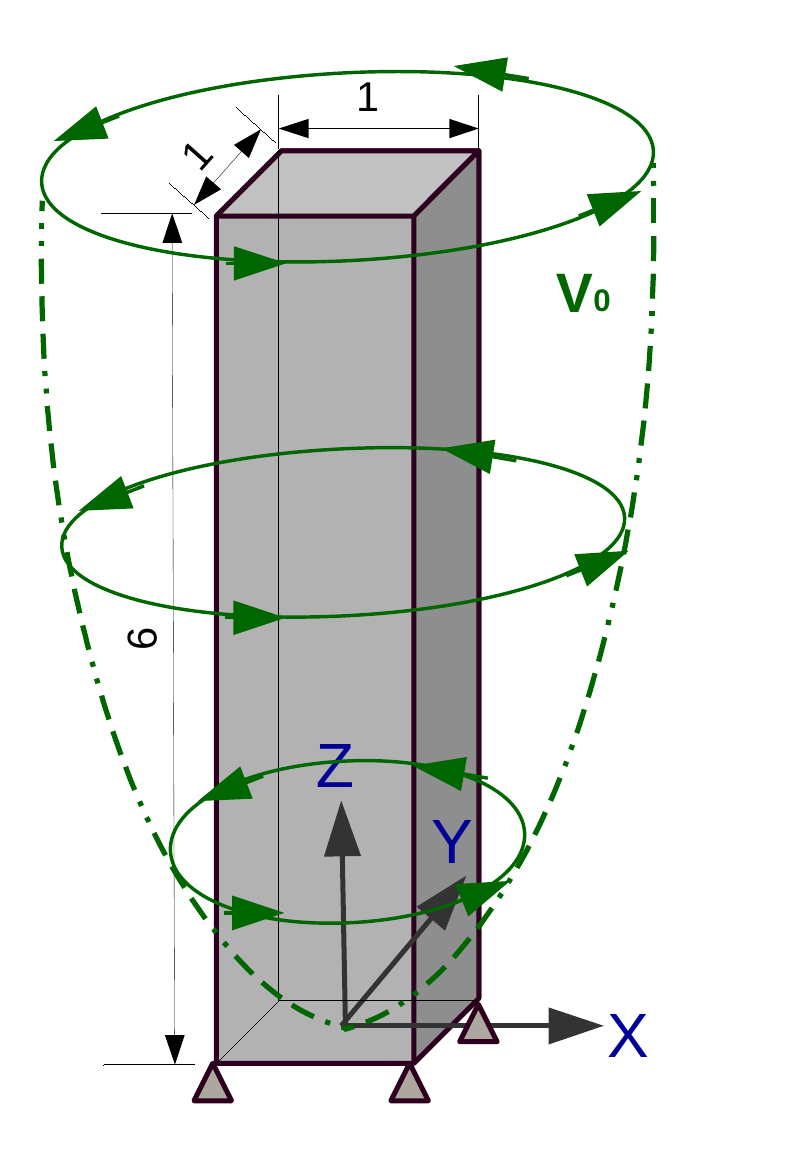}
    \end{minipage}
    &
    \begin{minipage}{.45\textwidth}
      \begin{itemize}
      \item[] Material properties:
      \item[] $G^{\infty}_{\mathrm{iso}} = \frac{c_1}{2} \left( \tilde{I}_1 - 3 \right) + \frac{c_2}{2} \left( \tilde{I}_2 - 3 \right)$,    
\item[] $\rho_0 = 1.1 \times 10^3$ kg/m$^3$,  
      \item[] $E = 1.7\times 10^7$ Pa, $c_1 = c_2 = E/6$,
      \item[] $\beta_1^{\infty} = 0.7$, $\mu^1 = c_1$, $\tau^1 = 0.1$ s,
      \item[] $\beta_2^{\infty} = 0.2$, $\mu^2 = c_1$, $\tau^2 = 0.5$ s,
      \item[] $\beta_3^{\infty} = 0.1$, $\mu^3 = c_1$, $\tau^3 = 1.0$ s,
      \item[] Reference scales: 
      \item[] $L_0 = 1$ m, $M_0 = 1$ kg, $T_0 = 1$ s. 
      \end{itemize}
	     
    \end{minipage}   
    \\ 
    \hline
  \end{tabular}
  \caption{Three-dimensional beam torsion: problem setting, boundary conditions, initial conditions, and material properties. Notice that the parameters $\beta_1^{\infty}$, $\beta_2^{\infty}$, and $\beta_3^{\infty}$ are only used in the IPC and MIPC models.} 
\label{table:3d_beam_torsion_benchmark_geometry}
\end{table}

\begin{figure}[!htbp]
\begin{center}
\begin{tabular}{c}
\includegraphics[angle=0, trim=85 85 120 80, clip=true, scale = 0.39]{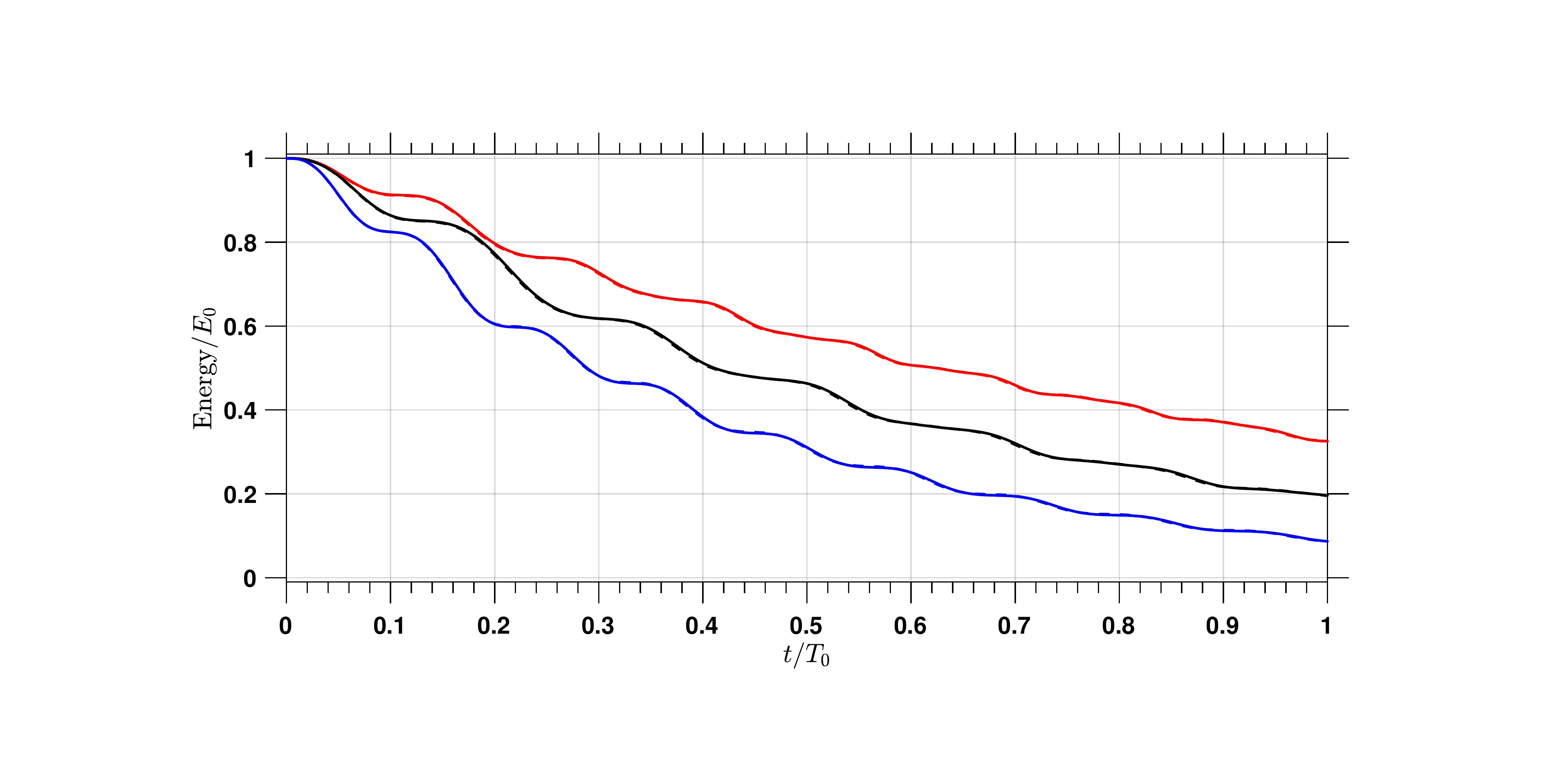}\\
(a) Total energies \\
\includegraphics[angle=0, trim=85 85 120 80, clip=true, scale = 0.39]{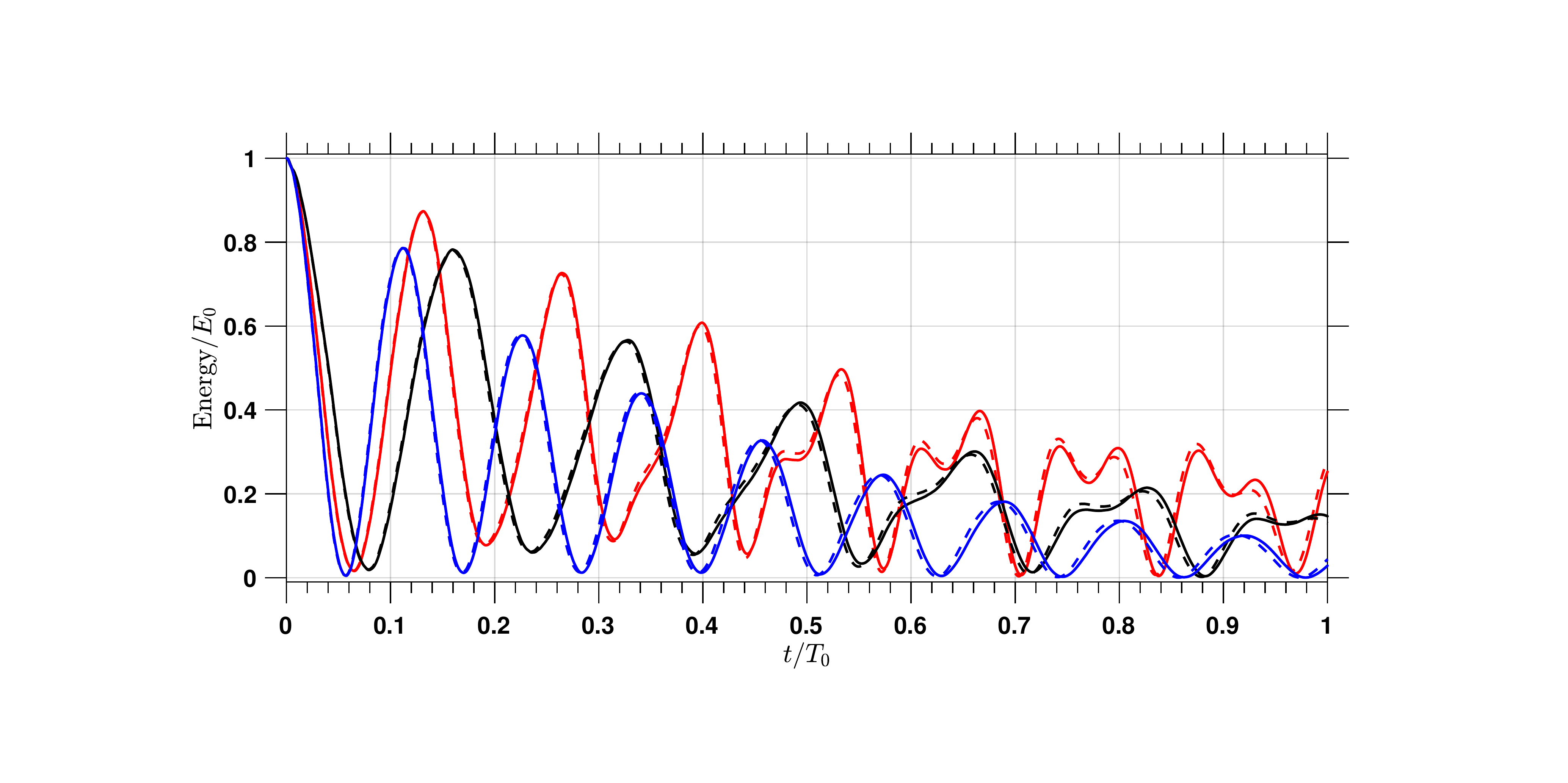} \\
(b) Kinetic energies \\
\includegraphics[angle=0, trim=85 85 120 80, clip=true, scale = 0.39]{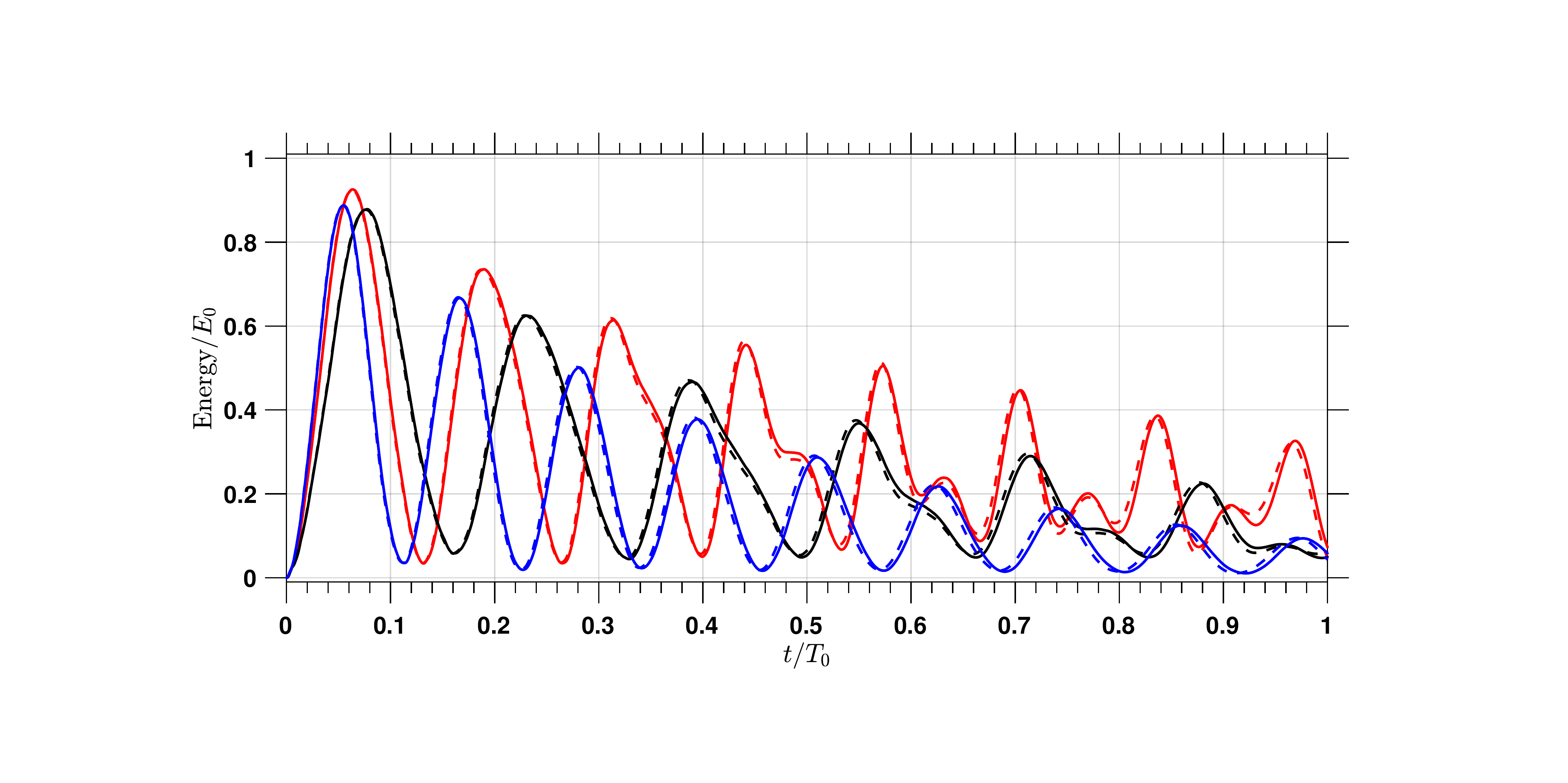}\\
(c) Potential energies
\end{tabular}
\caption{The total, kinetic, and potential energies (i.e., $G^{\infty}_{\mathrm{iso}}(\tilde{\bm C}_h) + \Upsilon^1(\tilde{\bm C}_h, \bm \Gamma^1_h)$) of the IPC (red), HS (blue), and MIPC (black) models over time for the beam torsion problem. The solid and dashed lines illustrate results obtained from the fine and coarse spatiotemporal discretizations. The reference value of the total energy $E_0$ is chosen to be the total energy at time $t=0$, which is $2.75\times 10^6$ kg m$^2$/s$^2$.} 
\label{fig:beam_torsion_energy}
\end{center}
\end{figure}

\begin{figure}[!htbp]
\begin{tabular}{ c c c c c c }
\multicolumn{6}{c}{
\includegraphics[angle=0, trim=0 0 280 750, clip=true, scale = 0.26]{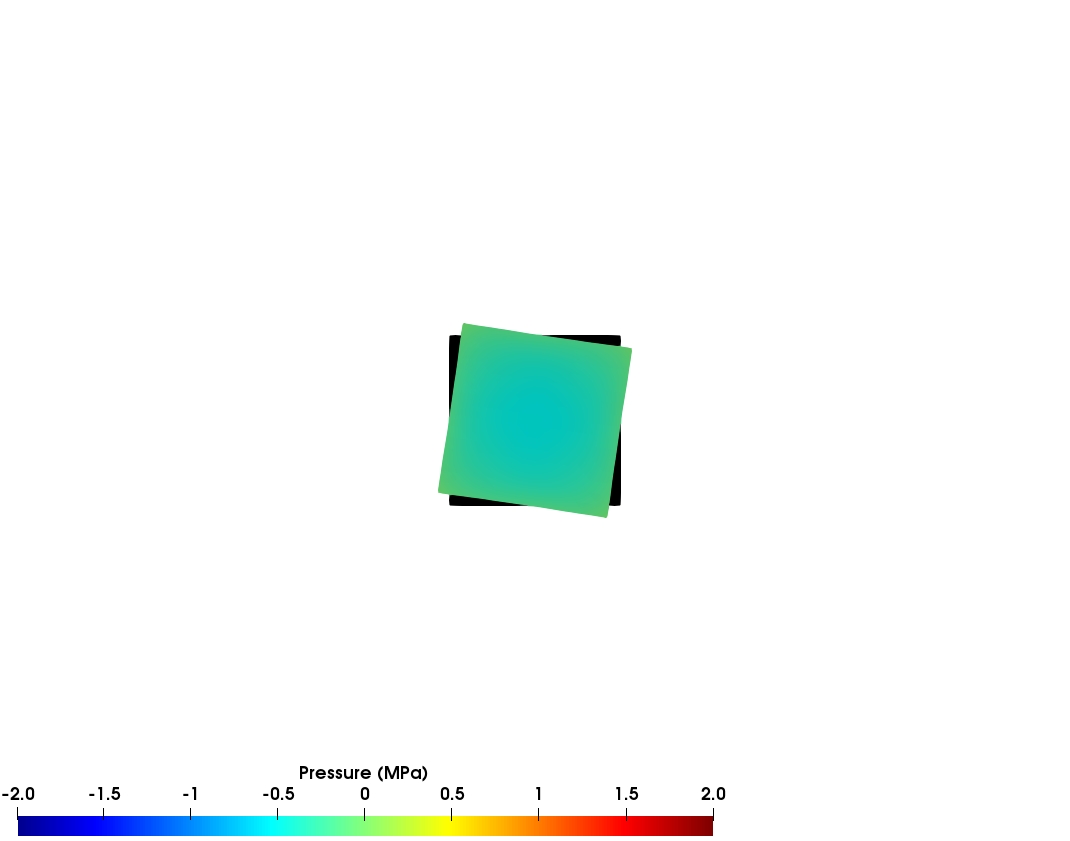}
} \\
\includegraphics[angle=0, trim=380 0 380 0, clip=true, scale = 0.21]{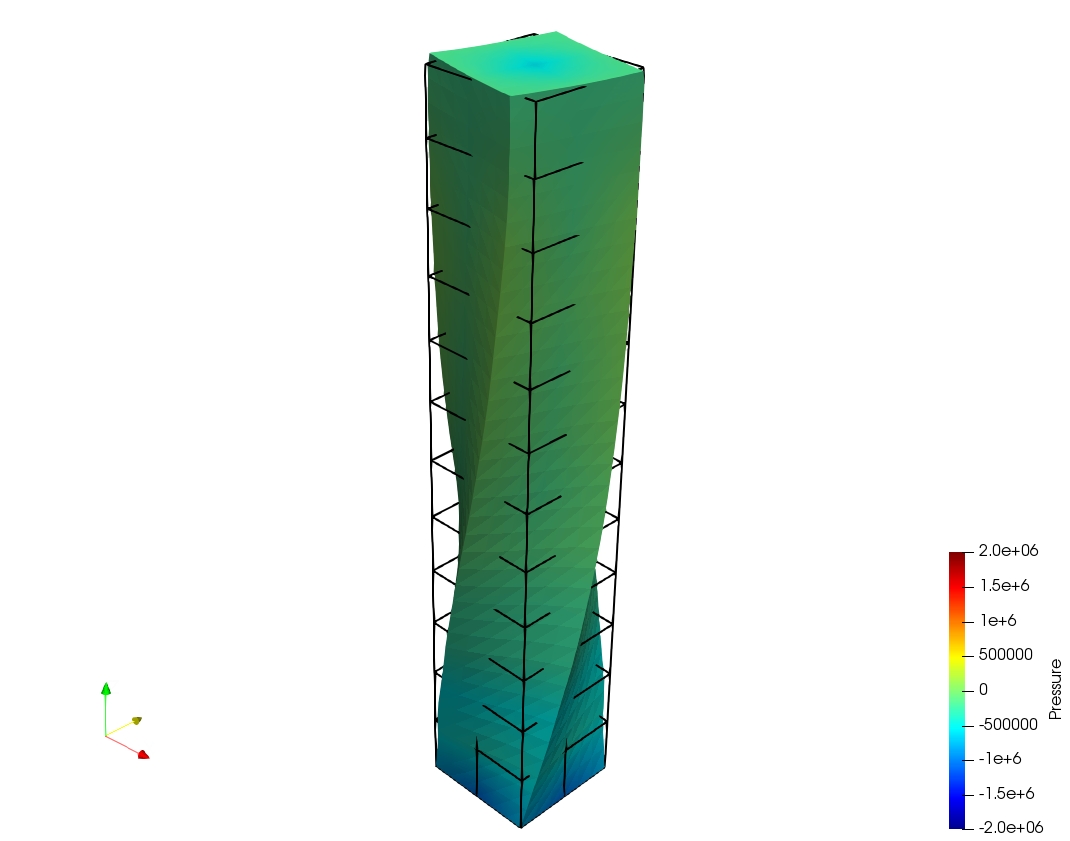} &
\includegraphics[angle=0, trim=380 0 380 0, clip=true, scale = 0.21]{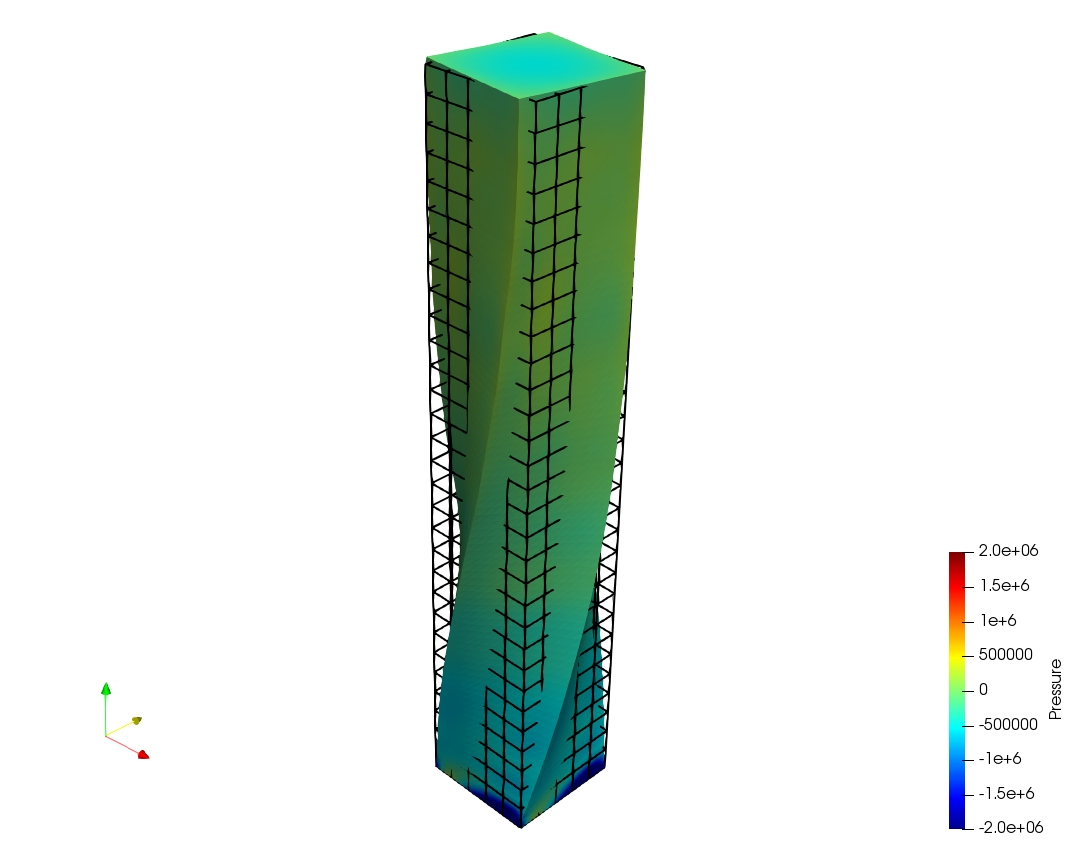} &
\includegraphics[angle=0, trim=380 0 380 0, clip=true, scale = 0.21]{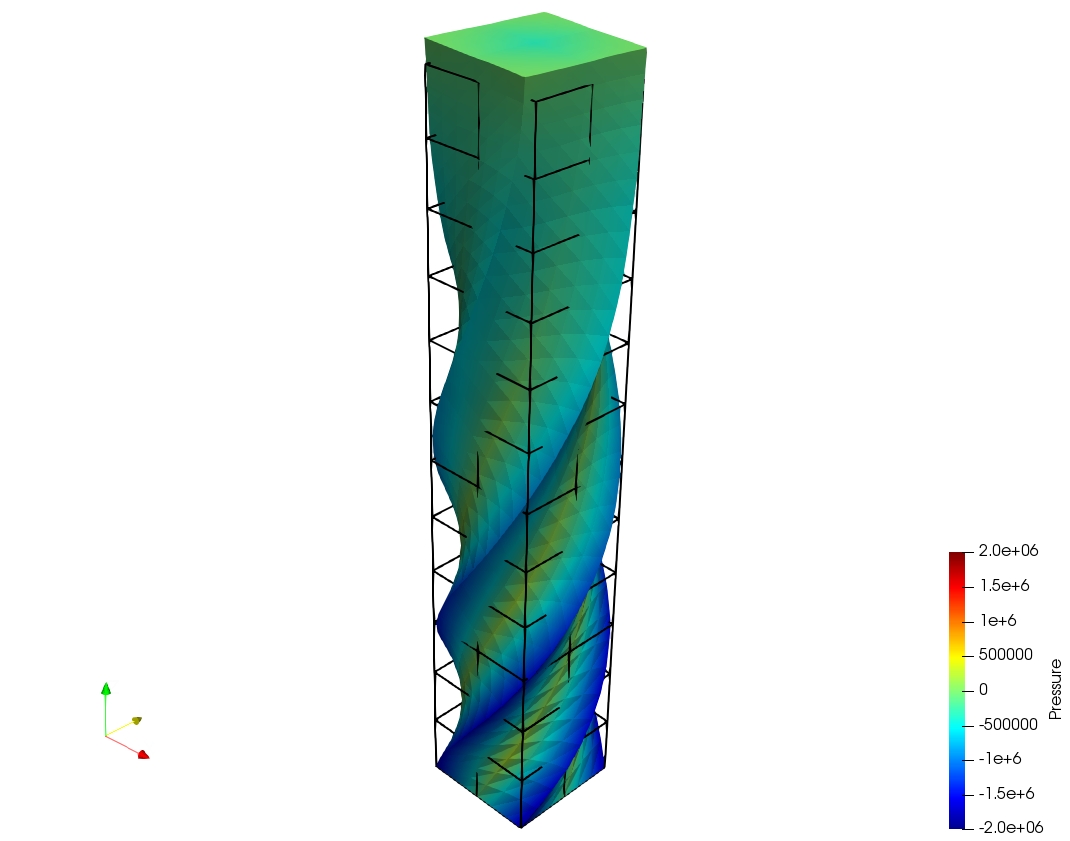} &
\includegraphics[angle=0, trim=380 0 380 0, clip=true, scale = 0.21]{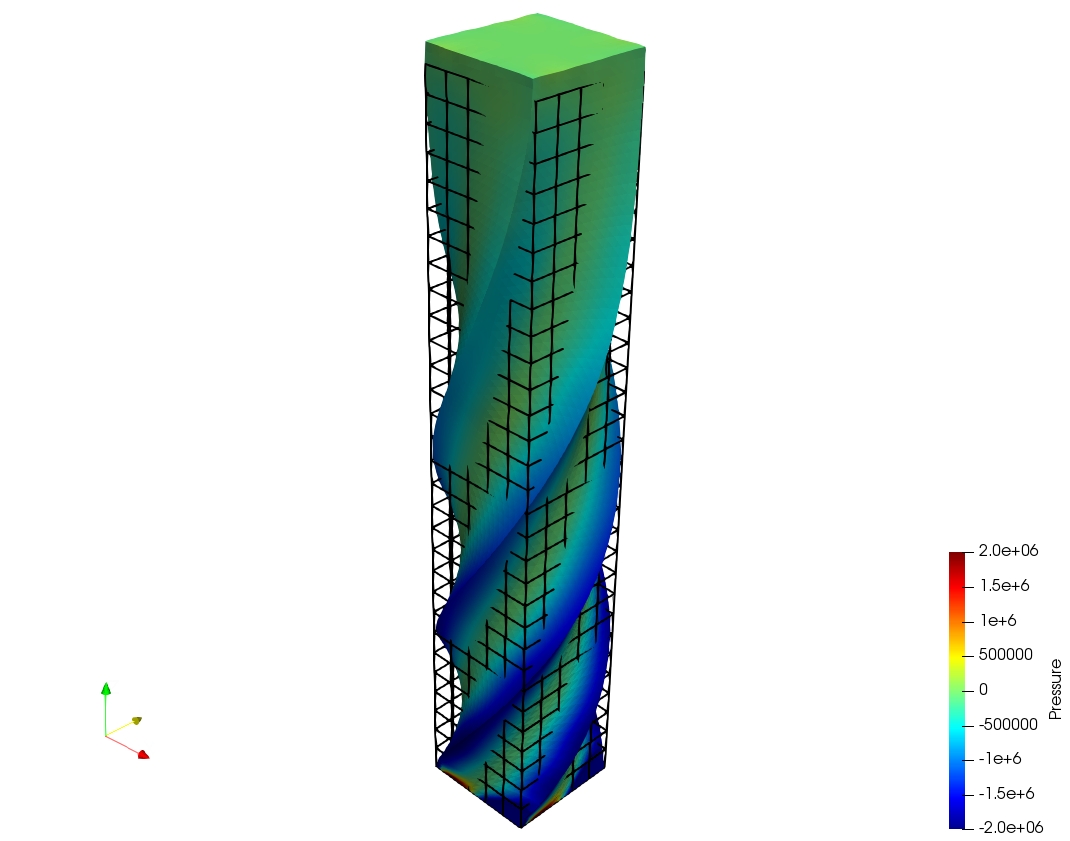} &
\includegraphics[angle=0, trim=380 0 380 0, clip=true, scale = 0.21]{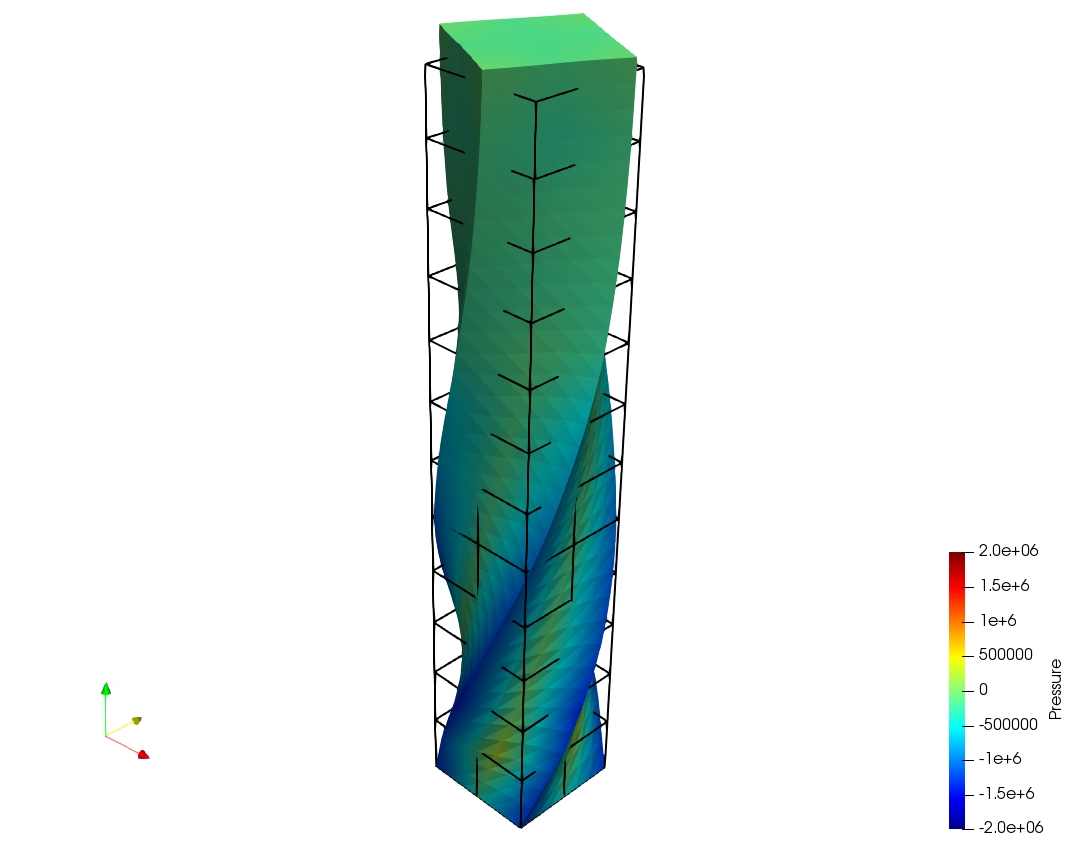} &
\includegraphics[angle=0, trim=380 0 380 0, clip=true, scale = 0.21]{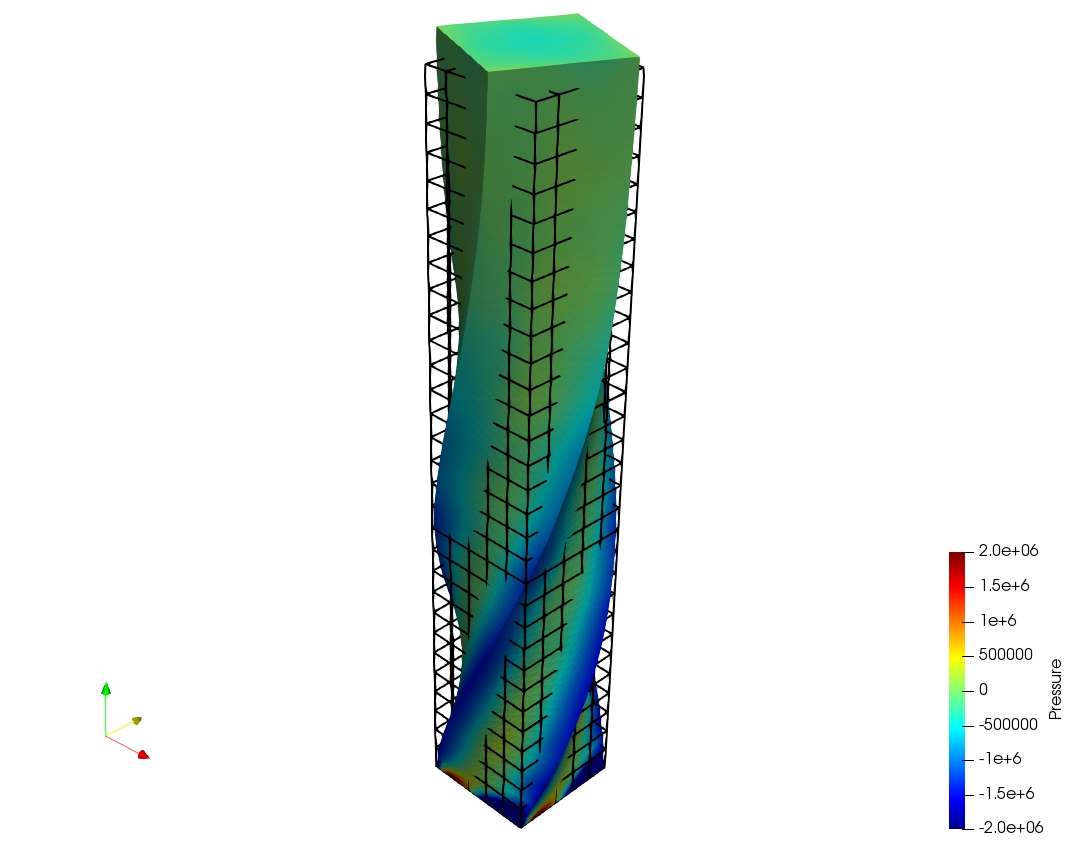} \\
\multicolumn{2}{c}{HS} & \multicolumn{2}{c}{MIPC} & \multicolumn{2}{c}{IPC}
\end{tabular}
\caption{The snapshots of the pressure fields plotted on the deformed configuration at time $t/T_0 = 0.1$ of the three models for the beam torsion problem. The results from the coarse and fine discretizations are shown, with the meshes at the initial time plotted as the black grid.} 
\label{fig:torsion_pressure_snapshot}
\end{figure}


\begin{figure}[!htbp]
\begin{tabular}{ c c c c c c }
\multicolumn{6}{c}{
\includegraphics[angle=0, trim=0 0 280 730, clip=true, scale = 0.26]{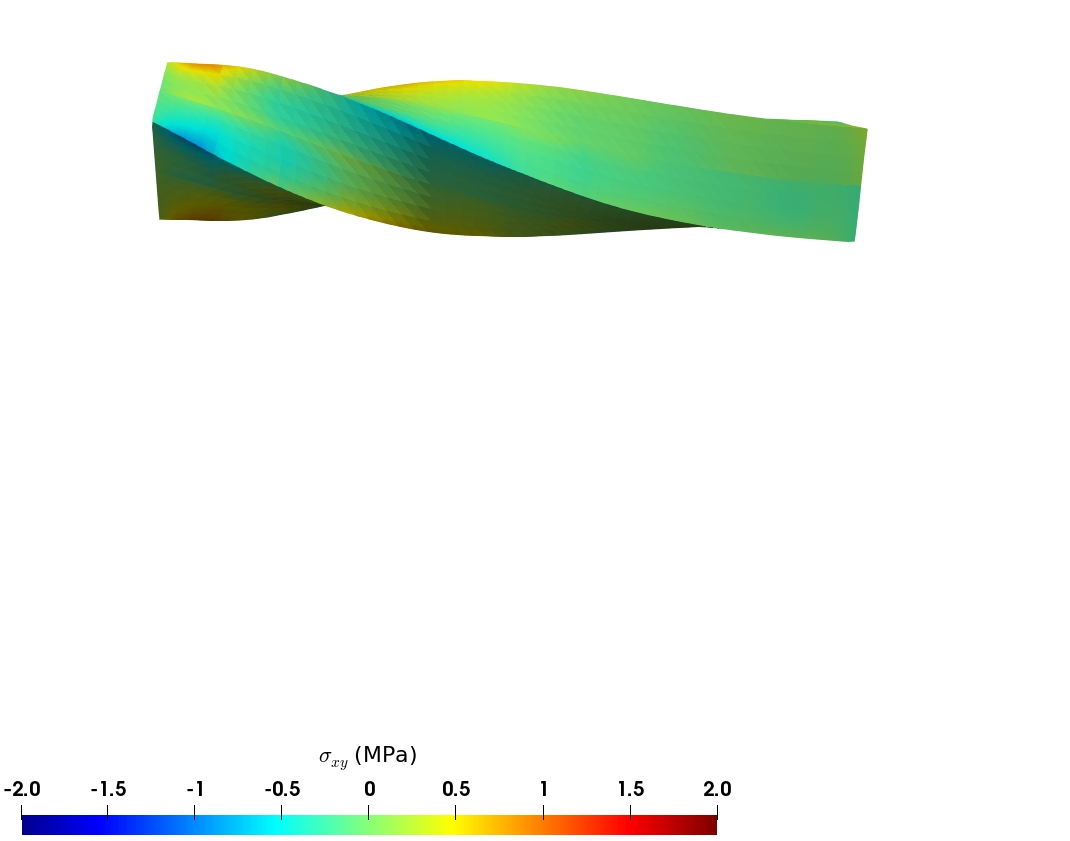}
} \\
\includegraphics[angle=0, trim=380 0 380 0, clip=true, scale = 0.21]{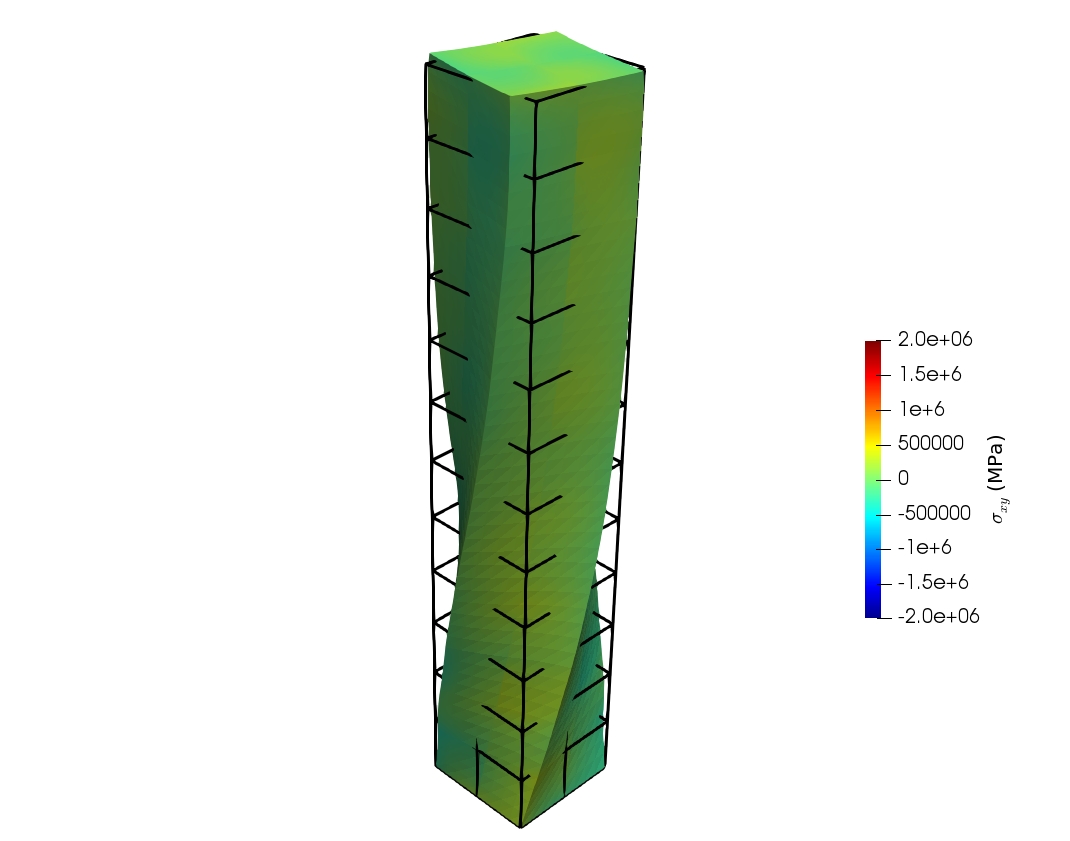} &
\includegraphics[angle=0, trim=380 0 380 0, clip=true, scale = 0.21]{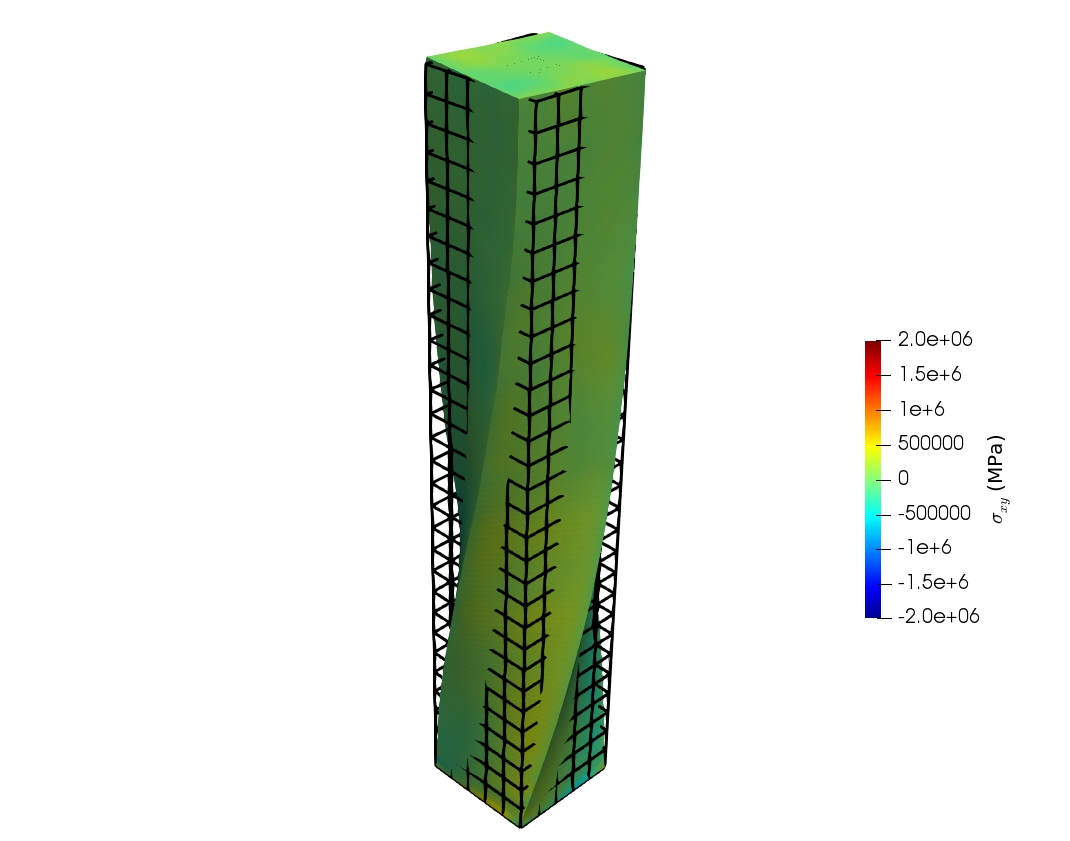} &
\includegraphics[angle=0, trim=380 0 380 0, clip=true, scale = 0.21]{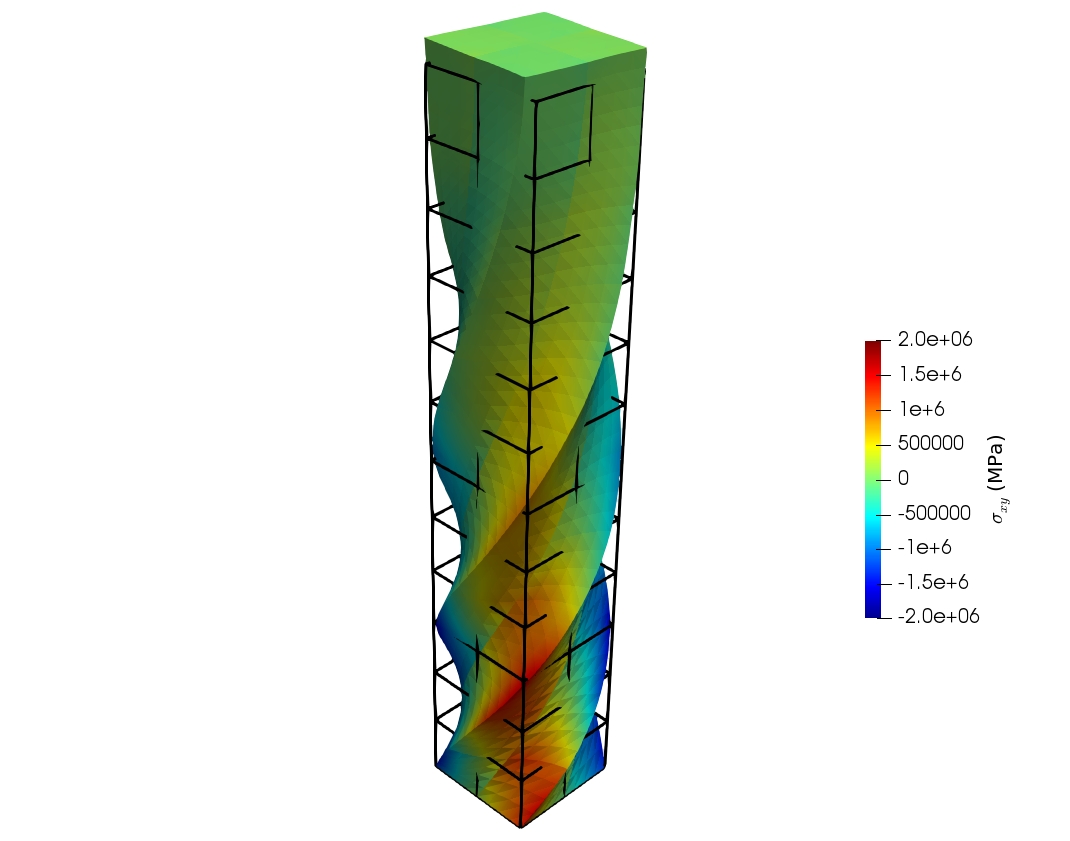} &
\includegraphics[angle=0, trim=380 0 380 0, clip=true, scale = 0.21]{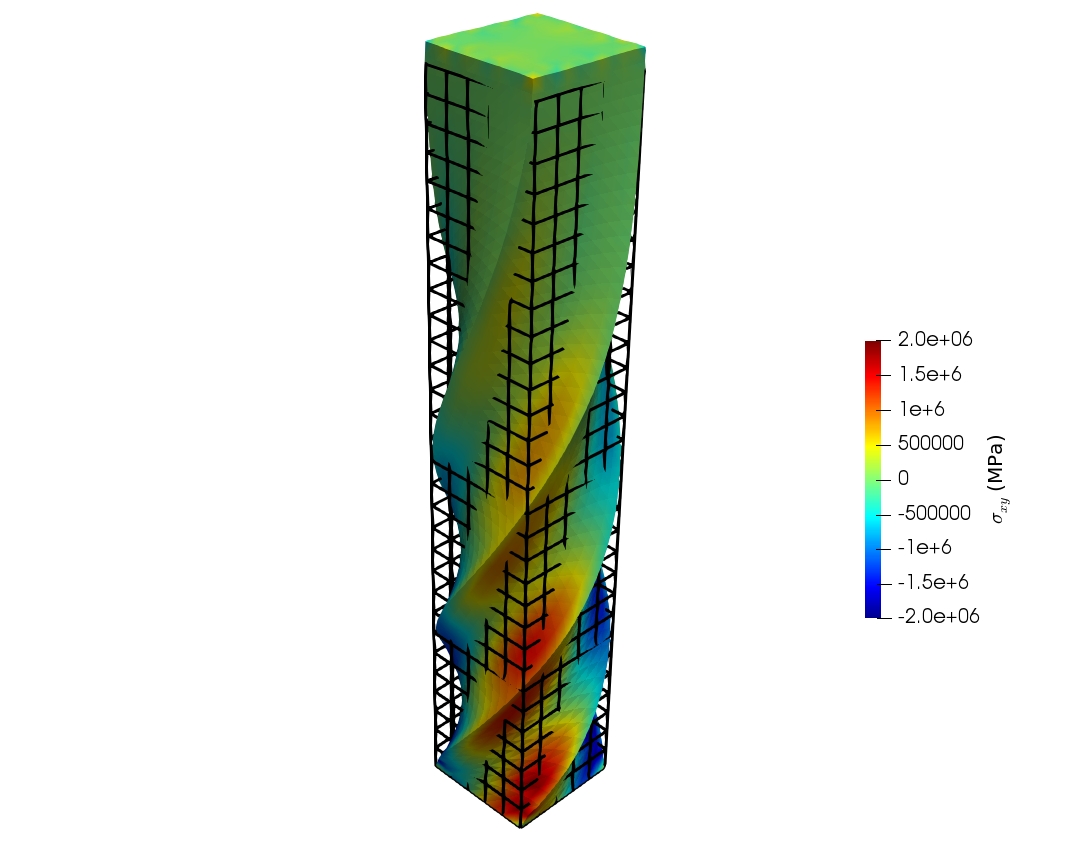} &
\includegraphics[angle=0, trim=380 0 380 0, clip=true, scale = 0.21]{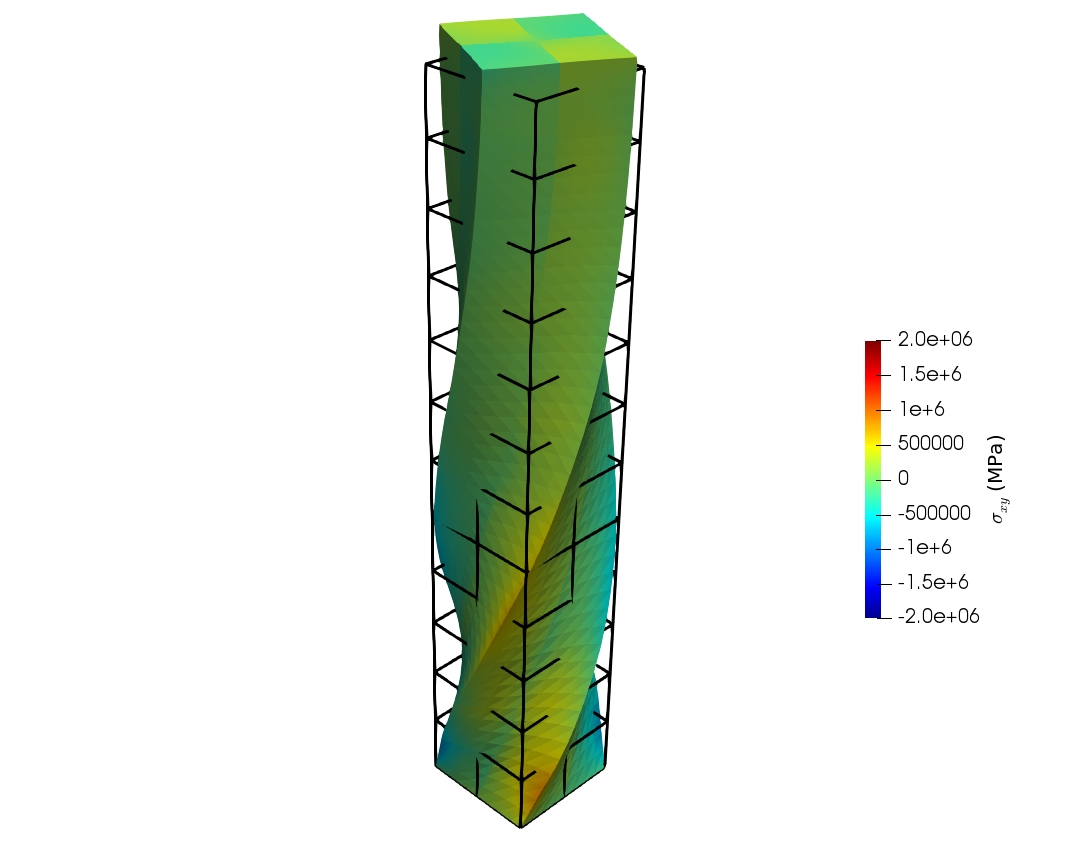} &
\includegraphics[angle=0, trim=380 0 380 0, clip=true, scale = 0.21]{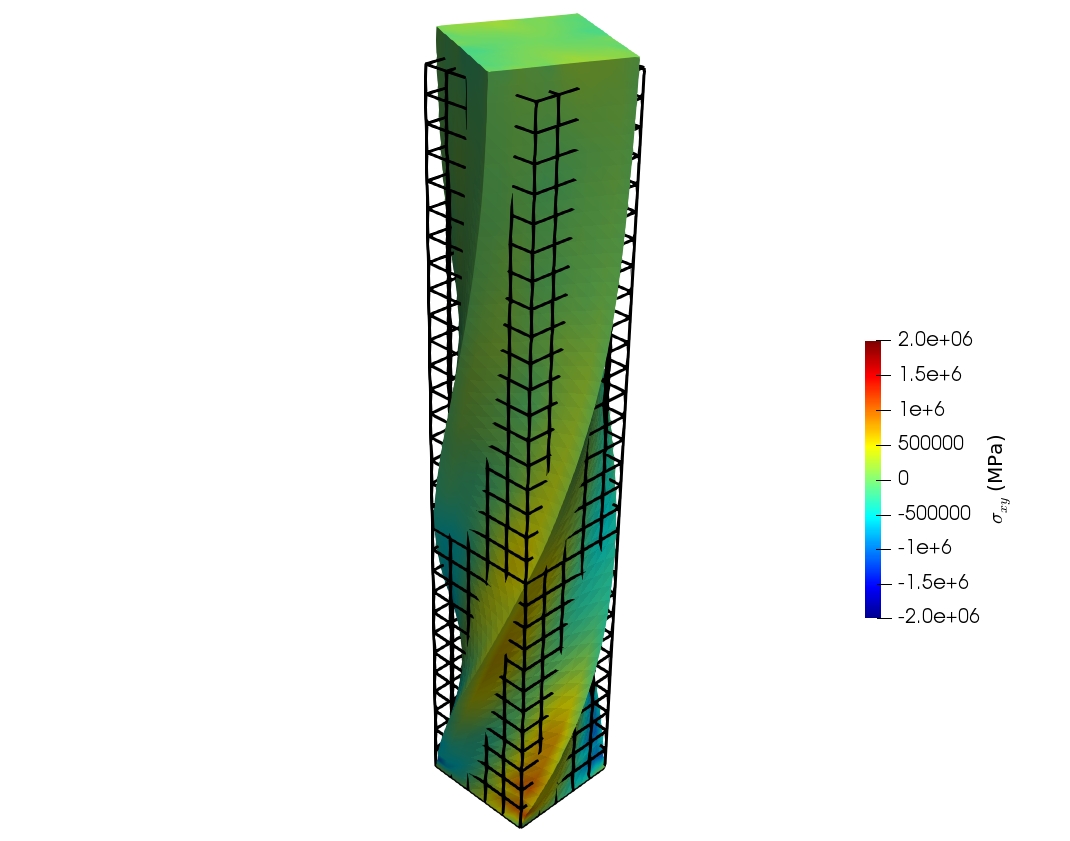} \\
\multicolumn{2}{c}{HS} & \multicolumn{2}{c}{MIPC} & \multicolumn{2}{c}{IPC}
\end{tabular}
\caption{The snapshots of the stress component $\sigma_{xy}$ plotted on the deformed configuration at time $t/T_0 = 0.1$ of the three models for the beam torsion problem. The results from the coarse and fine discretizations are shown, with the meshes at the initial time plotted as the black grid.} 
\label{fig:torsion_cauchy_xy_snapshot}
\end{figure}

\subsection{Beam torsion}
In this second example, we consider the torsion of a three-dimensional beam. The problem setting and material properties are illustrated in Table \ref{table:3d_beam_torsion_benchmark_geometry}. In particular, the boundary conditions are identical to those of the beam bending example, and the body is initially stress-free with zero displacements. The initial velocity is given by
\begin{align*}
\bm V(\bm X, 0) = V_0 \left( -\frac{Y}{L_0}, \frac{X}{L_0}, 0 \right)^T, \quad V_0 = 100 \sin\left( \frac{\pi Z}{12L_0} \right) \textup{m}/\textup{s}.
\end{align*}
Again, the problem constitutes an unforced mechanical system with the total energy monotonically decreasing with time. We numerically investigate this problem using two sets of spatiotemporal discretizations. In a coarse discretization, we use a fixed time step size $\Delta t/T_0 = 5.0 \times 10^{-4}$ and a spatial mesh with $\mathsf p = 1$ and $2 \times 2 \times 12$ elements; in a fine discretization, we use a fixed time step size $\Delta t/T_0 = 1.0 \times 10^{-4}$ and a spatial mesh with $\mathsf p = 2$ and $5 \times 5 \times 30$ elements. The time integration is performed with the mid-point rule (i.e., $\varrho_{\infty} = 1.0$). The evolutions of the total, kinetic, and potential energies are illustrated in Figure \ref{fig:beam_torsion_energy}, from which we may observe that the coarse and fine meshes produce indistinguishable results in terms of the energies. Again the total energy is justified to be monotonically decreasing with time from the numerical results, which corroborates the estimate made in Theorem \ref{the:energy-stability}. In Figures \ref{fig:torsion_pressure_snapshot} and \ref{fig:torsion_cauchy_xy_snapshot}, the snapshots of the pressure and the $xy$-component of the Cauchy stress are plotted on the current configuration at time $t/T_0 = 0.1$. For comparison purposes, the results of the three models calculated by the two discretizations are illustrated. We observe that both the pressure and the Cauchy stress are well resolved by the coarse mesh and are indistinguishable. This again demonstrates the superior capability of the chosen element technology in resolving the stresses.

\subsection{Clamped cylindrical support}
In the third example, we consider a clamped cylindrical support characterized by viscoelastic material behaviors with its inner surface connected to a vibrating device. This is a benchmark problem designed for examing the rate-dependent behavior with hysteresis loops \cite{Zeng2017,Fancello2006}. The geometrical setting of the cylindrical support together with its material properties are illustrated in Table \ref{table:3d_tube_support_benchmark_geometry}. The vibration of the device is described by
\begin{align*}
\bm U(\bm X, t) = \left( U_0 \sin\left( \omega t \right), 0, 0 \right)^T, \quad U_0 = 5.625 \times 10^{-3} \textup{m},
\end{align*}
which represents a translation along the $x$-direction, and, in this study, the strain rate $\omega$ takes the values of $5$ s$^{-1}$, $10$ s$^{-1}$, and $20$ s$^{-1}$, respectively. The outer surface of the support is fixed while traction-free boundary conditions are prescribed on the two end annular surfaces. The geometry of this cylindrical support can be exactly represented via NURBS. We use $24$ elements in the circumferential direction, $6$ elements in the radial direction, and $5$ elements in the axial direction. We choose $\mathsf p = 2$ and $3$ for the discrete pressure function space, respectively, to generate two sets of meshes. For the given strain rate, the period of one cycle is $2\pi / \omega$, and we simulate the problems for three cycles. For the mesh with $\mathsf p = 2$, the time step size is $\Delta t = 2.0 \times 10^{-4}$ s; for the mesh with $\mathsf p = 3$, the time step size is $\Delta t=1.0 \times 10^{-4}$. In this problem, we choose $\varrho_{\infty} = 0.0$ to achieve better algorithm robustness.

\begin{table}[!htbp]
  \centering
  \begin{tabular}{ m{.45\textwidth}   m{.45\textwidth} }
    \hline
    \begin{minipage}{.45\textwidth}
      \includegraphics[angle=0, trim=0 0 0 190, clip=true, scale = 0.3]{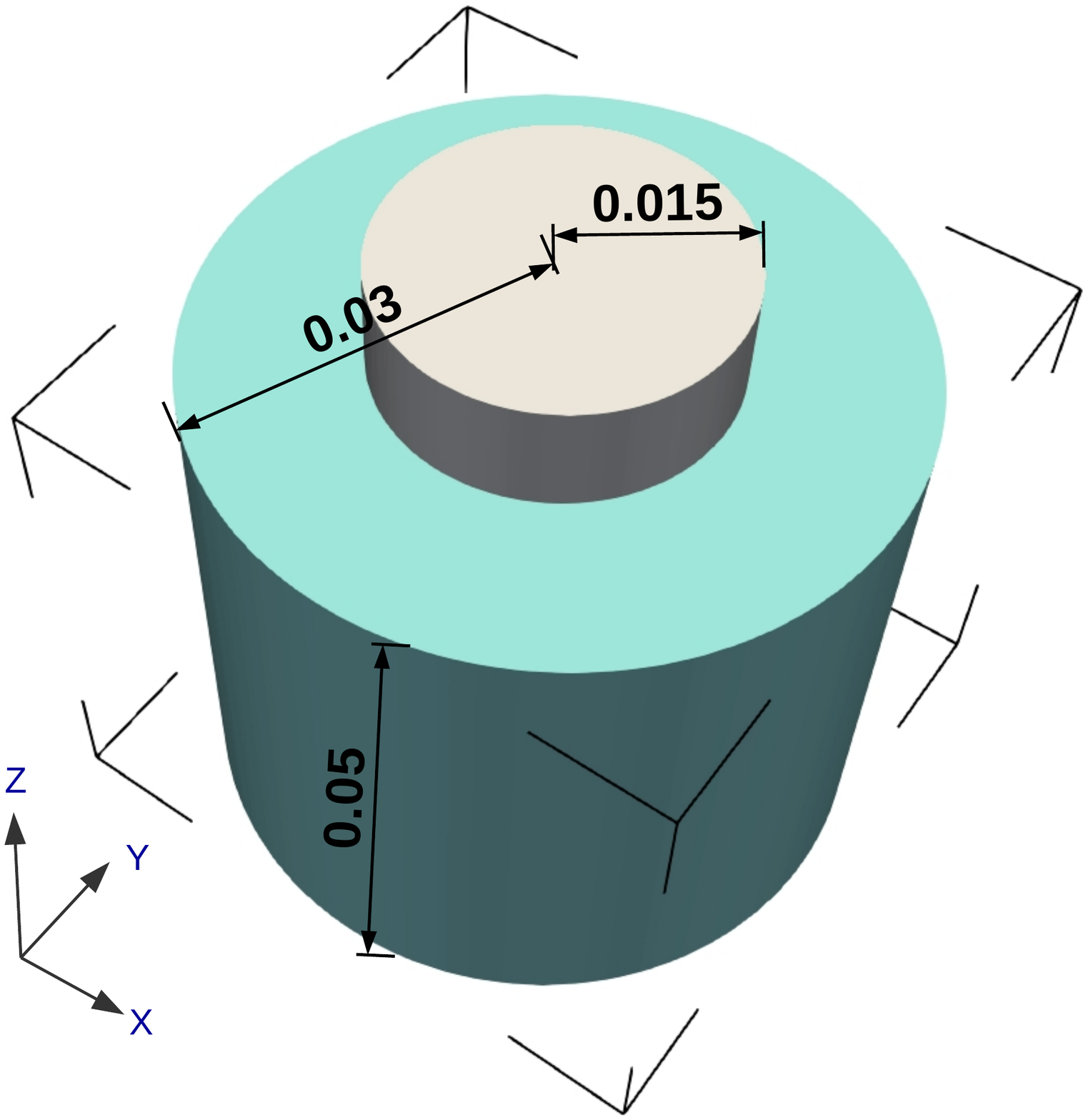}
    \end{minipage}
    &
    \begin{minipage}{.45\textwidth}
      \begin{itemize}
      \item[] Material properties:
      \item[] $G^{\infty}_{\mathrm{iso}} = \frac{c_1}{2} \left( \tilde{I}_1 - 3 \right) + \frac{c_2}{2} \left( \tilde{I}_2 - 3 \right)$,    
\item[] $\rho_0 = 1.1 \times 10^3$ kg/m$^3$,  
      \item[] $E = 1.2\times 10^5$ Pa, $c_1 = c_2 = E/6$,
      \item[] $\beta_1^{\infty} = 1.0$, $\mu^1 = 4c_1$, $\tau^1 = 0.4$ s,
      \item[] Reference scales: 
      \item[] $L_0 = 1$ m, $M_0 = 1$ kg, $T_0 = 1$ s. 
      \end{itemize}
    \end{minipage}   
    \\ 
    \hline
  \end{tabular}
  \caption{Three-dimensional clamped cylindrical support: problem setting, boundary conditions, initial conditions, and material properties. Notice that the parameter $\beta_1^{\infty}$ is only used in the MIPC models.} 
\label{table:3d_tube_support_benchmark_geometry}
\end{table}

\begin{figure}[!htbp]
\begin{center}
\begin{tabular}{ c c }
\includegraphics[angle=0, trim=120 100 160 100, clip=true, scale = 0.23]{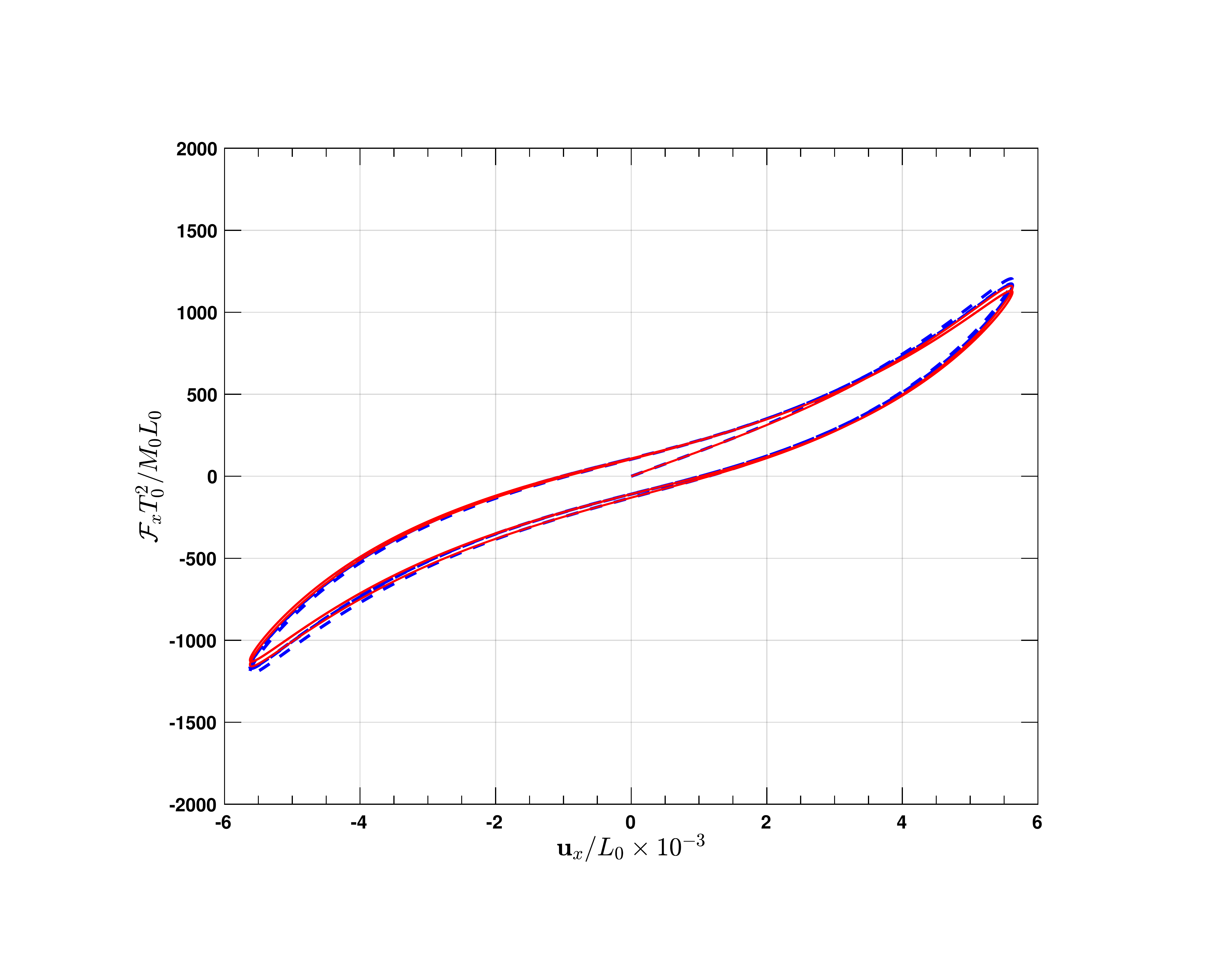} &
\includegraphics[angle=0, trim=120 100 160 100, clip=true, scale = 0.23]{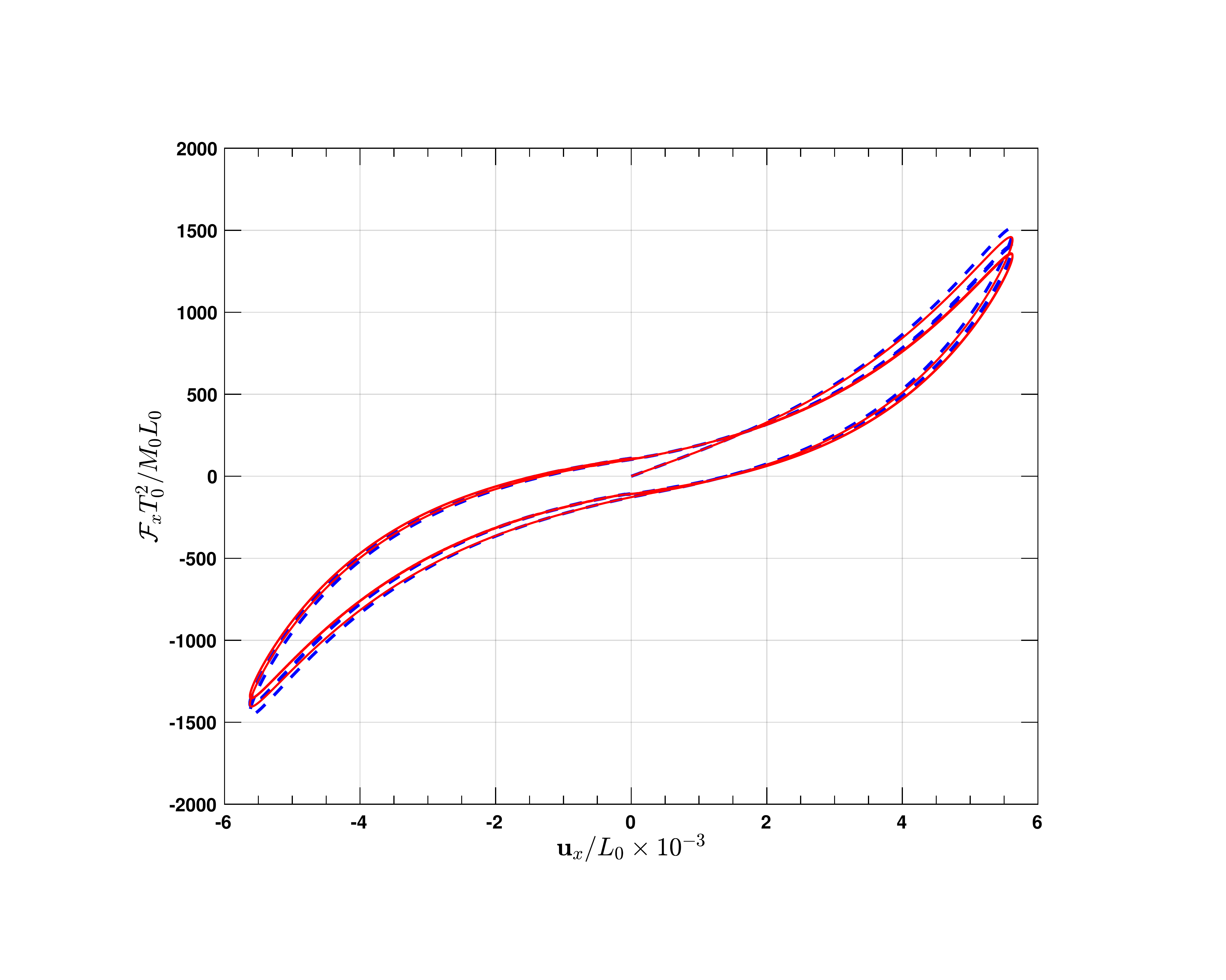} \\
\includegraphics[angle=0, trim=120 100 160 100, clip=true, scale = 0.23]{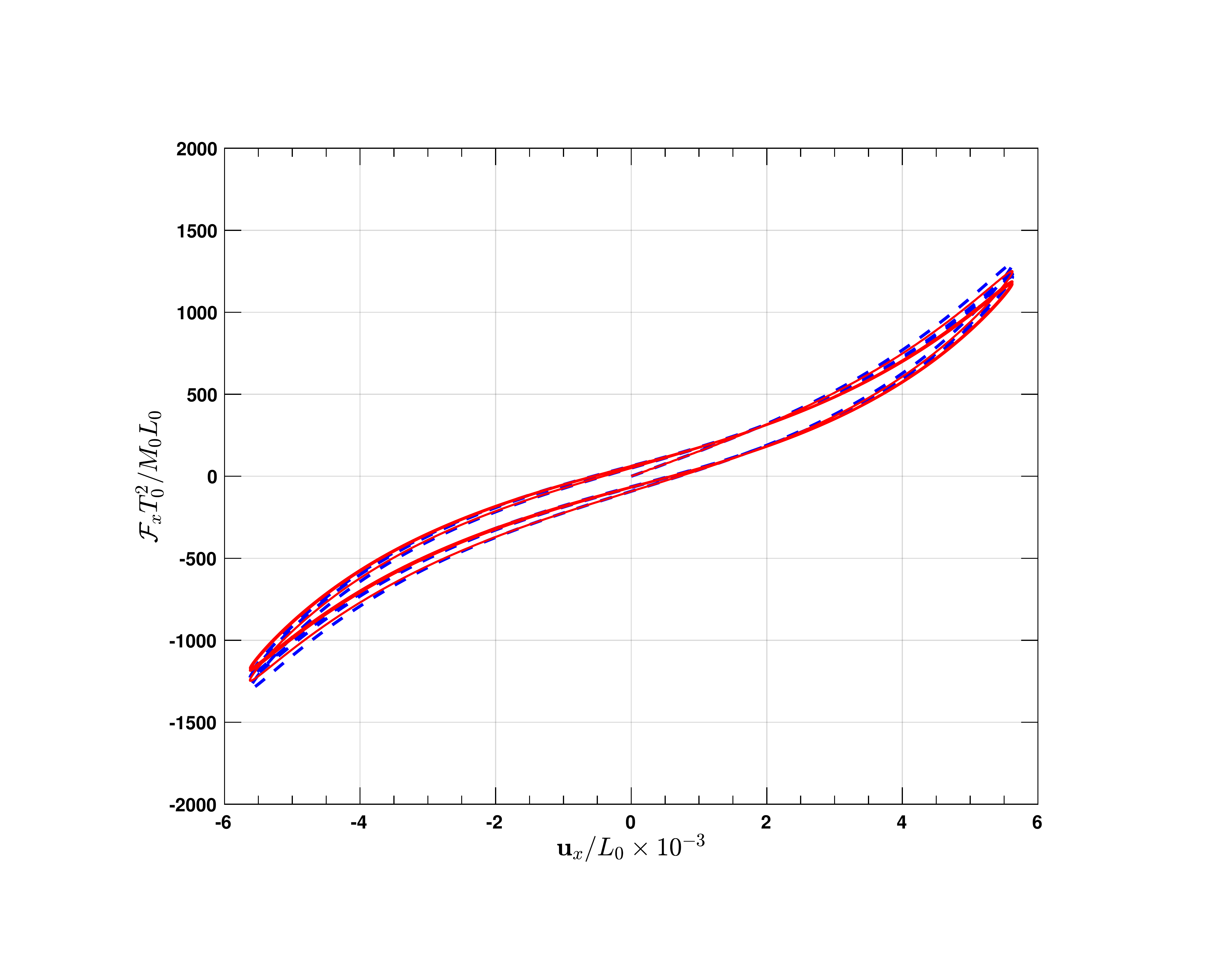} &
\includegraphics[angle=0, trim=120 100 160 100, clip=true, scale = 0.23]{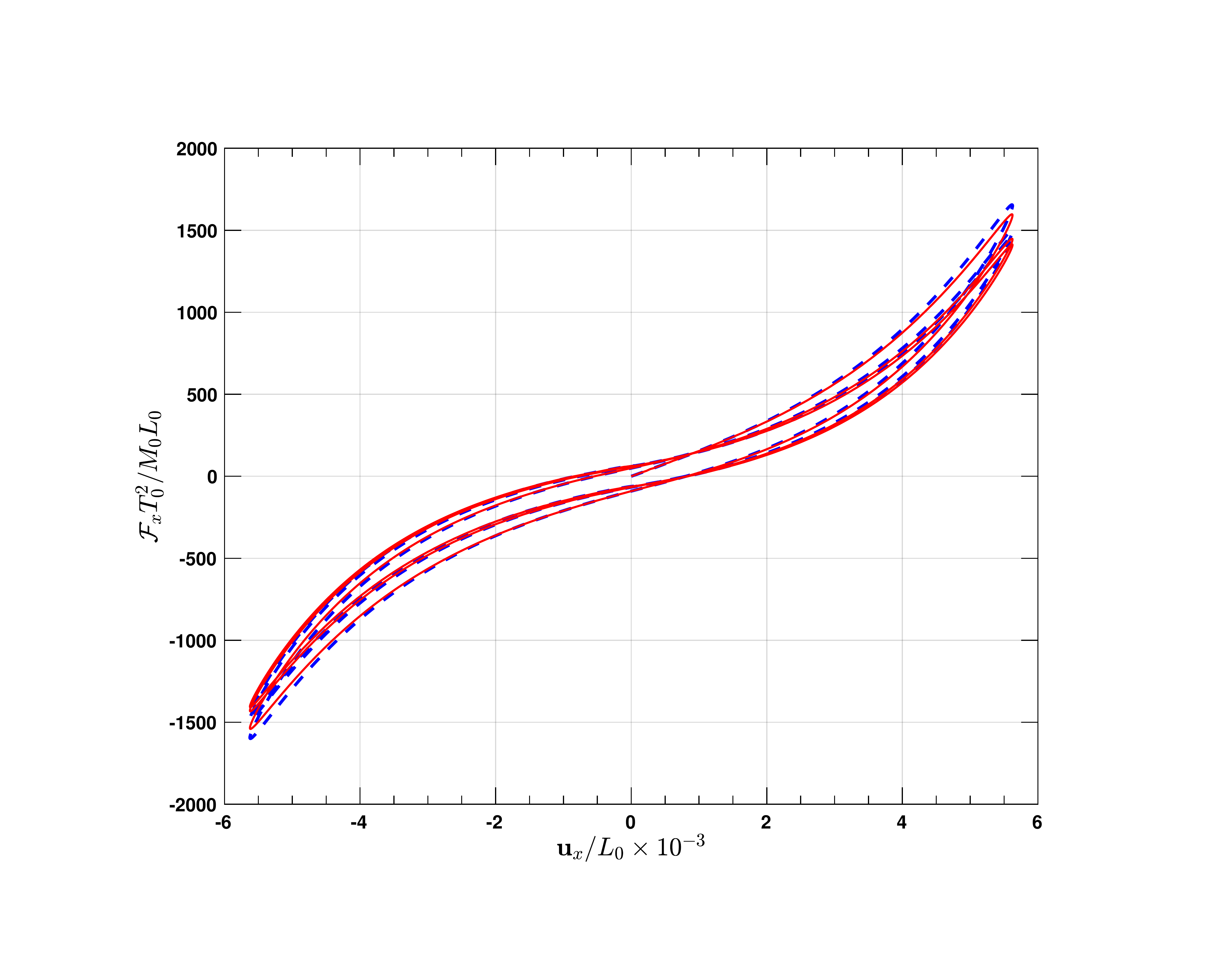} \\
\includegraphics[angle=0, trim=120 100 160 100, clip=true, scale = 0.23]{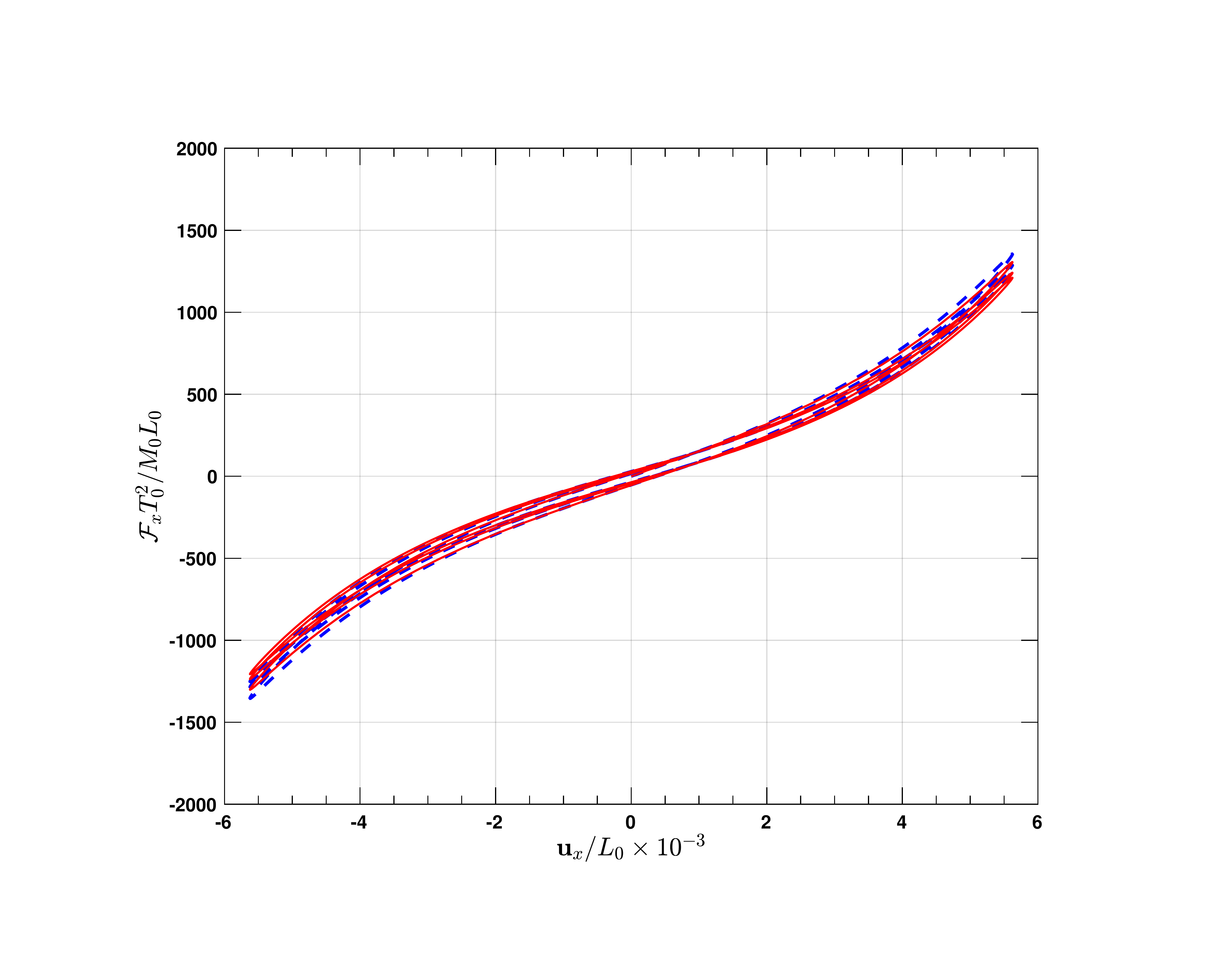} &
\includegraphics[angle=0, trim=120 100 160 100, clip=true, scale = 0.23]{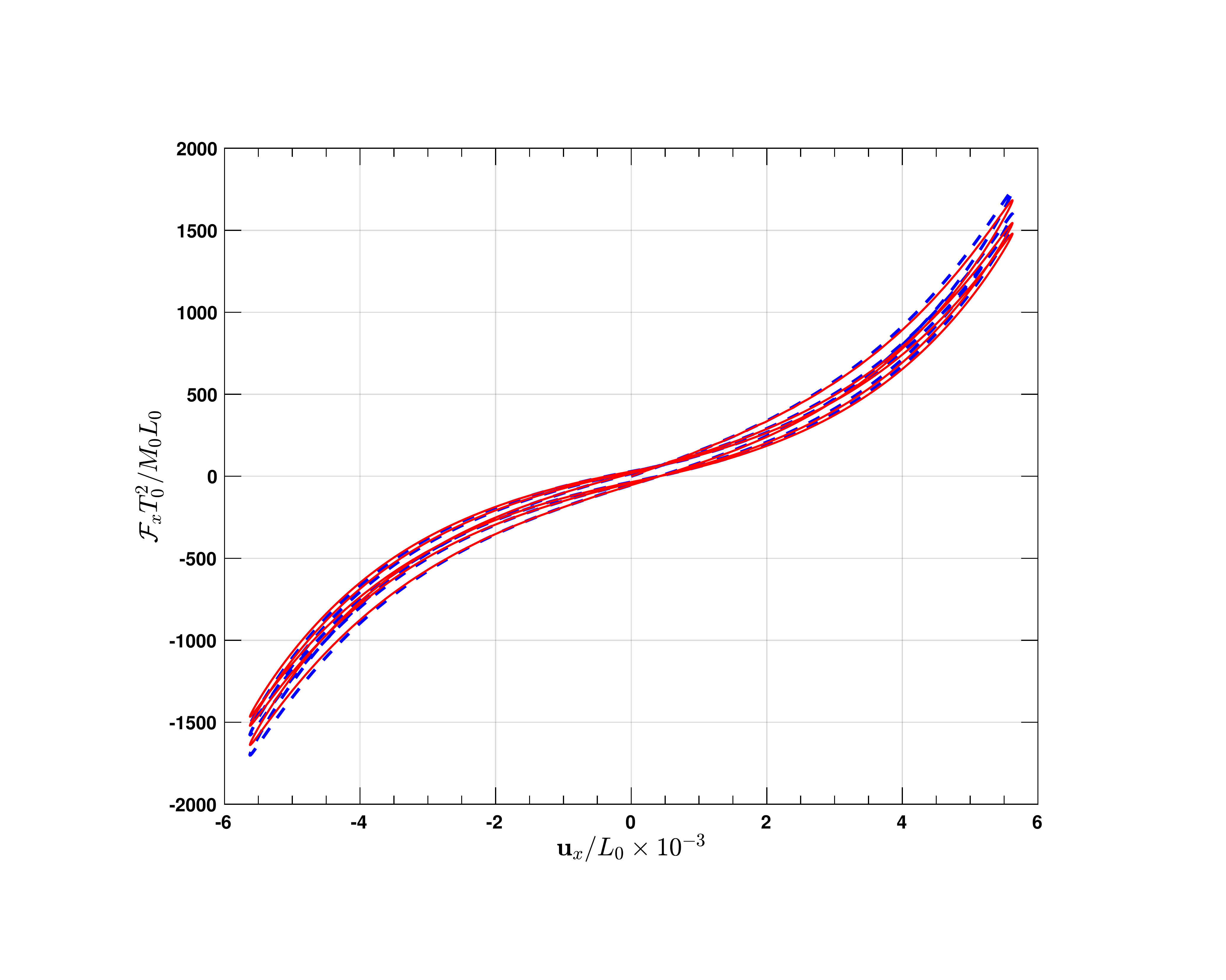} \\
HS & MIPC \\
\multicolumn{2}{c}{
\includegraphics[angle=0, trim=450 150 450 400, clip=true, scale = 0.36]{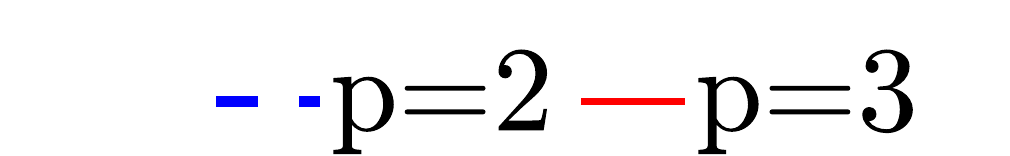}
}
\end{tabular}
\caption{The lateral force $\bm{\mathcal F}_x$ on $\Gamma_{\mathrm{inner}}$ plotted against the lateral displacement $\bm U_{x}$ for $\omega = 5$s$^{-1}$ (top row), $\omega = 10$s$^{-1}$ (middle row), $\omega = 20$s$^{-1}$ (bottom row). In the left column, the configurational free energy is given by the HS model while on the right column it is given by the MIPC model.} 
\label{fig:hysteresis}
\end{center}
\end{figure}

\begin{figure}[!htbp]
\begin{tabular}{ c c c c }
\multicolumn{2}{c}{
HS
} & 
\multicolumn{2}{c}{
MIPC
} \\
\includegraphics[angle=0, trim=130 30 160 30, clip=true, scale = 0.09]{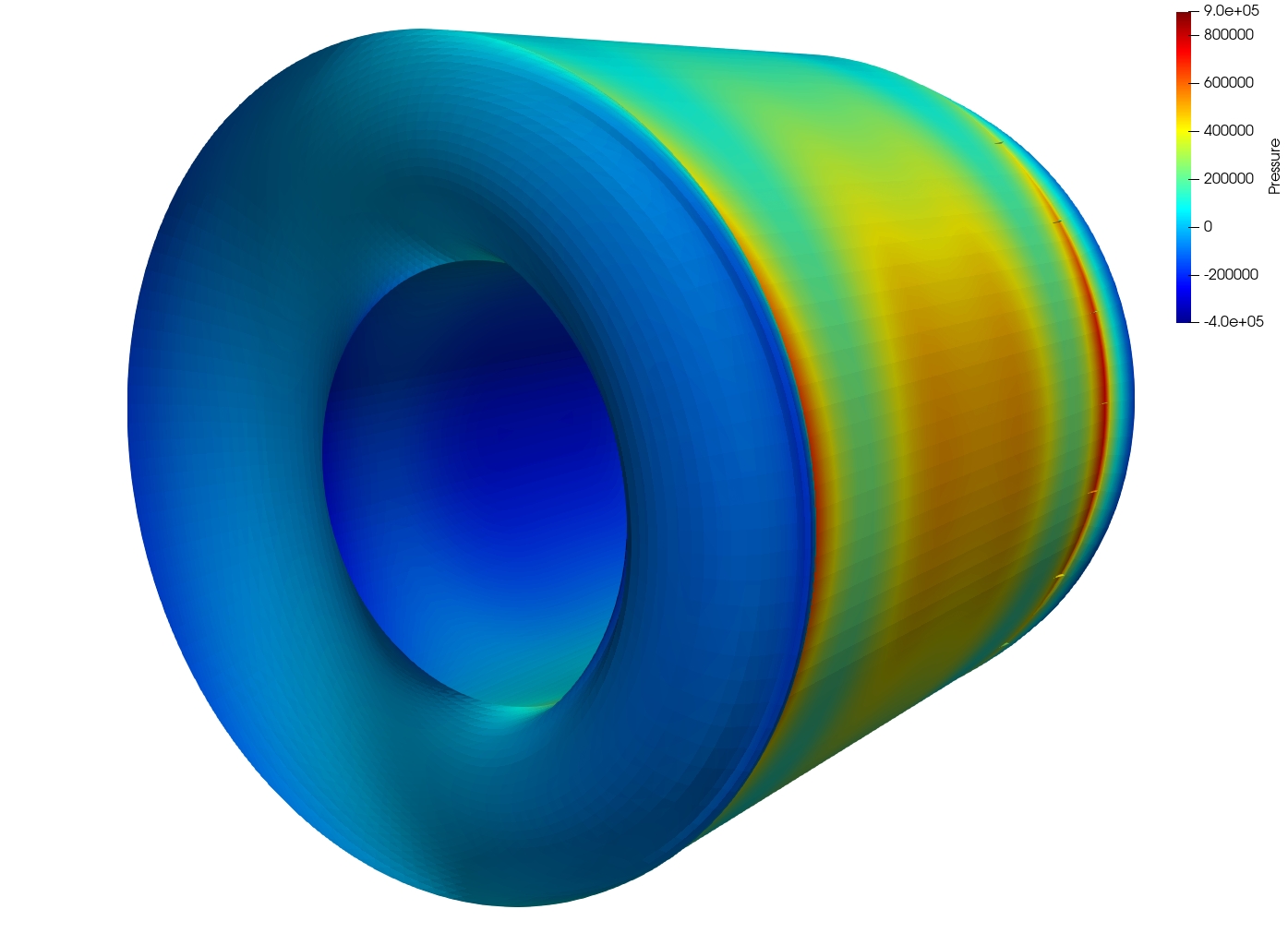} &
\includegraphics[angle=0, trim=130 30 160 30, clip=true, scale = 0.09]{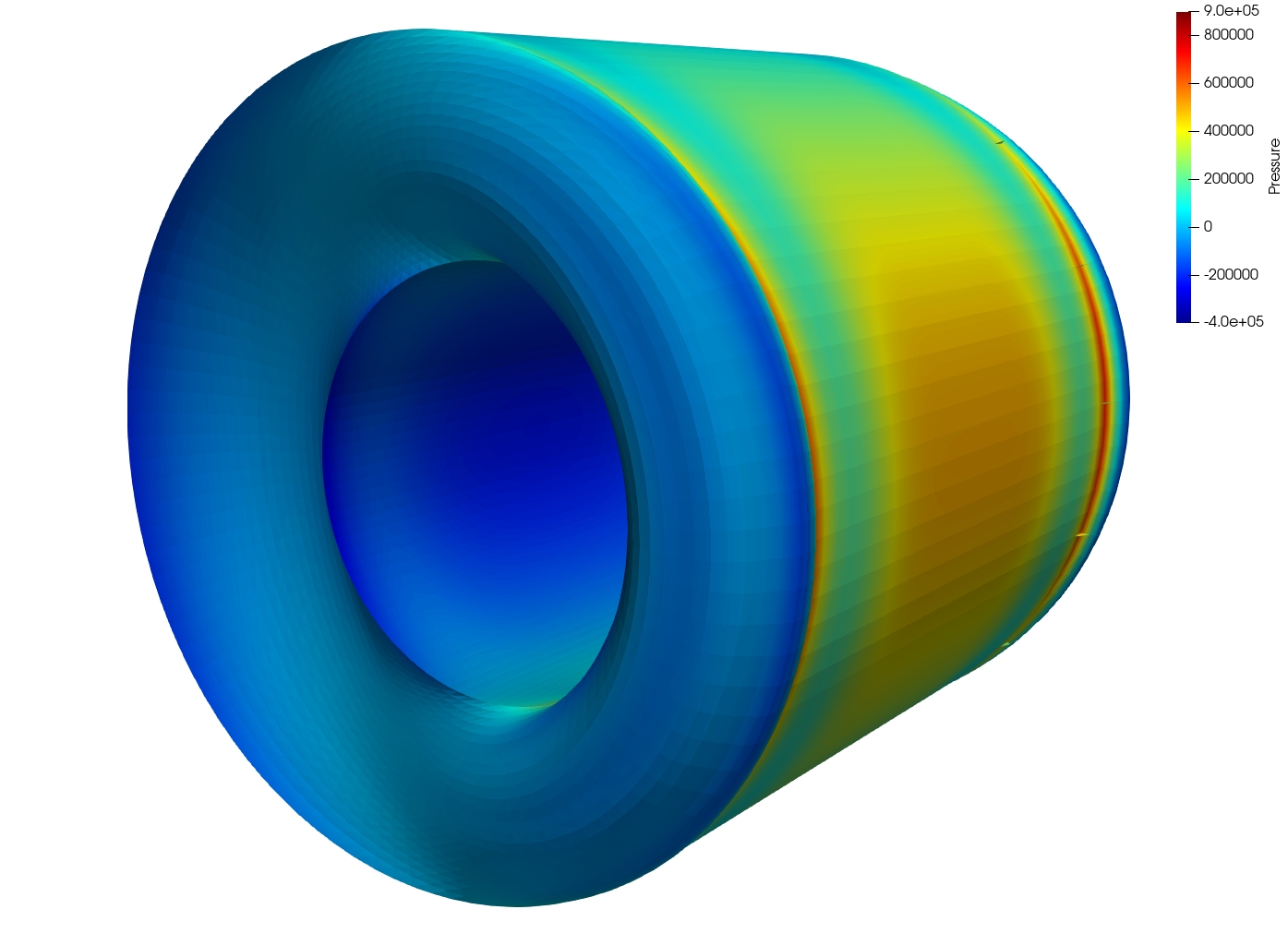} &
\includegraphics[angle=0, trim=130 30 160 30, clip=true, scale = 0.09]{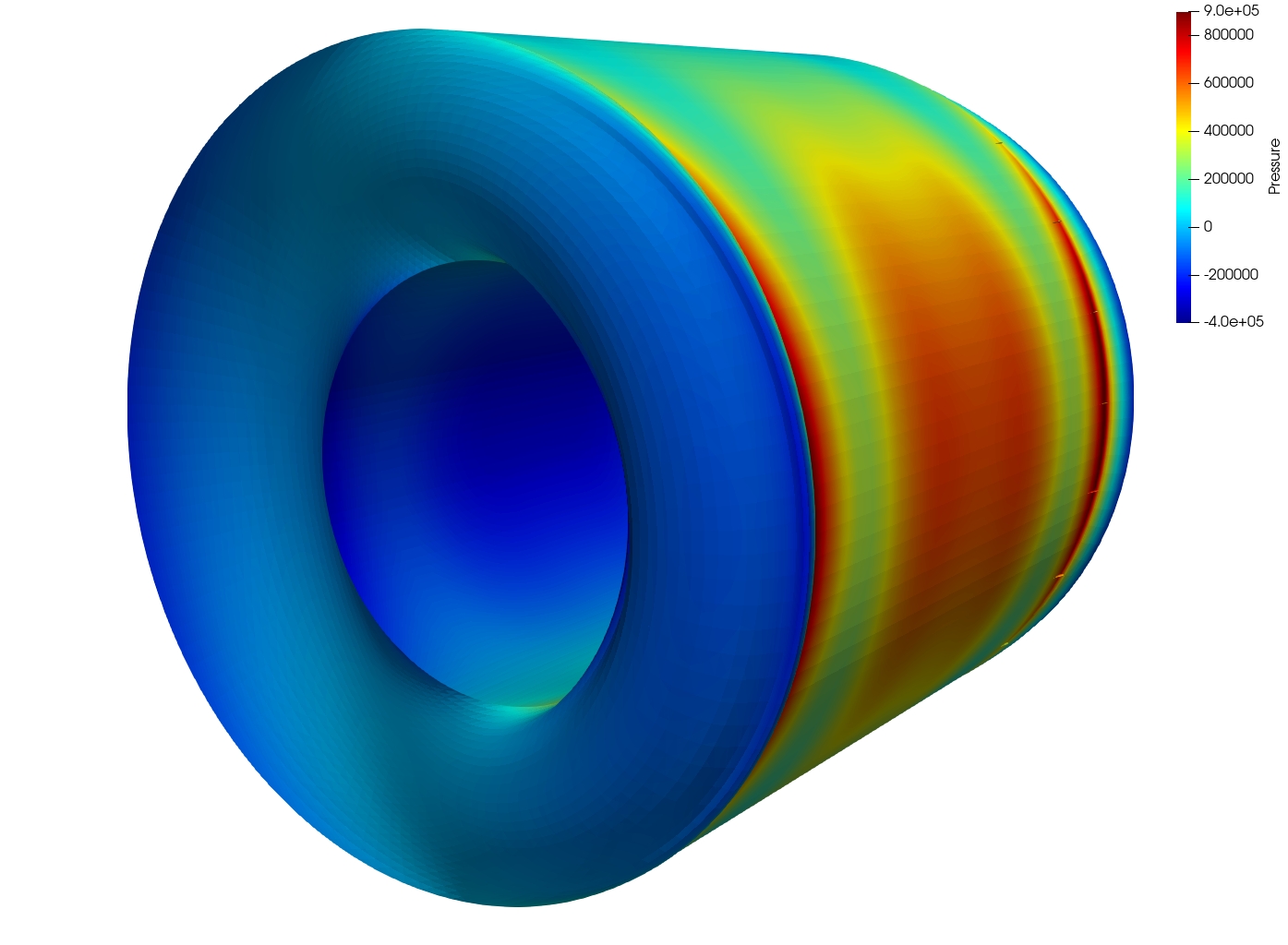} &\includegraphics[angle=0, trim=130 30 160 30, clip=true, scale = 0.09]{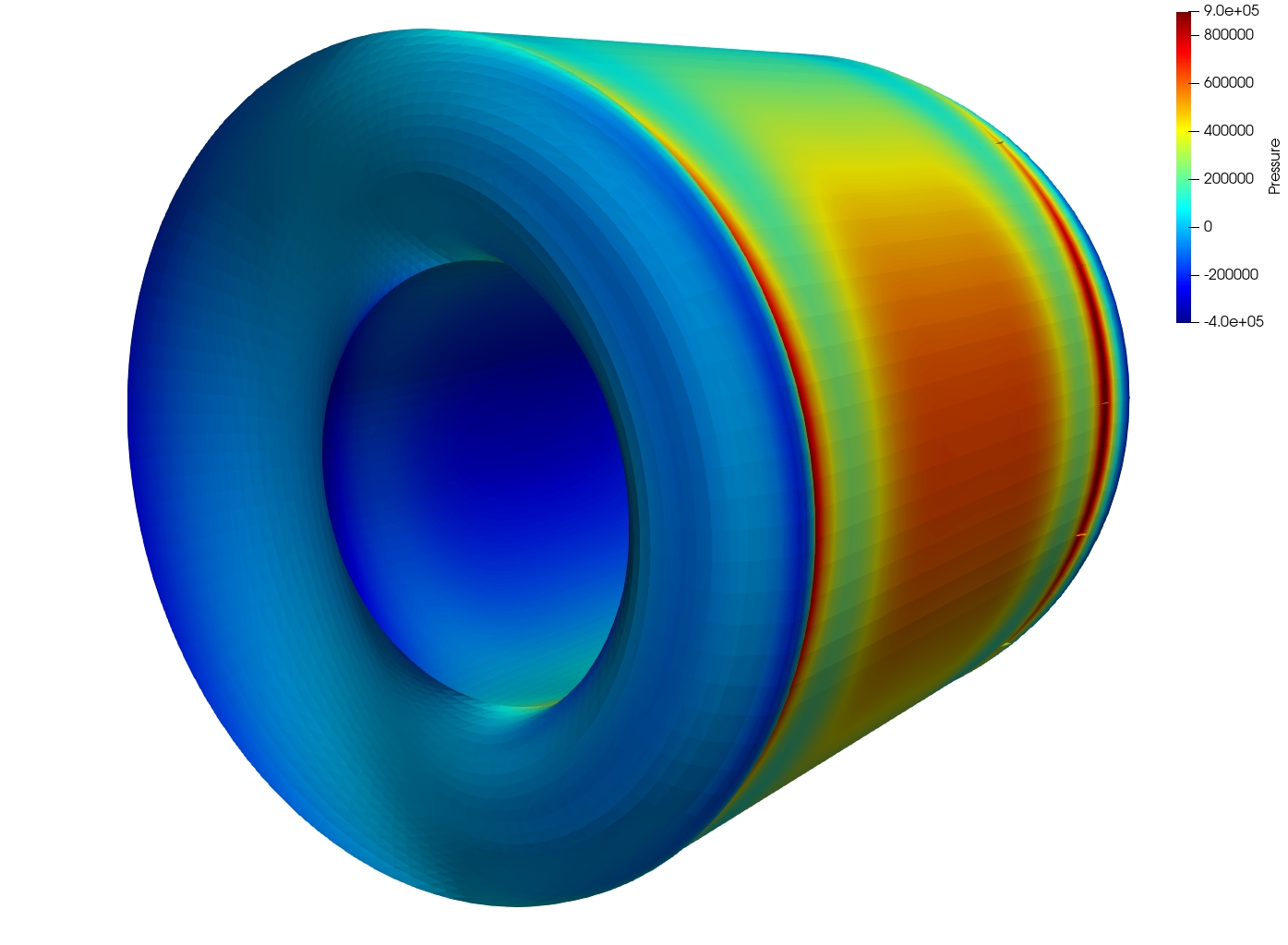} \\
$\omega = 5$ s$^{-1}$, $\mathsf p = 2$ & $\omega = 5$ s$^{-1}$, $\mathsf p = 3$ & $\omega = 5$ s$^{-1}$, $\mathsf p = 2$ & $\omega = 5$ s$^{-1}$, $\mathsf p = 3$ \\
\includegraphics[angle=0, trim=130 10 160 10, clip=true, scale = 0.09]{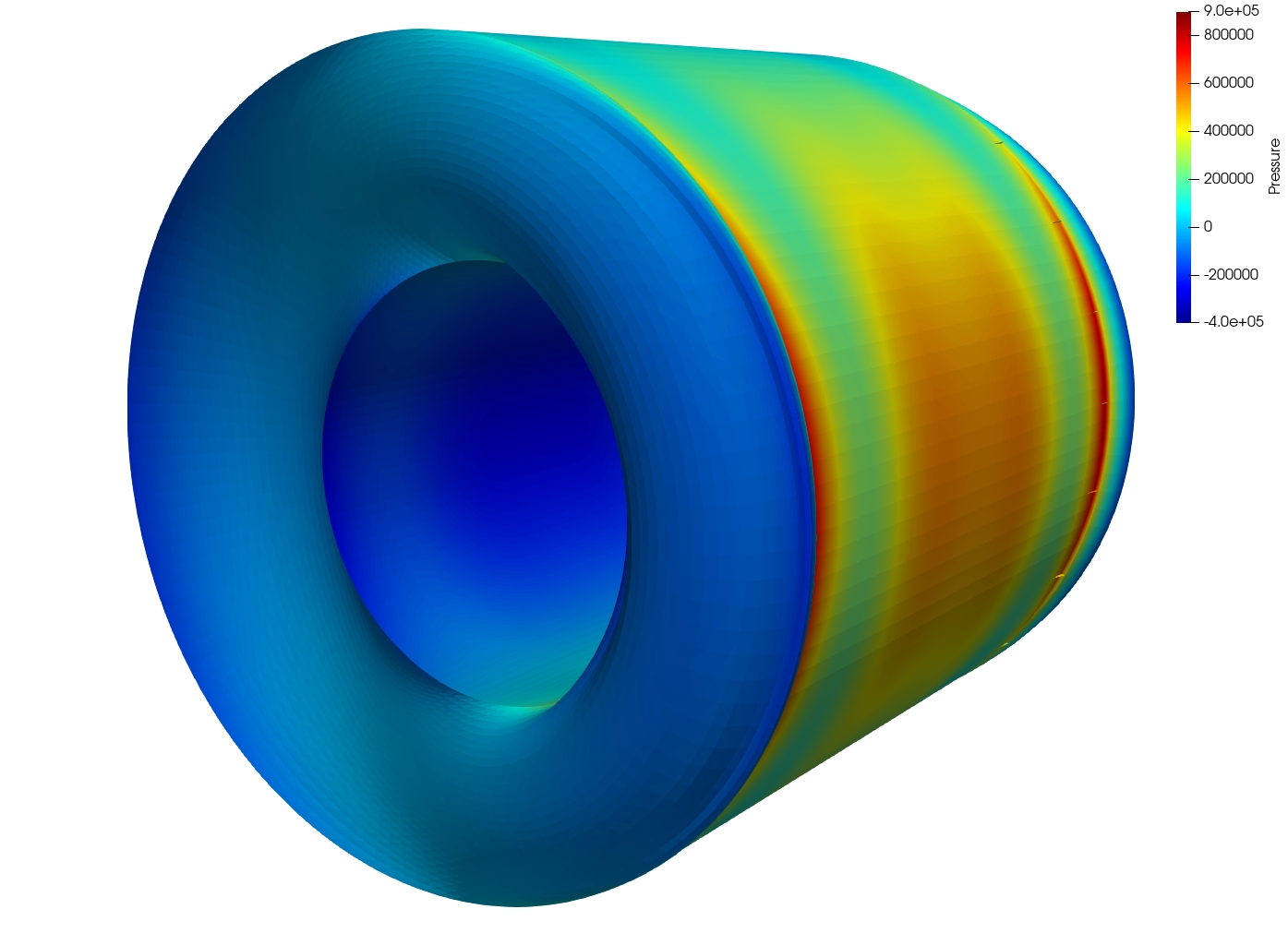} &
\includegraphics[angle=0, trim=130 10 160 10, clip=true, scale = 0.09]{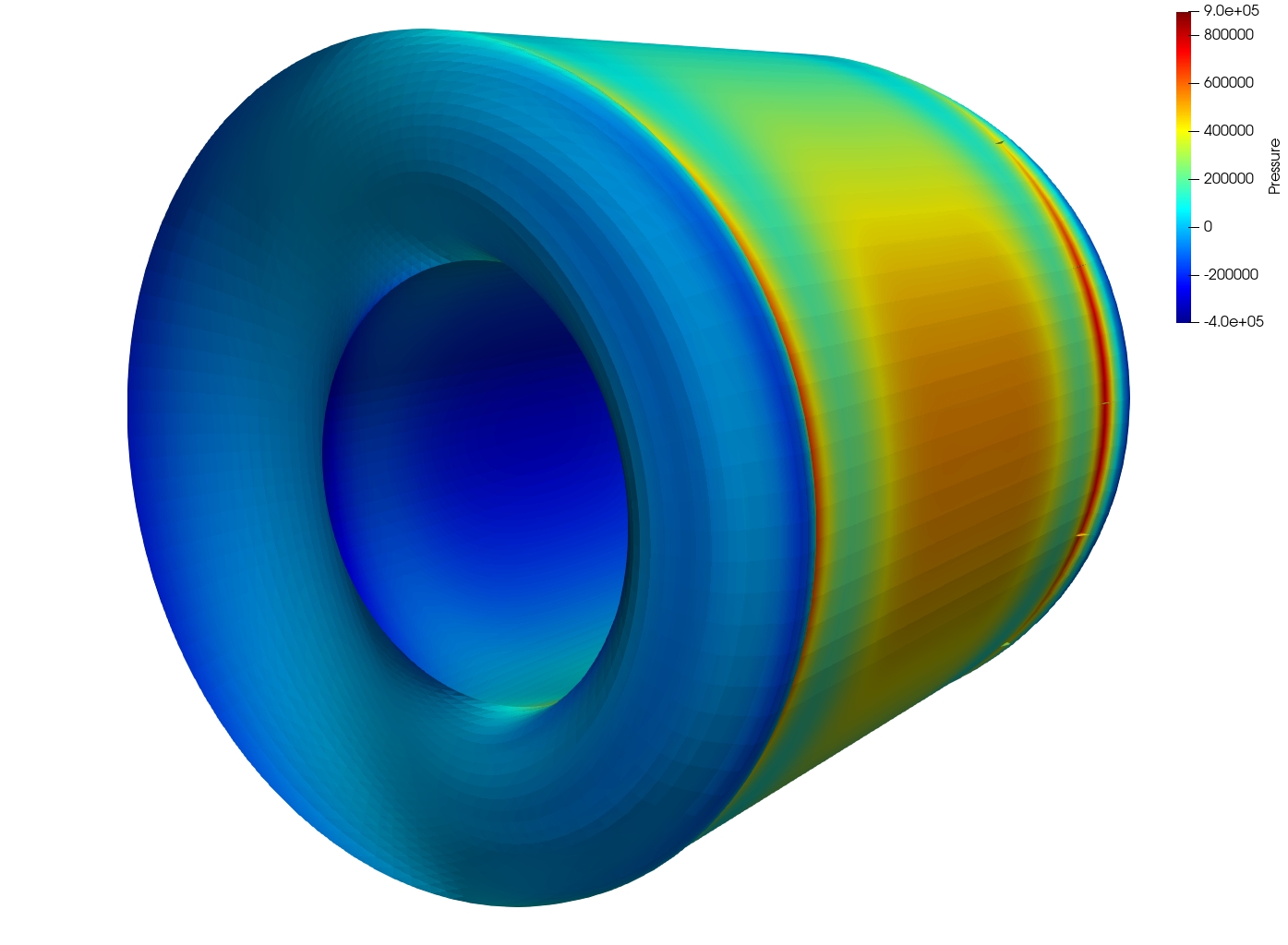} &
\includegraphics[angle=0, trim=130 10 160 10, clip=true, scale = 0.09]{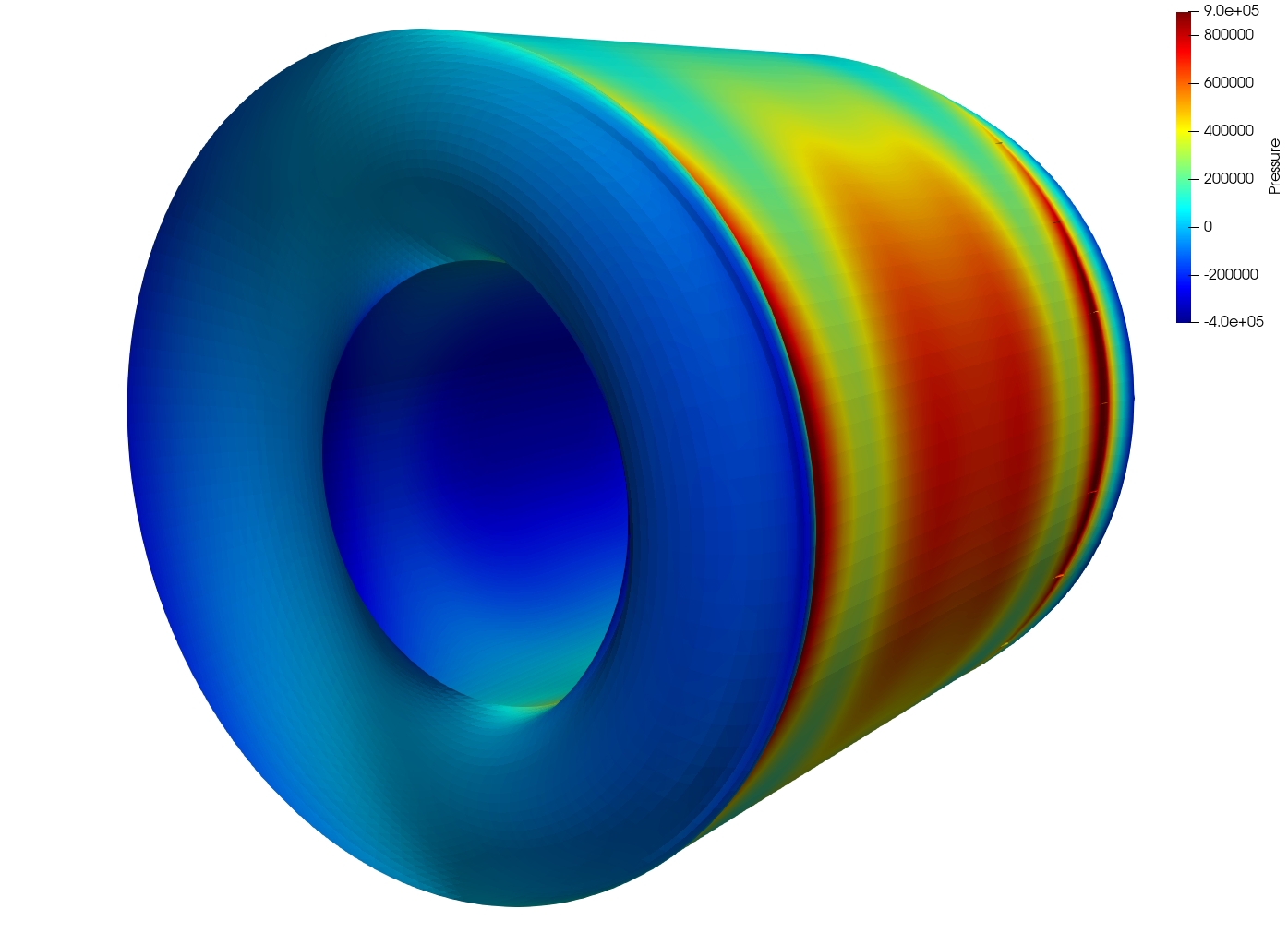} &\includegraphics[angle=0, trim=130 10 160 10, clip=true, scale = 0.09]{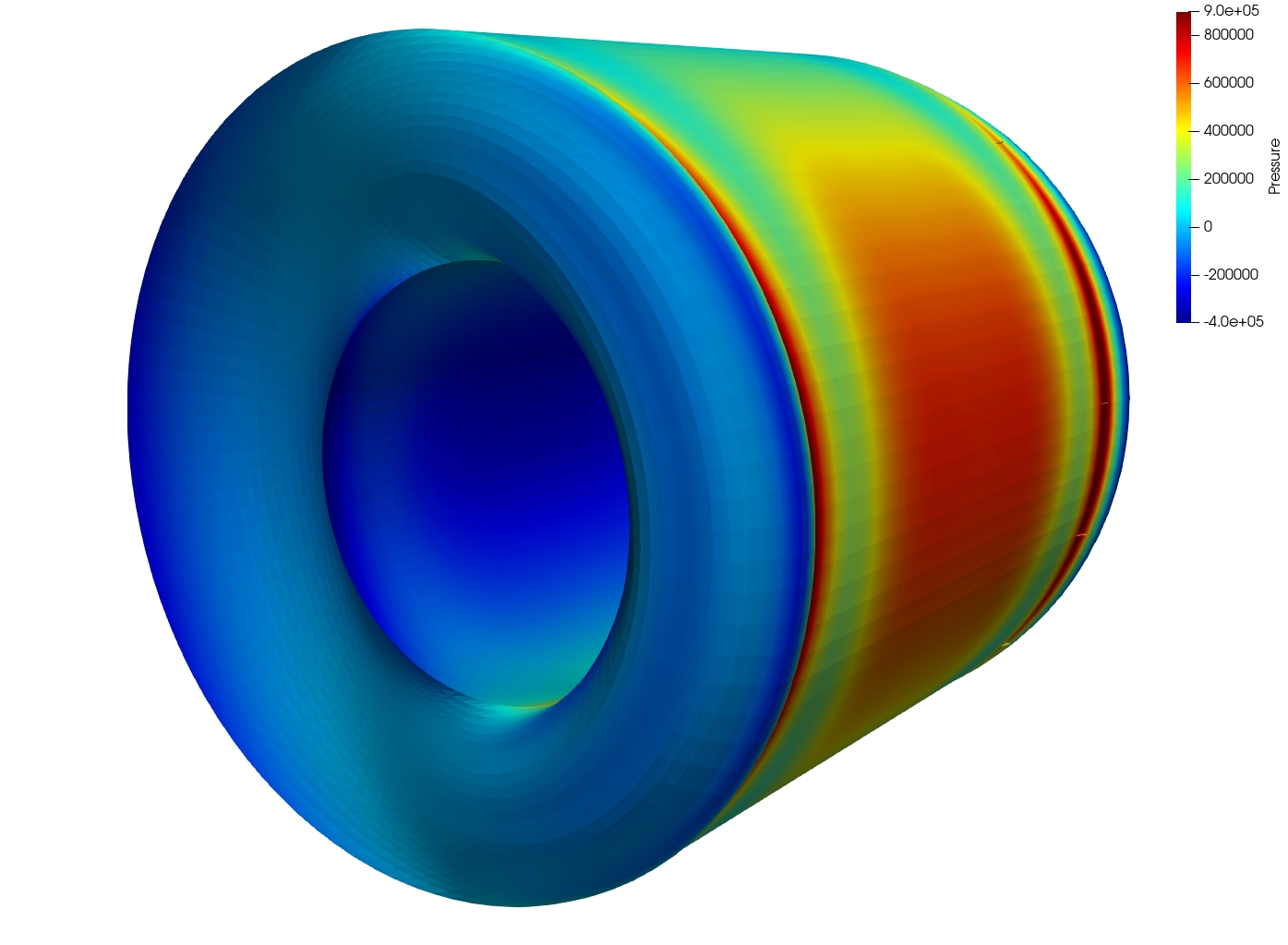} \\
$\omega = 10$ s$^{-1}$, $\mathsf p = 2$ & $\omega = 10$ s$^{-1}$, $\mathsf p = 3$ & $\omega = 10$ s$^{-1}$, $\mathsf p = 2$ & $\omega = 10$ s$^{-1}$, $\mathsf p = 3$ \\
\includegraphics[angle=0, trim=130 10 160 10, clip=true, scale = 0.09]{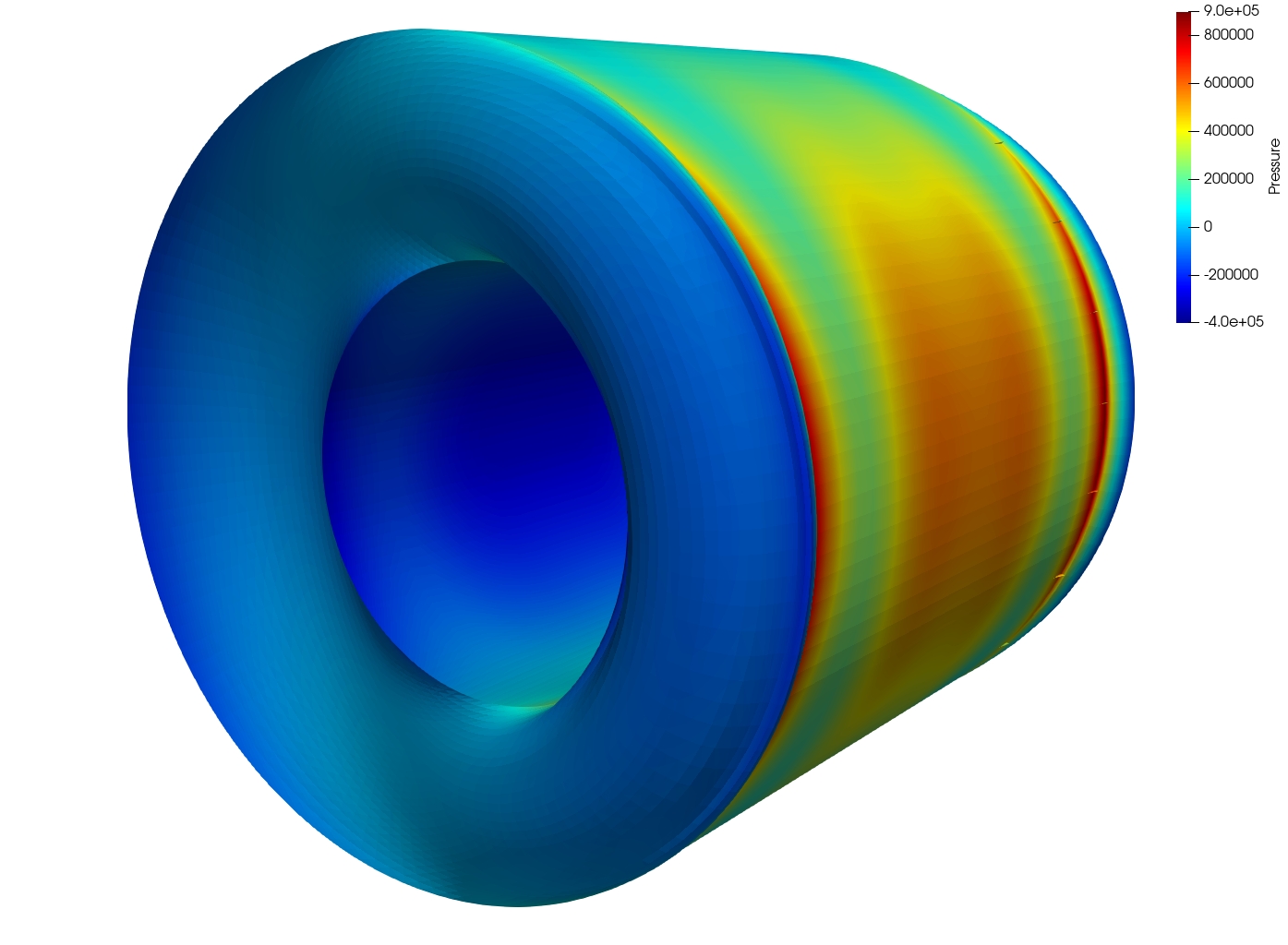} &
\includegraphics[angle=0, trim=130 10 160 10, clip=true, scale = 0.09]{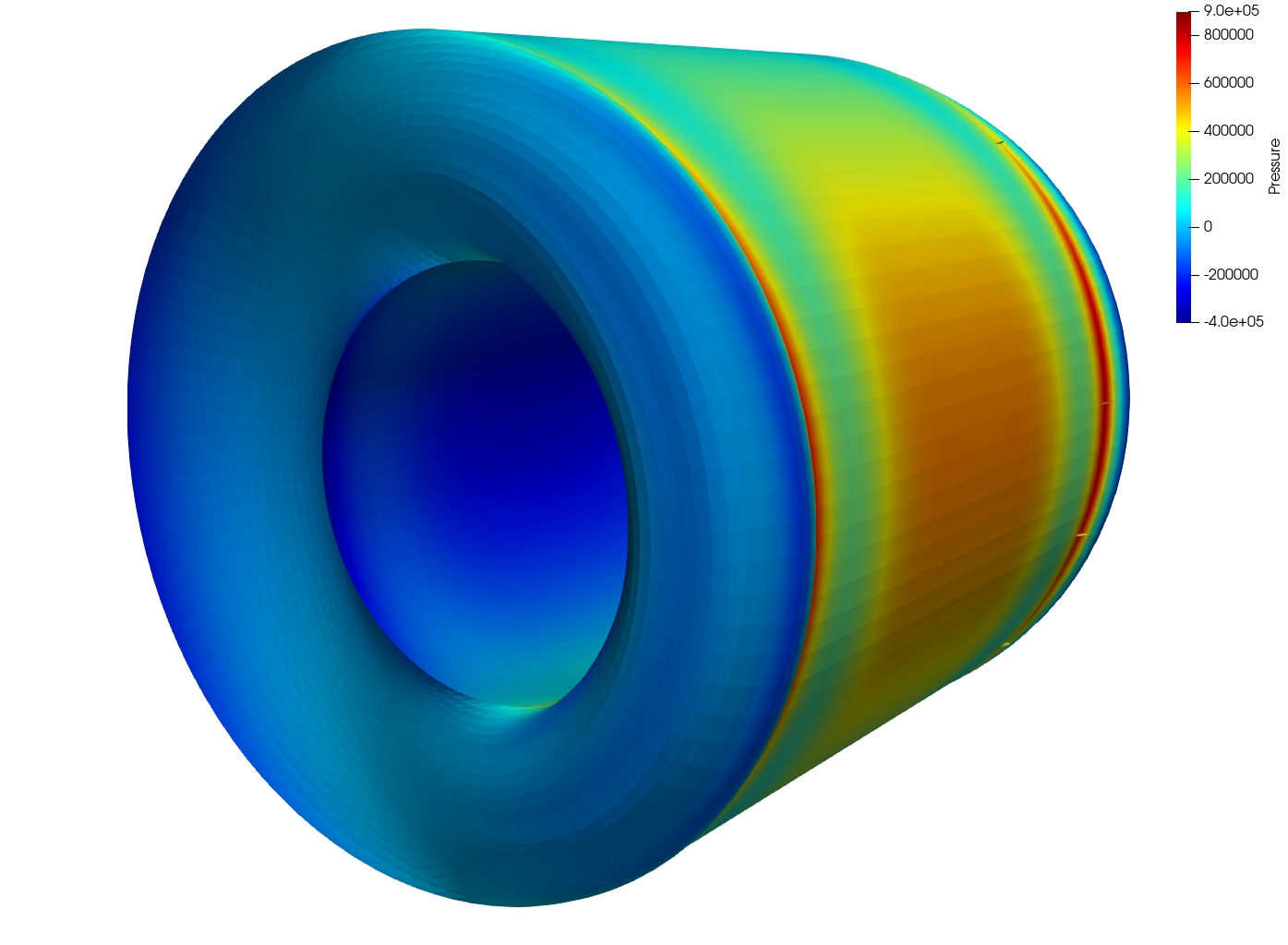} &
\includegraphics[angle=0, trim=130 10 160 10, clip=true, scale = 0.09]{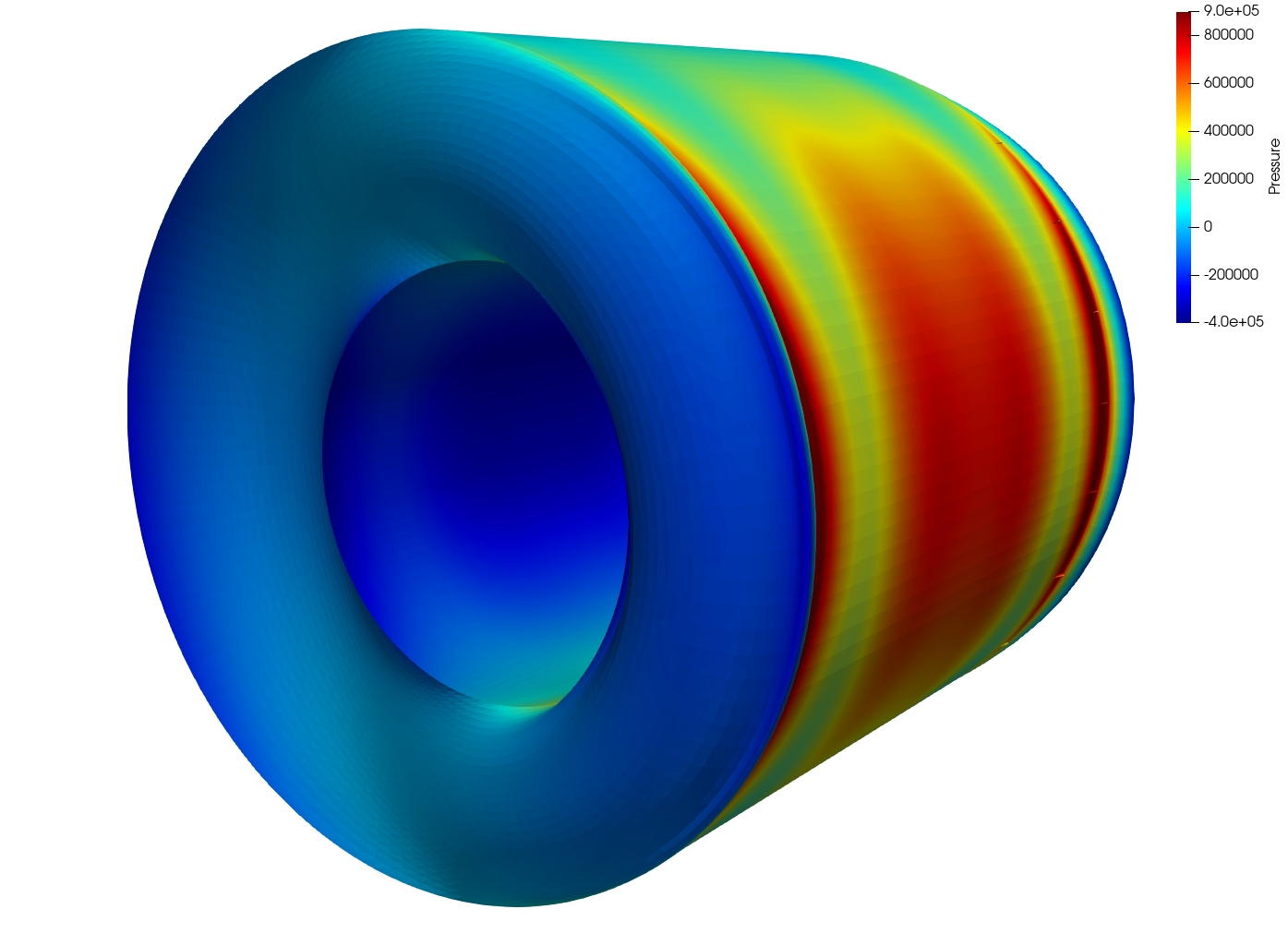} &\includegraphics[angle=0, trim=130 10 160 10, clip=true, scale = 0.09]{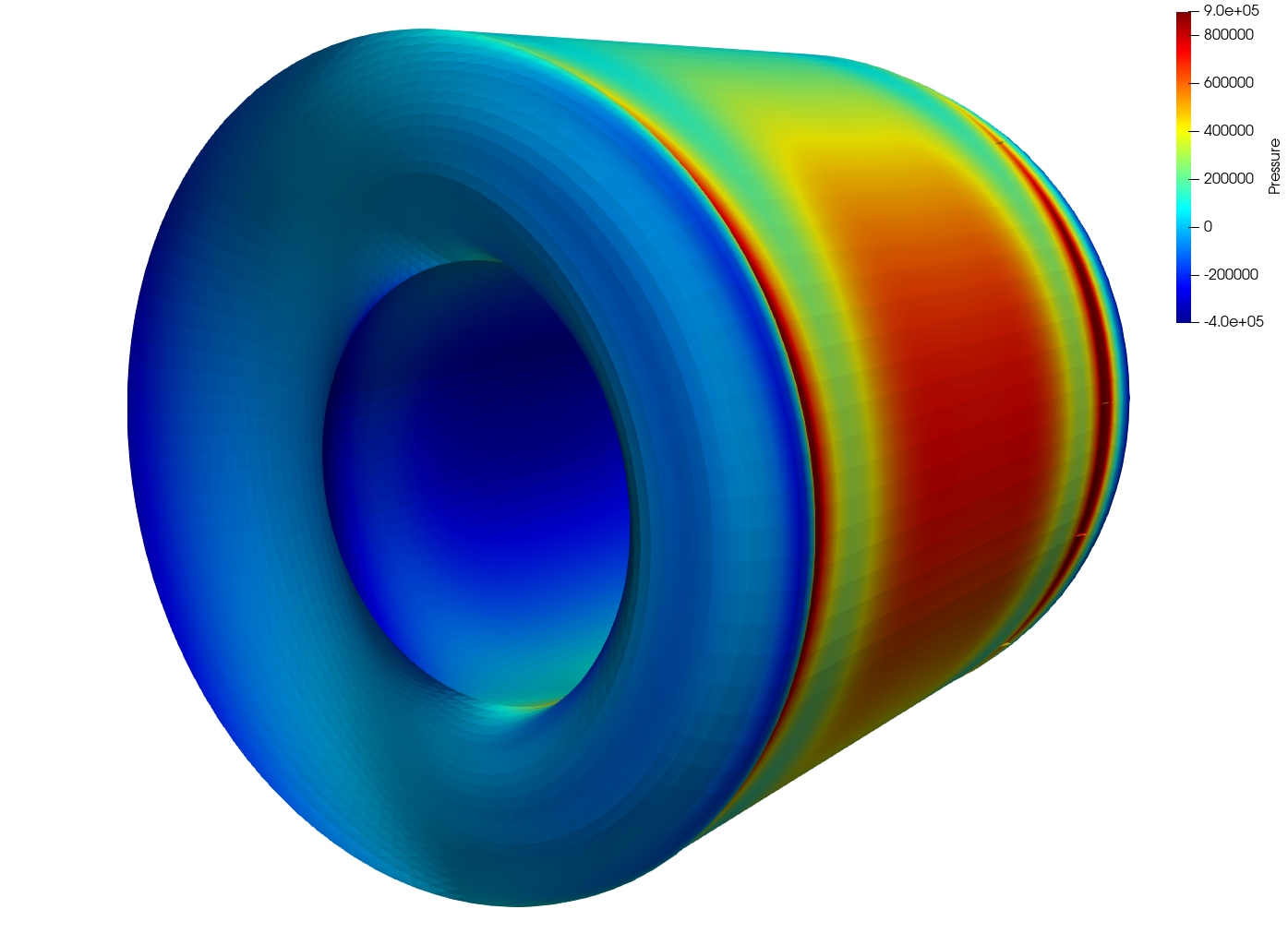} \\
$\omega = 20$ s$^{-1}$, $\mathsf p = 2$ & $\omega = 20$ s$^{-1}$, $\mathsf p = 3$ & $\omega = 20$ s$^{-1}$, $\mathsf p = 2$ & $\omega = 20$ s$^{-1}$, $\mathsf p = 3$ \\
\multicolumn{4}{c}{
\includegraphics[angle=0, trim=430 0 0 930, clip=true, scale = 0.3]{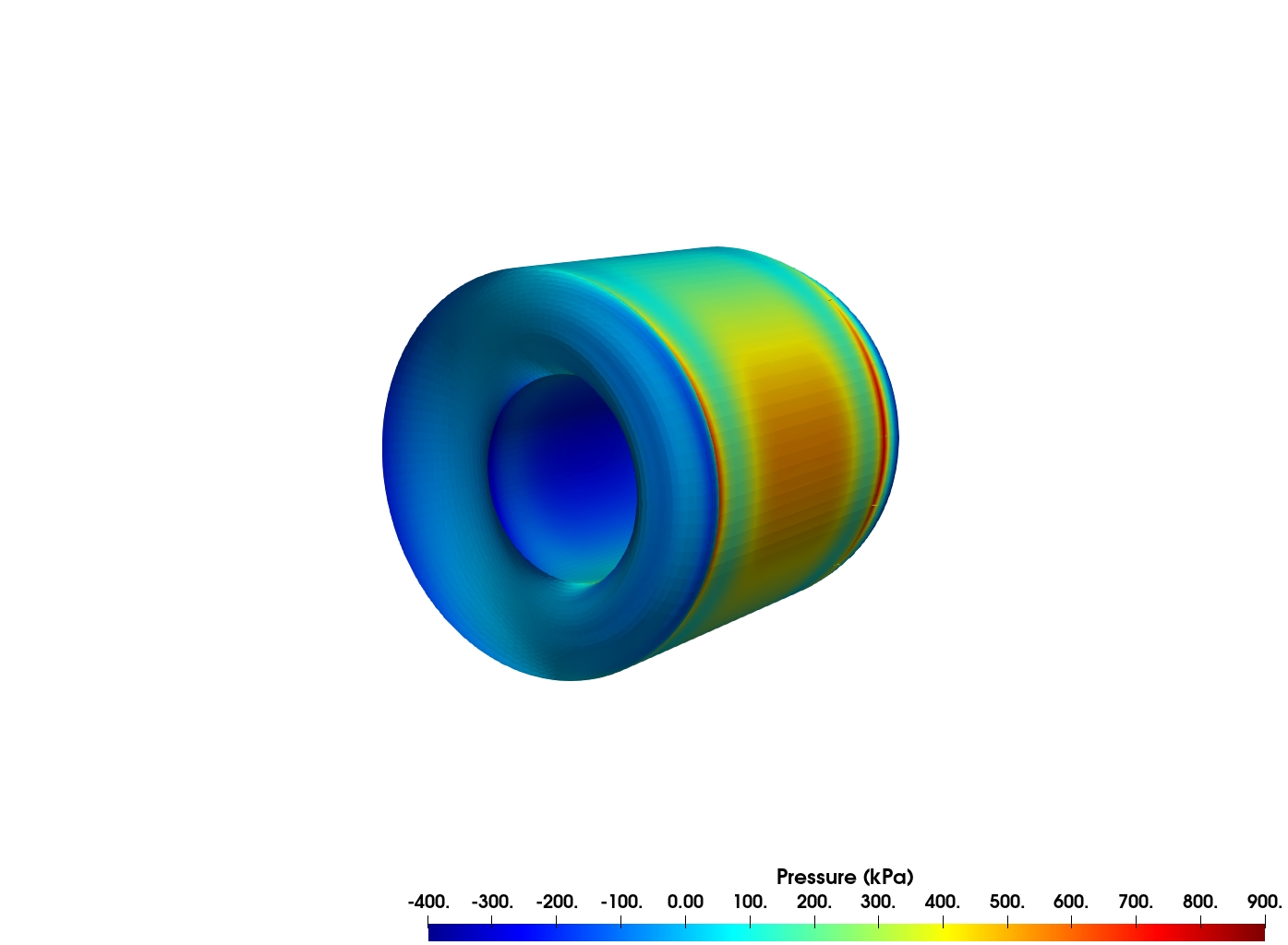}
} \\
\end{tabular}
\caption{The snapshots of the pressure plotted on the deformed configuration at time $t = \pi/2\omega$ of the HS and MIPC models for the clamped cylindrical support problem. The results from the meshes with $\mathsf p = 2$ and $3$ are shown for comparison.} 
\label{fig:hysteresis_pressure_snapshots}
\end{figure}

To investigate the viscous dissipative effect, we calculate the force on the inner surface $\Gamma_{\mathrm{inner}}$ as
\begin{align*}
\bm{\mathcal F} := \int_{\Gamma_{\mathrm{inner}}} \bm P \bm N d\Gamma.
\end{align*}
The lateral component in the $x$-direction $\bm{\mathcal F}_x$ is plotted against the lateral displacement $\bm U_{x}$ for the HS and MIPC models in Figure \ref{fig:hysteresis} for both meshes. The hysteretical curves are almost indistinguishable, suggesting the results are mesh independent. Comparing the results between the HS and MIPC models, we may observe that the difference in the configurational free energy engenders differences in the hysteretical loop, with the MIPC model giving higher values of the maximum force. With the increase of the strain rate $\omega$, both material models need more cycles to reach equilibrium; the increase in strain rate leads to an increase in force on the inner surface at the maximum strain; larger strain rates make the viscous effect less significant, making the material behavior closer to hyperelasticity that includes both equilibrium and non-equilibrium contributions. Those observations match well with prior studies \cite{Reese1998,Latorre2015}. The pressure fields on the deformed configuration when the support reaches the maximum strain are depicted in Figure \ref{fig:hysteresis_pressure_snapshots}. The results from the discretization with $\mathsf p =2$ and $3$ show good agreement for both models under the three strain rates. This again demonstrates the superior stress resolving property of the proposed discretization method \cite{Liu2019}.

\section{Conclusion}
\label{sec:conclusion}
In this work, we start by considering a general continuum theory for viscoelasticity with the viscous deformation characterized by a set of internal state variables. With the incompressible constraint as a common material property in mind, we choose to develop a theory based on the Gibbs free energy \cite{Liu2018,Liu2019a}, which leads to a pressure primitive variable formulation. A set of nonlinear evolution equations is derived for the viscous deformation. With that, we consider a special form of the configurational free energy, which leads to the definition of the finite linear viscoelasticity. This viscoelasticity theory is endowed with a set of linear evolution equations, which is amenable to finite element implementation and anisotropic material modeling. It is revealed through the derivation that the original viscoelasticity model developed in \cite{Simo1987,Simo2006} can be rectified to maintain thermodynamic consistency. In particular, the following rectifications need to be pointed out. 
\begin{enumerate}
\item The right-hand side of the evolution equations \eqref{eq:evolution-eqn-Upsilon-special} is driven by the \textit{fictitious} second Piloa-Kirchhoff stress $\tilde{\bm S}^{\alpha}_{\mathrm{iso}}$. In contrast, in the original model, the right-hand side is driven by $J^{-\frac23} \mathbb P : \tilde{\bm S}^{\alpha}_{\mathrm{iso}}$ \cite{Simo1987,Simo2006}.
\item The non-equilibrium stresses $\bm S^{\alpha}_{\mathrm{neq}}$ and the conjugate variables $\bm Q^{\alpha}$ are two different quantities. They become identical only when the configurational free energy takes a special form \eqref{eq:HS-model-G-alpha-form}.
\item The configurational free energy needs to satisfy the normalization condition \eqref{eq:Gibbs_free_energy_normalization_condition}; the relaxation of the non-equilibrium stresses poses an additional constraint on the form of the configurational free energy (see Section \ref{subsubsec:condition-vanish-non-equilibrium-stress}). In particular, the identical polymer chain model considered in \cite{Holzapfel1996} does not guarantee the vanishment of the viscous stress in the equilibrium limit, and we have identified an unstable solution of that model (see Figures \ref{fig:compare-IPC-MIPC} and \ref{fig:compare-IPC-MPIC-energy}).
\end{enumerate}
Indeed, the original finite linear viscoelasticity model has been criticized for the lack of a thermomechanical foundation, and a finite time blow-up solution has been previously identified. The above three points could be the potential source of the blow-up phenomenon observed by S. Govindjee, et al. in \cite{Govindjee2014}. Additionally, in one instantiation of the model, the free energy indicates that there is an \textit{additive} split of the elastic and viscous strain (see \eqref{eq:HS-model-configurational-energy-additive-split}), which makes this theory analogous to the elastoplasticity theory proposed by A. Green and P. Naghdi \cite{Green1965,Green1971}. Based on the consistent continuum theory, we construct a numerical formulation that inherits the stability to the semidiscrete formulation. We invoke the smooth generalization of the Taylor-Hood element based on NURBS for the spatial discretization. The well-known recurrence formula is utilized to integrate the constitutive relation, and the dynamic integration is performed by the generalized-$\alpha$ scheme. A variety of benchmark examples are presented and corroborate the properties of the continuum and numerical formulations in different deformation states. The superior stress accuracy of the adopted NURBS basis function is demonstrated as well. 

Based on this work, we will extend the proposed viscoelasticity theory to fiber-reinforced materials with a particular focus on arterial wall modeling and vascular fluid-structure interaction. It is also of interest to investigate the approximation of the viscoelastic constitutive relation directly from the hereditary integral, which may open the door for the numerical modeling of biological tissue growth and remodeling \cite{Humphrey2002}.

\section*{Acknowledgements}
We want to thank Prof. Jay D. Humphrey at Yale University for many helpful discussions. This work is supported by the National Institutes of Health under the award numbers 1R01HL121754, 1R01HL123689, R01EB01830204, the startup grant provided by the Southern University of Science and Technology under the award number Y01326127, the Guangdong-Hong Kong-Macao Joint Laboratory for Data-Driven Fluid Mechanics and Engineering Applications under the award number 2020B1212030001, the computational resources from the Center for Computational Science and Engineering at Southern University of Science and Technology, the Stanford Research Computing Center, and the Extreme Science and Engineering Discovery Environment supported by the National Science Foundation grant ACI-1053575.

\appendix

\section{An analysis of the null space of $\mathbb P$}
\label{appendix:an-analysis-of-the-null-space-of-P}
In this section, we show that a fictitious stress living in the null space of $\mathbb P$ has to be a zero stress. By definition, this stress satisfies
\begin{align*}
\bm O = \mathbb P : \tilde{\bm S} = \tilde{\bm S} - \frac13 \left( \tilde{\bm S} : \bm C \right) \bm C^{-1}.
\end{align*}
From the above definition, one has
\begin{align}
\label{eq:appendix-fictitious-stress-special-form}
\tilde{\bm S} = \frac13 \left( \tilde{\bm S} : \bm C \right) \bm C^{-1} = \frac13 \left( \tilde{\bm S} : \tilde{\bm C} \right) \tilde{\bm C}^{-1} = a \tilde{\bm C}^{-1}, \quad \mbox{ with } a := \frac13 \left( \tilde{\bm S} : \tilde{\bm C} \right).
\end{align}
Recall that the fictitious stress is given by an isochoric energy $G_{\mathrm{iso}}(\tilde{\bm C}, \Theta, \bm \Gamma^1, \cdots, \bm \Gamma^m)$,
\begin{align*}
\tilde{\bm S} = 2 \frac{\partial G_{\mathrm{iso}}}{\partial \tilde{\bm C} }.
\end{align*} 
Now, let $\bm C(t)$ be an arbitrary Green-Lagrange tensor that smoothly varies over time $t$. We consider a free energy $G$ defined for $\bm C(t)$ with a conjugate stress satisfying the relation \eqref{eq:appendix-fictitious-stress-special-form}. Therefore, one has
\begin{align*}
G(\bm C(t)) =& G(\bm C(t)) - G(\bm I) = \frac12 \int_{0}^t \bm S(\bm C(s)) : \dot{\bm C}(s) ds = \frac12 \int_{0}^t a \bm C^{-1}(s) : \dot{\bm C}(s) ds \nonumber \\
=& \int_{0}^t a \bm C^{-1}(s) : \bm F^T \bm d \bm F ds = \int_{0}^t a \mathrm{tr}[\bm d] ds,
\end{align*}
with $\bm d$ being the rate of deformation tensor.
Now we consider that the Green-Lagrange tensor characterizes an isochoric motion, implying $\mathrm{tr}[\bm d] = 0$. Therefore, we have $G(\tilde{\bm C}) = 0$ for $\tilde{\bm C}$ characterizing volume-preserving deformations. This suggests that for a fictitious stress satisfying the relation \eqref{eq:appendix-fictitious-stress-special-form}, it has to be a zero stress,
\begin{align*}
\tilde{\bm S} = 2\frac{\partial G_{\mathrm{iso}}}{\partial \tilde{\bm C}} = \bm O.
\end{align*}
In other words, an isochoric free energy cannot induce a hydrostatic fictitious stress.

\section{Algorithm for the stress and elasticity tensor in the identical polymer chain model with $G^{\alpha} = F^{\alpha}$}
\label{sec:algorithm-IPC}
In the original model based on the identical polymer chain assumption, it is taken that $G^{\alpha}(\tilde{\bm C},\Theta) = F^{\alpha}(\tilde{\bm C},\Theta) = \beta^{\infty}_{\alpha} G^{\infty}_{\mathrm{iso}}(\tilde{\bm C},\Theta)$ \cite{Holzapfel1996}. Then the configurational free energy takes the following form.
\begin{align*}
\Upsilon^{\alpha}\left(\tilde{\bm C}, \Theta, \bm \Gamma^1, \cdots, \bm \Gamma^m \right) =& H^{\alpha}(\bm \Gamma^{\alpha},\Theta) + \left( \hat{\bm S}^{\alpha}_0 - 2 \frac{\partial G^{\alpha}(\tilde{\bm C}, \Theta)}{\partial \tilde{\bm C}} \right) : \frac{\bm \Gamma^{\alpha} - \bm I}{2} + G^{\alpha}( \tilde{\bm C}, \Theta ), \\
=& \mu^{\alpha}(\Theta) \left\lvert \frac{\bm \Gamma^{\alpha} - \bm I}{2} \right\rvert^2 + \left( \hat{\bm S}^{\alpha}_0 - \beta^{\infty}_{\alpha} \tilde{\bm S}^{\infty}_{\mathrm{iso}} \right) : \frac{\bm \Gamma^{\alpha} - \bm I}{2} + \beta^{\infty}_{\alpha} G^{\infty}_{\mathrm{iso}}( \tilde{\bm C}, \Theta ). 
\end{align*}
To the best of our knowledge, numerical analysis for that model has not been performed, probably due to the appearance of a six-order tensor in the definition of the elasticity tensor. Here, we provide the algorithm for computing the stress as well as the elasticity tensor of this model. Given the time step $\Delta t_n$, and the value of $\tilde{\bm S}^{\infty}_{\mathrm{iso} \: n}$ and $\bm Q^{\alpha}_{n}$ at each quadrature points, proceed through the following steps to compute the isochoric part of the second Piola-Kirchhoff stress and the isochoric part of the elasticity tensor.
\begin{enumerate}
\item Calculate the deformation gradient and the strain measures based on the displacement $\bm U_{n+1}$,
\begin{align*}
\bm F_{n+1} = \bm I + \nabla_{\bm X} \bm U_{n+1}, \quad J_{n+1} = \mathrm{det}\left( \bm F_{n+1} \right), \quad \bm C_{n+1} = \bm F_{n+1}^T \bm F_{n+1}, \quad \tilde{\bm C}_{n+1} = J_{n+1}^{-2/3} \bm C_{n+1}.
\end{align*}
\item Calculate the fictitious second Piola-Kirchhoff stress
\begin{align*}
\tilde{\bm S}^{\infty}_{\mathrm{iso} \: n+1} = 2 \left( \frac{\partial G^{\infty}_{\mathrm{iso}}}{\partial \tilde{\bm C}} \right)_{n+1}.
\end{align*}
\item Calculate the fictitious elasticity tensor
\begin{align*}
\tilde{\mathbb C}^{\infty}_{\mathrm{iso} \: n+1} = 4 J^{-\frac{4}{3}}_{n+1} \left( \frac{\partial^2 G^{\infty}_{\mathrm{iso}}}{\partial \tilde{\bm C} \partial \tilde{\bm C}} \right)_{n+1}.
\end{align*}
\item For $\alpha = 1, \cdots, m$, calculate
\begin{align*}
\bm Q^{\alpha}_{n+1} = \beta^{\infty}_{\alpha} \mathrm{exp}\left( \xi^{\alpha} \right) \tilde{\bm S}^{\infty}_{\mathrm{iso} \: n+1} + \mathrm{exp}\left( \xi^{\alpha} \right) \left( \mathrm{exp}\left( \xi^{\alpha} \right) \bm Q^{\alpha}_{n} - \beta^{\infty}_{\alpha} \tilde{\bm S}^{\infty}_{\mathrm{iso} \: n} \right).
\end{align*}
\item For $\alpha = 1, \cdots, m$, calculate
\begin{align*}
\tilde{\bm S}^{\alpha}_{\mathrm{neq} \: n+1} = \beta^{\infty}_{\alpha} \tilde{\bm S}^{\infty}_{\mathrm{iso} \: n+1} - \frac{\beta^{\infty}_{\alpha} }{2\mu^{\alpha}}J^{\frac43}_{n+1} \tilde{\mathbb C}^{\infty}_{\mathrm{iso} \: n+1} : \left( \beta^{\infty}_{\alpha}\tilde{\bm S}^{\infty}_{\mathrm{iso} \: n+1} - \hat{\bm S}^{\alpha}_0 - \bm Q^{\alpha}_{n+1} \right).
\end{align*}
\item Calculate the fictitious stress
\begin{align*}
\tilde{\bm S}_{n+1} = \tilde{\bm S}^{\infty}_{\mathrm{iso} \: n+1} + \sum_{\alpha=1}^{m} \tilde{\bm S}^{\alpha}_{\mathrm{neq} \: n+1}.
\end{align*}
\item Calculate the projection tensor
\begin{align*}
\mathbb P_{n+1} = \mathbb I - \frac13 \bm C^{-1}_{n+1} \otimes \bm C_{n+1}.
\end{align*}
\item Calculate the isochoric part of the second Piola-Kirchhoff stress
\begin{align*}
\bm S_{\mathrm{iso} \: n+1} = J_{n+1}^{-\frac23} \mathbb P_{n+1} : \tilde{\bm S}_{n+1}.
\end{align*}
\item Calculate the fictitious elasticity tensors for the non-equilibrium part,
\begin{align*}
\tilde{\mathbb C}_{\mathrm{neq} \: n+1}^{\alpha} =& \beta^{\infty}_{\alpha} \tilde{\mathbb C}^{\infty}_{\mathrm{iso} \: n+1} - \frac{4 \beta^{\infty}_{\alpha}}{\mu^{\alpha}} J^{-\frac43}_{n+1} \left( \frac{\partial^3 G^{\infty}_{\mathrm{iso}}}{\partial \tilde{\bm C} \partial \tilde{\bm C} \partial \tilde{\bm C}} \right)_{n+1} : \left( \beta^{\infty}_{\alpha} \tilde{\bm S}^{\infty}_{\mathrm{iso} \: n+1} - \hat{\bm S}^{\alpha}_0 - \bm Q^{\alpha}_{n+1} \right) \\
&- \frac{\beta^{\infty}_{\alpha}}{2\mu^{\alpha}}J^{\frac43}_{n+1}\left( 1 - \delta_{\alpha} \right) \tilde{\mathbb C}^{\infty}_{\mathrm{iso} \: n+1} : \tilde{\mathbb C}^{\infty}_{\mathrm{iso} \: n+1}.
\end{align*}
\item Calculate the fictitious elasticity tensor 
\begin{align*}
\tilde{\mathbb C}_{n+1} = \tilde{\mathbb C}_{\mathrm{iso} \: n+1}^{\infty} + \sum_{\alpha=1}^{m} \tilde{\mathbb C}_{\mathrm{neq} \: n+1}^{\alpha}.
\end{align*}
\item Calculate the fourth-order tensor 
\begin{align*}
\tilde{\mathbb P}_{n+1} = \bm C^{-1}_{n+1} \odot \bm C^{-1}_{n+1} - \frac13 \bm C^{-1}_{n+1} \otimes \bm C^{-1}_{n+1}.
\end{align*}
\item Calculate the isochoric part of the elasticity tensor
\begin{align*}
\mathbb C_{\mathrm{iso} \: n+1} = \mathbb P_{n+1} : \tilde{\mathbb C}_{n+1} : \mathbb P^T_{n+1} + \frac23 \mathrm{Tr}\left( J^{-\frac23}_{n+1} \tilde{\bm S}_{n+1} \right) \tilde{\mathbb P}_{n+1} - \frac23 \left( \bm C^{-1}_{n+1} \otimes \bm S_{\mathrm{iso} \: n+1} + \bm S_{\mathrm{iso} \: n+1} \otimes \bm C^{-1}_{n+1} \right).
\end{align*}
\end{enumerate}

\section{Algorithm for the stress and elasticity tensor in the Holzapfel-Simo model}
\label{sec:algorithm-HS}
The Holzapfel-Simo model presented in Section \ref{sec:HS-model} is different from the classical one documented in classical works \cite[Chapter~6.10]{Holzapfel2000}. Here, we provide the algorithm for computing the stresses as well as the elasticity tensor. Given the time step $\Delta t_n$, and the value of $\tilde{\bm S}^{\infty}_{\mathrm{iso} \: n}$ and $\bm Q^{\alpha}_{n}$ at each quadrature points, proceed through the following steps to compute the isochoric part of the second Piola-Kirchhoff stress and the isochoric part of the elasticity tensor.
\begin{enumerate}
\item Calculate the deformation gradient and the strain measures based on the displacement $\bm U_{n+1}$,
\begin{align*}
\bm F_{n+1} = \bm I + \nabla_{\bm X} \bm U_{n+1}, \quad J_{n+1} = \mathrm{det}\left( \bm F_{n+1} \right), \quad \bm C_{n+1} = \bm F_{n+1}^T \bm F_{n+1}, \quad \tilde{\bm C}_{n+1} = J_{n+1}^{-2/3} \bm C_{n+1}.
\end{align*}
\item Calculate 
\begin{align*}
\tilde{\bm S}^{\infty}_{\mathrm{iso} \: n+1} = 2 \left( \frac{\partial G^{\infty}_{\mathrm{iso}}}{\partial \tilde{\bm C}} \right)_{n+1}.
\end{align*}
\item For $\alpha = 1, \cdots, m$, calculate
\begin{align*}
\bm Q^{\alpha}_{n+1} = \mu^{\alpha} \mathrm{exp}\left( \xi^{\alpha} \right) \tilde{\bm C}_{n+1} + \mathrm{exp}\left( \xi^{\alpha} \right) \left( \mathrm{exp}\left( \xi^{\alpha} \right) \bm Q^{\alpha}_{n} - \mu^{\alpha} \tilde{\bm C}_n \right).
\end{align*}
\item Calculate the fictitious stress
\begin{align*}
\tilde{\bm S}_{n+1} = \tilde{\bm S}^{\infty}_{\mathrm{iso} \: n+1} + \sum_{\alpha=1}^{m} \bm Q^{\alpha}_{n+1}.
\end{align*}
\item Calculate the projection tensor
\begin{align*}
\mathbb P_{n+1} = \mathbb I - \frac13 \bm C^{-1}_{n+1} \otimes \bm C_{n+1}.
\end{align*}
\item Calculate the isochoric part of the second Piola-Kirchhoff stress
\begin{align*}
\bm S_{\mathrm{iso} \: n+1} = J_{n+1}^{-\frac23} \mathbb P_{n+1} : \tilde{\bm S}_{n+1}.
\end{align*}
\item Calculate the fictitious elasticity tensors for the equilibrium and non-equilibrium parts,
\begin{align*}
\tilde{\mathbb C}_{\mathrm{iso} \: n+1}^{\infty} = 4 J^{-\frac43}_{n+1} \left( \frac{\partial^2 G^{\infty}_{\mathrm{iso}}}{\partial \tilde{\bm C} \partial \tilde{\bm C}} \right)_{n+1}, \quad \mbox{ and } \quad  \tilde{\mathbb C}_{\mathrm{neq} \: n+1}^{\alpha} = 2  \mu^{\alpha} \mathrm{exp}(\xi^{\alpha}) J_{n+1}^{-\frac43} \mathbb I, \quad \mbox{ for } \alpha = 1, \cdots, m.
\end{align*}
\item Calculate the fictitious elasticity tensor 
\begin{align*}
\tilde{\mathbb C}_{n+1} = \tilde{\mathbb C}_{\mathrm{iso} \: n+1}^{\infty} + \sum_{\alpha=1}^{m} \tilde{\mathbb C}_{\mathrm{neq} \: n+1}^{\alpha}.
\end{align*}
\item Calculate the fourth-order tensor 
\begin{align*}
\tilde{\mathbb P}_{n+1} = \bm C^{-1}_{n+1} \odot \bm C^{-1}_{n+1} - \frac13 \bm C^{-1}_{n+1} \otimes \bm C^{-1}_{n+1}.
\end{align*}
\item Calculate the isochoric part of the elasticity tensor
\begin{align*}
\mathbb C_{\mathrm{iso} \: n+1} = \mathbb P_{n+1} : \tilde{\mathbb C}_{n+1} : \mathbb P^T_{n+1} + \frac23 \mathrm{Tr}\left( J^{-\frac23}_{n+1} \tilde{\bm S}_{n+1} \right) \tilde{\mathbb P}_{n+1} - \frac23 \left( \bm C^{-1}_{n+1} \otimes \bm S_{\mathrm{iso} \: n+1} + \bm S_{\mathrm{iso} \: n+1} \otimes \bm C^{-1}_{n+1} \right).
\end{align*}
\end{enumerate}

\section{Algorithm for the stress and elasticity tensor in the modified identical polymer chain model}
\label{sec:algorithm-MIPC}
The modified identical polymer chain model presented in Section \ref{sec:MIPC-model} rectifies the non-physical behavior of the original identical polymer chain model in the thermodynamic equilibrium. Here we provide the algorithm for computing the stresses as well as the elasticity tensor of this model. Given the time step $\Delta t_n$, and the value of $\tilde{\bm S}^{\infty}_{\mathrm{iso} \: n}$ and $\bm Q^{\alpha}_{n}$ at each quadrature points, proceed through the following steps to compute the isochoric part of the second Piola-Kirchhoff stress and the isochoric part of the elasticity tensor.
\begin{enumerate}
\item Calculate the deformation gradient and the strain measures based on the displacement $\bm U_{n+1}$,
\begin{align*}
\bm F_{n+1} = \bm I + \nabla_{\bm X} \bm U_{n+1}, \quad J_{n+1} = \mathrm{det}\left( \bm F_{n+1} \right), \quad \bm C_{n+1} = \bm F_{n+1}^T \bm F_{n+1}, \quad \tilde{\bm C}_{n+1} = J_{n+1}^{-2/3} \bm C_{n+1}.
\end{align*}
\item Calculate the fictitious second Piola-Kirchhoff stress
\begin{align*}
\tilde{\bm S}^{\infty}_{\mathrm{iso} \: n+1} = 2 \left( \frac{\partial G^{\infty}_{\mathrm{iso}}}{\partial \tilde{\bm C}} \right)_{n+1}.
\end{align*}
\item Calculate the fictitious elasticity tensor
\begin{align*}
\tilde{\mathbb C}^{\infty}_{\mathrm{iso} \: n+1} = 4 J^{-\frac{4}{3}}_{n+1} \left( \frac{\partial^2 G^{\infty}_{\mathrm{iso}}}{\partial \tilde{\bm C} \partial \tilde{\bm C}} \right)_{n+1}.
\end{align*}
\item For $\alpha = 1, \cdots, m$, calculate
\begin{align*}
\bm Q^{\alpha}_{n+1} = \beta^{\infty}_{\alpha} \mathrm{exp}\left( \xi^{\alpha} \right) \tilde{\bm S}^{\infty}_{\mathrm{iso} \: n+1} + \mathrm{exp}\left( \xi^{\alpha} \right) \left( \mathrm{exp}\left( \xi^{\alpha} \right) \bm Q^{\alpha}_{n} - \beta^{\infty}_{\alpha} \tilde{\bm S}^{\infty}_{\mathrm{iso} \: n} \right).
\end{align*}
\item For $\alpha = 1, \cdots, m$, calculate
\begin{align*}
\tilde{\bm S}^{\alpha}_{\mathrm{neq} \: n+1} = \frac{ J^{\frac43}_{n+1} \beta^{\infty}_{\alpha} }{2\mu^{\alpha}} \tilde{\mathbb C}^{\infty}_{\mathrm{iso} \: n+1} : \bm Q^{\alpha}_{n+1}.
\end{align*}
\item Calculate the fictitious stress
\begin{align*}
\tilde{\bm S}_{n+1} = \tilde{\bm S}^{\infty}_{\mathrm{iso} \: n+1} + \sum_{\alpha=1}^{m} \tilde{\bm S}^{\alpha}_{\mathrm{neq} \: n+1}.
\end{align*}
\item Calculate the projection tensor
\begin{align*}
\mathbb P_{n+1} = \mathbb I - \frac13 \bm C^{-1}_{n+1} \otimes \bm C_{n+1}.
\end{align*}
\item Calculate the isochoric part of the second Piola-Kirchhoff stress
\begin{align*}
\bm S_{\mathrm{iso} \: n+1} = J_{n+1}^{-\frac23} \mathbb P_{n+1} : \tilde{\bm S}_{n+1}.
\end{align*}
\item Calculate the fictitious elasticity tensors for the non-equilibrium part,
\begin{align*}
\tilde{\mathbb C}^{\alpha}_{\mathrm{neq} \: n+1} = \frac{4\beta^{\infty}_{\alpha} J^{-\frac43}_{n+1}}{\mu^{\alpha}} \left(\frac{\partial^3 G^{\infty}_{\mathrm{iso}}}{\partial \tilde{\bm C} \partial \tilde{\bm C} \partial \tilde{\bm C}} \right)_{n+1} : \bm Q^{\alpha}_{n+1} + \frac{\delta_{\alpha} \beta^{\infty}_{\alpha} J^{\frac43}_{n+1} }{2\mu^{\alpha}} \tilde{\mathbb C}^{\infty}_{\mathrm{iso} \: n+1} : \tilde{\mathbb C}^{\infty}_{\mathrm{iso} \: n+1}.
\end{align*}
\item Calculate the fictitious elasticity tensor 
\begin{align*}
\tilde{\mathbb C}_{n+1} = \tilde{\mathbb C}_{\mathrm{iso} \: n+1}^{\infty} + \sum_{\alpha=1}^{m} \tilde{\mathbb C}_{\mathrm{neq} \: n+1}^{\alpha}.
\end{align*}
\item Calculate the fourth-order tensor 
\begin{align*}
\tilde{\mathbb P}_{n+1} = \bm C^{-1}_{n+1} \odot \bm C^{-1}_{n+1} - \frac13 \bm C^{-1}_{n+1} \otimes \bm C^{-1}_{n+1}.
\end{align*}
\item Calculate the isochoric part of the elasticity tensor
\begin{align*}
\mathbb C_{\mathrm{iso} \: n+1} = \mathbb P_{n+1} : \tilde{\mathbb C}_{n+1} : \mathbb P^T_{n+1} + \frac23 \mathrm{Tr}\left( J^{-\frac23}_{n+1} \tilde{\bm S}_{n+1} \right) \tilde{\mathbb P}_{n+1} - \frac23 \left( \bm C^{-1}_{n+1} \otimes \bm S_{\mathrm{iso} \: n+1} + \bm S_{\mathrm{iso} \: n+1} \otimes \bm C^{-1}_{n+1} \right).
\end{align*}
\end{enumerate}

\bibliographystyle{elsarticle-num}
\bibliography{viscoelaticity-stableFE.bib}

\end{document}